%% file: stt-with-uqe.tex
\documentclass[fleqn]{article}
\usepackage{latexsym}
\usepackage{amssymb}
\usepackage{stmaryrd}
\usepackage{graphicx}
\usepackage{color}
\usepackage{url}
\usepackage{phonetic}
\usepackage{amsmath}

\input{stt-with-uqe-def}

\title{{\bf Simple Type Theory with Undefinedness, Quotation, and
    Evaluation}\thanks{This research was supported by NSERC.}}

\author{William M. Farmer\thanks{Address: Department of Computing and
    Software, McMaster University, 1280 Main Street West, Hamilton,
    Ontario L8S 4K1, Canada.  E-mail:
    {\texttt{wmfarmer@mcmaster.ca}.}}}

\date{1 December 2016}

\begin{document}

\maketitle

\begin{abstract}
This paper presents a version of simple type theory called {\qzerouqe}
that is based on {\qzero}, the elegant formulation of Church's type
theory created and extensively studied by Peter~B.~Andrews.
{\qzerouqe} directly formalizes the \emph{traditional approach to
  undefinedness} in which undefined expressions are treated as
legitimate, nondenoting expressions that can be components of
meaningful statements.  {\qzerouqe} is also equipped with a facility
for reasoning about the syntax of expressions based on
\emph{quotation} and \emph{evaluation}.  Quotation is used to refer to
a syntactic value that represents the syntactic structure of an
expression, and evaluation is used to refer to the value of the
expression that a syntactic value represents.  With quotation and
evaluation it is possible to reason in {\qzerouqe} about the interplay
of the syntax and semantics of expressions and, as a result, to
formalize in {\qzerouqe} syntax-based mathematical algorithms.  The
paper gives the syntax and semantics of {\qzerouqe} as well as a proof
system for {\qzerouqe}.  The proof system is shown to be sound for all
formulas and complete for formulas that do not contain evaluations.
The paper also illustrates some applications of {\qzerouqe}.

\bigskip

\noindent
\textbf{Keywords:} Church's type theory, undefinedness, reasoning
about syntax, quotation, evaluation, truth predicates, substitution.

\end{abstract}

\newpage

\tableofcontents
\listoftables

\newpage

\section{Introduction}

A huge portion of mathematical reasoning is performed by
algorithmically manipulating the syntactic structure of mathematical
expressions.  For example, the derivative of a function is commonly
obtained using an algorithm that repeatedly applies syntactic
differentiation rules to an expression that represents the function.
The specification and analysis of a syntax-based mathematical
algorithm requires the ability to reason about the interplay of how
the expressions are manipulated and what the manipulations mean
mathematically.  This is challenging to do in a traditional logic like
first-order logic or simple type theory because there is no mechanism
for directly referring to the syntax of the expressions in the logic.

The standard approach for reasoning in a logic about a language $L$ of
expressions is to introduce another language $L_{\rm syn}$ to
represent the syntax of $L$.  The expressions in $L_{\rm syn}$ denote
certain \emph{syntactic values} (e.g., syntax trees) that represent
the syntactic structures of the expressions in $L$.  We will thus call
$L_{\rm syn}$ a \emph{syntax language}.  A syntax language like
$L_{\rm syn}$ is usually presented as an \emph{inductive type}.  The
members of $L$ are mapped by a \emph{quotation function} to members of
$L_{\rm syn}$, and members of $L_{\rm syn}$ are mapped by an
\emph{evaluation function} to members of $L$.  The language $L_{\rm
  syn}$ provides the means to indirectly reason about the members of
$L$ as syntactic objects, and the quotation and evaluation functions
link this reasoning directly to $L$ itself.  In computer science this
approach is called a \emph{deep embedding}~\cite{BoultonEtAl93}.  The
components of the standard approach --- a syntax language,
quotation function, and evaluation function --- form an instance of a
\emph{syntax framework}~\cite{FarmerLarjani13}, a mathematical
structure that models systems for reasoning about the syntax of a
interpreted language.

We will say that an implementation of the standard approach is
\emph{global} when $L$ is the entire language of the logic and is
\emph{local} otherwise.  For example, the use of G\"odel numbers to
represent the syntactic structure of expressions is usually a global
approach since every expression is assigned a G\"odel number.  We will
also say that an implementation of the standard approach is
\emph{internal} when the quotation and evaluation functions are
expressed as operators in the logic and is \emph{external} when they
are expressed only in the metalogic.  Let the \emph{replete approach}
be the standard approach restricted to implementations that are both
global and internal.  The components of an implementation of the
replete approach form an instance of a \emph{replete syntax
  framework}~\cite{FarmerLarjani13}.

It is a straightforward task to implement the local approach in a
traditional logic, but two significant shortcomings cannot be easily
avoided.  First, the implementation must be external since the
quotation function, and often the evaluation function as well, can
only be expressed in the metalogic, not in the logic itself.  Second,
the constructed syntax framework works only for $L$; another language
(e.g., a larger language that includes $L$) requires a new syntax
framework.  For instance, each time a defined constant is added to
$L$, the syntax language, quotation function, and evaluation function
must all be extended.  See~\cite{Farmer13} for a more detailed
presentation of the local approach.

Implementing the replete approach is much more ambitious: quotation
and evaluation operators are added to the logic and then a syntax
framework is built for the entire language of the logic.  We will
write the quotation and evaluation operators applied to an expression
$e$ as $\synbrack{e}$ and $\sembrack{e}$, respectively.  The replete
approach provides the means to directly reason about the syntax of the
entire language of the logic in the logic itself.  Moreover, the
syntax framework does not have to be extended whenever the language of
the logic is extended, and it can be used to express syntactic side
conditions, schemas, substitution operations, and other such things
directly in the logic.  In short, the replete approach enables
syntax-based reasoning to be moved from the metalogic to the logic
itself.

At first glance, the replete approach appears to solve the whole
problem of how to reason about the interplay of syntax and semantics.
However, the replete approach comes with an entourage of challenging
problems that stand in the way of an effective implementation.  Of
these, we are most concerned with the following two:

\be

  \item \emph{Evaluation Problem.}  Since a replete syntax framework
    works for the entire language of the logic, the evaluation
    operator is applicable to syntax values that represent formulas
    and thus is effectively a truth predicate.  Hence, by the proof of
    Alfred Tarski's theorem on the undefinability of
    truth~\cite{Tarski33,Tarski35,Tarski35a}, if the evaluation
    operator is total in the context of a sufficiently strong theory
    like first-order Peano arithmetic, then it is possible to express
    the liar paradox using the quotation and evaluation operators.
    Therefore, the evaluation operator must be partial and the law of
    disquotation cannot hold universally (i.e., for some expressions
    $e$, $\sembrack{\synbrack{e}} \not= e$).  As a result, reasoning
    with evaluation is cumbersome and leads to undefined expressions.

  \item \emph{Variable Problem.}  The variable $x$ is not free in the
    expression $\synbrack{x + 3}$ (or in any quotation).  However, $x$
    is free in $\sembrack{\synbrack{x + 3}}$ because
    $\sembrack{\synbrack{x + 3}} = x + 3$.  If the value of a constant
    $c$ is $\synbrack{x + 3}$, then $x$ is free in $\sembrack{c}$
    because $\sembrack{c} = \sembrack{\synbrack{x + 3}} = x + 3$.
    Hence, in the presence of an evaluation operator, whether or not a
    variable is free in an expression may depend on the values of the
    expression's components.  As a consequence, the substitution of an
    expression for the free occurrences of a variable in another
    expression depends on the semantics (as well as the syntax) of the
    expressions involved and must be integrated with the proof system
    of the logic.  Hence a logic with quotation and evaluation
    requires a semantics-dependent form of substitution in which side
    conditions, like whether a variable is free in an expression, are
    proved within the proof system.  This is a major departure from
    traditional logic.

\ee
See~\cite{Farmer13} for a more detailed presentation of the replete
approach including discussion of some other problems that come with
it.\footnote{The \emph{replete approach} is called the \emph{global
    approach} in~\cite{Farmer13}.}

There are several implementations of the replete approach in
programming languages.  The most well-known example is the Lisp
programming language with its quote and eval operators.  Other
examples are Agda~\cite{Norell07,Norell09}, Archon~\cite{Stump09},
Elixir~\cite{Elixir15}, F\#~\cite{FSharp15},
MetaML~\cite{TahaSheard00}, MetaOCaml~\cite{MetaOCaml11},
reFLect~\cite{GrundyEtAl06}, and Template
Haskell~\cite{SheardJones02}.

Implementations of the replete approach are much rarer in logics.  One
example is a logic called Chiron~\cite{Farmer07a,Farmer07} which is a
derivative of von-Neumann-Bernays-G\"odel ({\nbg}) set theory.  It
admits undefined expressions, has a rich type system, and contains the
machinery of a replete syntax framework.  As far as we know, there is
no implementation of the replete approach in simple type theory.
See~\cite{GieseBuchberger07,MelhamEtAl13} for research moving in this
direction.  Such an implementation would require significant changes
to the logic:

\be

  \item A syntax language that represents the set of expressions of
    the logic must be defined in the logic.

  \item The syntax and semantics of the logic must be modified to
    admit quotation and evaluation operators.

  \item The proof system of the logic must be extended to include the
    means to reason about quotation, evaluation, and substitution.

\ee 
Moreover, these changes must provide solutions to the Evaluation and
Variable Problems.

The purpose of this paper is to demonstrate how the replete approach
can be implemented in Church's type theory~\cite{Church40}, a version
of simple type theory with lambda-notation introduced by Alonzo Church
in 1940.  We start with {\qzero}, an especially elegant version of
Church's type theory formulated by Peter B. Andrews and meticulously
described and analyzed in~\cite{Andrews02}.  Since evaluation
unavoidably leads to undefined expressions, we modify {\qzero} so that
it formalizes the traditional approach to
undefinedness~\cite{Farmer04}.  This version of {\qzero} with
undefined expressions called {\qzerou} is presented
in~\cite{Farmer08a}.  ({\qzerou} is a simplified version of
{\lutins}~\cite{Farmer90,Farmer93b,Farmer94}, the logic of the the
{\imps} theorem proving system~\cite{FarmerEtAl93,FarmerEtAl96}.)
And, finally, we modify {\qzerou} so that it implements the replete
approach.  This version of {\qzero} with undefined expressions,
quotation, and evaluation called {\qzerouqe} is presented in this
paper.

{\qzerouqe} consists of three principal components: a syntax, a
semantics, and a proof system.  The syntax and semantics of
{\qzerouqe} are relatively straightforward extensions of the syntax
and semantics of {\qzerou}.  However, the proof system of {\qzerouqe}
is significantly more complicated than the proof system of {\qzerou}.
This is because the Variable Problem discussed above necessitates that
the proof system employ a semantics-dependent substitution mechanism.
The proof system of {\qzerouqe} can be used to effectively reason
about quotations and evaluations, but unlike the proof systems of
{\qzero} and {\qzerou} it is not complete.  However, we do show that
it is complete for formulas that do not contain evaluations.

The paper is organized as follows.  The syntax of {\qzerouqe} is
defined in section~\ref{sec:syntax}.  A Henkin-style general models
semantics for {\qzerouqe} is presented in section~\ref{sec:semantics}.
Section~\ref{sec:definitions} introduces several important defined
logical constants and abbreviations.  Section~\ref{sec:syn-frame}
shows that {\qzerouqe} embodies the structure of a replete syntax
framework.  Section~\ref{sec:normal} finishes the specification of the
logical constants of {\qzerouqe} and defines the notion of a normal
general model for {\qzerouqe}. The substitution mechanism for
{\qzerouqe} is presented in section~\ref{sec:substitution}.
Section~\ref{sec:proofsystem} defines {\pfsysuqe}, the proof system of
{\qzerouqe}.  {\pfsysuqe} is proved in section~\ref{sec:soundness} to
be sound with respect to the semantics of {\qzerouqe}.  Several
metatheorems of {\pfsysuqe} are proved in
section~\ref{sec:metatheorems}.  {\pfsysuqe} is proved in
section~\ref{sec:completeness} to be complete with respect to the
semantics of {\qzerouqe} for evaluation-free formulas.  Some
applications of {\qzerouqe} are illustrated in
section~\ref{sec:applications}.  And the paper ends with some final
remarks in section~\ref{sec:conclusion} including a brief discussion
on related and future work.

The great majority of the definitions for {\qzerouqe} are derived from
those for {\qzero} given in~\cite{Andrews02}.  In fact, many
{\qzerouqe} definitions are exactly the same as the {\qzero}
definitions.  Of these, we repeat only the most important and least
obvious definitions for {\qzero}; for the others the reader is
referred to~\cite{Andrews02}.

\section{Syntax} \label{sec:syntax}

The syntax of {\qzerouqe} includes the syntax of {\qzerou} plus
machinery for reasoning about the syntax of expressions (i.e., wffs in
Andrews' terminology) based on quotation and evaluation.

\subsection{Symbols}

A \emph{type symbol} of {\qzerouqe} is defined inductively as follows:

\be

  \item {\iotaAlt} is a type symbol.

  \item $\omicron$ is a type symbol.

  \item $\epsilon$ is a type symbol.

  \item If $\alpha$ and $\beta$ are type symbols, then $(\alpha\beta)$
  is a type symbol.

  \item If $\alpha$ and $\beta$ are type symbols, then
    $\seq{\alpha\beta}$ is a type symbol.

\ee
Let $\sT$ denote the set of type symbols.  $\alpha,\beta,\gamma,
\ldots$ are syntactic variables ranging over type symbols.  When there
is no loss of meaning, matching pairs of parentheses in type symbols
may be omitted.  We assume that type combination of the form
$(\alpha\beta)$ associates to the left so that a type of the form
$((\alpha\beta)\gamma)$ may be written as $\alpha\beta\gamma$.

\bsp The \emph{primitive symbols} of {\qzerouqe} are the following:

\be

  \item \emph{Improper symbols}: [, ], $\lambda$, \mname{c},
    \mname{q}, \mname{e}.

  \item A denumerable set of \emph{variables} of type $\alpha$ for
    each $\alpha \in \sT$: $f_\alpha$, $g_\alpha$, $h_\alpha$,
    $x_\alpha$, $y_\alpha$, $z_\alpha$, $f^{1}_{\alpha}$,
    $g^{1}_{\alpha}$, $h^{1}_{\alpha}$, $x^{1}_{\alpha}$,
    $y^{1}_{\alpha}$, $z^{1}_{\alpha}$, $\dots$\ .

  \item \emph{Logical constants}: see Table~\ref{tab:log-con}.

  \item An unspecified set of \emph{nonlogical constants}
    of various types.

\ee
The types of variables and constants are indicated by their
subscripts. $\textbf{f}_\alpha, \textbf{g}_\alpha, \textbf{h}_\alpha,
\textbf{x}_\alpha, \textbf{y}_\alpha, \textbf{z}_\alpha,\ldots$ are
syntactic variables ranging over variables of type~$\alpha$. \esp

\begin{table}
\bc
\begin{tabular}{|ll|}
\hline
$\mname{Q}_{((o\alpha)\alpha)}$ 
& for all $\alpha \in \sT$\\
$\iota_{(\alpha(o\alpha))}$ 
& for all $\alpha \in \sT$ with $\alpha \not= o$\\
$\mname{pair}_{((\seq{\alpha\beta}\beta)\alpha)}$
& for all $\alpha,\beta \in \sT$\\
$\mname{var}_{(o\epsilon)}$
&\\
$\mname{con}_{(o\epsilon)}$
&\\
$\mname{app}_{((\epsilon\epsilon)\epsilon)}$
&\\
$\mname{abs}_{((\epsilon\epsilon)\epsilon)}$
&\\
$\mname{cond}_{(((\epsilon\epsilon)\epsilon)\epsilon)}$
&\\
$\mname{quot}_{(\epsilon\epsilon)}$
&\\
$\mname{eval}_{((\epsilon\epsilon)\epsilon)}$
&\\
$\mname{eval-free}_{(o\epsilon)}$
&\\
$\mname{not-free-in}_{((o\epsilon)\epsilon)}$
&\\
$\mname{cleanse}_{(\epsilon\epsilon)}$
&\\
$\mname{sub}_{(((\epsilon\epsilon)\epsilon)\epsilon)}$
&\\
$\mname{wff}^{\alpha}_{(o\epsilon)}$
& for all $\alpha \in \sT$\\
\hline
\end{tabular}
\ec
\caption{Logical Constants}\label{tab:log-con}
\end{table}

\begin{note}[Iota Constants]\em
Only $\iota_{\iotaAltS(o\iotaAltS)}$ is a primitive logical constant
in {\qzero}; each other $\iota_{\alpha(o\alpha)}$ is a nonprimitive
logical constant in {\qzero} defined according to an inductive scheme
presented by Church in~\cite{Church40}
(see~\cite[pp.~233--4]{Andrews02}).  We will see in the next section
that the iota constants have a different semantics in {\qzerouqe} than
in {\qzero}.  As a result, it is not possible to define the iota
constants in {\qzerouqe} as they are defined in {\qzero}, and thus they
must be primitive in {\qzerouqe}.  Notice that $\iota_{o(oo)}$ is not a
primitive logical constant of {\qzerouqe}.  It has been left out because
it serves no useful purpose.  It can be defined as a nonprimitive
logical constant as in~\cite[p.~233]{Andrews02} if
desired.
\end{note}

\subsection{Wffs}

Following Andrews, we will call the expressions of {\qzerouqe}
\emph{well-formed formulas (wffs)}.  We are now ready to define a
\emph{wff of type $\alpha$} ($\textit{wff}_\alpha$) of {\qzerouqe}.
$\textbf{A}_\alpha, \textbf{B}_\alpha, \textbf{C}_\alpha, \ldots$ are
syntactic variables ranging over wffs of type $\alpha$.  A
$\textit{wff}_\alpha$ is defined inductively as follows:

\be

  \item A variable of type $\alpha$ is a $\wff{\alpha}$.

  \item A primitive constant of type $\alpha$ is a $\wff{\alpha}$.

  \item $[\textbf{A}_{\alpha\beta}\textbf{B}_\beta]$ is a
  $\wff{\alpha}$.

  \item $[\lambda\textbf{x}_\beta\textbf{A}_\alpha]$ is a
  $\wff{\alpha\beta}$.

  \item $[\mname{c}\textbf{A}_o\textbf{B}_\alpha\textbf{C}_\alpha]$ is
    a $\wff{\alpha}$.

  \item $[\mname{q}\textbf{A}_\alpha]$ is a $\wff{\epsilon}$.

  \item $[\mname{e}\textbf{A}_\epsilon\textbf{x}_\alpha]$ is a
    $\wff{\alpha}$.

\ee
A wff of the form $[\textbf{A}_{\alpha\beta}\textbf{B}_\beta]$,
$[\lambda\textbf{x}_\beta\textbf{A}_\alpha]$,
$[\mname{c}\textbf{A}_o\textbf{B}_\alpha\textbf{C}_\alpha]$,
$[\mname{q}\textbf{A}_\alpha]$, or
$[\mname{e}\textbf{A}_\epsilon\textbf{x}_\alpha]$ is called a
\emph{function application}, a \emph{function abstraction}, a
\emph{conditional}, a \emph{quotation}, or an \emph{evaluation},
respectively.  A \emph{formula} is a $\wff{o}$.  $\textbf{A}_\alpha$
is \emph{evaluation-free} if each occurrence of an evaluation in
$\textbf{A}_\alpha$ is within a quotation.  When there is no loss of
meaning, matching pairs of square brackets in wffs may be omitted.  We
assume that wff combination of the form
$[\textbf{A}_{\alpha\beta}\textbf{B}_\beta]$ associates to the left so
that a wff
$[[\textbf{C}_{\gamma\beta\alpha}\textbf{A}_\alpha]\textbf{B}_\beta]$
may be written as
$\textbf{C}_{\gamma\beta\alpha}\textbf{A}_\alpha\textbf{B}_\beta$.

The \emph{size} of $\textbf{A}_\alpha$ is the number of variables and
primitive constants occurring in $\textbf{A}_\alpha$.  The
\emph{complexity} of $\textbf{A}_\alpha$ is the ordered pair $(m,n)$
of natural numbers such that $m$ is the number of evaluations
occurring in $\textbf{A}_\alpha$ that are not within a quotation and
$n$ is the size of $\textbf{A}_\alpha$.  Complexity pairs are ordered
lexicographically.  The complexity of an evaluation-free wff is a pair
$(0,n)$ where $n$ is the size of the wff.

\begin{note}[Type $\epsilon$]\label{note:new-syntax}\em
The type $\epsilon$ denotes an inductively defined set $\sD_\epsilon$
of values called \emph{constructions} that represent the syntactic
structures of wffs.  The constants
$\mname{app}_{\epsilon\epsilon\epsilon}$,
$\mname{abs}_{\epsilon\epsilon\epsilon}$,
$\mname{cond}_{\epsilon\epsilon\epsilon\epsilon}$,
$\mname{quot}_{\epsilon\epsilon}$, and
$\mname{eval}_{\epsilon\epsilon\epsilon}$ are used to build wffs that
denote constructions representing function applications, function
abstractions, conditionals, quotations, and evaluations, respectively.
An alternate approach would be to have a type $\epsilon_\alpha$ of
constructions that represent the syntactic structures of
$\wffs{\alpha}$ for each $\alpha \in \sT$.
\end{note}

\vspace{-2ex}

\begin{note}[Type $(\alpha\beta)$]\label{note:type-functions}\em
A type $(\alpha\beta)$ denotes a set of partial and total functions
from values of $\alpha$ to values of type $\beta$.  $\beta \tarrow
\alpha$ is an alternate notation for $(\alpha\beta)$.
\end{note}

\vspace{-2ex}

\begin{note}[Type $\seq{\alpha\beta}$]\label{note:type-pairs}\em
A type $\seq{\alpha\beta}$ denotes the set of ordered pairs
$\seq{a,b}$ where $a$ is a value of type $\alpha$ and $b$ is a value
of type $\beta$.  $\alpha \times \beta$ is an alternate notation for
$\seq{\alpha\beta}$.  The constant
$\mname{pair}_{\seq{\alpha\beta}\beta\alpha}$ is used to construct
ordered pairs of type $\seq{\alpha\beta}$.
\end{note}

\vspace{-2ex}

\begin{note}[Conditionals]\label{note:cond-wff}\em
We will see that
$[\mname{c}\textbf{A}_o\textbf{B}_\alpha\textbf{C}_\alpha]$ is a
conditional that is not strict with respect to undefinedness.  For
instance, if $\textbf{A}_o$ is true, then
$[\mname{c}\textbf{A}_o\textbf{B}_\alpha\textbf{C}_\alpha]$ denotes
the value of $\textbf{B}_\alpha$ even when $\textbf{C}_\alpha$ is
undefined.  We construct conditionals using a primitive wff
constructor instead of using a primitive or defined constant since
constants always denote functions that are effectively strict with
respect to undefinedness.
\end{note}

\vspace{-2ex}

\begin{note}[Evaluation Syntax]\label{note:eval-syn}\em
The sole purpose of the variable $\textbf{x}_\alpha$ in an evaluation
$[\mname{e}\textbf{A}_\epsilon\textbf{x}_\alpha]$ is to designate the
type of the evaluation.  We will see in the next section that this
evaluation is defined (true if $\alpha = o$) only if
$\textbf{A}_\epsilon$ denotes a construction representing a
$\wff{\alpha}$.  Hence, if $\textbf{A}_\epsilon$ does denote a
construction representing a $\wff{\alpha}$,
$[\mname{e}\textbf{A}_\epsilon\textbf{x}_\beta]$ is undefined (false
if $\alpha = o$) for all $\beta \in \sT$ with $\beta \not= \alpha$.
\end{note}

\section{Semantics} \label{sec:semantics}

The semantics of {\qzerouqe} is obtained by making three principal
changes to the semantics of {\qzerou}: (1) The semantics of the type
$\epsilon$ is defined to be a domain $D_\epsilon$ of values such that,
for each wff $\textbf{A}_\alpha$ of {\qzerouqe}, there is a unique
member of $D_\epsilon$ that represents the syntactic structure of
$\textbf{A}_\alpha$.  (2) The semantics of the type constructor
$\seq{\alpha\beta}$ is defined to be a domain of ordered pairs.  (3)
The valuation function for wffs is extended to include conditionals,
quotations, and evaluations in its domain.

\subsection{Frames}

Let $\sE$ be the function from the set of wffs to the set of
$\wffs{\epsilon}$ defined inductively as follows:

\be

  \item $\sE(\textbf{x}_\alpha) = [\mname{q}\textbf{x}_\alpha]$.

  \item $\sE(\textbf{c}_\alpha) = [\mname{q}\textbf{c}_\alpha]$ {\sglsp}
    where $\textbf{c}_\alpha$ is a primitive constant.

  \item $\sE([\textbf{A}_{\alpha\beta}\textbf{B}_\beta]) =
    [\mname{app}_{\epsilon\epsilon\epsilon} \,
    \sE(\textbf{A}_{\alpha\beta}) \, \sE(\textbf{B}_\beta)]$.

  \item $\sE([\lambda\textbf{x}_\beta\textbf{A}_\alpha]) =
    [\mname{abs}_{\epsilon\epsilon\epsilon} \, \sE(\textbf{x}_\beta)
    \, \sE(\textbf{A}_\alpha)]$.

  \item $\sE([\mname{c}\textbf{A}_o\textbf{B}_\alpha\textbf{C}_\alpha]) =
    [\mname{cond}_{\epsilon\epsilon\epsilon\epsilon} \, 
    \sE(\textbf{A}_o) \, \sE(\textbf{B}_\alpha) \, \sE(\textbf{C}_\alpha)]$. 

  \item $\sE([\mname{q}\textbf{A}_\alpha]) =
    [\mname{quot}_{\epsilon\epsilon} \, \sE(\textbf{A}_\alpha)]$.

  \item $\sE([\mname{e}\textbf{A}_\epsilon\textbf{x}_\alpha]) =
    [\mname{eval}_{\epsilon\epsilon\epsilon} \,
    \sE(\textbf{A}_\epsilon) \, \sE(\textbf{x}_\alpha)]$.

\ee
$\sE$ is obviously an injective, total function whose range is a
proper subset of the set of $\wffs{\epsilon}$.  The $\wff{\epsilon}$
$\sE(\textbf{A}_\alpha)$ represents the syntactic structure of the wff
$\textbf{A}_\alpha$.  The seven kinds of wffs and their syntactic representations
are given in Table~\ref{tab:wffs}.

\begin{table}
\bc
\begin{tabular}{|lll|}
\hline

\textbf{Kind}
& \textbf{Syntax}
& \textbf{Syntactic Representation}\\

Variable
& $\textbf{x}_\alpha$
& $[\mname{q}\textbf{x}_\alpha]$\\

Primitive constant
& $\textbf{c}_\alpha$
& $[\mname{q}\textbf{c}_\alpha]$\\

Function application
& $[\textbf{A}_{\alpha\beta}\textbf{B}_\beta]$
& $[\mname{app}_{\epsilon\epsilon\epsilon} \, \sE(\textbf{A}_{\alpha\beta}) \, \sE(\textbf{B}_\beta)]$\\

Function abstraction
& $[\lambda\textbf{x}_\beta\textbf{A}_\alpha]$
& $[\mname{abs}_{\epsilon\epsilon\epsilon} \, \sE(\textbf{x}_\beta) \, \sE(\textbf{A}_\alpha)]$\\

Conditional  
& $[\mname{c}\textbf{A}_o\textbf{B}_\alpha\textbf{C}_\alpha]$ 
& $[\mname{cond}_{\epsilon\epsilon\epsilon\epsilon} \, 
\sE(\textbf{A}_o) \, \sE(\textbf{B}_\alpha) \, \sE(\textbf{C}_\alpha)]$\\

Quotation
& $[\mname{q}\textbf{A}_\alpha]$
& $[\mname{quot}_{\epsilon\epsilon} \, \sE(\textbf{A}_\alpha)]$\\

Evaluation
& $[\mname{e}\textbf{A}_\epsilon\textbf{x}_\alpha]$
& $[\mname{eval}_{\epsilon\epsilon\epsilon} \, \sE(\textbf{A}_\epsilon) \, \sE(\textbf{x}_\alpha)]$\\

\hline
\end{tabular}
\ec
\caption{Seven Kinds of Wffs}\label{tab:wffs}
\end{table}

A \emph{frame} of {\qzerouqe} is a collection $\set{\sD_\alpha \;|\;
  \alpha \in \sT}$ of nonempty domains such that:

\be

  \item $\sD_o = \set{\mname{T},\mname{F}}$.

  \item $\set{\sE(\textbf{A}_\alpha) \;|\; \textbf{A}_\alpha \mbox{ is
      a wff}} \subseteq \sD_\epsilon$.

  \item For $\alpha, \beta \in \sT$, $\sD_{(\alpha\beta)}$ is some set
    of \emph{total} functions from $\sD_\beta$ to $\sD_\alpha$ if
    $\alpha = o$ and is some set of \emph{partial and total} functions
    from $\sD_\beta$ to $\sD_\alpha$ if $\alpha \not= o$.

  \item For $\alpha, \beta \in \sT$, $\sD_{\seq{\alpha\beta}}$ is the
    set of all ordered pairs $\seq{a,b}$ such that $a \in \sD_\alpha$
    and $b \in \sD_\beta$.

\ee
$\sD_{\iotaAltS}$ is the \emph{domain of individuals}, $\sD_o$ is the
\emph{domain of truth values}, $\sD_\epsilon$ is the \emph{domain of
  constructions}, and, for $\alpha, \beta \in \sT$,
$\sD_{(\alpha\beta)}$ is a \emph{function domain} and
$\sD_{\seq{\alpha\beta}}$ is a \emph{ordered pair domain}.  For all
$\alpha \in \sT$, the \emph{identity relation} on $\sD_\alpha$ is the
total function $q \in \sD_{o\alpha\alpha}$ such that, for all $x,y \in
\sD_\alpha$, $q(x)(y) = \mname{T}$ iff $x = y$.  For all $\alpha \in
\sT$ with $\alpha \not= o$, the \emph{unique member selector} on
$\sD_\alpha$ is the partial function $f \in \sD_{\alpha(o\alpha)}$
such that, for all $s \in \sD_{o\alpha}$, if the predicate $s$
represents a singleton $\set{x} \subseteq \sD_\alpha$, then $f(s) =
x$, and otherwise $f(s)$ is undefined.  For all $\alpha, \beta \in
\sT$, the \emph{pairing function} on $\sD_\alpha$ and $\sD_\beta$ is
the total function $f \in \sD_{\seq{\alpha\beta}\beta\alpha}$ such
that, for all $a \in \sD_\alpha$ and $b \in \sD_\beta$, $f(a)(b) =
\seq{a,b}$, the ordered pair of $a$ and~$b$.

\begin{note}[Function Domains]\em
In a {\qzero} frame a function domain $\sD_{\alpha\beta}$ contains
only total functions, while in a {\qzerouqe} (and {\qzerou}) frame a
function domain $\sD_{o\beta}$ contains only total functions but a
function domain $\sD_{\alpha\beta}$ with $\alpha \not= o$ contains
partial functions as well as total functions.
\end{note}

\subsection{Interpretations}

An \emph{interpretation} $\seq{\set{\sD_\alpha \;|\; \alpha \in \sT},
  \sJ}$ of {\qzerouqe} consists of a frame and an interpretation
function $\sJ$ that maps each primitive constant of {\qzerouqe} of
type $\alpha$ to an element of $\sD_\alpha$ such that:

\be

  \item $\sJ(\mname{Q}_{o\alpha\alpha})$ is the identity relation on
    $\sD_\alpha$ for all $\alpha \in \sT$.

  \item $\sJ(\iota_{\alpha(o\alpha)})$ is the unique member selector
    on $\sD_\alpha$ for all $\alpha \in \sT$ with $\alpha \not= o$.

  \item $\sJ(\mname{pair}_{\seq{\alpha\beta}\beta\alpha})$ is the
    pairing function on $\sD_\alpha$ and $\sD_\beta$ for all
    $\alpha,\beta \in \sT$.

\ee
The other 12 logical constants involving the type $\epsilon$ will be
specified later via axioms in section~\ref{subsec:specs}.

\begin{note}[Definite Description Operators]\em
The $\iota_{\alpha(o\alpha)}$ in {\qzero} are \emph{description
  operators}: if $\textbf{A}_{o\alpha}$ denotes a singleton, then the
value of $\iota_{\alpha(o\alpha)}\textbf{A}_{o\alpha}$ is the unique
member of the singleton, and otherwise the value of
$\iota_{\alpha(o\alpha)}\textbf{A}_{o\alpha}$ is \emph{unspecified}.
In contrast, the $\iota_{\alpha(o\alpha)}$ in {\qzerouqe} (and
{\qzerou}) are \emph{definite description operators}: if
$\textbf{A}_{o\alpha}$ denotes a singleton, then the value of
$\iota_{\alpha(o\alpha)}\textbf{A}_{o\alpha}$ is the unique member of
the singleton, and otherwise the value of
$\iota_{\alpha(o\alpha)}\textbf{A}_{o\alpha}$ is \emph{undefined}.
\end{note}

An \emph{assignment} into a frame $\set{\sD_\alpha \;|\; \alpha \in
  \sT}$ is a function $\phi$ whose domain is the set of variables of
{\qzerouqe} such that, for each variable $\textbf{x}_\alpha$,
$\phi(\textbf{x}_\alpha) \in \sD_\alpha$.  Given an assignment $\phi$,
a variable $\textbf{x}_\alpha$, and $d \in \sD_\alpha$, let
$\phi[\textbf{x}_\alpha \mapsto d]$ be the assignment $\psi$ such that
$\psi(\textbf{x}_\alpha) = d$ and $\psi(\textbf{y}_\beta) =
\phi(\textbf{y}_\beta)$ for all variables $\textbf{y}_\beta \not=
\textbf{x}_\alpha$.  Given an interpretation $\sM =
\seq{\set{\sD_\alpha \;|\; \alpha \in \sT}, \sJ}$,
$\mname{assign}(\sM)$ is the set of assignments into the frame of
$\sM$.

\subsection{General and Evaluation-Free Models}

An interpretation $\sM = \seq{\set{\sD_\alpha \;|\; \alpha \in \sT},
  \sJ}$ is a \emph{general model} for {\qzerouqe} if there is a binary
valuation function $\sV^{\cal M}$ such that, for each assignment $\phi
\in \mname{assign}(\sM)$ and wff $\textbf{D}_\delta$, either
$\sV^{\cal M}_{\phi}(\textbf{D}_\delta) \in \sD_\delta$ or $\sV^{\cal
  M}_{\phi}(\textbf{D}_\delta)$ is undefined and the following
conditions are satisfied for all assignments $\phi \in
\mname{assign}(\sM)$ and all wffs $\textbf{D}_\delta$:

\be

  \item Let $\textbf{D}_\delta$ be a variable of {\qzerouqe}.  Then
    $\sV^{\cal M}_{\phi}(\textbf{D}_\delta) =
    \phi(\textbf{D}_\delta)$.

  \item Let $\textbf{D}_\delta$ be a primitive constant of
    {\qzerouqe}.  Then $\sV^{\cal M}_{\phi}(\textbf{D}_\delta) =
    \sJ(\textbf{D}_\delta)$.

  \item Let $\textbf{D}_\delta$ be $[\textbf{A}_{\alpha\beta}
    \textbf{B}_\beta]$. If $\sV^{\cal
    M}_{\phi}(\textbf{A}_{\alpha\beta})$ is defined, $\sV^{\cal
    M}_{\phi}(\textbf{B}_\beta)$ is defined, and the function
    $\sV^{\cal M}_{\phi}(\textbf{A}_{\alpha\beta})$ is defined at the
    argument $\sV^{\cal M}_{\phi}(\textbf{B}_\beta)$, then
  \[\sV^{\cal M}_{\phi}(\textbf{D}_\delta) = \sV^{\cal
  M}_{\phi}(\textbf{A}_{\alpha\beta})(\sV^{\cal
    M}_{\phi}(\textbf{B}_\beta)),\] the value of the function
  $\sV^{\cal M}_{\phi}(\textbf{A}_{\alpha\beta})$ at the argument
  $\sV^{\cal M}_{\phi}(\textbf{B}_\beta)$.  Otherwise, $\sV^{\cal
    M}_{\phi}(\textbf{D}_\delta) = \mname{F}$ if $\alpha = o$ and
  $\sV^{\cal M}_{\phi}(\textbf{D}_\delta)$ is undefined if $\alpha
  \not= o$.

  \item Let $\textbf{D}_\delta$ be
    $[\lambda\textbf{x}_\beta\textbf{B}_\alpha]$.  Then $\sV^{\cal
    M}_{\phi}(\textbf{D}_\delta)$ is the (partial or total) function
    $f \in \sD_{\alpha\beta}$ such that, for each $d \in \sD_\beta$,
    $f(d) = \sV^{\cal M}_{\phi[{\bf x}_\alpha \mapsto
      d]}(\textbf{B}_\alpha)$ if $\sV^{\cal M}_{\phi[{\bf x}_\alpha
      \mapsto d]}(\textbf{B}_\alpha)$ is defined and $f(d)$ is
    undefined if $\sV^{\cal M}_{\phi[{\bf x}_\alpha \mapsto
        d]}(\textbf{B}_\alpha)$ is undefined.

  \item Let $\textbf{D}_\delta$ be
    $[\mname{c}\textbf{A}_o\textbf{B}_\alpha\textbf{C}_\alpha]$.  If
    $\sV^{\cal M}_{\phi}(\textbf{A}_o) = \mname{T}$ and $\sV^{\cal
    M}_{\phi}(\textbf{B}_\alpha)$ is defined, then $\sV^{\cal
    M}_{\phi}(\textbf{D}_\delta) = \sV^{\cal
    M}_{\phi}(\textbf{B}_\alpha)$. If $\sV^{\cal
    M}_{\phi}(\textbf{A}_o) = \mname{T}$ and $\sV^{\cal
    M}_{\phi}(\textbf{B}_\alpha)$ is undefined, then $\sV^{\cal
    M}_{\phi}(\textbf{D}_\delta)$ is undefined.  If $\sV^{\cal
    M}_{\phi}(\textbf{A}_o) = \mname{F}$ and $\sV^{\cal
    M}_{\phi}(\textbf{C}_\alpha)$ is defined, then $\sV^{\cal
    M}_{\phi}(\textbf{D}_\delta) = \sV^{\cal
    M}_{\phi}(\textbf{C}_\alpha)$. If $\sV^{\cal
    M}_{\phi}(\textbf{A}_o) = \mname{F}$ and $\sV^{\cal
    M}_{\phi}(\textbf{C}_\alpha)$ is undefined, then $\sV^{\cal
    M}_{\phi}(\textbf{D}_\delta)$ is undefined.

  \item Let $\textbf{D}_\delta$ be $[\mname{q}\textbf{A}_\alpha]$.
    Then $\sV^{\cal M}_{\phi}(\textbf{D}_\delta) =
    \sE(\textbf{A}_\alpha)$.

  \item Let $\textbf{D}_\delta$ be
    $[\mname{e}\textbf{A}_\epsilon\textbf{x}_\alpha]$.  If $\sV^{\cal
    M}_{\phi}(\textbf{A}_\epsilon)$ is defined, $\sE^{-1}(\sV^{\cal
    M}_{\phi}(\textbf{A}_\epsilon))$ is an evaluation-free
    $\wff{\alpha}$, and $\sV^{\cal M}_{\phi}(\sE^{-1}(\sV^{\cal
    M}_{\phi}(\textbf{A}_\epsilon)))$ is defined, then
    \[\sV^{\cal M}_{\phi}(\textbf{D}_\delta) = \sV^{\cal
    M}_{\phi}(\sE^{-1}(\sV^{\cal M}_{\phi}(\textbf{A}_\epsilon))).\]
    Otherwise, $\sV^{\cal M}_{\phi}(\textbf{D}_\delta) = \mname{F}$ if
    $\alpha = o$ and $\sV^{\cal M}_{\phi}(\textbf{D}_\delta)$ is
    undefined if $\alpha \not= o$.
        
\ee  

\begin{prop}\label{prop:gen-models-exist}
General models for {\qzerouqe} exist.
\end{prop}

\begin{proof}
It is easy to construct an interpretation $\sM = \seq{\set{\sD_\alpha
    \;|\; \alpha \in \sT}, \sJ}$ that is a general model for
{\qzerouqe} for which $\sD_{\alpha\beta}$ is the set of all
\emph{total} functions from $\sD_\beta$ to $\sD_\alpha$ if $\alpha =
o$ and is the set of all \emph{partial and total} functions from
$\sD_\beta$ to $\sD_\alpha$ if $\alpha \not= o$ for all $\alpha,\beta
\in \sT$.
\end{proof}

\bigskip

An interpretation $\sM = \seq{\set{\sD_\alpha \;|\; \alpha \in \sT},
  \sJ}$ is an \emph{evaluation-free model} for {\qzerouqe} if there is
a binary valuation function $\sV^{\cal M}$ such that, for each
assignment $\phi \in \mname{assign}(\sM)$ and evaluation-free wff
$\textbf{D}_\delta$, either $\sV^{\cal M}_{\phi}(\textbf{D}_\delta)
\in \sD_\delta$ or $\sV^{\cal M}_{\phi}(\textbf{D}_\delta)$ is
undefined and conditions 1--6 above are satisfied for all assignments
$\phi \in \mname{assign}(\sM)$ and all evaluation-free wffs
$\textbf{D}_\delta$.  A general model is also an evaluation-free
model.

\begin{note}[Valuation Function]\em
In {\qzero}, if $\sM$ is a general model, then $\sV^{\cal M}$ is total
and the value of $\sV^{\cal M}$ on a function abstraction is always a
total function.  In {\qzerouqe}, if $\sM$ is a general model, then
$\sV^{\cal M}$ is partial and the value of $\sV^{\cal M}$ on a
function abstraction can be either a partial or a total
function.
\end{note}

\begin{prop}
Let $\sM$ be a general model for {\qzerouqe}.  Then $\sV^{\cal M}$ is
defined on all variables, primitive constants, function applications
of type $o$, function abstractions, conditionals of type $o$,
quotations, and evaluations of type $o$ and is defined on only a
proper subset of function applications of type $\alpha \not= o$, a
proper subset of conditionals of type $\alpha \not= o$, and a proper
subset of evaluations of type $\alpha \not= o$.
\end{prop}

\begin{note}[Traditional Approach]\em
{\qzerouqe} satisfies the three principles of the traditional approach
to undefinedness stated in~\cite{Farmer04}.  Like other traditional
logics, {\qzero} only satisfies the first principle.
\end{note}

\vspace{-2ex}

\begin{note}[Theories of Quotation]\em
The semantics of the quotation operator \mname{q} is based on the
\emph{disquotational theory of quotation}~\cite{Quotation12}.
According to this theory, a quotation of an expression $e$ is an
expression that denotes $e$ itself.  In our definition of a syntax
framework, $[\mname{q}\textbf{A}_\alpha]$ denotes a value that
represents $\textbf{A}_\alpha$ as a syntactic entity.  Andrew Polonsky
presents in~\cite{Polonsky11} a set of axioms for quotation operators
of this kind.  There are several other theories of quotation that have
been proposed~\cite{Quotation12}.
\end{note}

\vspace{-2ex}

\begin{note}[Theories of Truth]\label{note:truth}\em
$[\mname{e}\textbf{A}_\epsilon\textbf{x}_o]$ asserts the truth of the
  formula represented by $\textbf{A}_\epsilon$.  Thus the evaluation
  operator \mname{e} is a \emph{truth predicate}~\cite{Truth13}.  A
  truth predicate is the face of a \emph{theory of truth}: the
  properties of a truth predicate characterize a theory of
  truth~\cite{Leitgeb07}.  What truth is and how it can be formalized
  is a fundamental research area of logic, and avoiding
  inconsistencies derived from the liar paradox and similar statements
  is one of the major research issues in the area
  (see~\cite{Halbach11}).
\end{note}

\vspace{-2ex}

\begin{note}[Evaluation Semantics]\label{note:eval-sem}\em
An evaluation of type $\alpha$ is undefined (false if $\alpha = o$)
whenever its (first) argument represents a non-evaluation-free
$\wff{\alpha}$.  This idea avoids the Evaluation Problem discussed in
section 1.  The origin of this idea is found in Tarski's famous paper
on the concept of
truth~\cite[Theorem~III]{Tarski33,Tarski35,Tarski35a}.
See~\cite{Halbach09} for a different approach for overcoming the
Evaluation Problem in which the argument of an evaluation is
restricted to wffs that only contain positive occurrences of
evaluations.
\end{note}

$\sV^{\cal M}_\phi(\textbf{A}_\alpha) \QuasiEqual \sV^{\cal
  M}_\phi(\textbf{B}_\alpha)$ means either $\sV^{\cal
  M}_\phi(\textbf{A}_\alpha)$ and $\sV^{\cal
  M}_\phi(\textbf{B}_\alpha)$ are both defined and equal or $\sV^{\cal
  M}_\phi(\textbf{A}_\alpha)$ and $\sV^{\cal
  M}_\phi(\textbf{B}_\alpha)$ are both undefined.  Given a set $X$ of
variables, $\textbf{A}_\alpha$ is \emph{independent of $X$ in $\sM$}
if $\sV^{\cal M}_{\phi}(\textbf{A}_\alpha) \QuasiEqual \sV^{\cal
  M}_{\phi'}(\textbf{A}_\alpha)$ for all $\phi, \phi' \in
\mname{assign}(\sM)$ such that $\phi(\textbf{x}_\alpha) =
\phi'(\textbf{x}_\alpha)$ whenever $\textbf{x}_\alpha \not\in X$.
$\textbf{A}_\alpha$ is \emph{semantically closed} if
$\textbf{A}_\alpha$ is independent of $X$ in every general model for
     {\qzerouqe} where $X$ is the set of all variables.  A
     \emph{sentence} is a semantically closed formula.
     $\textbf{A}_\alpha$ is \emph{invariable} if $\sV^{\cal
       M}_\phi(\textbf{A}_\alpha)$ is the same value or undefined for
     every general model $\sM$ for {\qzerouqe} and every $\phi \in
     \mname{assign}(\sM)$.  If $\textbf{A}_\alpha$ is invariable,
     $\textbf{A}_\alpha$ is said to \emph{denote} the value $\sV^{\cal
       M}_\phi(\textbf{A}_\alpha)$ when $\sV^{\cal
       M}_\phi(\textbf{A}_\alpha)$ is defined and to be
     \emph{undefined} otherwise.

\begin{prop} \label{prop:sem-closed}
A wff that contains variables only within a quotation or the second
argument of an evaluation is semantically closed.
\end{prop}

\begin{prop} \label{prop:invariable}
Quotations and tautologous formulas are invariable.
\end{prop}

Let $\sH$ be a set of $\wffs{o}$ and $\sM$ be a general model for
{\qzerouqe}.  $\textbf{A}_o$ is \emph{valid} in $\sM$, written $\sM
\models \textbf{A}_o$, if $\sV^{\cal M}_{\phi}(\textbf{A}_o) =
\mname{T}$ for all assignments $\phi \in \mname{assign}(\sM)$.  $\sM$
is a \emph{general model for $\sH$}, written $\sM \models \sH$, if
$\sM \models \textbf{B}_o$ for all $\textbf{B}_o \in \sH$.  We write
$\sH \models \textbf{A}_o$ to mean $\sM \models \textbf{A}_o$ for
every general model $\sM$ for {\sH}.  We write ${} \models
\textbf{A}_o$ to mean $\emptyset \models \textbf{A}_o$.

Now let $\textbf{A}_o$ be evaluation-free, $\sH$ be a set of
evaluation-free $\wffs{o}$, and $\sM$ be an evaluation-free model for
{\qzerouqe}.  $\textbf{A}_o$ is \emph{valid} in $\sM$, written $\sM
\models \textbf{A}_o$, if $\sV^{\cal M}_{\phi}(\textbf{A}_o) =
\mname{T}$ for all assignments $\phi \in \mname{assign}(\sM)$.  $\sM$
is an \emph{evaluation-free model for $\sH$}, written $\sM \models
\sH$, if $\sM \models \textbf{B}_o$ for all $\textbf{B}_o \in \sH$.
We write $\sH \modelsa \textbf{A}_o$ to mean $\sM \models
\textbf{A}_o$ for every evaluation-free model $\sM$ for {\sH}.  We
write ${} \modelsa \textbf{A}_o$ to mean $\emptyset \modelsa
\textbf{A}_o$.

\begin{note}[Semantically Closed]\em
Andrews shows in~\cite{Andrews02} that {\qzero} is undecidable.  Hence
it is undecidable whether a formula of {\qzero} is valid in all
general models for {\qzero}.  By similar reasoning, it is undecidable
whether a formula of {\qzerouqe} is valid in all general models for
{\qzerouqe}.  This implies that it is undecidable whether a
conditional of the form
$\mname{c}\textbf{A}_o\textbf{c}_\alpha\textbf{x}_\alpha$, where
$\textbf{c}_\alpha$ is a primitive constant, is semantically closed.
(Primitive constants are semantically closed by
Proposition~\ref{prop:sem-closed}.)  Therefore, more generally, it is
undecidable whether a given wff is semantically closed.  See also
Note~\ref{note:syn-closed} in section~\ref{sec:substitution}.
\end{note}

\subsection{Standard Models}

An interpretation $\sM = \seq{\set{\sD_\alpha \;|\; \alpha \in \sT},
  \sJ}$ is a \emph{standard model} for {\qzerouqe} if
$\sD_{\alpha\beta}$ is the set of all \emph{total} functions from
$\sD_\beta$ to $\sD_\alpha$ if $\alpha = o$ and is the set of all
\emph{partial and total} functions from $\sD_\beta$ to $\sD_\alpha$ if
$\alpha \not= o$ for all $\alpha,\beta \in \sT$.

\begin{lem}
A standard model for {\qzerouqe} is also a general model for
{\qzerouqe}.
\end{lem}

\begin{proof}
Let $\sM$ be a standard model for {\qzerouqe}.  It is easy to show
that $\sV^{\cal M}_{\phi}(\textbf{D}_\delta)$ is well defined by
induction on the complexity of $\textbf{D}_\delta$.
\end{proof}

\bigskip

By the proof of Proposition~\ref{prop:gen-models-exist}, standard
models for {\qzerouqe} exist.  A general model for {\qzerouqe} is a
\emph{nonstandard model} for {\qzerouqe} if it is not a standard
model.

\section{Definitions and Abbreviations} \label{sec:definitions}

As Andrews does in~\cite[p.~212]{Andrews02}, we introduce in
Table~\ref{tab:defs} several defined logical constants and
abbreviations.  The former includes constants for true and false, the
propositional connectives, a canonical undefined wff, the projection
functions for pairs, and some predicates for values of type
$\epsilon$.  The latter includes notation for equality, the
propositional connectives, universal and existential quantification,
defined and undefined wffs, quasi-equality, definite description,
conditionals, quotation, and evaluation.

\begin{table}
\bc
\begin{tabular}{|lll|}
\hline

$[\textbf{A}_\alpha = \textbf{B}_\alpha]$ 
& stands for 
& $[\mname{Q}_{o\alpha\alpha}\textbf{A}_\alpha\textbf{B}_\alpha]$.\\

$[\textbf{A}_o \Iff \textbf{B}_o]$ 
& stands for
& $[\mname{Q}_{ooo}\textbf{A}_o\textbf{B}_o]$.\\

$T_o$ 
& stands for
& $[\mname{Q}_{ooo} = \mname{Q}_{ooo}]$.\\

$F_o$ 
& stands for
& $[\lambda x_o T_o] = [\lambda x_o x_o]$.\\

$[\Forall\textbf{x}_\alpha\textbf{A}_o]$ 
& stands for
& $[\lambda y_{\alpha} T_o] = [\lambda \textbf{x}_\alpha\textbf{A}_o]$.\\

$\wedge_{ooo}$ 
& stands for
& $[\lambda x_o \lambda y_o [[\lambda g_{ooo} [g_{ooo} T_o T_o]] = [\lambda g_{ooo} [g_{ooo} x_o y_o]]]]$.\\

$[\textbf{A}_o \Andd \textbf{B}_o]$ 
& stands for
& $[\wedge_{ooo}\textbf{A}_o\textbf{B}_o]$.\\

$\Implies_{ooo}$ 
& stands for
& $[\lambda x_o \lambda y_o [x_o = [x_o \Andd y_o]]]$.\\

$[\textbf{A}_o \Implies \textbf{B}_o]$ 
& stands for
& $[{\Implies_{ooo}}\textbf{A}_o\textbf{B}_o]$.\\

$\NegAlt_{oo}$ 
& stands for
& $[\mname{Q}_{ooo}F_o]$.\\

$[\NegAlt\textbf{A}_o]$ 
& stands for
& $[\NegAlt_{oo}\textbf{A}_o]$.\\

$\vee_{ooo}$ 
& stands for
& $[\lambda x_o \lambda y_o [\NegAlt[[\NegAlt x_o] \Andd [\NegAlt y_o]]]]$.\\

$[\textbf{A}_o \Or \textbf{B}_o]$ 
& stands for
& $[\vee_{ooo}\textbf{A}_o\textbf{B}_o]$.\\

$[\Forsome\textbf{x}_\alpha\textbf{A}_o]$ 
& stands for
& $[\NegAlt[\Forall\textbf{x}_\alpha\NegAlt\textbf{A}_o]]$.\\

$[\Forsome_1\textbf{x}_\alpha\textbf{A}_o]$ 
& stands for
& $[\Forsome \textbf{x}_\alpha [[\lambda\textbf{x}_\alpha\textbf{A}_o] = 
\mname{Q}_{o\alpha\alpha} \textbf{x}_\alpha]]$.\\

$[\textbf{A}_\alpha \not= \textbf{B}_\alpha]$ 
& stands for 
& $[\NegAlt[\textbf{A}_\alpha = \textbf{B}_\alpha]]$.\\

$[{\textbf{A}_\alpha\IsDefApp}]$
& stands for
& $[\textbf{A}_\alpha = \textbf{A}_\alpha]$.\\

$[{\textbf{A}_\alpha\IsUndefApp}]$
& stands for
& $[\NegAlt[{\textbf{A}_\alpha\IsDefApp}]]$.\\

$[\textbf{A}_\alpha \QuasiEqual \textbf{B}_\alpha]$
& stands for
& $[{\textbf{A}_\alpha\IsDefApp} \Or {\textbf{B}_\alpha\IsDefApp}]
\Implies [\textbf{A}_\alpha = \textbf{B}_\alpha]$.\\

$[\Iota\textbf{x}_\alpha\textbf{A}_o]$
& stands for
& $[\iota_{\alpha(o\alpha)}[\lambda\textbf{x}_\alpha\textbf{A}_o]]$
  {\sglsp} where $\alpha \not= o$.\\

$\Undefined_o$ 
& stands for
& $F_o$.\\

$\Undefined_\alpha$ 
& stands for
& $[\Iota x_\alpha [x_\alpha \not= x_\alpha]]$ {\sglsp} where $\alpha \not= o$.\\

$[\If \, \textbf{A}_o \, \textbf{B}_\alpha \, \textbf{C}_\alpha]$
& stands for
& $[\mname{c}\textbf{A}_o\textbf{B}_\alpha\textbf{C}_\alpha]$.\\

$\synbrack{\textbf{A}_\alpha}$ 
& stands for
& $[\mname{q}\textbf{A}_\alpha]$.\\

$\sembrack{\textbf{A}_\epsilon}_\alpha$ 
& stands for
& $[\mname{e}\textbf{A}_\epsilon x_\alpha]$.\\

$\mname{fst}_{(\alpha\seq{\alpha\beta})}$
& stands for
& $\lambda z_{\seq{\alpha\beta}}
\Iota x_\alpha \Forsome y_\beta 
[z_{\seq{\alpha\beta}} = \mname{pair}_{\seq{\alpha\beta}\beta\alpha} \, x_\alpha \, y_\beta]$.\\

$\mname{snd}_{(\beta\seq{\alpha\beta})}$
& stands for
& $\lambda z_{\seq{\alpha\beta}}
\Iota y_\beta \Forsome x_\alpha 
[z_{\seq{\alpha\beta}} = \mname{pair}_{\seq{\alpha\beta}\beta\alpha} \, x_\alpha \, y_\beta]$.\\

$\mname{var}^{\alpha}_{(o\epsilon)}$
& stands for
& $\lambda x_\epsilon 
[\mname{var}_{(o\epsilon)} \, x_\epsilon \Andd
\mname{wff}^{\alpha}_{(o\epsilon)} \, x_\epsilon]$.\\

$\mname{con}^{\alpha}_{(o\epsilon)}$
& stands for
& $\lambda x_\epsilon 
[\mname{con}_{(o\epsilon)} \, x_\epsilon \Andd
\mname{wff}^{\alpha}_{(o\epsilon)} \, x_\epsilon]$.\\

$\mname{eval-free}^{\alpha}_{(o\epsilon)}$
& stands for
& $\lambda x_\epsilon 
[\mname{eval-free}_{o\epsilon} \, x_\epsilon \Andd
\mname{wff}^{\alpha}_{o\epsilon} \, x_\epsilon]$.\\

$\mname{syn-closed}_{(o\epsilon)}$
& stands for
& $\lambda x_\epsilon 
\Forall y_\epsilon 
[\mname{var}_{o\epsilon} \, y_\epsilon \Implies
\mname{not-free-in}_{o\epsilon\epsilon} \, 
y_\epsilon \, x_\epsilon]$.\\

\hline
\end{tabular}
\ec
\caption{Definitions and Abbreviations}\label{tab:defs}
\end{table}

$[\Forsome_1\textbf{x}_\alpha\textbf{A}_o]$ asserts that there is a
unique $\textbf{x}_\alpha$ that satisfies $\textbf{A}_o$.

$[\Iota\textbf{x}_\alpha\textbf{A}_o]$ is called a \emph{definite
  description}.  It denotes the unique $\textbf{x}_\alpha$ that
satisfies $\textbf{A}_o$.  If there is no or more than one such
$\textbf{x}_\alpha$, it is undefined.  Following Bertrand Russell and
Church, Andrews denotes this definite description operator as an
inverted lower case iota (\invertediota).  We represent this operator
by an (inverted) capital iota ($\Iota$).

$[{\textbf{A}_\alpha\IsDefApp}]$ says that $\textbf{A}_\alpha$ is
defined, and similarly, $[{\textbf{A}_\alpha\IsUndefApp}]$ says that
$\textbf{A}_\alpha$ is undefined.  $[\textbf{A}_\alpha \QuasiEqual
  \textbf{B}_\alpha]$ says that $\textbf{A}_\alpha$ and
$\textbf{B}_\alpha$ are \emph{quasi-equal}, i.e., that
$\textbf{A}_\alpha$ and $\textbf{B}_\alpha$ are either both defined
and equal or both undefined.  The defined constant $\Undefined_\alpha$
is a canonical undefined wff of type $\alpha$.

\begin{note}[Definedness Notation]\em
In {\qzero}, $[{\textbf{A}_\alpha\IsDefApp}]$ is always true,
$[{\textbf{A}_\alpha\IsUndefApp}]$ is always false, $[\textbf{A}_\alpha
\QuasiEqual \textbf{B}_\alpha]$ is always equal to $[\textbf{A}_\alpha
= \textbf{B}_\alpha]$, and $\Undefined_\alpha$ denotes an unspecified
value.
\end{note}

\section{Syntax Frameworks} \label{sec:syn-frame}

In this section we will show that {\qzerouqe} with a fixed general
model and assignment is an instance of a replete syntax
framework~\cite{FarmerLarjani13}.  We assume that the reader is
familiar with the definitions in~\cite{FarmerLarjani13}.

Fix a general model $\sM = \seq{\set{\sD_\alpha \;|\; \alpha \in \sT},
  \sJ}$ for {\qzerouqe} and an assignment $\phi \in
\mname{assign}(\sM)$.  Let $\sL$ to be the set of wffs, $\sL_\alpha$
to be the set of $\wffs{\alpha}$, and $\sD =
\bigcup_\alpha\sD_\alpha$.  Choose some value $\bot \not\in \sD$.
Define $\sW^{\cal M}_\phi : \sL \tarrow \sD \cup \set{\bot}$ to be
function such that, for all wffs $\textbf{D}_\delta$, $\sW^{\cal
  M}_\phi (\textbf{D}_\delta) = \sV^{\cal M}_\phi (\textbf{D}_\delta)$
if $\sV^{\cal M}_\phi (\textbf{D}_\delta)$ is defined and $\sW^{\cal
  M}_\phi (\textbf{D}_\delta) = \bot$ otherwise. It is then easy to
prove the following three propositions:

\begin{prop}
$I= (\sL,\sD \cup \set{\bot},\sW^{\cal M}_{\phi})$ is an interpreted
  language.
\end{prop}

\begin{prop} \label{prop:synrep}
$R= (\sD_\epsilon \cup \set{\bot},\sE)$ is a syntax representation of
  $\sL$.
\end{prop}

\begin{prop} \label{prop:synlang}
$(\sL_\epsilon,I)$ is a syntax language for
  $R$.
\end{prop}

We will now define quotation and evaluation functions.  Let $Q : \sL
\tarrow \sL_\epsilon$ be the injective, total function that maps each
wff $\textbf{D}_\delta$ to its quotation
$\synbrack{\textbf{D}_\delta}$.  Let $E: \sL_\epsilon \tarrow \sL$ be
the partial function that maps each $\wff{\epsilon}$
$\textbf{A}_\epsilon$ to $\sembrack{\textbf{A}_\epsilon}_\alpha$ if
$\sV^{\cal M}_\phi(\sembrack{\textbf{A}_\epsilon}_\alpha)$ is defined
for some $\alpha \in \sT$ and is undefined otherwise.  $E$ is well
defined since $\sE^{-1}(\sV^{\cal M}_{\phi}(\textbf{A}_\epsilon))$ is
a wff of a most one type.

\begin{thm}[Replete Syntax Framework] 
$F = (\sD_\epsilon \cup \set{\bot}, \sE, \sL_\epsilon, Q, E)$ is a
  replete syntax framework for $(\sL,I)$.
\end{thm}

\begin{proof}
$F$ is a syntax framework since it satisfies the following conditions:
\be

  \item $R$ is a syntax representation of $\sL$ by
    Proposition~\ref{prop:synrep}.

  \item $(\sL_\epsilon,I)$ is a syntax language for $R$ by
    Proposition~\ref{prop:synlang}.

  \item For all $\synbrack{\textbf{D}_\delta} \in \sL$, \[\sW^{\cal
    M}_\phi (Q(\textbf{D}_\delta)) = \sW^{\cal M}_\phi
    (\synbrack{\textbf{D}_\delta})) = \sV^{\cal M}_\phi
    (\synbrack{\textbf{D}_\delta})) = \sE(\textbf{D}_\delta),\] i.e.,
    the Quotation Axiom holds.

  \item For all $\synbrack{\textbf{A}_\epsilon} \in \sL_\epsilon$,
    \begin{align*}
    &
    \sW^{\cal M}_\phi (E(\textbf{A}_\epsilon)) \\
    &= 
    \sW^{\cal M}_\phi(\sembrack{\textbf{A}_\epsilon}_\alpha)\\
    &=  
    \sV^{\cal M}_\phi(\sembrack{\textbf{A}_\epsilon}_\alpha)\\
    &=  
    V^{\cal M}_\phi(\sE^{-1}(\sV^{\cal M}_{\phi}(\textbf{A}_\epsilon)))\\
    &=  
    \sW^{\cal M}_\phi(\sE^{-1}(\sW^{\cal M}_{\phi}(\textbf{A}_\epsilon)))
    \end{align*}
    if $E(\textbf{A}_\epsilon)$ is defined, i.e., the Evaluation Axiom
    holds.

\ee
Finally, $F$ is replete since $\sL$ is both the object and full
language of $F$ and $F$ has build-in quotation and evaluation.
\end{proof}

\section{Normal Models}\label{sec:normal}

\subsection{Specifications}\label{subsec:specs}

In a general or evaluation-free model, the first three logical
constants are specified as part of the definition of an
interpretation, but the remaining 12 logical constants, which involve
the type $\epsilon$, are not specified.  In this section, each of
these latter logical constants is specified below via a set of
formulas called \emph{specifying axioms}.  Formula schemas are used to
present the specifying axioms.

\bi

  \item[] \textbf{Specification 1 (Quotation)}

  \be

    \item[] $\synbrack{\textbf{A}_\alpha} = \sE(\textbf{A}_\alpha).$

  \ee

  \item[] \textbf{Specification 2 ($\mname{var}_{o\epsilon}$)}

  \be

    \item $\mname{var}_{o\epsilon} \, \synbrack{\textbf{x}_\alpha}.$

    \vspace{1ex}

    \item $\NegAlt[\mname{var}_{o\epsilon} \,
      \synbrack{\textbf{A}_\alpha}]$ {\sglsp} where 
      $\textbf{A}_\alpha$ is not a variable.

  \ee

  \item[] \textbf{Specification 3 ($\mname{con}_{o\epsilon}$)}

  \be

    \item $\mname{con}_{o\epsilon} \, \synbrack{\textbf{c}_\alpha}$
      {\sglsp} where $\textbf{c}_\alpha$ is a primitive constant.

    \vspace{1ex}

    \item $\NegAlt[\mname{con}_{o\epsilon} \,
      \synbrack{\textbf{A}_\alpha}]$ {\sglsp} where 
      $\textbf{A}_\alpha$ is not a primitive constant.

  \ee

  \item[] \textbf{Specification 4 ($\epsilon$)}

  \be

    \item $\NegAlt[\mname{var}_{o\epsilon} \, \textbf{A}_\epsilon \Andd
      \mname{con}_{o\epsilon} \, \textbf{A}_\epsilon]$.

    \vspace{1ex}

    \item $\NegAlt[\mname{var}_{o\epsilon} \, \textbf{A}_\epsilon \Andd
      \textbf{A}_\epsilon = \mname{app}_{\epsilon\epsilon\epsilon} \,
      \textbf{D}_\epsilon \, \textbf{E}_\epsilon].$

    \vspace{1ex}

    \item $\NegAlt[\mname{var}_{o\epsilon} \, \textbf{A}_\epsilon \Andd
      \textbf{A}_\epsilon = \mname{abs}_{\epsilon\epsilon\epsilon} \,
      \textbf{D}_\epsilon \, \textbf{E}_\epsilon].$

    \vspace{1ex}

    \item $\NegAlt[\mname{var}_{o\epsilon} \, \textbf{A}_\epsilon \Andd
      \textbf{A}_\epsilon = \mname{cond}_{\epsilon\epsilon\epsilon\epsilon} \,
      \textbf{D}_\epsilon \, \textbf{E}_\epsilon \, \textbf{F}_\epsilon].$

    \vspace{1ex}

    \item $\NegAlt[\mname{var}_{o\epsilon} \, \textbf{A}_\epsilon \Andd
      \textbf{A}_\epsilon = \mname{quot}_{\epsilon\epsilon} \,
      \textbf{D}_\epsilon].$

    \vspace{1ex}

    \item $\NegAlt[\mname{var}_{o\epsilon} \, \textbf{A}_\epsilon \Andd
      \textbf{A}_\epsilon = \mname{eval}_{\epsilon\epsilon\epsilon} \,
      \textbf{D}_\epsilon \, \textbf{E}_\epsilon].$

    \vspace{1ex}

    \item $\NegAlt[\mname{con}_{o\epsilon} \, \textbf{A}_\epsilon \Andd
      \textbf{A}_\epsilon = \mname{app}_{\epsilon\epsilon\epsilon} \,
      \textbf{D}_\epsilon \, \textbf{E}_\epsilon].$

    \vspace{1ex}

    \item $\NegAlt[\mname{con}_{o\epsilon} \, \textbf{A}_\epsilon \Andd
      \textbf{A}_\epsilon = \mname{abs}_{\epsilon\epsilon\epsilon} \,
      \textbf{D}_\epsilon \, \textbf{E}_\epsilon].$

    \item $\NegAlt[\mname{con}_{o\epsilon} \, \textbf{A}_\epsilon \Andd
      \textbf{A}_\epsilon = \mname{cond}_{\epsilon\epsilon\epsilon\epsilon} \,
      \textbf{D}_\epsilon \, \textbf{E}_\epsilon \, \textbf{F}_\epsilon].$

    \vspace{1ex}

    \vspace{1ex}

    \item $\NegAlt[\mname{con}_{o\epsilon} \, \textbf{A}_\epsilon \Andd
      \textbf{A}_\epsilon = \mname{quot}_{\epsilon\epsilon} \,
      \textbf{D}_\epsilon].$

    \vspace{1ex}

    \item $\NegAlt[\mname{con}_{o\epsilon} \, \textbf{A}_\epsilon \Andd
      \textbf{A}_\epsilon = \mname{eval}_{\epsilon\epsilon\epsilon} \,
      \textbf{D}_\epsilon \, \textbf{E}_\epsilon].$

    \vspace{1ex}

    \item $\mname{app}_{\epsilon\epsilon\epsilon} \, \textbf{A}_\epsilon \,
      \textbf{B}_\epsilon \not= \mname{abs}_{\epsilon\epsilon\epsilon} \,
      \textbf{D}_\epsilon \, \textbf{E}_\epsilon.$

    \item $\mname{app}_{\epsilon\epsilon\epsilon} \, \textbf{A}_\epsilon \,
      \textbf{B}_\epsilon \not= \mname{cond}_{\epsilon\epsilon\epsilon} \,
      \textbf{D}_\epsilon \, \textbf{E}_\epsilon \, \textbf{F}_\epsilon.$

    \vspace{1ex}

    \vspace{1ex}

    \item $\mname{app}_{\epsilon\epsilon\epsilon} \, \textbf{A}_\epsilon \,
      \textbf{B}_\epsilon \not= \mname{quot}_{\epsilon\epsilon} \, \textbf{D}_\epsilon.$

    \vspace{1ex}

    \item $\mname{app}_{\epsilon\epsilon\epsilon} \, \textbf{A}_\epsilon \,
      \textbf{B}_\epsilon \not= \mname{eval}_{\epsilon\epsilon\epsilon} \,
      \textbf{D}_\epsilon \, \textbf{E}_\epsilon.$

    \vspace{1ex}

    \item $\mname{abs}_{\epsilon\epsilon\epsilon} \, \textbf{A}_\epsilon \,
      \textbf{B}_\epsilon \not= \mname{cond}_{\epsilon\epsilon\epsilon\epsilon} \,
      \textbf{D}_\epsilon \, \textbf{E}_\epsilon \, \textbf{F}_\epsilon.$

    \vspace{1ex}

    \item $\mname{abs}_{\epsilon\epsilon\epsilon} \, \textbf{A}_\epsilon \,
      \textbf{B}_\epsilon \not= \mname{quot}_{\epsilon\epsilon} \, \textbf{D}_\epsilon.$

    \vspace{1ex}

    \item $\mname{abs}_{\epsilon\epsilon\epsilon} \, \textbf{A}_\epsilon \,
      \textbf{B}_\epsilon \not= \mname{eval}_{\epsilon\epsilon\epsilon} \,
      \textbf{D}_\epsilon \, \textbf{E}_\epsilon.$

    \vspace{1ex}

    \item $\mname{cond}_{\epsilon\epsilon\epsilon\epsilon} \,
      \textbf{A}_\epsilon \, \textbf{B}_\epsilon \, \textbf{C}_\epsilon\not=
      \mname{quot}_{\epsilon} \, \textbf{D}_\epsilon.$

    \vspace{1ex}

    \item $\mname{cond}_{\epsilon\epsilon\epsilon\epsilon} \,
      \textbf{A}_\epsilon \, \textbf{B}_\epsilon \, \textbf{C}_\epsilon\not=
      \mname{eval}_{\epsilon\epsilon} \, \textbf{D}_\epsilon \, \textbf{E}_\epsilon.$

    \vspace{1ex}

    \item $\mname{quot}_{\epsilon\epsilon} \, \textbf{A}_\epsilon \not=
      \mname{eval}_{\epsilon\epsilon\epsilon} \, \textbf{D}_\epsilon \,
      \textbf{E}_\epsilon.$

    \vspace{1ex}

    \item $\synbrack{\textbf{x}_\alpha} \not=
      \synbrack{\textbf{y}_\beta}$ {\sglsp} where $\textbf{x}_\alpha
      \not= \textbf{y}_\alpha$.

    \vspace{1ex}

    \item $\synbrack{\textbf{c}_\alpha} \not=
      \synbrack{\textbf{d}_\beta}$ {\sglsp} where $\textbf{c}_\alpha$
      and $\textbf{d}_\alpha$ are different primitive constants.

    \vspace{1ex}

    \item $\mname{app}_{\epsilon\epsilon\epsilon} \, \textbf{A}_\epsilon \,
      \textbf{B}_\epsilon = \mname{app}_{\epsilon\epsilon\epsilon} \,
      \textbf{D}_\epsilon \, \textbf{E}_\epsilon \Implies [\textbf{A}_\epsilon = \textbf{D}_\epsilon \Andd
        \textbf{B}_\epsilon = \textbf{E}_\epsilon].$

    \vspace{1ex}

    \item $\mname{abs}_{\epsilon\epsilon\epsilon} \, \textbf{A}_\epsilon \,
      \textbf{B}_\epsilon = \mname{abs}_{\epsilon\epsilon\epsilon} \,
      \textbf{D}_\epsilon \, \textbf{E}_\epsilon \Implies [\textbf{A}_\epsilon = \textbf{D}_\epsilon \Andd
        \textbf{B}_\epsilon = \textbf{E}_\epsilon].$

    \vspace{1ex}

    \item $\mname{cond}_{\epsilon\epsilon\epsilon} \, \textbf{A}_\epsilon \,
      \textbf{B}_\epsilon \, \textbf{C}_\epsilon = \mname{cond}_{\epsilon\epsilon\epsilon} \,
      \textbf{D}_\epsilon \, \textbf{E}_\epsilon \, \textbf{F}_\epsilon \Implies {}$ \\
      $[\textbf{A}_\epsilon = \textbf{D}_\epsilon \Andd
      \textbf{B}_\epsilon = \textbf{E}_\epsilon \Andd
      \textbf{C}_\epsilon = \textbf{F}_\epsilon].$

    \vspace{1ex}

    \item $\mname{quot}_{\epsilon\epsilon} \, \textbf{A}_\epsilon =
      \mname{quot}_{\epsilon\epsilon} \, \textbf{D}_\epsilon \Implies \textbf{A}_\epsilon
      = \textbf{D}_\epsilon.$

    \vspace{1ex}

    \item $\mname{eval}_{\epsilon\epsilon\epsilon} \, \textbf{A}_\epsilon \,
      \textbf{B}_\epsilon = \mname{eval}_{\epsilon\epsilon\epsilon} \,
      \textbf{D}_\epsilon \, \textbf{E}_\epsilon \Implies [\textbf{A}_\epsilon = \textbf{D}_\epsilon \Andd
        \textbf{B}_\epsilon = \textbf{E}_\epsilon].$

    \vspace{1ex}

    \item $[\textbf{A}^{1}_o \Andd \textbf{A}^{2}_o \Andd
        \textbf{A}^{3}_o \Andd \textbf{A}^{4}_o \Andd \textbf{A}^{5}_o
        \Andd \textbf{A}^{6}_o \Andd \textbf{A}^{7}_o] \Implies \Forall x_\epsilon
      [p_{o\epsilon} x_\epsilon]$ {\sglsp} where:

    \be

      \item[] $\textbf{A}^{1}_o$ \sglsp is \sglsp $\Forall
        x_\epsilon[\mname{var}_{o\epsilon} \, x_\epsilon \Implies
          p_{o\epsilon} x_\epsilon]$.

      \item[] $\textbf{A}^{2}_o$ \sglsp is \sglsp $\Forall
        x_\epsilon[\mname{con}_{o\epsilon} \, x_\epsilon \Implies
          p_{o\epsilon} x_\epsilon]$.

      \item[] $\textbf{A}^{3}_o$ \sglsp is \sglsp $\Forall x_\epsilon
        \Forall y_\epsilon [[p_{o\epsilon} x_\epsilon \Andd p_{o\epsilon} y_\epsilon \Andd
        {[\mname{app}_{\epsilon\epsilon\epsilon} \, x_\epsilon \, y_\epsilon]\IsDefApp}] \Implies
        p_{o\epsilon}[\mname{app}_{\epsilon\epsilon\epsilon} \, x_\epsilon \, y_\epsilon]]$.

      \item[] $\textbf{A}^{4}_o$ \sglsp is \sglsp $\Forall x_\epsilon
        \Forall y_\epsilon [[{p_{o\epsilon} x_\epsilon} \Andd p_{o\epsilon} y_\epsilon \Andd 
        {[\mname{abs}_{\epsilon\epsilon\epsilon} \, x_\epsilon \, y_\epsilon]\IsDefApp}] \Implies
        p_{o\epsilon}[\mname{abs}_{\epsilon\epsilon\epsilon} \, x_\epsilon \, y_\epsilon]]$.

      \item[] $\textbf{A}^{5}_o$ \sglsp is \sglsp $\Forall x_\epsilon \Forall y_\epsilon \Forall z_\epsilon 
        [[{p_{o\epsilon} x_\epsilon} \Andd p_{o\epsilon} y_\epsilon \Andd p_{o\epsilon} z_\epsilon \Andd
        {[\mname{cond}_{\epsilon\epsilon\epsilon\epsilon} \, x_\epsilon \, y_\epsilon \, z_\epsilon ]\IsDefApp}] 
        \Implies {} \\ 
        \hspace*{18.8ex} p_{o\epsilon}[\mname{cond}_{\epsilon\epsilon\epsilon\epsilon} \, 
        x_\epsilon \, y_\epsilon \, z_\epsilon]]$.

      \item[] $\textbf{A}^{6}_o$ \sglsp is \sglsp $\Forall x_\epsilon
        [p_{o\epsilon} x_\epsilon \Implies p_{o\epsilon}[\mname{quot}_{\epsilon\epsilon} \,  x_\epsilon]]$.

      \item[] $\textbf{A}^{7}_o$ \sglsp is \sglsp $\Forall x_\epsilon
        \Forall y_\epsilon [[{p_{o\epsilon} x_\epsilon} \Andd {p_{o\epsilon} y_\epsilon} \Andd
        {[\mname{eval}_{\epsilon\epsilon\epsilon} \, x_\epsilon \, y_\epsilon]\IsDefApp}] \Implies
        p_{o\epsilon}[\mname{eval}_{\epsilon\epsilon\epsilon} \, x_\epsilon \, y_\epsilon]]$.

    \ee

  \ee

  \item[] \textbf{Specification 5 ($\mname{eval-free}_{o\epsilon}$)}

  \be

    \item $\mname{var}_{o\epsilon} \, \textbf{A}_\epsilon \Implies
      \mname{eval-free}_{o\epsilon} \, \textbf{A}_\epsilon.$

    \vspace{1ex}

    \item $\mname{con}_{o\epsilon} \, \textbf{A}_\epsilon \Implies
      \mname{eval-free}_{o\epsilon} \, \textbf{A}_\epsilon.$

    \vspace{1ex}

    \item ${[\mname{app}_{\epsilon\epsilon\epsilon} \, \textbf{A}_\epsilon \,
      \textbf{B}_\epsilon] \IsDefApp} \Implies {}$ \\[.5ex]
      \hspace*{4ex}
      $\mname{eval-free}_{o\epsilon} \,
      [\mname{app}_{\epsilon\epsilon\epsilon} \, \textbf{A}_\epsilon \,
      \textbf{B}_\epsilon] \Iff [\mname{eval-free}_{o\epsilon} \,
      \textbf{A}_\epsilon \Andd \mname{eval-free}_{o\epsilon} \,
      \textbf{B}_\epsilon].$

    \vspace{1ex}

   \item ${[\mname{abs}_{\epsilon\epsilon\epsilon} \, \textbf{A}_\epsilon \,
       \textbf{B}_\epsilon] \IsDefApp} \Implies {}$\\[.5ex]
       \hspace*{4ex}
       $\mname{eval-free}_{o\epsilon} \,
       [\mname{abs}_{\epsilon\epsilon\epsilon} \, \textbf{A}_\epsilon \,
       \textbf{B}_\epsilon] \Iff \mname{eval-free}_{o\epsilon} \,
       \textbf{B}_\epsilon.$

    \vspace{1ex}

   \item ${[\mname{cond}_{\epsilon\epsilon\epsilon\epsilon} \, \textbf{A}_\epsilon \,
       \textbf{B}_\epsilon \, \textbf{C}_\epsilon] \IsDefApp} \Implies {}$\\[.5ex]
       \hspace*{4ex}
       $\mname{eval-free}_{o\epsilon} \,
       [\mname{cond}_{\epsilon\epsilon\epsilon\epsilon} \, 
       \textbf{A}_\epsilon \, \textbf{B}_\epsilon \, \textbf{C}_\epsilon] \Iff {}$\\
       \hspace*{4ex}
       $[\mname{eval-free}_{o\epsilon} \, \textbf{A}_\epsilon \Andd 
       \mname{eval-free}_{o\epsilon} \, \textbf{B}_\epsilon \Andd 
       \mname{eval-free}_{o\epsilon} \, \textbf{C}_\epsilon].$

    \vspace{1ex}

   \item ${\textbf{A}_\epsilon \IsDefApp} \Implies \mname{eval-free}_{o\epsilon} \,
     [\mname{quot}_{\epsilon\epsilon} \, \textbf{A}_\epsilon].$

    \vspace{1ex}

   \item $\NegAlt[\mname{eval-free}_{o\epsilon} \,
     [\mname{eval}_{\epsilon\epsilon\epsilon} \, \textbf{A}_\epsilon \,
       \textbf{B}_\epsilon]].$
  
  \ee

  \item[] \textbf{Specification 6 ($\mname{wff}^{\alpha}_{o\epsilon}$)}

  \be

    \item $\mname{wff}^{\alpha}_{o\epsilon} \, \synbrack{\textbf{x}_\alpha}.$

    \vspace{1ex}

    \item $\mname{wff}^{\alpha}_{o\epsilon} \,
      \synbrack{\textbf{c}_\alpha}$ {\sglsp} where $\textbf{c}_\alpha$
      is a primitive constant.

    \vspace{1ex}

    \item $[\mname{wff}^{\alpha\beta}_{o\epsilon} \, \textbf{A}_\epsilon \Andd
      \mname{wff}^{\beta}_{o\epsilon} \, \textbf{B}_\epsilon] \Implies
      \mname{wff}^{\alpha}_{o\epsilon} \,
            [\mname{app}_{\epsilon\epsilon\epsilon} \, \textbf{A}_\epsilon \,
              \textbf{B}_\epsilon].$

    \vspace{1ex}

    \item $[\mname{wff}^{\iota}_{o\epsilon} \, \textbf{A}_\epsilon \Or
      \mname{wff}^{o}_{o\epsilon} \, \textbf{A}_\epsilon \Or
      \mname{wff}^{\epsilon}_{o\epsilon} \, \textbf{A}_\epsilon \Or
      \mname{wff}^{\seq{\alpha\beta}}_{o\epsilon} \, \textbf{A}_\epsilon]
      \Implies [\mname{app}_{\epsilon\epsilon\epsilon} \, \textbf{A}_\epsilon
        \, \textbf{B}_\epsilon] \IsUndefApp.$

    \vspace{1ex}

    \item $[\mname{wff}^{\alpha\beta}_{o\epsilon} \, \textbf{A}_\epsilon \Andd
      \NegAlt[\mname{wff}^{\beta}_{o\epsilon} \, \textbf{B}_\epsilon]] \Implies
      [\mname{app}_{\epsilon\epsilon\epsilon} \, \textbf{A}_\epsilon \,
        \textbf{B}_\epsilon] \IsUndefApp.$

    \vspace{1ex}

    \item $[\mname{var}^{\alpha}_{o\epsilon} \, \textbf{A}_\epsilon \Andd
      \mname{wff}^{\beta}_{o\epsilon} \, \textbf{B}_\epsilon] \Implies
      \mname{wff}^{\beta\alpha}_{o\epsilon} \,
            [\mname{abs}_{\epsilon\epsilon\epsilon} \, \textbf{A}_\epsilon \,
              \textbf{B}_\epsilon].$

    \vspace{1ex}

    \item $\NegAlt[\mname{var}_{o\epsilon} \, \textbf{A}_\epsilon] \Implies
      [\mname{abs}_{\epsilon\epsilon\epsilon} \, \textbf{A}_\epsilon \,
        \textbf{B}_\epsilon] \IsUndefApp.$

    \vspace{1ex}

    \item $[\mname{wff}^{o}_{o\epsilon} \, \textbf{A}_\epsilon \Andd
      \mname{wff}^{\alpha}_{o\epsilon} \, \textbf{B}_\epsilon \Andd
      \mname{wff}^{\alpha}_{o\epsilon} \, \textbf{C}_\epsilon] \Implies
      \mname{wff}^{\alpha}_{o\epsilon} \,
      [\mname{cond}_{\epsilon\epsilon\epsilon\epsilon} \, \textbf{A}_\epsilon \, \textbf{B}_\epsilon \, \textbf{C}_\epsilon].$

    \vspace{1ex}

    \item $[\NegAlt[\mname{wff}^{o}_{o\epsilon} \, \textbf{A}_\epsilon] \Or 
      [\mname{wff}^{\alpha}_{o\epsilon} \, \textbf{B}_\epsilon \Andd \mname{wff}^{\beta}_{o\epsilon} \, \textbf{C}_\epsilon]]
      \Implies
      [\mname{cond}_{\epsilon\epsilon\epsilon\epsilon} \, 
      \textbf{A}_\epsilon \, \textbf{B}_\epsilon \, \textbf{C}_\epsilon] \IsUndefApp$ \\
      where $\alpha \not= \beta$.

    \vspace{1ex}

    \item ${\textbf{A}_\epsilon \IsDefApp} \Implies \mname{wff}^{\epsilon}_{o\epsilon} \,
      [\mname{quot}_{\epsilon\epsilon} \, \textbf{A}_\epsilon].$

    \vspace{1ex}

    \item $[\mname{wff}^{\epsilon}_{o\epsilon} \, \textbf{A}_\epsilon \Andd
      \mname{var}^{\alpha}_{o\epsilon} \, \textbf{B}_\epsilon] \Implies
      \mname{wff}^{\alpha}_{o\epsilon} \,
            [\mname{eval}_{\epsilon\epsilon\epsilon} \, \textbf{A}_\epsilon \,
              \textbf{B}_\epsilon].$

    \vspace{1ex}

    \item $[\NegAlt[\mname{wff}^{\epsilon}_{o\epsilon} \,
        \textbf{A}_\epsilon] \Or \NegAlt[\mname{var}_{o\epsilon} \,
        \textbf{B}_\epsilon]] \Implies [\mname{eval}_{\epsilon\epsilon\epsilon}
      \, \textbf{A}_\epsilon \, \textbf{B}_\epsilon] \IsUndefApp.$

    \vspace{1ex}

    \item $\NegAlt[\mname{wff}^{\alpha}_{o\epsilon} \, \textbf{A}_\epsilon \Andd
      \mname{wff}^{\beta}_{o\epsilon} \, \textbf{A}_\epsilon]$ {\sglsp} where
      $\alpha \not= \beta$.

    \vspace{1ex}

  \ee

  \item[] \textbf{Specification 7 ($\mname{not-free-in}_{o\epsilon\epsilon}$)}

  \be

    \item $\mname{var}_{o\epsilon} \, \textbf{A}_\epsilon \Implies
      \NegAlt[\mname{not-free-in}_{o\epsilon\epsilon} \, \textbf{A}_\epsilon \,
        \textbf{A}_\epsilon].$

    \vspace{1ex}

    \item $[\mname{var}_{o\epsilon} \, \textbf{A}_\epsilon \Andd
      \mname{var}_{o\epsilon} \, \textbf{B}_\epsilon \Andd \textbf{A}_\epsilon \not=
      \textbf{B}_\epsilon] \Implies \mname{not-free-in}_{o\epsilon\epsilon} \,
      \textbf{A}_\epsilon \, \textbf{B}_\epsilon.$

    \vspace{1ex}

    \item $[\mname{var}_{o\epsilon} \, \textbf{A}_\epsilon \Andd
      \mname{con}_{o\epsilon} \, \textbf{B}_\epsilon] \Implies
      \mname{not-free-in}_{o\epsilon\epsilon} \, \textbf{A}_\epsilon \,
      \textbf{B}_\epsilon.$

    \vspace{1ex}

    \item $[\mname{var}_{o\epsilon} \, \textbf{A}_\epsilon \Andd
      [\mname{app}_{\epsilon\epsilon\epsilon} \, \textbf{B}_\epsilon \,
      \textbf{C}_\epsilon] \IsDefApp] \Implies {}$ \\[.5ex] 
      \hspace*{4ex}
      $\mname{not-free-in}_{o\epsilon\epsilon} \,
      \textbf{A}_\epsilon \, [\mname{app}_{\epsilon\epsilon\epsilon} \,
      \textbf{B}_\epsilon \, \textbf{C}_\epsilon] \Iff {}$ \\[.5ex]
      \hspace*{4ex}
      $[\mname{not-free-in}_{o\epsilon\epsilon} \,
      \textbf{A}_\epsilon \,\textbf{B}_\epsilon \Andd
      \mname{not-free-in}_{o\epsilon\epsilon} \, \textbf{A}_\epsilon
      \,\textbf{C}_\epsilon].$

    \vspace{1ex}

    \item ${[\mname{abs}_{\epsilon\epsilon\epsilon} \,
      \textbf{A}_\epsilon \, \textbf{B}_\epsilon] \IsDefApp} \Implies
      \mname{not-free-in}_{o\epsilon\epsilon} \, \textbf{A}_\epsilon
      \, [\mname{abs}_{\epsilon\epsilon\epsilon} \,
        \textbf{A}_\epsilon \, \textbf{B}_\epsilon].$

    \vspace{1ex}

    \item $[\mname{var}_{o\epsilon} \, \textbf{A}_\epsilon \Andd 
      [{\mname{abs}_{\epsilon\epsilon\epsilon} \,
      \textbf{B}_\epsilon \, \textbf{C}_\epsilon] \IsDefApp} \Andd 
      {\textbf{A}_\epsilon \not= \textbf{B}_\epsilon}] \Implies {}$ \\[.5ex]
      \hspace*{4ex}
      $\mname{not-free-in}_{o\epsilon\epsilon} \,
      \textbf{A}_\epsilon \, [\mname{abs}_{\epsilon\epsilon\epsilon} \,
      \textbf{B}_\epsilon \, \textbf{C}_\epsilon] \Iff
      \mname{not-free-in}_{o\epsilon\epsilon} \, \textbf{A}_\epsilon \,
      \textbf{C}_\epsilon.$

    \vspace{1ex}

    \item $[\mname{var}_{o\epsilon} \, \textbf{A}_\epsilon \Andd
      [\mname{cond}_{\epsilon\epsilon\epsilon\epsilon} \, \textbf{D}_\epsilon \,
      \textbf{E}_\epsilon \, \textbf{F}_\epsilon] \IsDefApp] \Implies {}$ \\[.5ex] 
      \hspace*{4ex}
      $\mname{not-free-in}_{o\epsilon\epsilon} \,
      \textbf{A}_\epsilon \, [\mname{cond}_{\epsilon\epsilon\epsilon\epsilon} \,
      \textbf{D}_\epsilon \, \textbf{E}_\epsilon \, \textbf{F}_\epsilon] \Iff {}$ \\[.5ex]
      \hspace*{4ex}
      $[\mname{not-free-in}_{o\epsilon\epsilon} \, \textbf{A}_\epsilon \, \textbf{D}_\epsilon \Andd
      \mname{not-free-in}_{o\epsilon\epsilon} \, \textbf{A}_\epsilon \, \textbf{E}_\epsilon \Andd
      \mname{not-free-in}_{o\epsilon\epsilon} \, \textbf{A}_\epsilon \, \textbf{F}_\epsilon].$

    \vspace{1ex}

    \item $[\mname{var}_{o\epsilon} \, \textbf{A}_\epsilon \Andd \textbf{B}_\epsilon \IsDefApp] \Implies
      \mname{not-free-in}_{o\epsilon\epsilon} \, \textbf{A}_\epsilon \,
            [\mname{quot}_{\epsilon\epsilon} \, \textbf{B}_\epsilon].$

    \vspace{1ex}

    \item $[\mname{var}_{o\epsilon} \, \textbf{A}_\epsilon \Andd
      \mname{var}^{\alpha}_{o\epsilon} \, \textbf{C}_\epsilon \Andd
      [\mname{eval}_{\epsilon\epsilon\epsilon} \, \textbf{B}_\epsilon \,
      \textbf{C}_\epsilon] \IsDefApp] \Implies {}$ \\[.5ex]
      \hspace*{4ex}
      $\mname{not-free-in}_{o\epsilon\epsilon} \,
      \textbf{A}_\epsilon \, [\mname{eval}_{\epsilon\epsilon\epsilon} \,
      \textbf{B}_\epsilon \, \textbf{C}_\epsilon] \Iff {}$\\[.5ex]
      \hspace*{4ex}
      $[\mname{syn-closed}_{o\epsilon} \, \textbf{B}_\epsilon \Andd
      \mname{eval-free}^{\epsilon}_{o\epsilon} \, \textbf{B}_\epsilon \Andd {}$\\[.5ex]
      \hspace*{4ex}
      $\mname{eval-free}^{\alpha}_{o\epsilon} \,
      \sembrack{\textbf{B}_\epsilon}_\epsilon \Andd
      \mname{not-free-in}_{o\epsilon\epsilon} \, \textbf{A}_\epsilon \,
      \sembrack{\textbf{B}_\epsilon}_\epsilon].$

    \vspace{1ex}

    \item $\NegAlt[\mname{var}_{o\epsilon} \, \textbf{A}_\epsilon] \Implies
      \mname{not-free-in}_{o\epsilon\epsilon} \, \textbf{A}_\epsilon \,
      \textbf{B}_\epsilon.$

  \ee

  \item[] \textbf{Specification 8 ($\mname{cleanse}_{\epsilon\epsilon}$)}

  \be

    \item $\mname{var}_{o\epsilon} \, \textbf{A}_\epsilon \Implies
      \mname{cleanse}_{\epsilon\epsilon} \, \textbf{A}_\epsilon = \textbf{A}_\epsilon.$

    \vspace{1ex}

    \item $\mname{con}_{o\epsilon} \, \textbf{A}_\epsilon \Implies
      \mname{cleanse}_{\epsilon\epsilon} \, \textbf{A}_\epsilon = \textbf{A}_\epsilon.$

    \vspace{1ex}

    \item ${[\mname{app}_{\epsilon\epsilon\epsilon} \, \textbf{A}_\epsilon \,
      \textbf{B}_\epsilon] \IsDefApp} \Implies {}$ \\[.5ex]
      \hspace*{4ex}
      $\mname{cleanse}_{\epsilon\epsilon} \,
      [\mname{app}_{\epsilon\epsilon\epsilon} \, \textbf{A}_\epsilon \, \textbf{B}_\epsilon] \QuasiEqual
      \mname{app}_{\epsilon\epsilon\epsilon} \,
      [\mname{cleanse}_{\epsilon\epsilon} \, \textbf{A}_\epsilon] \, 
      [\mname{cleanse}_{\epsilon\epsilon} \, \textbf{B}_\epsilon].$

    \vspace{1ex}

    \item ${[\mname{abs}_{\epsilon\epsilon\epsilon} \, \textbf{A}_\epsilon \,
      \textbf{B}_\epsilon] \IsDefApp} \Implies {}$ \\[.5ex]
      \hspace*{4ex}
      $\mname{cleanse}_{\epsilon\epsilon} \,
      [\mname{abs}_{\epsilon\epsilon\epsilon} \, \textbf{A}_\epsilon \, \textbf{B}_\epsilon] \QuasiEqual
      \mname{abs}_{\epsilon\epsilon\epsilon} \, \textbf{A}_\epsilon \,
      [\mname{cleanse}_{\epsilon\epsilon} \, \textbf{B}_\epsilon].$

    \vspace{1ex}

    \item ${[\mname{cond}_{\epsilon\epsilon\epsilon\epsilon} \, \textbf{A}_\epsilon \,
      \textbf{B}_\epsilon \, \textbf{C}_\epsilon] \IsDefApp} \Implies {}$ \\[.5ex]
      \hspace*{4ex}
      $\mname{cleanse}_{\epsilon\epsilon} \,
      [\mname{cond}_{\epsilon\epsilon\epsilon\epsilon} \, 
      \textbf{A}_\epsilon \, \textbf{B}_\epsilon \, \textbf{C}_\epsilon] 
      \QuasiEqual {}$ \\[.5ex]
      \hspace*{4ex}
      $\mname{cond}_{\epsilon\epsilon\epsilon\epsilon} \,
      [\mname{cleanse}_{\epsilon\epsilon} \, \textbf{A}_\epsilon] \, 
      [\mname{cleanse}_{\epsilon\epsilon} \, \textbf{B}_\epsilon] \,
      [\mname{cleanse}_{\epsilon\epsilon} \, \textbf{C}_\epsilon].$

    \item $\mname{cleanse}_{\epsilon\epsilon} \,
      [\mname{quot}_{\epsilon\epsilon} \, \textbf{A}_\epsilon] \QuasiEqual
      [\mname{quot}_{\epsilon\epsilon} \, \textbf{A}_\epsilon].$

    \vspace{1ex}

    \item $[\mname{var}^{\alpha}_{o\epsilon} \, \textbf{B}_\epsilon \Andd
      [\mname{eval}_{\epsilon\epsilon\epsilon} \, \textbf{A}_\epsilon \,
      \textbf{B}_\epsilon] \IsDefApp] \Implies {}$ \\[.5ex]
      \hspace*{4ex}
      $\mname{cleanse}_{\epsilon\epsilon} \,
      [\mname{eval}_{\epsilon\epsilon\epsilon} \, \textbf{A}_\epsilon \,
      \textbf{B}_\epsilon] \QuasiEqual {}$ \\[.5ex] 
      \hspace*{4ex}
      $\If \, [\mname{syn-closed}_{o\epsilon} \, \textbf{E}_\epsilon \Andd 
      \mname{eval-free}^{\alpha}_{o\epsilon} \, \sembrack{\textbf{E}_\epsilon}_\epsilon] \,
      \sembrack{\textbf{E}_\epsilon}_\epsilon \, \Undefined_\epsilon$ \\[.5ex]
      $\mbox{where } \textbf{E}_\epsilon \mbox{ is }
      [\mname{cleanse}_{\epsilon\epsilon} \, \textbf{A}_\epsilon].$

  \ee

  \item[] \textbf{Specification 9 ($\mname{sub}_{\epsilon\epsilon\epsilon\epsilon}$)}

  \be

    \item $[\mname{wff}^{\alpha}_{o\epsilon} \, \textbf{A}_\epsilon \Andd
      \mname{var}^{\alpha}_{o\epsilon} \, \textbf{B}_\epsilon] \Implies {}$ \\[.5ex]
      \hspace*{4ex}
      $\mname{sub}_{\epsilon\epsilon\epsilon\epsilon} \,
      \textbf{A}_\epsilon \, \textbf{B}_\epsilon \, \textbf{B}_\epsilon =
      \mname{cleanse}_{\epsilon\epsilon} \, \textbf{A}_\epsilon.$

    \vspace{1ex}

    \item $[\mname{wff}^{\alpha}_{o\epsilon} \, \textbf{A}_\epsilon \Andd
      \mname{var}^{\alpha}_{o\epsilon} \, \textbf{B}_\epsilon \Andd
      \mname{var}_{o\epsilon} \, \textbf{C}_\epsilon \Andd \textbf{B}_\epsilon \not=
      \textbf{C}_\epsilon] \Implies {}$ \\[.5ex]
      \hspace*{4ex}
      $\mname{sub}_{\epsilon\epsilon\epsilon\epsilon} \,
      \textbf{A}_\epsilon \, \textbf{B}_\epsilon \, \textbf{C}_\epsilon = \textbf{C}_\epsilon.$

    \vspace{1ex}

    \item $[\mname{wff}^{\alpha}_{o\epsilon} \, \textbf{A}_\epsilon \Andd
      \mname{var}^{\alpha}_{o\epsilon} \, \textbf{B}_\epsilon \Andd
      \mname{con}_{o\epsilon} \, \textbf{C}_\epsilon] \Implies {}$ \\[.5ex]
      \hspace*{4ex}
      $\mname{sub}_{\epsilon\epsilon\epsilon\epsilon} \,
      \textbf{A}_\epsilon \, \textbf{B}_\epsilon \, \textbf{C}_\epsilon = \textbf{C}_\epsilon.$

    \vspace{1ex}

    \item $[\mname{wff}^{\alpha}_{o\epsilon} \, \textbf{A}_\epsilon \Andd
      \mname{var}^{\alpha}_{o\epsilon} \, \textbf{B}_\epsilon \Andd
      [\mname{app}_{\epsilon\epsilon\epsilon} \, \textbf{D}_\epsilon \,
      \textbf{E}_\epsilon] \IsDefApp] \Implies {}$ \\[.5ex]
      \hspace*{4ex}
      $\mname{sub}_{\epsilon\epsilon\epsilon\epsilon} \,  \textbf{A}_\epsilon \, \textbf{B}_\epsilon \,
      [\mname{app}_{\epsilon\epsilon\epsilon} \, \textbf{D}_\epsilon \, \textbf{E}_\epsilon] \QuasiEqual {}$ \\
      \hspace*{4ex}
      $\mname{app}_{\epsilon\epsilon\epsilon} \, 
      [\mname{sub}_{\epsilon\epsilon\epsilon\epsilon} \, 
      \textbf{A}_\epsilon \, \textbf{B}_\epsilon \, \textbf{D}_\epsilon] \,
      [\mname{sub}_{\epsilon\epsilon\epsilon\epsilon} \, 
      \textbf{A}_\epsilon \, \textbf{B}_\epsilon \, \textbf{E}_\epsilon].$

    \vspace{1ex}

    \item $[\mname{wff}^{\alpha}_{o\epsilon} \, \textbf{A}_\epsilon \Andd
      \mname{var}^{\alpha}_{o\epsilon} \, \textbf{B}_\epsilon \Andd
      [\mname{abs}_{\epsilon\epsilon\epsilon} \, \textbf{B}_\epsilon \,
      \textbf{E}_\epsilon] \IsDefApp] \Implies {}$ \\[.5ex]
      \hspace*{4ex}
      $\mname{sub}_{\epsilon\epsilon\epsilon\epsilon} \,
      \textbf{A}_\epsilon \, \textbf{B}_\epsilon \,
      [\mname{abs}_{\epsilon\epsilon\epsilon} \, \textbf{B}_\epsilon \,
      \textbf{E}_\epsilon] \QuasiEqual \mname{abs}_{\epsilon\epsilon\epsilon} \, \textbf{B}_\epsilon \, 
      [\mname{cleanse}_{\epsilon\epsilon} \, \textbf{A}_\epsilon].$

    \vspace{1ex}

    \item $[\mname{wff}^{\alpha}_{o\epsilon} \, \textbf{A}_\epsilon \Andd
      \mname{var}^{\alpha}_{o\epsilon} \, \textbf{B}_\epsilon \Andd
      \mname{var}_{o\epsilon} \, \textbf{D}_\epsilon \Andd \textbf{B}_\epsilon \not=
      \textbf{D}_\epsilon \Andd [\mname{abs}_{\epsilon\epsilon\epsilon} \,
      \textbf{D}_\epsilon \, \textbf{E}_\epsilon] \IsDefApp] \Implies {}$ \\[.5ex]
      \hspace*{4ex}
      $\mname{sub}_{\epsilon\epsilon\epsilon\epsilon} \,
      \textbf{A}_\epsilon \, \textbf{B}_\epsilon \,
      [\mname{abs}_{\epsilon\epsilon\epsilon} \, \textbf{D}_\epsilon \,
      \textbf{E}_\epsilon] \QuasiEqual {}$ \\[.5ex]
      \hspace*{4ex}
      $\If \, [\mname{not-free-in}_{o\epsilon\epsilon} \, \textbf{B}_\epsilon \,
      \textbf{E}_\epsilon \Or \mname{not-free-in}_{o\epsilon\epsilon} \, 
      \textbf{D}_\epsilon \, \textbf{A}_\epsilon]$ \\[.5ex]
      \hspace*{9ex}
      $[\mname{abs}_{\epsilon\epsilon\epsilon} \,
      \textbf{D}_\epsilon \, [\mname{sub}_{\epsilon\epsilon\epsilon\epsilon} \,
      \textbf{A}_\epsilon \, \textbf{B}_\epsilon \, \textbf{E}_\epsilon]]$ \\[.5ex]
      \hspace*{9ex}
      $\Undefined_\epsilon.$

    \vspace{1ex}

    \item $[\mname{wff}^{\alpha}_{o\epsilon} \, \textbf{A}_\epsilon \Andd
      \mname{var}^{\alpha}_{o\epsilon} \, \textbf{B}_\epsilon \Andd
      [\mname{cond}_{\epsilon\epsilon\epsilon\epsilon} \, \textbf{D}_\epsilon \,
      \textbf{E}_\epsilon \, \textbf{F}_\epsilon] \IsDefApp] \Implies {}$ \\[.5ex]
      \hspace*{4ex}
      $\mname{sub}_{\epsilon\epsilon\epsilon\epsilon} \, \textbf{A}_\epsilon \, \textbf{B}_\epsilon \,
      [\mname{cond}_{\epsilon\epsilon\epsilon\epsilon} \, \textbf{D}_\epsilon \, \textbf{E}_\epsilon \, \textbf{F}_\epsilon] 
      \QuasiEqual {}$ \\[.5ex]
      \hspace*{4ex}
      $\mname{cond}_{\epsilon\epsilon\epsilon\epsilon} \, 
      [\mname{sub}_{\epsilon\epsilon\epsilon\epsilon} \, 
      \textbf{A}_\epsilon \, \textbf{B}_\epsilon \, \textbf{D}_\epsilon] \,
      [\mname{sub}_{\epsilon\epsilon\epsilon\epsilon} \, 
      \textbf{A}_\epsilon \, \textbf{B}_\epsilon \, \textbf{E}_\epsilon] \,
      [\mname{sub}_{\epsilon\epsilon\epsilon\epsilon} \, 
      \textbf{A}_\epsilon \, \textbf{B}_\epsilon \, \textbf{F}_\epsilon].$

    \vspace{1ex}

    \item $[\mname{wff}^{\alpha}_{o\epsilon} \, \textbf{A}_\epsilon \Andd
      \mname{var}^{\alpha}_{o\epsilon} \, \textbf{B}_\epsilon \Andd \textbf{C}_\epsilon \IsDefApp] \Implies {}$ \\[.5ex]
      \hspace*{4ex}
      $\mname{sub}_{\epsilon\epsilon\epsilon\epsilon} \,
      \textbf{A}_\epsilon \, \textbf{B}_\epsilon \, [\mname{quot}_{\epsilon\epsilon} \,
      \textbf{C}_\epsilon] = \mname{quot}_{\epsilon\epsilon} \, \textbf{C}_\epsilon.$

    \vspace{1ex}

    \item $[\mname{wff}^{\alpha}_{o\epsilon} \, \textbf{A}_\epsilon \Andd
      \mname{var}^{\alpha}_{o\epsilon} \, \textbf{B}_\epsilon \Andd
      \mname{var}^{\beta}_{o\epsilon} \, \textbf{E}_\epsilon \Andd
      [\mname{eval}_{\epsilon\epsilon\epsilon} \, \textbf{D}_\epsilon \,
      \textbf{E}_\epsilon] \IsDefApp] \Implies {}$ \\[.5ex]
      \hspace*{4ex}
      $\mname{sub}_{\epsilon\epsilon\epsilon\epsilon} \,
      \textbf{A}_\epsilon \, \textbf{B}_\epsilon \,
      [\mname{eval}_{\epsilon\epsilon\epsilon} \, \textbf{D}_\epsilon \,
      \textbf{E}_\epsilon] \QuasiEqual {}$ \\[.5ex]
      \hspace*{4ex}
      $\If \, [\mname{syn-closed}_{o\epsilon} \, \textbf{E}^{1}_\epsilon
      \Andd \mname{eval-free}^{\beta}_{o\epsilon} \,
      \sembrack{\textbf{E}^{1}_\epsilon}_\epsilon] \,
      \textbf{E}^{2}_\epsilon \, \Undefined_\epsilon$ \\[.5ex]
      where: \\[.5ex]
      \hspace*{4ex}
      $\textbf{E}^{1}_\epsilon$ \sglsp is \sglsp
      $[\mname{sub}_{\epsilon\epsilon\epsilon\epsilon} \, \textbf{A}_\epsilon \,
      \textbf{B}_\epsilon \, \textbf{D}_\epsilon].$ \\[.5ex]
      \hspace*{4ex}
      $\textbf{E}^{2}_\epsilon$ \sglsp is \sglsp
      $[\mname{sub}_{\epsilon\epsilon\epsilon\epsilon} \, \textbf{A}_\epsilon \,
      \textbf{B}_\epsilon \, \sembrack{\textbf{E}^{1}_\epsilon}_\epsilon].$

    \vspace{1ex}

    \item $[\mname{wff}^{\alpha}_{o\epsilon} \, \textbf{A}_\epsilon \Andd
      \NegAlt[\mname{var}^{\alpha}_{o\epsilon} \, \textbf{B}_\epsilon] \Implies {}$ \\[.5ex]
      \hspace*{4ex}
      $[\mname{sub}_{\epsilon\epsilon\epsilon\epsilon} \,
      \textbf{A}_\epsilon \, \textbf{B}_\epsilon \, \textbf{C}_\epsilon] \IsUndefApp.$

  \ee

\ei

\subsection{Normal General and Evaluation-Free Models}

Let $\sS$ be the total set of specifying axioms given above.  A
general model $\sM$ for {\qzerouqe} is \emph{normal} if $\sM \models
\textbf{A}_o$ for all $\textbf{A}_o \in \sS$.  We write $\sH \modelsn
\textbf{A}_o$ to mean $\sM \models \textbf{A}_o$ for every normal
general model $\sM$ for {\sH} where $\sH$ is a set of $\wffs{o}$.  We
write ${} \modelsn \textbf{A}_o$ to mean $\emptyset \modelsn
\textbf{A}_o$.  $\textbf{A}_o$ is \emph{valid} in {\qzerouqe} if ${}
\modelsn \textbf{A}_o$.  

An evaluation-free model $\sM$ for {\qzerouqe} is \emph{normal} if
$\sM \models \textbf{A}_o$ for all evaluation-free $\textbf{A}_o \in
\sS$.  We write $\sH \modelsna \textbf{A}_o$ to mean $\sM \models
\textbf{A}_o$ for every normal evaluation-free model $\sM$ for {\sH}
where $\textbf{A}_o$ is evaluation-free and $\sH$ is a set of
evaluation-free $\wffs{o}$.  We write ${} \modelsna \textbf{A}_o$ to
mean $\emptyset \modelsna \textbf{A}_o$.

Since standard models exist, normal general models (and hence normal
evaluation-free models) exist by Corollary~\ref{cor:nor-stand-models}
given below.

\begin{prop}\label{prop:sem-of-E}\bsp
Let $\sM$ be a normal general model for {\qzerouqe}.  Then $\sV^{\cal
  M}_{\phi}(\sE(\textbf{A}_\alpha)) = \sE(\textbf{A}_\alpha)$ for all
$\phi \in \mname{assign}(\sM)$ and $\textbf{A}_\alpha$.\esp
\end{prop}

\begin{proof}
Immediate from the Specification 1 and the semantics of quotation.
\end{proof}

\begin{note}[Construction Literals]\em
The previous proposition says that a wff of the form
$\sE(\textbf{A}_\alpha)$ denotes itself.  Thus each image of $\sE$ is
a \emph{literal}: its value is directly represented by its syntax.
Quotation can be viewed as an operation that constructs literals for
syntactic values.  Florian Rabe explores in~\cite{Rabe15} a kind of
quotation that constructs literals for syntactic values.
\end{note}

\begin{note}[Quasiquotation]\em
Quasiquotation is a parameterized form of quotation in which the
parameters serve as holes in a quotation that are filled with the
values of expressions.  It is a very powerful syntactic device for
specifying expressions and defining macros.  Quasiquotation was
introduced by Willard Van Orman Quine in 1940 in the first version of
his book \emph{Mathematical Logic}~\cite{Quine03}.  It has been
extensively employed in the Lisp family of programming
languages~\cite{Bawden99}.\footnote{In Lisp, the standard symbol for
  quasiquotation is the backquote (\texttt{`}) symbol, and thus in
  Lisp, quasiquotation is usually called \emph{backquote}.}  A
\emph{quasiquotation} in {\qzerouqe} is a wff of the form
$\sE(\textbf{A}_\alpha)$ where some of its subwffs have been replaced
by $\wffs{\epsilon}$.  As an example, suppose $\textbf{A}_\alpha$ is
$\wedge_{ooo} F_o T_o$ and so $\sE(\textbf{A}_\alpha)$ is
\[\mname{app}_{\epsilon\epsilon\epsilon} \, 
[\mname{app}_{\epsilon\epsilon\epsilon} \synbrack{\wedge_{ooo}} \, \sE(F_o)] \, \sE(T_o).\]
Then 
\[\mname{app}_{\epsilon\epsilon\epsilon} \, 
[\mname{app}_{\epsilon\epsilon\epsilon} \synbrack{\wedge_{ooo}} \, \textbf{B}_\epsilon] \, \textbf{C}_\epsilon\]
is a quasiquotation that we will write in the more suggestive form
\[\synbrack{\wedge_{ooo}\commabrack{\textbf{B}_\epsilon}\commabrack{\textbf{C}_\epsilon}}.\]
$\commabrack{\textbf{B}_\epsilon}$ and
$\commabrack{\textbf{C}_\epsilon}$ are holes in the quotation
$\synbrack{\textbf{A}_\alpha}$ that are filled with the values of
$\textbf{B}_\epsilon$ and $\textbf{C}_\epsilon$.  For instance, if
$\textbf{B}_\epsilon$ and $\textbf{C}_\epsilon$ are $\synbrack{\textbf{D}_o}$ and
$\synbrack{\textbf{E}_o}$, then 
\[\synbrack{\wedge_{ooo}\commabrack{\textbf{B}_\epsilon}\commabrack{\textbf{C}_\epsilon}} =
\synbrack{\wedge_{ooo}\commabrack{\synbrack{\textbf{D}_o}}\commabrack{\synbrack{\textbf{E}_o}}} =
\synbrack{\wedge_{ooo}\textbf{D}_o\textbf{E}_o}.\]
\end{note}

\begin{lem}\bsp
Let $\sM$ be a standard model,
$\textbf{c}^{1}_{\alpha_1},\ldots,\textbf{c}^{11}_{\alpha_{11}}$ be
the 11 logical constants $\mname{var}_{o\epsilon}$,
$\mname{con}_{o\epsilon}$, $\mname{app}_{\epsilon\epsilon\epsilon}$,
$\mname{abs}_{\epsilon\epsilon\epsilon}$,
$\mname{cond}_{\epsilon\epsilon\epsilon\epsilon}$,
$\mname{quot}_{\epsilon\epsilon}$,
$\mname{eval}_{\epsilon\epsilon\epsilon}$,
$\mname{eval-free}_{o\epsilon}$,
$\mname{not-free-in}_{o\epsilon\epsilon}$,
$\mname{cleanse}_{\epsilon\epsilon}$, and
$\mname{sub}_{\epsilon\epsilon\epsilon\epsilon}$, and
$\textbf{d}^{\alpha}_{\beta}$ be the logical constant
$\mname{wff}^{\alpha}_{o\epsilon}$ for each $\alpha \in \sT$.  Then
there are unique functions $f^1 \in \sD_{\alpha_1}, \ldots, f^{11} \in
\sD_{\alpha_{11}}$ and $g^\alpha \in \sD_\beta$ for each $\alpha \in
\sT$ such that the members of $\sS$ are satisfied when
$\textbf{c}^{1}_{\alpha_1},\ldots,\textbf{c}^{11}_{\alpha_{11}}$ and
$\textbf{d}^{\alpha}_{\beta}$ for each $\alpha \in \sT$ are
interpreted in $\sM$ by $f^1, \ldots, f^{11}$ and $g^\alpha$ for each
$\alpha \in \sT$, respectively.\esp
\end{lem}

\begin{proof} 
Let $\sM = \seq{\set{\sD_\alpha \;|\; \alpha \in \sT}, \sJ}$ be a
standard model for {\qzerouqe}.  Then $\sD_\epsilon =
\set{\sE(\textbf{A}_\alpha) \;|\; \textbf{A}_\alpha \mbox{ is a wff}}$
by the Proposition~\ref{prop:standard-epsilon} stated below.  $f^1$ is
the predicate $p \in \sD_{o\epsilon}$ such that, for all wffs
$\textbf{A}_\alpha$, $p(\sE(\textbf{A}_\alpha)) = \mname{T}$ iff
$\textbf{A}_\alpha$ is a variable.  $f^2$ is the predicate $p \in
\sD_{o\epsilon}$ such that, for all wffs $\textbf{A}_\alpha$,
$p(\sE(\textbf{A}_\alpha)) = \mname{T}$ iff $\textbf{A}_\alpha$ is a
primitive constant.

$f^3$ is the function $f \in \sD_{\epsilon\epsilon\epsilon}$ such
that, for all wffs $\textbf{A}_\alpha$ and $\textbf{B}_\beta$, if
$[\textbf{A}_\alpha\textbf{B}_\beta]$ is a wff, then
$f(\sE(\textbf{A}_\alpha))(\sE(\textbf{B}_\beta))$ is the wff
$[\mname{app}_{\epsilon\epsilon\epsilon} \, \sE(\textbf{A}_\alpha) \,
  \sE(\textbf{B}_\beta)]$, and otherwise
$f(\sE(\textbf{A}_\alpha))(\sE(\textbf{B}_\beta))$ is undefined.
$f^4$ is the function $f \in \sD_{\epsilon\epsilon\epsilon}$ such
that, for all wffs $\textbf{A}_\alpha$ and $\textbf{B}_\beta$, if
$[\lambda\textbf{A}_\alpha\textbf{B}_\beta]$ is a wff, then
$f(\sE(\textbf{A}_\alpha))(\sE(\textbf{B}_\beta))$ is the wff
$[\mname{abs}_{\epsilon\epsilon\epsilon} \, \sE(\textbf{A}_\alpha) \,
  \sE(\textbf{B}_\beta)]$, and otherwise
$f(\sE(\textbf{A}_\alpha))(\sE(\textbf{B}_\beta))$ is undefined.
$f^5$ is the function $f \in \sD_{\epsilon\epsilon\epsilon\epsilon}$
such that, for all wffs $\textbf{A}_o$, $\textbf{B}_\alpha$, and
$\textbf{C}_\alpha$, if $[\mname{c} \textbf{A}_o \textbf{B}_\alpha
  \textbf{C}_\alpha]$ is a wff, then
$f(\sE(\textbf{A}_o))(\sE(\textbf{B}_\alpha))(\sE(\textbf{C}_\alpha))$
is the wff $[\mname{cond}_{\epsilon\epsilon\epsilon\epsilon} \,
  \textbf{A}_o \, \textbf{B}_\alpha \, \textbf{C}_\alpha]$, and
otherwise
$f(\sE(\textbf{A}_o))(\sE(\textbf{B}_\alpha))(\sE(\textbf{C}_\alpha))$
is undefined.  $f^6$ is the function $f \in \sD_{\epsilon\epsilon}$
such that, for all wffs $\textbf{A}_\alpha$,
$f(\sE(\textbf{A}_\alpha))$ is the wff
$[\mname{quot}_{\epsilon\epsilon} \, \sE(\textbf{A}_\alpha)]$.  $f^7$
is the function $f \in \sD_{\epsilon\epsilon\epsilon}$ such that, for
all wffs $\textbf{A}_\alpha$ and $\textbf{B}_\beta$, if
$[\mname{e}\textbf{A}_\alpha\textbf{B}_\beta]$ is a wff, then
$f(\sE(\textbf{A}_\alpha))(\sE(\textbf{B}_\beta))$ is the wff
$[\mname{eval}_{\epsilon\epsilon\epsilon} \, \sE(\textbf{A}_\alpha) \,
  \sE(\textbf{B}_\beta)]$, and otherwise
$f(\sE(\textbf{A}_\alpha))(\sE(\textbf{B}_\beta))$ is undefined.

$f^8$ is the predicate $p \in \sD_{o\epsilon}$ such that, for all wffs
$\textbf{A}_\alpha$, $p(\sE(\textbf{A}_\alpha)) = \mname{T}$ iff
$\textbf{A}_\alpha$ is evaluation-free.  And, for each $\alpha$,
$g^\alpha$ is the predicate $p \in \sD_{o\epsilon}$ such that, for all
wffs $\textbf{A}_\beta$, $p(\sE(\textbf{A}_\beta)) = \mname{T}$ iff
$\beta = \alpha$.  All of these functions above clearly satisfy the
specifying axioms in $\sS$ that pertain to them.

$f^9$ is the unique function constructed by defining
$f^9(\sE(\textbf{A}_\alpha))(\sE(\textbf{B}_\beta))$ for all wffs
$\textbf{A}_\alpha$ and $\textbf{B}_\beta$ by recursion on the
complexity of $\textbf{B}_\beta$ in accordance with Specification 7.
$f^{10}$ and $f^{11}$ are constructed similarly.
\end{proof}

\begin{cor}\label{cor:nor-stand-models}
If $\sM$ is a standard model for {\qzerouqe}, then there is normal
standard model $\sM'$ for {\qzerouqe} having the same frame as $\sM$.
\end{cor}

A normal general model or evaluation-free model is a general or
evaluation-free model $\sM = \seq{\set{\sD_\alpha \;|\; \alpha \in
    \sT}, \sJ}$ in which the structure of the domain $\sD_\epsilon$ is
accessible via the logical constants involving $\epsilon$.  From this
point on, we will only be interested in general and evaluation-free
models that are normal.

\subsection{Nonstandard Constructions}\label{subsec:nonstand-cons}

Let $\sM = \seq{\set{\sD_\alpha \;|\; \alpha \in \sT}, \sJ}$ be a
normal general model and $d \in \sD_\epsilon$.  The construction $d$
is \emph{standard} if $d = \sE(\textbf{A}_\alpha)$ for some wff
$\textbf{A}_\alpha$ and is \emph{nonstandard} if it is not standard.
That is, if $d$ is nonstandard, then $d \in \sD_\epsilon \setminus
\set{\sE(\textbf{A}_\alpha) \;|\; \textbf{A}_\alpha \mbox{ is a
    wff}}$.

One might think that Specification 4.29, the induction principle for
the type~$\epsilon$, would rule out the possibility of nonstandard
constructions in $\sM$.  This is the case only when $\sD_{o\epsilon}$
contains all possible predicates.  Thus the following proposition
holds:

\begin{prop}\label{prop:standard-epsilon}
If $\sM$ is a normal standard model for {\qzerouqe}, then
$\sD_\epsilon = \set{\sE(\textbf{A}_\alpha) \;|\; \textbf{A}_\alpha
  \mbox{ is a wff}}$, i.e., $\sM$ contains no nonstandard
constructions.
\end{prop}

The variables of type $\epsilon$ in the specifying axioms given by
Specifications 1--9 thus range over both standard and nonstandard
constructions in a normal general model with nonstandard
constructions.  We will examine some basic results about having
nonstandard constructions present in a normal general model.

\begin{lem}\label{lem:eval-nonstand}
Let $\sM$ be a normal general model for {\qzerouqe} and $\phi \in
\mname{assign}(\sM)$.  Suppose $\sV^{\cal
  M}_{\phi}(\textbf{A}_\epsilon)$ is a nonstandard construction.  Then
$\sV^{\cal M}_{\phi}(\sembrack{\textbf{A}_\epsilon}_\gamma) =
\mname{F}$ if $\gamma = o$ and $\sV^{\cal
  M}_{\phi}(\sembrack{\textbf{A}_\epsilon}_\gamma)$ is undefined if
$\gamma \not= o$.
\end{lem}

\begin{proof}
Immediate from the semantics of evaluation.
\end{proof}

\begin{lem}\label{lem:stand-constructions}
Let $\sM$ be a normal general model for {\qzerouqe} and $\phi \in
\mname{assign}(\sM)$.  

\be

  \item If $\sV^{\cal M}_{\phi}(\mname{app}_{\epsilon\epsilon\epsilon}
    \, \textbf{x}_\epsilon \, \textbf{y}_\epsilon)$ is defined, then
    $\phi(\textbf{x}_\epsilon)$ and $\phi(\textbf{y}_\epsilon)$ are
    standard constructions iff $\sV^{\cal
      M}_{\phi}(\mname{app}_{\epsilon\epsilon\epsilon} \,
    \textbf{x}_\epsilon \, \textbf{y}_\epsilon)$ is a standard
    construction.

  \item If $\sV^{\cal M}_{\phi}(\mname{abs}_{\epsilon\epsilon\epsilon}
    \, \textbf{x}_\epsilon \, \textbf{y}_\epsilon)$ is defined, then
    $\phi(\textbf{x}_\epsilon)$ and $\phi(\textbf{y}_\epsilon)$ are
    standard constructions iff $\sV^{\cal
      M}_{\phi}(\mname{abs}_{\epsilon\epsilon\epsilon} \,
    \textbf{x}_\epsilon \, \textbf{y}_\epsilon)$ is a standard
    construction.

  \item If $\sV^{\cal
    M}_{\phi}(\mname{cond}_{\epsilon\epsilon\epsilon\epsilon} \,
    \textbf{x}_\epsilon \, \textbf{y}_\epsilon \,
    \textbf{z}_\epsilon)$ is defined, then
    $\phi(\textbf{x}_\epsilon)$, $\phi(\textbf{y}_\epsilon)$, and
    $\phi(\textbf{z}_\epsilon)$ are standard constructions iff
    $\sV^{\cal M}_{\phi}(\mname{app}_{\epsilon\epsilon\epsilon} \,
    \textbf{x}_\epsilon \, \textbf{y}_\epsilon \,
    \textbf{z}_\epsilon)$ is a standard construction.

  \item $\phi(\textbf{x}_\epsilon)$ is a standard construction iff
    $\sV^{\cal M}_{\phi}(\mname{quot}_{\epsilon\epsilon} \,
    \textbf{x}_\epsilon)$ is a standard construction.

  \item If $\sV^{\cal
    M}_{\phi}(\mname{eval}_{\epsilon\epsilon\epsilon} \,
    \textbf{x}_\epsilon \, \textbf{y}_\epsilon)$ is defined, then
    $\phi(\textbf{x}_\epsilon)$ and $\phi(\textbf{y}_\epsilon)$ are
    standard constructions iff $\sV^{\cal
      M}_{\phi}(\mname{eval}_{\epsilon\epsilon\epsilon} \,
    \textbf{x}_\epsilon \, \textbf{y}_\epsilon)$ is a standard
    construction.

\ee
\end{lem}

\begin{proof} 

\medskip

\noindent \textbf{Part 1} {\sglsp} Let $\sV^{\cal
  M}_{\phi}(\mname{app}_{\epsilon\epsilon\epsilon} \,
\textbf{x}_\epsilon \, \textbf{y}_\epsilon)$ be defined.  Assume
$\phi(\textbf{x}_\epsilon)$ and $\phi(\textbf{y}_\epsilon)$ are
standard constructions.  Then $\phi(\textbf{x}_\epsilon) =
\sE(\textbf{A}_{\alpha\beta})$ and $\phi(\textbf{y}_\epsilon) =
\sE(\textbf{B}_{\beta})$ for some wffs $\textbf{A}_{\alpha\beta}$ and
$\textbf{B}_{\beta}$ by Specifications 6.4 and 6.5.  Hence, by the
definition of $\sE$,
\begin{align*}
&
\sV^{\cal M}_{\phi}(\mname{app}_{\epsilon\epsilon\epsilon} \, 
\textbf{x}_\epsilon \, \textbf{y}_\epsilon) \\
&= 
\sV^{\cal M}_{\phi}(\mname{app}_{\epsilon\epsilon\epsilon} \,
\sE(\textbf{A}_{\alpha\beta}) \, \sE(\textbf{B}_{\beta})) \\
&= 
\sV^{\cal M}_{\phi}(\sE(\textbf{A}_{\alpha\beta}\textbf{B}_{\beta})),
\end{align*}
which is clearly a standard construction.

Now assume $\sV^{\cal M}_{\phi}(\mname{app}_{\epsilon\epsilon\epsilon}
\, \textbf{x}_\epsilon \, \textbf{y}_\epsilon)$ is a standard
construction.  Then, by Specifications 4.1--21 and Specifications 6.4
and 6.5,
\begin{align*}
&
\sV^{\cal M}_{\phi}(\mname{app}_{\epsilon\epsilon\epsilon})(\phi(\textbf{x}_\epsilon))(\phi(\textbf{y}_\epsilon)) \\
&=
\sV^{\cal M}_{\phi}(\mname{app}_{\epsilon\epsilon\epsilon} \,
\textbf{x}_\epsilon \, \textbf{y}_\epsilon) \\
&=
\sV^{\cal M}_{\phi}(\mname{app}_{\epsilon\epsilon\epsilon} \,
\sE(\textbf{A}_{\alpha\beta}) \, \sE(\textbf{B}_{\beta})) \\
&=
\sV^{\cal M}_{\phi}(\mname{app}_{\epsilon\epsilon\epsilon})
(\sE(\textbf{A}_{\alpha\beta}))(\sE(\textbf{B}_{\beta}))
\end{align*}
for some wffs $\textbf{A}_{\alpha\beta}$ and $\textbf{B}_{\beta}$.
Hence $\phi(\textbf{x}_\epsilon) = \sE(\textbf{A}_{\alpha\beta})$ and
$\phi(\textbf{y}_\epsilon) = \sE(\textbf{B}_{\beta})$ by Specification
4.24 and are thus standard constructions.

\medskip

\noindent \textbf{Parts 2--5} {\sglsp} Similar to Part 1. 

\end{proof}

\bigskip

\bsp Let $\phi \in \mname{assign}(\sM)$.  Suppose $\sV^{\cal
  M}_{\phi}(\mname{sub}_{\epsilon\epsilon\epsilon\epsilon} \,
\textbf{x}_\epsilon \, \textbf{y}_\epsilon \, \textbf{z}_\epsilon)$ is
a standard construction.  Does this imply that
$\phi(\textbf{x}_\epsilon)$, $\phi(\textbf{y}_\epsilon)$, and
$\phi(\textbf{z}_\epsilon)$ are standard constructions?  The answer is
no: Let $\phi(\textbf{x}_\epsilon) = \sE(\textbf{c}_\alpha)$ for some
constant $\textbf{c}_\alpha$ and $\phi(\textbf{y}_\epsilon) =
\phi(\textbf{z}_\epsilon)$ be a nonstandard construction such that
$\sV^{\cal M}_{\phi}(\mname{var}^{\alpha}_{o\epsilon} \,
\textbf{y}_\epsilon) = \mname{T}$.  Then $\sV^{\cal
  M}_{\phi}(\mname{sub}_{\epsilon\epsilon\epsilon\epsilon} \,
\textbf{x}_\epsilon \, \textbf{y}_\epsilon \, \textbf{z}_\epsilon) =
\sE(\textbf{c}_\alpha)$ by Specifications 3.1, 6.2, 8.2, and 9.1.\esp

However, the following result does hold:

\begin{lem}\label{lem:nonstand-sub}
\bsp Let $\sM$ be a normal general model for {\qzerouqe} and $\phi \in
\mname{assign}(\sM)$.  If $\phi(\textbf{x}_\epsilon)$,
$\phi(\textbf{y}_\epsilon)$, and $\sV^{\cal
  M}_{\phi}(\mname{sub}_{\epsilon\epsilon\epsilon\epsilon} \,
\textbf{x}_\epsilon \, \textbf{y}_\epsilon \, \textbf{z}_\epsilon)$
are standard constructions and $\sV^{\cal
  M}_{\phi}(\mname{eval-free}_{o\epsilon} \, \textbf{z}_\epsilon) =
\mname{T}$, then $\phi(\textbf{z}_\epsilon)$ is a standard
construction.\esp
\end{lem}

\begin{proof}
Let $\sV^{\cal
  M}_{\phi}(\mname{sub}_{\epsilon\epsilon\epsilon\epsilon} \,
\textbf{x}_\epsilon \, \textbf{y}_\epsilon \, \textbf{z}_\epsilon)$ =
$\sE(\textbf{A}_\alpha)$ for some wff $\textbf{A}_\alpha$.  Then the
proof of the lemma is by induction on the size of $\textbf{A}_\alpha$.
\end{proof}

\subsection{Example: Infinite Dependency}\label{subsec:inf-dep}

Having specified the logical constant $\mname{var}_{o\epsilon}$ in
this section, we are now ready to present the following simple, but
very interesting example.

Let $\sM = \seq{\set{\sD_\alpha \;|\; \alpha \in \sT}, \sJ}$ be a
normal general model for {\qzerouqe} with $\sD_\epsilon =
\set{\sE(\textbf{A}_\alpha) \;|\; \textbf{A}_\alpha \mbox{ is a wff}}$
and $\phi \in \mname{assign}(\sM)$.  Let $\textbf{A}_o$ be the simple
formula
\[\Forall x_\epsilon [\mname{var}^{o}_{o\epsilon} \, x_\epsilon \Implies
    \sembrack{x_\epsilon}_o]\] involving evaluation.  If we forget
about evaluation, $\textbf{A}_o$ looks like a semantically close
formula --- which is not the case!  By the semantics of universal
quantification $\sV^{\cal M}_{\phi}(\textbf{A}_o) = \mname{T}$ iff
$\sV^{\cal M}_{\phi[x_\epsilon \mapsto {\cal E}({\bf
      B}_\alpha)]}(\mname{var}^{o}_{o\epsilon} \, x_\epsilon \Implies
\sembrack{x_\epsilon}_o) = \mname{T}$ for every wff
$\textbf{B}_\alpha$.  If $\textbf{B}_\alpha$ is not a variable of type
$o$, then $\sV^{\cal M}_{\phi[x_\epsilon \mapsto {\cal E}({\bf
      B}_\alpha)]}(\mname{var}^{o}_{o\epsilon} \, x_\epsilon) =
\mname{F}$, and so $\sV^{\cal M}_{\phi[x_\epsilon \mapsto {\cal
      E}({\bf B}_\alpha)]}(\mname{var}^{o}_{o\epsilon} \, x_\epsilon
\Implies \sembrack{x_\epsilon}_o) = \mname{T}$.  If
$\textbf{B}_\alpha$ is a variable $\textbf{y}_o$, then
\begin{align*}
& 
\sV^{\cal M}_{\phi[x_\epsilon \mapsto {\cal E}({\bf y}_o)]}
([\mname{var}^{o}_{o\epsilon} \, x_\epsilon \Implies \sembrack{x_\epsilon}_o]) \\
&= 
\sV^{\cal M}_{\phi[x_\epsilon \mapsto {\cal E}({\bf y}_o)]}
(\sembrack{x_\epsilon}_o) \\
&= 
\sV^{\cal M}_{\phi[x_\epsilon \mapsto {\cal E}({\bf y}_o)]}
(\sE^{-1}(\sV^{\cal M}_{\phi[x_\epsilon \mapsto {\cal E}({\bf y}_o)]}
(x_\epsilon))) \\
&= 
\sV^{\cal M}_{\phi[x_\epsilon \mapsto {\cal E}({\bf y}_o)]}
(\sE^{-1}(\sE({\bf y}_o))) \\
&= 
\sV^{\cal M}_{\phi[x_\epsilon \mapsto {\cal E}({\bf y}_o)]}(\textbf{y}_o) \\
&= 
\phi(\textbf{y}_o).
\end{align*}
Hence $\sV^{\cal M}_{\phi}(\textbf{A}_o) = \mname{T}$ iff
$\phi(\textbf{y}_o) = \mname{T}$ for all variables $\textbf{y}_o$ of
type $o$.  Therefore, not only is $\textbf{A}_o$ not semantically
closed, its value in $\sM$ depends on the values assigned to
infinitely many variables.  In contrast, the value of any
evaluation-free wff depends on at most finitely many variables.

\section{Substitution} \label{sec:substitution}

Our next task is to construct a proof system {\pfsysuqe} for
{\qzerouqe} based on the proof system of {\qzerou}.  We need a
mechanism for substituting a wff $\textbf{A}_\alpha$ for a free
variable $\textbf{x}_\alpha$ in another wff $\textbf{B}_\alpha$ so
that we can perform beta-reduction in {\pfsysuqe}.  Beta-reduction
is performed in the proof system of {\qzero} in a purely syntactic way
using the basic properties of lambda-notation stated as Axioms
$4_1$--$4_5$ in~\cite{Andrews02}.  Due to the Variable Problem
discussed in section~1, {\pfsysuqe} requires a semantics-dependent
form of substitution.  There is no easy way of extending or modifying
Axioms $4_1$--$4_5$ to cover all function abstractions that contain
evaluations.  Instead, we will utilize a form of explicit
substitution~\cite{AbadiEtAl91}.  We will also utilize as well the
basic properties of lambda-notation that remain valid in
{\qzerouqe}.

The law of beta-reduction for {\qzerou} is expressed as the
schema \[{\textbf{A}_\alpha \IsDefApp} \Implies [[\lambda
    \textbf{x}_\alpha \textbf{B}_\beta]\textbf{A}_\alpha \QuasiEqual
  \mname{S}^{{\bf x}_\alpha}_{{\bf A}_\alpha} \textbf{B}_\beta]\]
where $\textbf{A}_\alpha$ is free for $\textbf{x}_\alpha$ in
$\textbf{B}_\beta$ and $\mname{S}^{{\bf x}_\alpha}_{{\bf
    A}_\alpha}\textbf{B}_\beta$ is the result of substituting
$\textbf{A}_\alpha$ for each free occurrence of $\textbf{x}_\alpha$ in
$\textbf{B}_\beta$.\footnote{Andrews uses {\sub} (with a dot) instead
  of $\mname{S}$ for substitution in~\cite{Andrews02}.}  The law of
beta-reduction for {\qzerouqe} will be expressed by the schema
\[[{\textbf{A}_\alpha \IsDefApp} \Andd 
   {\mname{sub}_{\epsilon\epsilon\epsilon\epsilon} \,
     \synbrack{\textbf{A}_\alpha} \, \synbrack{\textbf{x}_\alpha} \,
     \synbrack{\textbf{B}_\beta} = \synbrack{\textbf{C}_\beta}}]
\Implies [\lambda \textbf{x}_\alpha \textbf{B}_\beta]\textbf{A}_\alpha
\QuasiEqual \textbf{C}_\beta\] without the syntactic side condition
that $\textbf{A}_\alpha$ is free for $\textbf{x}_\alpha$ in
$\textbf{B}_\beta$ and with the result of the substitution expressed
by the wff $\mname{sub}_{\epsilon\epsilon\epsilon\epsilon} \,
\synbrack{\textbf{A}_\alpha} \, \synbrack{\textbf{x}_\alpha} \,
\synbrack{\textbf{B}_\beta}$.  The logical constant
$\mname{sub}_{\epsilon\epsilon\epsilon\epsilon}$ was specified in the
previous section.  We will prove in this section that the law of
beta-reduction for {\qzerouqe} stated above --- in which substitution
is represented by $\mname{sub}_{\epsilon\epsilon\epsilon\epsilon}$ ---
is valid in {\qzerouqe}.

\subsection{Requirements for $\mname{sub}_{\epsilon\epsilon\epsilon\epsilon}$}\label{subsec:requirements}

The specification of $\mname{sub}_{\epsilon\epsilon\epsilon\epsilon}$
needs to satisfy the following requirements:

\bi

  \item[] \textbf{Requirement 1} {\sglsp} \emph{When
    $\mname{sub}_{\epsilon\epsilon\epsilon\epsilon} \,
    \synbrack{\textbf{A}_\alpha} \, \synbrack{\textbf{x}_\alpha} \,
    \synbrack{\textbf{B}_\beta}$ is defined, its value must represent
    the $\textit{wff}_{\beta}$ that results from substituting
    $\textbf{A}_\alpha$ for each free occurrence of
    $\textbf{x}_\alpha$ in $\textbf{B}_\beta$.}  More precisely, for
    any normal general model $\sM$ for {\qzerouqe}, if \[\sM \models
    [\mname{sub}_{\epsilon\epsilon\epsilon\epsilon} \,
      \synbrack{\textbf{A}_\alpha} \, \synbrack{\textbf{x}_\alpha} \,
      \synbrack{\textbf{B}_\beta}] \IsDefApp,\] then \[\sV^{\cal
      M}_{\phi}(\sembrack{\mname{sub}_{\epsilon\epsilon\epsilon\epsilon}
      \, \synbrack{\textbf{A}_\alpha} \, \synbrack{\textbf{x}_\alpha}
      \, \synbrack{\textbf{B}_\beta}}_\beta) \QuasiEqual \sV^{\cal
      M}_{\phi[{\bf x}_\alpha \mapsto {\cal V}^{\cal M}_{\phi}({\bf
          A}_\alpha)]}(\textbf{B}_\beta)\] must be true for all $\phi
    \in \mname{assign}(\sM)$ such that $\sV^{\cal
      M}_{\phi}(\textbf{A}_\alpha)$ is defined.  Satisfying this
    requirement is straightforward when $\textbf{A}_\alpha$ and
    $\textbf{B}_\beta$ are evaluation-free.  Since the semantics of
    evaluation involves a double application of $\sV^{\cal M}_{\phi}$,
    the specification of
    $\mname{sub}_{\epsilon\epsilon\epsilon\epsilon} \,
    \synbrack{\textbf{A}_\alpha} \, \synbrack{\textbf{x}_\alpha} \,
    \synbrack{\textbf{B}_\beta}$ must include a double substitution
    when $\textbf{B}_\beta$ is an evaluation.

  \item[] \textbf{Requirement 2} {\sglsp} 
    \emph{$\mname{sub}_{\epsilon\epsilon\epsilon\epsilon} \,
      \synbrack{\textbf{A}_\alpha} \, \synbrack{\textbf{x}_\alpha} \,
      \synbrack{\textbf{B}_\beta}$ must be undefined when substitution
      would result in a variable capture.}  To avoid variable capture
    we need to check whether a variable does not occur freely in a
    wff.  We have specified the logical constant
    $\mname{not-free-in}_{o\epsilon\epsilon}$ to do this.

  \item[]\bsp \textbf{Requirement 3} {\sglsp} \emph{When
    $\mname{sub}_{\epsilon\epsilon\epsilon\epsilon} \,
    \synbrack{\textbf{A}_\alpha} \, \synbrack{\textbf{x}_\alpha} \,
    \synbrack{\textbf{B}_\beta}$ is defined, its value must represent
    an evaluation-free $\textit{wff}_{\beta}$.}  Otherwise
    $\sembrack{\mname{sub}_{\epsilon\epsilon\epsilon\epsilon} \,
    \synbrack{\textbf{A}_\alpha} \, \synbrack{\textbf{x}_\alpha} \,
    \synbrack{\textbf{B}_\beta}}_\beta$ will be undefined.  We will
    ``cleanse'' any evaluations that remain after a substitution by
    effectively replacing each wff of the form
    $\synbrack{\sembrack{\textbf{A}_{\epsilon}}_\alpha}$ with
    \[[\If \, [\mname{eval-free}^{\alpha}_{o\epsilon} \, \textbf{A}_\epsilon] \,
      \textbf{A}_\epsilon \, \Undefined_\epsilon].\] We have specified
      the logical constant $\mname{cleanse}_{\epsilon\epsilon}$ to do
      this.\esp

  \item[] \textbf{Requirement 4} {\sglsp} \emph{When
    $\mname{sub}_{\epsilon\epsilon\epsilon\epsilon} \,
    \synbrack{\textbf{A}_\alpha} \, \synbrack{\textbf{x}_\alpha} \,
    \synbrack{\textbf{B}_\beta}$ is defined, its value must be
    semantically closed.}  That is, the variables occurring in
    $\textbf{A}_\alpha$ or $\textbf{B}_\beta$ must not be allowed to
    escape outside of a quotation. To avoid such variable escape when
    a wff of the form
    $\synbrack{\sembrack{\textbf{A}_{\epsilon}}_\alpha}$ is cleansed
    as noted above, we need to enforce that $\textbf{A}_{\epsilon}$ is
    semantically closed.  We have used the defined constant
    $\mname{syn-closed}_{o\epsilon}$ to do this.

  \item[] \bsp \textbf{Requirement 5} {\sglsp} \emph{
    $\mname{sub}_{\epsilon\epsilon\epsilon\epsilon} \,
    \synbrack{\textbf{A}_\alpha} \, \synbrack{\textbf{x}_\alpha} \,
    \synbrack{\textbf{B}_\beta}$ is defined in the cases corresponding
    to when substitution is defined in {\qzero}.}  More precisely,
    $\mname{sub}_{\epsilon\epsilon\epsilon\epsilon} \,
    \synbrack{\textbf{A}_\alpha} \, \synbrack{\textbf{x}_\alpha} \,
    \synbrack{\textbf{B}_\beta}$ is defined whenever
    $\textbf{A}_\alpha$ and $\textbf{B}_\beta$ are evaluation-free,
    $\textbf{A}_\alpha$ is defined, and substituting
    $\textbf{A}_\alpha$ for each free occurrence of
    $\textbf{x}_\alpha$ in $\textbf{B}_\beta$ does not result in a
    variable capture.\esp

\ei

\bsp We will prove a series of lemmas that show (1) the properties
that $\mname{not-free-in}_{o\epsilon\epsilon}$,
$\mname{cleanse}_{\epsilon\epsilon}$, and
$\mname{sub}_{\epsilon\epsilon\epsilon\epsilon}$ have and (2) that
$\mname{sub}_{\epsilon\epsilon\epsilon\epsilon}$ satisfies
Requirements 1--5.\esp

\subsection{Evaluation-Free Wffs}

\begin{prop} [Meaning of $\mname{eval-free}^{\alpha}_{o\epsilon}$] \label{prop:eval-free}
\bsp Let $\sM$ be a normal general model for {\qzerouqe}.  $\sM \models
\mname{eval-free}^{\alpha}_{o\epsilon} \,
\synbrack{\textbf{A}_\alpha}$ iff $\textbf{A}_\alpha$ is
evaluation-free.\esp
\end{prop}

\begin{proof}
Immediate from the specification of $\mname{eval-free}^{\alpha}_{o\epsilon}$.
\end{proof}

\begin{lem} [Evaluation-Free] \label{lem:evaluation-free}
Let $\sM$ be a normal general model for {\qzerouqe} and
$\textbf{A}_\alpha$ and $\textbf{B}_\beta$ be evaluation-free.
\be

  \item \bsp $\mname{not-free-in}_{o\epsilon\epsilon} \,
    \synbrack{\textbf{x}_\alpha} \, \synbrack{\textbf{B}_\beta}$,
    $\mname{syn-closed}_{o\epsilon} \, \synbrack{\textbf{A}_\alpha}$,
    $\mname{cleanse}_{\epsilon\epsilon} \,
    \synbrack{\textbf{B}_\beta}$, and
    $\mname{sub}_{\epsilon\epsilon\epsilon\epsilon} \,
    \synbrack{\textbf{A}_\alpha} \, \synbrack{\textbf{x}_\alpha} \,
    \synbrack{\textbf{B}_\beta}$ are invariable. \esp

  \item If $\sM \models \mname{not-free-in}_{o\epsilon\epsilon} \,
    \synbrack{\textbf{x}_\alpha} \, \synbrack{\textbf{B}_\beta}$, then
    $\textbf{B}_\beta$ is independent of $\set{\textbf{x}_\alpha}$ in
    $\sM$.

  \item If $\sM \models \mname{not-free-in}_{o\epsilon\epsilon} \,
    \synbrack{\textbf{x}_\alpha} \, \synbrack{\textbf{B}_\beta}$, then
    $\sM \models \mname{sub}_{\epsilon\epsilon\epsilon\epsilon} \,
    \synbrack{\textbf{A}_\alpha} \, \synbrack{\textbf{x}_\alpha} \,
    \synbrack{\textbf{B}_\beta} = \synbrack{\textbf{B}_\beta}$.

  \item $\sM \models \mname{cleanse}_{\epsilon\epsilon} \,
    \synbrack{\textbf{B}_\beta} = \synbrack{\textbf{B}_\beta}$.

  \item Either $\sM \models
    \mname{sub}_{\epsilon\epsilon\epsilon\epsilon} \,
    \synbrack{\textbf{A}_\alpha} \, \synbrack{\textbf{x}_\alpha} \,
    \synbrack{\textbf{B}_\beta} = \synbrack{\textbf{C}_\beta}$ for
    some evaluation-free $\textbf{C}_\beta$ or $\sM \models
    [\mname{sub}_{\epsilon\epsilon\epsilon\epsilon} \,
      \synbrack{\textbf{A}_\alpha} \, \synbrack{\textbf{x}_\alpha} \,
      \synbrack{\textbf{B}_\beta}] \IsUndefApp$.

  \item $\sM \models \NegAlt[\mname{not-free-in}_{o\epsilon\epsilon} \,
    \synbrack{\textbf{x}_\alpha} \, \synbrack{\textbf{B}_\beta}]$ for
    at most finitely many variables $\textbf{x}_\alpha$.

\ee
\end{lem}

\begin{proof}
Parts 1--5 follow straightforwardly by induction on the size of
$\synbrack{\textbf{B}_\beta}$.  Part~6 follows from the fact that $\sM
\models \NegAlt[\mname{not-free-in}_{o\epsilon\epsilon} \,
  \synbrack{\textbf{x}_\alpha} \, \synbrack{\textbf{B}_\beta}]$
implies $\synbrack{\textbf{x}_\alpha}$ occurs in
$\synbrack{\textbf{B}_\beta}$.
\end{proof}

\bigskip

\bsp By virtue of Lemma~\ref{lem:evaluation-free} (particularly part
1), several standard definitions of predicate logic that are not
applicable to wffs in general are applicable to evaluation-free wffs.
Let $\textbf{A}_\alpha$, $\textbf{B}_\beta$, and $\textbf{C}_o$ be
evaluation-free wffs.  A variable $\textbf{x}_\alpha$ is \emph{bound
  in $\textbf{B}_\beta$} if $\mname{not-free-in}_{o\epsilon\epsilon}
\, \synbrack{\textbf{x}_\alpha} \, \synbrack{\textbf{B}_\beta}$
denotes $\mname{T}$ and is \emph{free in $\textbf{B}_\beta$} if
$\mname{not-free-in}_{o\epsilon\epsilon} \,
\synbrack{\textbf{x}_\alpha} \, \synbrack{\textbf{B}_\beta}$ denotes
$\mname{F}$.  $\textbf{A}_\alpha$ is \emph{syntactically closed} if
$\mname{syn-closed}_{o\epsilon} \, \synbrack{\textbf{A}_\alpha}$
denotes $\mname{T}$.  A \emph{universal closure} of $\textbf{C}_o$ is
a formula
\[\Forall \textbf{x}^{1}_{\alpha_1} \cdots \Forall \textbf{x}^{n}_{\alpha_n} 
\textbf{C}_o\] such that $\textbf{y}_{\beta}$ is free in
$\textbf{C}_o$ iff $\textbf{y}_{\beta} \in
\set{\textbf{x}^{1}_{\alpha_1}, \ldots,
  \textbf{x}^{n}_{\alpha_n}}$. \esp

\begin{lem} [Universal Closures]\label{lem:uni-close}
Let $\sM$ be a normal general model for {\qzerouqe}, $\textbf{A}_o$ be
an evaluation-free formula, and $\textbf{B}_o$ be a universal closure
of $\textbf{A}_o$.

\be

  \item $\textbf{B}_o$ is syntactically closed.

  \item $\sM \models \textbf{A}_o$ iff $\sM \models \textbf{B}_o$.

\ee
\end{lem}

\begin{proof} Part 1 follows from the definitions of universal closure 
and syntactically closed.  Part 2 follows from the semantics of
universal quantification.
\end{proof}

\begin{note}[Syntactically Closed]\label{note:syn-closed}\em\bsp
It is clearly decidable whether an evaluation-free wff is
syntactically closed.  Is it also decidable whether a
non-evaluation-free wff $\textbf{A}_\alpha$ is syntactically closed
(i.e., ${} \models \mname{syn-closed}_{o\epsilon} \,
\synbrack{\textbf{A}_\alpha}$ holds)?  Since {\qzerouqe} is
undecidable, it follows that it is undecidable whether ${} \models
\mname{syn-closed}_{o\epsilon} \, \synbrack{\textbf{A}_\alpha}$ holds
when $\textbf{A}_\alpha$ has the form \[\sembrack{\If \, \textbf{B}_o
  \, \synbrack{\textbf{c}_\alpha} \,
  \synbrack{\textbf{x}_\alpha}}_\alpha,\] where $\textbf{c}_\alpha$ is
a primitive constant.  Therefore, it undecidable whether a
non-evaluation-free wff is syntactically closed.\esp
\end{note}

\begin{lem} [Semantically Closed]\label{lem:sem-closed}
Let $\sM$ be a normal general model for {\qzerouqe}.

\be

  \item If $\textbf{A}_\alpha$ is evaluation-free and syntactically
    closed, then $\textbf{A}_\alpha$ is semantically closed.

  \item If $\textbf{A}_\epsilon$ is semantically closed, then either
    $\sM \models \textbf{A}_\epsilon = \synbrack{\textbf{B}_\beta}$
    for some $\textbf{B}_\beta$ or $\sM \models
    \sembrack{\textbf{A}_\epsilon}_\gamma \QuasiEqual
    \Undefined_\gamma$ for all $\gamma \in \sT$.

  \item \bsp If $\textbf{A}_\epsilon$ is semantically closed, $\sM \models
    \mname{syn-closed}_{o\epsilon} \, \textbf{A}_\epsilon$, and $\sM
    \models \mname{eval-free}^{\alpha}_{o\epsilon} \,
    \textbf{A}_\epsilon$, then $\sembrack{\textbf{A}_\epsilon}_\alpha$
    is semantically closed.\esp

\ee

\end{lem}

\begin{proof}

\medskip

\noindent \textbf{Part 1} {\sglsp} Follows immediately from part 2 of
Lemma~\ref{lem:evaluation-free}.

\medskip

\noindent \textbf{Part 2} {\sglsp} Assume $\textbf{A}_\epsilon$ is
semantically closed.  Let $\phi \in \mname{assign}(\sM)$.  If $V^{\cal
  M}_{\phi}(\textbf{A}_\epsilon)$ is undefined or $\sE^{-1}(V^{\cal
  M}_{\phi}(\textbf{A}_\epsilon))$ is undefined, then $\sM \models
\sembrack{\textbf{A}_\epsilon}_\gamma \QuasiEqual \Undefined_\gamma$
for all $\gamma \in \sT$.  So we may assume $\sE^{-1}(V^{\cal
  M}_{\phi}(\textbf{A}_\epsilon))$ is some wff $\textbf{B}_\beta$.
Then $V^{\cal M}_{\phi}(\synbrack{\textbf{B}_\beta}) =
\sE(\textbf{B}_\beta) = \sE(\sE^{-1}(V^{\cal
  M}_{\phi}(\textbf{A}_\epsilon))) = V^{\cal
  M}_{\phi}(\textbf{A}_\epsilon)$.  The hypothesis implies
$\sE^{-1}(V^{\cal M}_{\phi}(\textbf{A}_\epsilon))$ does not depend on
$\phi$.  Hence $\sM \models \textbf{A}_\epsilon =
\synbrack{\textbf{B}_\beta}$.

\medskip

\noindent \textbf{Part 3} {\sglsp} Assume (a) $\textbf{A}_\epsilon$ is
semantically closed, (b) $\sM \models \mname{syn-closed}_{o\epsilon}
\, \textbf{A}_\epsilon$, and (c)~$\sM \models
\mname{eval-free}^{\alpha}_{o\epsilon} \, \textbf{A}_\epsilon$.  (a)
and part 2 of this lemma imply either there is some
$\textbf{B}_\alpha$ such that (d) $\sM \models \textbf{A}_\epsilon =
\synbrack{\textbf{B}_\alpha}$ or $\sM \models
\sembrack{\textbf{A}_\epsilon}_\alpha \QuasiEqual \Undefined_\alpha$.
$\Undefined_\alpha$ is semantically closed, so we may assume (d).
(b), (c), and (d) imply (e) $\sM \models
\mname{syn-closed}_{o\epsilon} \, \synbrack{\textbf{B}_\alpha}$ and
(f) $\sM \models \mname{eval-free}_{o\epsilon} \,
\synbrack{\textbf{B}_\alpha}$.  (f) implies $\textbf{B}_\alpha$ is
evaluation-free by Proposition~\ref{prop:eval-free}, and this and (e)
imply $\textbf{B}_\alpha$ is syntactically closed by part 1 of
Lemma~\ref{lem:evaluation-free}.  Thus $\textbf{B}_\alpha$ is
semantically closed by part 1 of this lemma.  Therefore,
$\sembrack{\textbf{A}_\epsilon}_\alpha$ is semantically closed since
$\sM \models \textbf{B}_\alpha \QuasiEqual
\sembrack{\synbrack{\textbf{B}_\alpha}}_\alpha$ by~(f) and $\sM
\models \sembrack{\synbrack{\textbf{B}_\alpha}}_\alpha =
\sembrack{\textbf{A}_\epsilon}_\alpha$ by (d).
\end{proof}

\subsection{Properties of $\mname{not-free-in}_{o\epsilon\epsilon}$}

\begin{lem} [Not Free In] \label{lem:not-free-in}
Let $\sM$ be a normal general model for {\qzerouqe}.

\be

  \item If $X$ is a set of variables such that $\sM \models
    \mname{not-free-in}_{o\epsilon\epsilon} \,
    \synbrack{\textbf{x}_\alpha} \, \synbrack{\textbf{B}_\beta}$ for
    all $\textbf{x}_\alpha \in X$, then $\textbf{B}_\beta$ is
    independent of $X$ in $\sM$.

  \item If $\sM \models \mname{not-free-in}_{o\epsilon\epsilon} \,
    \synbrack{\textbf{x}_\alpha} \, \synbrack{\textbf{B}_\beta}$,
    then \[V^{\cal M}_{\phi}(\textbf{B}_\beta) \QuasiEqual V^{\cal
      M}_{\phi[{\bf x}_\alpha \mapsto d]}(\textbf{B}_\beta)\] for all
    $\phi \in \mname{assign}(\sM)$ and all $d \in \sD_\alpha$.

\ee
\end{lem}

\begin{proof} 

\medskip

\noindent \textbf{Part 1} {\sglsp} Let $X$ be a set of variables.
Without loss of generality, we may assume that $X$ is nonempty.  We
will show that, if \[\sM \models
\mname{not-free-in}_{o\epsilon\epsilon} \,
\synbrack{\textbf{x}_\alpha} \, \synbrack{\textbf{D}_\delta} \mbox{
  for all $\textbf{x}_\alpha \in X$} {\dblsp} \mbox{[designated
    $H(\synbrack{\textbf{D}_\delta},X)$]},\] then
\[\textbf{D}_\delta \mbox{ is independent of $X$ in $\sM$} {\dblsp} \mbox{[designated
        $C(\textbf{D}_\delta,X)$]}.\] Our proof is by induction on the
complexity of $\textbf{D}_\delta$.  There are 9 cases corresponding to
the 9 parts of Specification 7 used to specify
$\mname{not-free-in}_{o\epsilon\epsilon} \,
\synbrack{\textbf{x}_\alpha} \, \synbrack{\textbf{D}_\delta}$.

\be

  \item[] \textbf{Case 1}: $\textbf{D}_\delta$ is a variable
    $\textbf{x}_\alpha$.  Assume $H(\synbrack{\textbf{x}_\alpha},X)$
    is true.  Then $\textbf{x}_\alpha \not\in X$ by the specification of
    $\mname{not-free-in}_{o\epsilon\epsilon}$.  Hence
    $C(\textbf{x}_\alpha,X)$ is obviously true.

  \item[] \textbf{Case 2}: $\textbf{D}_\delta$ is a primitive constant
    $\textbf{c}_\alpha$.  Then $C(\textbf{c}_\alpha,X)$ is true since
    every primitive constant is semantically closed by
    Proposition~\ref{prop:sem-closed}.

  \item[] \bsp\textbf{Case 3}: $\textbf{D}_\delta$ is
    $\textbf{A}_{\alpha\beta}\textbf{B}_\beta$.  Assume
    $H(\synbrack{\textbf{A}_{\alpha\beta}\textbf{B}_\beta},X)$ is
    true.  Then $H(\synbrack{\textbf{A}_{\alpha\beta}},X)$ and
    $H(\synbrack{\textbf{B}_\beta},X)$ are true by the specification of
    $\mname{not-free-in}_{o\epsilon\epsilon}$.  Hence
    $C(\textbf{A}_{\alpha\beta},X)$ and $C(\textbf{B}_\beta,X)$ are
    true by the induction hypothesis.  These imply
    $C(\textbf{A}_{\alpha\beta}\textbf{B}_\beta,X)$ by the semantics
    of function application. \esp

  \item[] \textbf{Case 4}: $\textbf{D}_\delta$ is
    $\lambda\textbf{x}_\alpha \textbf{A}_\beta$.  Assume
    $H(\synbrack{\lambda\textbf{x}_\alpha \textbf{A}_\beta},X)$ is
    true.  $C(\lambda\textbf{x}_\alpha
    \textbf{A}_\beta,\set{\textbf{x}_\alpha})$ is true by the
    semantics of function abstraction.
    $H(\synbrack{\lambda\textbf{x}_\alpha \textbf{A}_\beta},X)$
    implies $H(\synbrack{\textbf{A}_\beta},X \setminus
    \set{\textbf{x}_\alpha})$ by the specification of
    $\mname{not-free-in}_{o\epsilon\epsilon}$.  Hence
    $C(\textbf{A}_\beta,X \setminus \set{\textbf{x}_\alpha})$ is true
    by the induction hypothesis.  This implies
    $C(\lambda\textbf{x}_\alpha \textbf{A}_\beta,X \setminus
    \set{\textbf{x}_\alpha})$ by the semantics of function
    abstraction.  Therefore, $C(\lambda\textbf{x}_\alpha
    \textbf{A}_\beta,X)$ holds.

  \item[] \textbf{Case 5}: $\textbf{D}_\delta$ is $\If \, \textbf{A}_o
    \, \textbf{B}_\alpha \, \textbf{C}_\alpha$.  Similar to Case 3.

  \item[] \textbf{Case 6}: $\textbf{D}_\delta$ is
    $\synbrack{\textbf{A}_\alpha}$.  Then
    $C(\synbrack{\textbf{A}_\alpha},X)$ is true since every quotation
    is semantically closed by Proposition~\ref{prop:sem-closed}.

  \item[] \bsp\textbf{Case 7}: $\textbf{D}_\delta$ is
    $\sembrack{\textbf{A}_{\epsilon}}_\alpha$.  Assume
    $H(\synbrack{\sembrack{\textbf{A}_{\epsilon}}_\alpha},X)$ is true.
    Then (a)~$\sM \models \mname{syn-closed}_{o\epsilon} \,
    \synbrack{\textbf{A}_{\epsilon}}$, (b)~$\sM \models
    \mname{eval-free}^{\epsilon}_{o\epsilon} \,
    \synbrack{\textbf{A}_{\epsilon}}$, (c)~$\sM \models
    \mname{eval-free}^{\alpha}_{o\epsilon} \, \textbf{A}_{\epsilon}$,
    and (d) $H(\textbf{A}_{\epsilon},X)$ by the specification of
    $\mname{not-free-in}_{o\epsilon\epsilon}$ and the fact $X$ is
    nonempty.  (a) and (b) imply (e) $\textbf{A}_{\epsilon}$ is
    semantically closed by Proposition~\ref{prop:eval-free} and part~1
    of Lemma~\ref{lem:sem-closed}.  (e) and part 2 of
    Lemma~\ref{lem:sem-closed} implies either
    $\sembrack{\textbf{A}_{\epsilon}}_\alpha$ is semantically closed
    or (f) $\sE^{-1}(\sV^{\cal M}_{\phi}(\textbf{A}_{\epsilon}))$ is
    defined for all $\phi \in \mname{assign}(\sM)$.  So we may assume
    (f).  (c) and (f) imply (g)~$\sE^{-1}(\sV^{\cal
      M}_{\phi}(\textbf{A}_{\epsilon}))$ is an evaluation-free
    $\wff{\alpha}$ for all $\phi \in \mname{assign}(\sM)$, and thus
    the complexity of $\sE^{-1}(\sV^{\cal
      M}_{\phi}(\textbf{A}_{\epsilon}))$ is less than the complexity
    of $\sembrack{\textbf{A}_{\epsilon}}_\alpha$ (for any $\phi \in
    \mname{assign}(\sM)$).  Hence (d) implies $C(\sE^{-1}(\sV^{\cal
      M}_{\phi}(\textbf{A}_{\epsilon})),X)$ by the induction
    hypothesis.  Let $\phi,\phi' \in \mname{assign}(\sM)$ such that
    $\phi(\textbf{x}_\alpha) = \phi'(\textbf{x}_\alpha)$ whenever
    $\textbf{x}_\alpha \not\in X$.  Then
    \begin{align} \setcounter{equation}{0}
    &
    \sV^{\cal M}_{\phi}(\sembrack{\textbf{A}_{\epsilon}}_\beta) \\
    &\QuasiEqual
    \sV^{\cal M}_{\phi}(\sE^{-1}(\sV^{\cal M}_{\phi}(\textbf{A}_{\epsilon})))\\
    &\QuasiEqual 
    \sV^{\cal M}_{\phi'}(\sE^{-1}(\sV^{\cal M}_{\phi}(\textbf{A}_{\epsilon})))\\
    &\QuasiEqual 
    \sV^{\cal M}_{\phi'}(\sE^{-1}(\sV^{\cal M}_{\phi'}(\textbf{A}_{\epsilon})))\\
    &\QuasiEqual 
    \sV^{\cal M}_{\phi'}(\sembrack{\textbf{A}_{\epsilon}}_\beta).
    \end{align}
    (2) is by (g) and the semantics of evaluation; (3) is by
    $C(\sE^{-1}(\sV^{\cal M}_{\phi}(\textbf{A}_{\epsilon}))),X)$; (4)
    is by (e); and (5) is again by (g) and the semantics of
    evaluation.  This implies
    $C(\sembrack{\textbf{A}_{\epsilon}}_\beta,X)$. \esp

\ee

\medskip

\noindent \textbf{Part 2} {\sglsp} This part of the lemma is the
special case of part~1 when $X$ is a singleton.
\end{proof}

\subsection{Properties of $\mname{cleanse}_{\epsilon\epsilon}$}

\begin{lem} [Cleanse] \label{lem:cleanse}
Let $\sM$ be a normal general model for {\qzerouqe}.
\be

  \item If $\sM \models [\mname{cleanse}_{\epsilon\epsilon} \,
    \synbrack{\textbf{D}_\delta}] \IsDefApp$, then
    $\mname{cleanse}_{\epsilon\epsilon} \,
    \synbrack{\textbf{D}_\delta}$ is semantically closed and \[\sM
    \models \mname{eval-free}^{\delta}_{o\epsilon} \,
            [\mname{cleanse}_{\epsilon\epsilon} \,
              \synbrack{\textbf{D}_\delta}].\]

  \item Either $\sM \models \mname{cleanse}_{\epsilon\epsilon} \,
    \synbrack{\textbf{A}_\alpha} = \synbrack{\textbf{B}_\alpha}$ for
    some evaluation-free $\textbf{B}_\alpha$ or $\sM \models
    \sembrack{\mname{cleanse}_{\epsilon\epsilon} \,
      \synbrack{\textbf{A}_\alpha}}_\gamma \QuasiEqual
    \Undefined_\gamma$ for all $\gamma \in \sT$.

    \item \bsp If $\textbf{C}_\gamma$ contains an evaluation
    $\sembrack{\textbf{A}_\epsilon}_\alpha$ not in a quotation such
    that, for some variable $\textbf{x}_\beta$, $\sM \models
    \NegAlt[\mname{not-free-in}_{o\epsilon\epsilon} \, \synbrack{\textbf{x}_\beta}
    \, \synbrack{\textbf{A}_\epsilon}]$, then \[\sM \models
       [\mname{cleanse}_{\epsilon\epsilon} \,
         \synbrack{\textbf{C}_\gamma}]\IsUndefApp.\]\esp

  \item If $\sM \models [\mname{cleanse}_{\epsilon\epsilon} \,
    \synbrack{\textbf{D}_\delta}]\IsDefApp$, then \[\sM \models
    \sembrack{\mname{cleanse}_{\epsilon\epsilon} \,
      \synbrack{\textbf{D}_\delta}}_\delta \QuasiEqual
    \textbf{D}_\delta\]

\ee
\end{lem}

\begin{proof} 
Let $A(\synbrack{\textbf{D}_\delta})$ mean
$\mname{cleanse}(\synbrack{\textbf{D}_\delta})$.

\medskip

\noindent \textbf{Part 1} {\sglsp} Our proof is by induction on the
complexity of $\textbf{D}_\delta$.  There are 7 cases corresponding to
the 7 parts of Specification 8 used to specify
$A(\synbrack{\textbf{D}_\delta})$.

\be

  \item[] \textbf{Cases 1, 2, and 6}: $\textbf{D}_\delta$ is a
    variable, primitive constant, or quotation.  Then $\sM \models
    A(\synbrack{\textbf{D}_\delta}) = \synbrack{\textbf{D}_\delta}$ by
    the specification of $\mname{cleanse}_{\epsilon\epsilon}$.  Hence
    $A(\synbrack{\textbf{D}_\delta})$ is semantically closed since a
    quotation is semantically closed by
    Proposition~\ref{prop:sem-closed} and $\sM \models
    \mname{eval-free}^{\delta}_{o\epsilon} \,
    A(\synbrack{\textbf{D}_\delta})$ since a variable, primitive
    constant, or quotation is evaluation-free.

  \item[] \bsp \textbf{Case 3}: $\textbf{D}_\delta$ is
    $\textbf{A}_{\alpha\beta}\textbf{B}_\beta$.  Assume $\sM \models
    A(\synbrack{\textbf{A}_{\alpha\beta}\textbf{B}_\beta})\IsDefApp$.
    Then $\sM \models A(\synbrack{\textbf{A}_{\alpha\beta}})\IsDefApp$
    and $\sM \models A(\synbrack{\textbf{B}_\beta})\IsDefApp$ by the
    specification of $\mname{cleanse}_{\epsilon\epsilon}$.  It follows
    that $A(\synbrack{\textbf{A}_{\alpha\beta}\textbf{B}_\beta})$ is
    semantically closed and $\sM \models
    \mname{eval-free}^{\alpha}_{o\epsilon} \,
    A(\synbrack{\textbf{A}_{\alpha\beta}\textbf{B}_\beta})$ by the
    induction hypothesis and the specification of
    $\mname{cleanse}_{\epsilon\epsilon}$. \esp

  \item[] \textbf{Case 4}: $\textbf{D}_\delta$ is
    $\lambda\textbf{x}_\beta\textbf{A}_\alpha$.  Similar to Case 3.

  \item[] \textbf{Case 5}: $\textbf{D}_\delta$ is $\If \,
    \textbf{A}_o \, \textbf{B}_\alpha \, \textbf{C}_\alpha$.  Similar
      to the proof of Case 3.

  \item[]\bsp \textbf{Case 7}: $\textbf{D}_\delta$ is
    $\sembrack{\textbf{A}_{\epsilon}}_\alpha$.  Assume (a) $\sM
    \models
    A(\synbrack{\sembrack{\textbf{A}_{\epsilon}}_\alpha})\IsDefApp$.
    (a) implies (b)~$\sM \models \mname{syn-closed}_{o\epsilon} \,
    A(\synbrack{\textbf{A}_{\epsilon}})$, (c)~$\sM \models
    \mname{eval-free}^{\alpha}_{o\epsilon} \,
    \sembrack{A(\synbrack{\textbf{A}_{\epsilon}})}_\epsilon$, and 
    \[(\mbox{d})~\sM \models A(\synbrack{\sembrack{\textbf{A}_{\epsilon}}_\alpha}) \QuasiEqual
    \sembrack{A(\synbrack{\textbf{A}_{\epsilon}})}_\epsilon\] by the
    specification of $\mname{cleanse}_{\epsilon\epsilon}$.
    (a)~implies (e) $\sM \models
    A(\synbrack{\textbf{A}_{\epsilon}})\IsDefApp$, and (e)~implies (f)
    $A(\synbrack{\textbf{A}_{\epsilon}})$ is semantically closed and
    (g) $\sM \models \mname{eval-free}^{\epsilon}_{o\epsilon} \,
    A(\synbrack{\textbf{A}_{\epsilon}})$ by the induction hypothesis.
    (b), (f), and (g) imply (h)
    $\sembrack{A(\synbrack{\textbf{A}_{\epsilon}})}_\epsilon$ is
    semantically closed by part 3 of Lemma~\ref{lem:sem-closed}.
    Therefore, $A(\synbrack{\sembrack{\textbf{A}_{\epsilon}}_\alpha})$
    is semantically closed by (d) and (h). \esp

\ee

\medskip

\noindent \textbf{Part 2} {\sglsp} Follows easily from part 1 of this
lemma and part 2 of Lemma~\ref{lem:sem-closed}.

\medskip

\noindent \textbf{Part 3} {\sglsp} Follows immediately from the
specification of $\mname{cleanse}_{\epsilon\epsilon}$.

\medskip

\noindent \textbf{Part 4} {\sglsp} Assume 
\[\sM \models A(\synbrack{\textbf{D}_\delta})\IsDefApp \dblsp \mbox{[designated
    $H(\synbrack{\textbf{D}_\delta})$]}.\] We must show that \[\sM
\models \sembrack{A(\synbrack{\textbf{D}_\delta})}_\delta \QuasiEqual
\textbf{D}_\delta \dblsp \mbox{[designated
    $C(\synbrack{\textbf{D}_\delta})$]}.\] Our proof is by induction
on the complexity of $\textbf{D}_\delta$.  There are 7 cases
corresponding to the 7 parts of Specification 8 used to specify
$A(\synbrack{\textbf{D}_\delta})$.  Let $\phi \in
\mname{assign}(\sM)$.

\be

  \item[] \textbf{Case 1}: $\textbf{D}_\delta$ is $\textbf{x}_\alpha$.
    Then 
    \begin{align} \setcounter{equation}{0}
    &
    \sV^{\cal M}_{\phi}(\sembrack{A(\synbrack{\textbf{x}_\alpha})}_\alpha)\\
    &\QuasiEqual 
    \sV^{\cal  M}_{\phi}(\sembrack{\synbrack{\textbf{x}_\alpha}}_\alpha)\\
    &\QuasiEqual
    \sV^{\cal M}_{\phi}(\textbf{x}_\alpha).
    \end{align}
    (2) is by the specification of
    $\mname{cleanse}_{\epsilon\epsilon}$, and (3) is by the fact that
    $\textbf{x}_\alpha$ is evaluation-free and the semantics of
    evaluation.  Therefore, $C(\synbrack{\textbf{x}_\alpha})$ holds.

  \item[] \textbf{Case 2}: $\textbf{D}_\delta$ is a primitive constant
    $\textbf{c}_\alpha$.  Similar to Case 1.

  \item[] \textbf{Case 3}: $\textbf{D}_\delta$ is
    $\textbf{A}_{\alpha\beta}\textbf{B}_\beta$.
    $H(\synbrack{\textbf{A}_{\alpha\beta}\textbf{B}_\beta})$ implies
    $H(\synbrack{\textbf{A}_{\alpha\beta}})$ and
    $H(\synbrack{\textbf{B}_\beta})$ by the specification of
    $\mname{cleanse}_{\epsilon\epsilon}$.  These imply
    $C(\synbrack{\textbf{A}_{\alpha\beta}})$ and
    $C(\synbrack{\textbf{B}_\beta})$ by the induction hypothesis.
    Then
    \begin{align} \setcounter{equation}{0}
    &
    \sV^{\cal M}_{\phi}(\sembrack{A(\synbrack{\textbf{A}_{\alpha\beta}
    \textbf{B}_\beta})}_\alpha)\\
    &\QuasiEqual
    \sV^{\cal  M}_{\phi}(\sembrack{\mname{app}_{\epsilon\epsilon\epsilon} \, 
    A(\synbrack{\textbf{A}_{\alpha\beta}}) \, A(\synbrack{\textbf{B}_\beta})}_\alpha)\\
    &\QuasiEqual
    \sV^{\cal  M}_{\phi}(\sembrack{A(\synbrack{\textbf{A}_{\alpha\beta}})}_{\alpha\beta}
    \sembrack{A(\synbrack{\textbf{B}_\beta})}_\beta)\\
    &\QuasiEqual
    \sV^{\cal M}_{\phi}(\textbf{A}_{\alpha\beta}\textbf{B}_\beta).
    \end{align}
    \bsp
    (2) is by the specification of
    $\mname{cleanse}_{\epsilon\epsilon}$; (3) is by the semantics of
    $\mname{app}_{\epsilon\epsilon\epsilon}$ and evaluation; and (4)
    is by $C(\synbrack{\textbf{A}_{\alpha\beta}})$ and
    $C(\synbrack{\textbf{B}_\beta})$.  Therefore,
    $C(\synbrack{\textbf{A}_{\alpha\beta} \textbf{B}_\beta})$ holds.
    \esp

  \item[] \textbf{Case 4}: $\textbf{D}_\delta$ is
    $\lambda\textbf{x}_\beta\textbf{A}_\alpha$.
    $H(\synbrack{\lambda\textbf{x}_\beta\textbf{A}_\alpha})$ implies
    $H(\synbrack{\textbf{A}_\alpha})$ by the specification of
    $\mname{cleanse}_{\epsilon\epsilon}$.  This implies
    $C(\synbrack{\textbf{A}_\alpha})$ by the induction hypothesis and
    $A(\synbrack{\textbf{A}_\alpha})$ is semantically closed by part 1
    of this lemma.  Then
    \begin{align} \setcounter{equation}{0}
    &
    \sV^{\cal M}_{\phi}(\sembrack{A(\synbrack{\lambda\textbf{x}_\beta\textbf{A}_\alpha})}_{\alpha\beta})\\
    &\QuasiEqual
    \sV^{\cal M}_{\phi}(\sembrack{\mname{abs}_{\epsilon\epsilon\epsilon} \, 
    \synbrack{\textbf{x}_\beta} \, A(\synbrack{\textbf{A}_\alpha})}_{\alpha\beta})\\
    &\QuasiEqual
    \sV^{\cal M}_{\phi}(\lambda\textbf{x}_\beta
    \sembrack{A(\synbrack{\textbf{A}_\alpha})}_\alpha)\\
    &\QuasiEqual
    \sV^{\cal M}_{\phi}(\lambda\textbf{x}_\beta\textbf{A}_\alpha).
    \end{align}
    (2) is by the specification of
    $\mname{cleanse}_{\epsilon\epsilon}$; (3) is by the semantics of
    $\mname{abs}_{\epsilon\epsilon\epsilon}$ and evaluation and the
    fact that $A(\synbrack{\textbf{A}_\alpha})$ is semantically
    closed; and (4) is by $C(\synbrack{\textbf{A}_\alpha})$.
    Therefore,
    $C(\synbrack{\lambda\textbf{x}_\beta\textbf{A}_\alpha})$ holds.

  \item[] \textbf{Case 5}: $\textbf{D}_\delta$ is $\If \,
    \textbf{A}_o \, \textbf{B}_\alpha \, \textbf{C}_\alpha$.  Similar
      to Case 3.

  \item[] \textbf{Case 6}: $\textbf{D}_\delta$ is
    $\synbrack{\textbf{A}_{\alpha}}$.  Similar to Case 1.

  \item[] \textbf{Case 7}: $\textbf{D}_\delta$ is
    $\sembrack{\textbf{A}_{\epsilon}}_\alpha$.
    $H(\synbrack{\textbf{A}_{\epsilon}})$ is true by the proof for
    Case 7 of Part 1, and hence $C(\synbrack{\textbf{A}_{\epsilon}})$
    is true by the induction hypothesis.  Then
    \begin{align} \setcounter{equation}{0}
    &
    \sV^{\cal M}_{\phi}(\sembrack{A(\synbrack{\sembrack{\textbf{A}_{\epsilon}}_\alpha})}_\alpha)\\
    &\QuasiEqual 
    \sV^{\cal M}_{\phi}(\sembrack{\sembrack{A(\synbrack{\textbf{A}_\epsilon})}_\epsilon}_\alpha) \\
    &\QuasiEqual 
    \sV^{\cal M}_{\phi}(\sembrack{\textbf{A}_\epsilon}_\alpha).
    \end{align}
    (2) is by \[\sM \models
    A(\synbrack{\sembrack{\textbf{A}_{\epsilon}}_\alpha}) \QuasiEqual
    \sembrack{A(\synbrack{\textbf{A}_{\epsilon}})}_\epsilon\] shown in
    the proof for Case 7 of Part 1, and (3) is by
    $C(\synbrack{\textbf{A}_{\epsilon}})$.  Therefore,
    $C(\synbrack{\sembrack{\textbf{A}_{\epsilon}}_\alpha})$ holds.

\ee
\end{proof}

\subsection{Properties of $\mname{sub}_{\epsilon\epsilon\epsilon\epsilon}$}

\begin{lem} [Substitution] \label{lem:sub}
Let $\sM$ be a normal general model for {\qzerouqe}.
\be

  \item \bsp If $\sM \models
    [\mname{sub}_{\epsilon\epsilon\epsilon\epsilon} \,
      \synbrack{\textbf{A}_\alpha} \, \synbrack{\textbf{x}_\alpha} \,
      \synbrack{\textbf{B}_\beta}]\IsDefApp$, then
    $\mname{sub}_{\epsilon\epsilon\epsilon\epsilon} \,
    \synbrack{\textbf{A}_\alpha} \, \synbrack{\textbf{x}_\alpha} \,
    \synbrack{\textbf{B}_\beta}$ is semantically closed and \[\sM
    \models \mname{eval-free}^{\beta}_{o\epsilon} \,
            [\mname{sub}_{\epsilon\epsilon\epsilon\epsilon} \,
              \synbrack{\textbf{A}_\alpha} \,
              \synbrack{\textbf{x}_\alpha} \,
              \synbrack{\textbf{B}_\beta}].\] \esp

  \item Either $\sM \models
    \mname{sub}_{\epsilon\epsilon\epsilon\epsilon} \,
    \synbrack{\textbf{A}_\alpha} \, \synbrack{\textbf{x}_\alpha} \,
    \synbrack{\textbf{B}_\beta} = \synbrack{\textbf{C}_\beta}$ for
    some evaluation-free $\textbf{C}_\beta$ or $\sM \models
    \sembrack{\mname{sub}_{\epsilon\epsilon\epsilon\epsilon} \,
      \synbrack{\textbf{A}_\alpha} \, \synbrack{\textbf{x}_\alpha} \,
      \synbrack{\textbf{B}_\beta}}_\gamma \QuasiEqual
    \Undefined_\gamma$ for all $\gamma \in \sT$.

  \item \bsp If $\sM \models \mname{sub}_{\epsilon\epsilon\epsilon\epsilon}
    \, \synbrack{\textbf{A}_\alpha} \, \synbrack{\textbf{x}_\alpha} \,
    \textbf{B}_\epsilon = \synbrack{\textbf{C}_\beta}$ for some
    $\textbf{C}_\beta$ and $\sM \models
    \mname{eval-free}^{\beta}_{o\epsilon} \, \textbf{B}_\epsilon$,
    then $\sM \models \textbf{B}_\epsilon =
    \synbrack{\textbf{D}_\beta}$ for some evaluation-free
    $\textbf{D}_\beta$.\esp

  \item If $\textbf{C}_\gamma$ contains an evaluation
    $\sembrack{\textbf{B}_\epsilon}_\beta$ not in a quotation such
    that, for some variable $\textbf{y}_\gamma$ with
    $\textbf{x}_\alpha \not= \textbf{y}_\gamma$, \[\sM
    \models \NegAlt[\mname{not-free-in}_{o\epsilon\epsilon} \,
      \synbrack{\textbf{y}_\gamma} \,
               [\mname{sub}_{\epsilon\epsilon\epsilon\epsilon} \,
                 \synbrack{\textbf{A}_\alpha} \,
                 \synbrack{\textbf{x}_\alpha} \,
                 \synbrack{\textbf{B}_\epsilon}]],\] then \[\sM
    \models [\mname{sub}_{\epsilon\epsilon\epsilon\epsilon} \,
      \synbrack{\textbf{A}_\alpha} \, \synbrack{\textbf{x}_\alpha} \,
      \synbrack{\textbf{C}_\gamma}]\IsUndefApp.\]

  \item If $\sM \models \mname{sub}_{\epsilon\epsilon\epsilon\epsilon}
    \, \synbrack{\textbf{A}_\alpha} \, \synbrack{\textbf{x}_\alpha} \,
    \synbrack{\textbf{D}_\delta} = \synbrack{\textbf{E}_\delta}$ for
    some $\textbf{E}_\delta$ and \[\sM \models
    \mname{not-free-in}_{o\epsilon\epsilon} \,
    \synbrack{\textbf{x}_\alpha} \, \synbrack{\textbf{D}_\delta},\]
    then \[\sM \models
    \sembrack{\mname{sub}_{\epsilon\epsilon\epsilon\epsilon} \,
      \synbrack{\textbf{A}_\alpha} \, \synbrack{\textbf{x}_\alpha} \,
      \synbrack{\textbf{D}_\delta}}_\delta \QuasiEqual
    \textbf{D}_\delta.\]

  \item If $\sM \models \mname{sub}_{\epsilon\epsilon\epsilon\epsilon}
    \, \synbrack{\textbf{A}_\alpha} \, \synbrack{\textbf{x}_\alpha} \,
    \synbrack{\textbf{D}_\delta} = \synbrack{\textbf{E}_\delta}$ for
    some $\textbf{E}_\delta$ , then \[\sV^{\cal
      M}_{\phi}(\sembrack{\mname{sub}_{\epsilon\epsilon\epsilon\epsilon}
      \, \synbrack{\textbf{A}_\alpha} \, \synbrack{\textbf{x}_\alpha}
      \, \synbrack{\textbf{D}_\delta}}_\delta) \QuasiEqual \sV^{\cal
      M}_{\phi[{\bf x}_\alpha \mapsto {\cal V}^{\cal M}_{\phi}({\bf
          A}_\alpha)]}(\textbf{D}_\delta)\] for all $\phi \in
    \mname{assign}(\sM)$ such that $\sV^{\cal
      M}_{\phi}(\textbf{A}_\alpha)$ is defined.

  \item $\sM \models [\mname{sub}_{\epsilon\epsilon\epsilon\epsilon}
    \, \synbrack{\textbf{A}_\alpha} \, \synbrack{\textbf{x}_\alpha} \,
    \synbrack{\textbf{B}_\beta}]\IsDefApp$ whenever
    $\textbf{A}_\alpha$ and $\textbf{B}_\beta$ are evaluation-free,
    $\sV^{\cal M}_{\phi}(\textbf{A}_\alpha)$ is defined, and
    substituting $\textbf{A}_\alpha$ for each free occurrence of
    $\textbf{x}_\alpha$ in $\textbf{B}_\beta$ does not result in a
    variable capture.

\ee
\end{lem}

\begin{proof} 
Let $S(\synbrack{\textbf{D}_\delta})$ mean
$\mname{sub}_{\epsilon\epsilon\epsilon\epsilon} \,
\synbrack{\textbf{A}_\alpha} \, \synbrack{\textbf{x}_\alpha} \,
\synbrack{\textbf{D}_\delta}$.

\medskip

\noindent \textbf{Part 1} {\sglsp} Similar to the proof of part 1 of
Lemma~\ref{lem:cleanse}.

\medskip

\noindent \textbf{Part 2} {\sglsp} Follows easily from part 1 of this
lemma and part 2 of Lemma~\ref{lem:sem-closed}.

\medskip

\noindent \textbf{Part 3} {\sglsp} Follows from
Lemma~\ref{lem:nonstand-sub}.

\medskip

\noindent \textbf{Part 4} {\sglsp} Follows immediately from the
specification of $\mname{sub}_{\epsilon\epsilon\epsilon\epsilon}$.

\medskip

\noindent \textbf{Part 5} {\sglsp} Assume \[\sM \models
S(\synbrack{\textbf{D}_\delta}) = \synbrack{\textbf{E}_\delta} \mbox{
  for some } \textbf{E}_\delta \dblsp \mbox{[designated
    $H_1(\synbrack{\textbf{D}_\delta})$]}\] and \[\sM \models
\mname{not-free-in}_{o\epsilon\epsilon} \,
\synbrack{\textbf{x}_\alpha} \, \synbrack{\textbf{D}_\delta} \dblsp
\mbox{[designated $H_2(\synbrack{\textbf{D}_\delta})$]}.\] We must
show that \[\sM \models
\sembrack{S(\synbrack{\textbf{D}_\delta})}_\delta \QuasiEqual
\textbf{D}_\delta \dblsp \mbox{[designated
    $C(\synbrack{\textbf{D}_\delta})$]}.\] Our proof is by induction
on the complexity of $\textbf{D}_\delta$.  There are 9 cases
corresponding to the 9 parts of Specification 9 used to specify
$S(\synbrack{\textbf{D}_\delta})$.  Let $\phi \in
\mname{assign}(\sM)$.

\be

  \item[] \textbf{Case 1}: $\textbf{D}_\delta$ is $\textbf{x}_\alpha$.
    By the specification of $\mname{not-free-in}_{o\epsilon\epsilon}$,
    $H_2(\synbrack{\textbf{x}_\alpha})$ does not hold in this case.

  \item[] \textbf{Case 2}: $\textbf{D}_\delta$ is
    $\textbf{y}_\beta$ where $\textbf{x}_\alpha \not=
    \textbf{y}_\beta$.
    Then 
    \begin{align} \setcounter{equation}{0}
    &
    \sV^{\cal M}_{\phi}(\sembrack{S(\synbrack{\textbf{y}_\beta})}_\beta)\\
    &\QuasiEqual
    \sV^{\cal  M}_{\phi}(\sembrack{\synbrack{\textbf{y}_\beta}}_\beta)\\
    &\QuasiEqual
    \sV^{\cal M}_{\phi}(\textbf{y}_\beta).
    \end{align}
    (2) is by the specification of
    $\mname{sub}_{\epsilon\epsilon\epsilon\epsilon}$, and (3) is by
    semantics of evaluation and the fact that $\textbf{y}_\beta$ is
    evaluation-free.  Therefore, $C(\synbrack{\textbf{y}_\beta})$
    holds.

  \item[] \textbf{Case 3}: $\textbf{D}_\delta$ is a primitive constant
    $\textbf{c}_\beta$. Similar to Case 2.

  \item[] \bsp \textbf{Case 4}: $\textbf{D}_\delta$ is
    $\textbf{B}_{\beta\gamma}\textbf{D}_\delta$.
    $H_1(\synbrack{\textbf{B}_{\beta\gamma}\textbf{D}_\delta})$
    implies $H_1(\synbrack{\textbf{B}_{\beta\gamma}})$ and
    $H_1(\synbrack{\textbf{D}_\delta})$ by the specification of
    $\mname{sub}_{\epsilon\epsilon\epsilon\epsilon}$.
    $H_2(\synbrack{\textbf{B}_{\beta\gamma}\textbf{D}_\delta})$
    implies $H_2(\synbrack{\textbf{B}_{\beta\gamma}})$ and
    $H_2(\synbrack{\textbf{D}_\delta})$ by the specification of
    $\mname{not-free-in}_{o\epsilon\epsilon}$.  These imply
    $C(\synbrack{\textbf{B}_{\beta\gamma}})$ and
    $C(\synbrack{\textbf{D}_\delta})$ by the induction hypothesis.
    Then
    \begin{align} \setcounter{equation}{0}
    &
    \sV^{\cal M}_{\phi}(\sembrack{S(\synbrack{\textbf{B}_{\beta\gamma}
    \textbf{D}_\delta})}_\beta)\\
    &\QuasiEqual
    \sV^{\cal  M}_{\phi}(\sembrack{\mname{app}_{\epsilon\epsilon\epsilon} \, 
    S(\synbrack{\textbf{B}_{\beta\gamma}}) \, S(\synbrack{\textbf{D}_\delta})}_\alpha)\\
    &\QuasiEqual
    \sV^{\cal  M}_{\phi}(\sembrack{S(\synbrack{\textbf{B}_{\beta\gamma}})}_{\beta\gamma}
    \sembrack{S(\synbrack{\textbf{D}_\delta})}_\gamma)\\
    &\QuasiEqual
    \sV^{\cal M}_{\phi}(\textbf{B}_{\beta\gamma}\textbf{D}_\delta).
    \end{align}
    (2) is by the specification of
    $\mname{sub}_{\epsilon\epsilon\epsilon\epsilon}$; (3) is by the
    semantics of $\mname{app}_{\epsilon\epsilon\epsilon}$ and
    evaluation; and (4) is by $C(\synbrack{\textbf{B}_{\beta\gamma}})$
    and $C(\synbrack{\textbf{D}_\delta})$.  Therefore,
    $C(\synbrack{\textbf{B}_{\beta\gamma} \textbf{D}_\delta})$
    holds. \esp

  \item[] \textbf{Case 5}: $\textbf{D}_\delta$ is $\lambda
    \textbf{x}_\alpha \textbf{B}_\beta$.
    $H_1(\synbrack{\lambda\textbf{x}_\alpha\textbf{B}_\beta})$ implies
    $\sM \models [\mname{cleanse}_{\epsilon\epsilon} \,
      \synbrack{\textbf{B}_\beta}] \IsDefApp$ by the specification of
    $\mname{sub}_{\epsilon\epsilon\epsilon\epsilon}$.  This implies
    that $\mname{cleanse}_{\epsilon\epsilon} \,
    \synbrack{\textbf{B}_\beta}$ is semantically closed by part 1 of
    Lemma~\ref{lem:cleanse}.  Then
    \begin{align} \setcounter{equation}{0}
    &
    \sV^{\cal M}_{\phi}(\sembrack{S(\lambda\textbf{x}_\alpha\textbf{B}_\beta)}_{\beta\alpha})\\
    &\QuasiEqual
    \sV^{\cal M}_{\phi}(\sembrack{\mname{abs}_{\epsilon\epsilon\epsilon} \, 
    \synbrack{\textbf{x}_\alpha} \, \mname{cleanse}_{\epsilon\epsilon} \, 
    \synbrack{\textbf{B}_\beta}}_{\beta\alpha})\\
    &\QuasiEqual
    \sV^{\cal M}_{\phi}(\lambda\textbf{x}_\alpha
    \sembrack{\mname{cleanse}_{\epsilon\epsilon} \, 
    \synbrack{\textbf{B}_\beta}}_\beta)\\
    &\QuasiEqual
    \sV^{\cal M}_{\phi}(\lambda\textbf{x}_\alpha\textbf{B}_\beta).
    \end{align}
    (2) is by the specification of
    $\mname{sub}_{\epsilon\epsilon\epsilon\epsilon}$; (3) is by the
    semantics of $\mname{abs}_{\epsilon\epsilon\epsilon}$ and
    evaluation and the fact that $\mname{cleanse}_{\epsilon\epsilon}
    \, \synbrack{\textbf{B}_\beta}$ is semantically closed; and (4) is
    by part 4 of Lemma~\ref{lem:cleanse}.  Therefore,
    $C(\synbrack{\lambda\textbf{x}_\alpha\textbf{B}_\beta})$ holds.

  \item[] \bsp \textbf{Case 6}: $\textbf{D}_\delta$ is $\lambda
    \textbf{y}_\beta \textbf{B}_\gamma$ where $\textbf{x}_\alpha \not=
    \textbf{y}_\beta$.
    $H_1(\synbrack{\lambda\textbf{y}_\beta\textbf{B}_\gamma})$ implies
    \[\sV^{\cal
      M}_{\phi}(S(\synbrack{\lambda\textbf{y}_\beta\textbf{B}_\gamma}))
    \QuasiEqual \sV^{\cal
      M}_{\phi}(\mname{abs}_{\epsilon\epsilon\epsilon} \,
    \synbrack{\textbf{y}_\beta} \, S(\synbrack{\textbf{B}_\gamma}))\]
    and $H_1(\synbrack{\textbf{B}_\gamma})$ by the specification of
    $\mname{sub}_{\epsilon\epsilon\epsilon\epsilon}$.
    $H_2(\synbrack{\lambda\textbf{y}_\beta\textbf{B}_\gamma})$ implies
    $H_2(\synbrack{\textbf{B}_\gamma})$ by the specification of
    $\mname{not-free-in}_{o\epsilon\epsilon}$.  These imply
    $C(\synbrack{\textbf{B}_\gamma})$ by the induction hypothesis and
    $S(\synbrack{\textbf{B}_\gamma})$ is semantically closed by part 1
    of this lemma.  Then
    \begin{align} \setcounter{equation}{0}
    &
    \sV^{\cal M}_{\phi}(\sembrack{S(\lambda\textbf{y}_\beta\textbf{B}_\gamma)}_{\gamma\beta})\\
    &\QuasiEqual
    \sV^{\cal M}_{\phi}(\sembrack{\mname{abs}_{\epsilon\epsilon\epsilon} \, 
    \synbrack{\textbf{y}_\beta} \, S(\synbrack{\textbf{B}_\gamma})}_{\gamma\beta})\\
    &\QuasiEqual
    \sV^{\cal M}_{\phi}(\lambda\textbf{y}_\beta
    \sembrack{S(\synbrack{\textbf{B}_\gamma})}_\gamma)\\
    &\QuasiEqual
    \sV^{\cal M}_{\phi}(\lambda\textbf{y}_\beta\textbf{B}_\gamma).
    \end{align}
    (2) is by the equation shown above; (3) is by the semantics of
    $\mname{abs}_{\epsilon\epsilon\epsilon}$ and evaluation and the
    fact that $S(\synbrack{\textbf{B}_\gamma})$ is semantically
    closed; and (4) is by $C(\synbrack{\textbf{B}_\gamma})$.
    Therefore,
    $C(\synbrack{\lambda\textbf{y}_\beta\textbf{B}_\gamma})$
    holds.\esp

  \item[] \textbf{Case 7}: $\textbf{D}_\delta$ is $\If \, \textbf{A}_o
    \, \textbf{B}_\alpha \, \textbf{C}_\alpha$.  Similar to Case 4.

  \item[] \textbf{Case 8}: $\textbf{D}_\delta$ is
    $\synbrack{\textbf{B}_\beta}$.  Similar to Case 2.

  \item[] \bsp \textbf{Case 9}: $\textbf{D}_\delta$ is
    $\sembrack{\textbf{B}_\epsilon}_\beta$.
    $H_1(\synbrack{\sembrack{\textbf{B}_\epsilon}_\beta})$
    implies \[(\mbox{a})~\sM \models
    \mname{eval-free}^{\beta}_{o\epsilon} \,
    \sembrack{S(\synbrack{\textbf{B}_\epsilon})}_\epsilon\]
    and \[(\mbox{b})~\sM \models
    S(\synbrack{\sembrack{\textbf{B}_\epsilon}_\beta}) =
    S(\sembrack{S(\synbrack{\textbf{B}_\epsilon})}_\epsilon)\] by the
    specification of $\mname{sub}_{\epsilon\epsilon\epsilon\epsilon}$.
    (a), (b), and
    $H_1(\synbrack{\sembrack{\textbf{B}_\epsilon}_\beta})$ imply
    $H_1(\sembrack{S(\synbrack{\textbf{B}_\epsilon})}_\epsilon)$ by
    part 3 of this lemma, and so (c)~$\sM \models
    \sembrack{S(\synbrack{\textbf{B}_\epsilon})}_\epsilon =
    \synbrack{\textbf{C}_\beta}$ for some evaluation-free
    $\textbf{C}_\beta$.
    $H_1(\sembrack{S(\synbrack{\textbf{B}_\epsilon})}_\epsilon)$
    implies $H_1(\synbrack{\textbf{B}_\epsilon})$ by the specification
    of $\mname{sub}_{\epsilon\epsilon\epsilon\epsilon}$.  By the
    specification of $\mname{not-free-in}_{o\epsilon\epsilon}$,
    $H_2(\synbrack{\sembrack{\textbf{B}_\epsilon}_\beta})$ implies
    $\sM \models \mname{syn-closed}_{o\epsilon} \,
    \synbrack{\textbf{B}_\epsilon}$ (hence
    $H_2(\synbrack{\textbf{B}_\epsilon})$ by the definition of
    $\mname{syn-closed}_{o\epsilon}$), $\sM \models
    \mname{eval-free}^{\epsilon}_{o\epsilon} \,
    \synbrack{\textbf{B}_\epsilon}$, and
    $H_2(\sembrack{\synbrack{\textbf{B}_\epsilon}}_\epsilon)$ (hence
    $H_2(\textbf{B}_\epsilon)$ by the semantics of evaluation).
    $H_1(\synbrack{\textbf{B}_\epsilon})$ and
    $H_2(\synbrack{\textbf{B}_\epsilon})$ imply
    $C(\synbrack{\textbf{B}_\epsilon})$ by the induction hypothesis,
    and so (d) $\sM \models
    \sembrack{S(\synbrack{\textbf{B}_\epsilon})}_\epsilon \QuasiEqual
    \textbf{B}_\epsilon$.  (c) and (d) imply (e) $\sM \models
    \textbf{B}_\epsilon = \synbrack{\textbf{C}_\beta}$.
    $H_1(\sembrack{S(\synbrack{\textbf{B}_\epsilon})}_\epsilon)$ and
    (c) imply $H_1(\synbrack{\textbf{C}_\beta})$.
    $H_2(\textbf{B}_\epsilon)$ and (e) implies
    $H_2(\synbrack{\textbf{C}_\beta})$.
    $H_1(\synbrack{\textbf{C}_\beta})$ and
    $H_2(\synbrack{\textbf{C}_\beta})$ imply
    $C(\synbrack{\textbf{C}_\beta})$ by the inductive hypothesis.
    Then
    \begin{align} \setcounter{equation}{0}
    &
    \sV^{\cal M}_{\phi}(\sembrack{S(\synbrack{\sembrack{\textbf{B}_\epsilon}_\beta})}_\beta)\\
    &\QuasiEqual
    \sV^{\cal M}_{\phi}(\sembrack{S(\sembrack{S(\synbrack{\textbf{B}_\epsilon})}_\epsilon)}_\beta)\\
    &\QuasiEqual
    \sV^{\cal M}_{\phi}(\sembrack{S(\synbrack{\textbf{C}_\beta})}_\beta)\\
    &\QuasiEqual
    \sV^{\cal M}_{\phi}(\textbf{C}_\beta)\\
    &\QuasiEqual
    \sV^{\cal M}_{\phi}(\sembrack{\synbrack{\textbf{C}_\beta}}_\beta)\\
    &\QuasiEqual
    \sV^{\cal M}_{\phi}(\sembrack{\textbf{B}_\epsilon}_\beta).
    \end{align}
    (2) is by (b); (3) is by (d) and (e); (4) is by
    $C(\synbrack{\textbf{C}_\beta})$; (5) is by the semantics of
    evaluation and the fact $\textbf{C}_\beta$ is evaluation-free; and
    (6) is by (e).  Therefore,
    $C(\synbrack{\sembrack{\textbf{B}_\epsilon}_\beta})$ holds.\esp

\ee

\medskip

\noindent \textbf{Part 6} {\sglsp} Assume \[\sM \models
S(\synbrack{\textbf{D}_\delta}) = \synbrack{\textbf{E}_\delta} \mbox{
  for some } \textbf{E}_\delta \dblsp \mbox{[designated
    $H(\synbrack{\textbf{D}_\delta})$]}\] We must show that

\bi

  \item[] $\sM \models \sV^{\cal
    M}_{\phi}(\sembrack{S(\synbrack{\textbf{D}_\delta})}_\delta)
    \QuasiEqual \sV^{\cal M}_{\phi[{\bf x}_\alpha \mapsto {\cal
          V}^{\cal M}_{\phi}({\bf A}_\alpha)]}(\textbf{D}_\delta)$ for all
    $\phi \in \mname{assign}(\sM)$ such that $\sV^{\cal
      M}_{\phi}(\textbf{A}_\alpha)$ is defined \sglsp
    \mbox{[designated $C(\synbrack{\textbf{D}_\delta})$]}.

\ei
Our proof is by induction on the complexity of $\textbf{D}_\delta$.
There are 9 cases corresponding to the 9 parts of Specification 9 used
to specify $S(\synbrack{\textbf{D}_\delta})$.  Let $\phi \in
\mname{assign}(\sM)$ such that $\sV^{\cal
  M}_{\phi}(\textbf{A}_\alpha)$ is defined.

\be

  \item[] \textbf{Case 1}: $\textbf{D}_\delta$ is $\textbf{x}_\alpha$.  Then
    \begin{align} \setcounter{equation}{0}
    &
    \sV^{\cal M}_{\phi}(\sembrack{S(\synbrack{\textbf{x}_\alpha})}_\alpha)\\
    &\QuasiEqual
    \sV^{\cal  M}_{\phi}(\sembrack{\mname{cleanse}_{\epsilon\epsilon}\, 
    \synbrack{\textbf{A}_\alpha}}_\alpha)\\
    &\QuasiEqual
    \sV^{\cal  M}_{\phi}(\textbf{A}_\alpha)\\
    &\QuasiEqual
    \sV^{\cal M}_{\phi[{\bf x}_\alpha \mapsto {\cal V}^{\cal M}_{\phi}({\bf A}_\alpha)]}(\textbf{x}_\alpha).
    \end{align}
    (2) is by the specification of
    $\mname{sub}_{\epsilon\epsilon\epsilon\epsilon}$; (3) is by
    $H(\synbrack{\textbf{x}_\alpha})$ and part 4 of
    Lemma~\ref{lem:cleanse}; and (4) is by the semantics of variables.
    Therefore, $C(\synbrack{\textbf{x}_\alpha})$ holds.

  \item[] \textbf{Case 2}: $\textbf{D}_\delta$ is
    $\textbf{y}_\beta$ where $\textbf{x}_\alpha \not=
    \textbf{y}_\beta$. Then
    \begin{align} \setcounter{equation}{0}
    &
    \sV^{\cal M}_{\phi}(\sembrack{S(\synbrack{\textbf{y}_\beta})}_\beta)\\
    &\QuasiEqual
    \sV^{\cal  M}_{\phi}(\sembrack{\synbrack{\textbf{y}_\beta}}_\beta)\\
    &\QuasiEqual
    \sV^{\cal  M}_{\phi}(\textbf{y}_\beta)\\
    &\QuasiEqual
    \sV^{\cal M}_{\phi[{\bf x}_\alpha \mapsto {\cal V}^{\cal M}_{\phi}({\bf A}_\alpha)]}(\textbf{y}_\beta).
    \end{align}
    (2) is by the specification of
    $\mname{sub}_{\epsilon\epsilon\epsilon\epsilon}$; (3) is by the
    semantics of evaluation and that fact that $\textbf{y}_\beta$ is
    evaluation-free; and (4) follows from $\textbf{x}_\alpha \not=
    \textbf{y}_\beta$.  Therefore, $C(\synbrack{\textbf{y}_\beta})$
    holds.

  \item[] \textbf{Case 3}: $\textbf{D}_\delta$ is a primitive constant
    $\textbf{c}_\beta$.  Similar to Case 2.

  \item[] \textbf{Case 4}: $\textbf{D}_\delta$ is
    $\textbf{B}_{\beta\gamma}\textbf{D}_\delta$.
    $H(\synbrack{\textbf{B}_{\beta\gamma}\textbf{D}_\delta})$ implies
    $H(\synbrack{\textbf{B}_{\beta\gamma}})$ and
    $H(\synbrack{\textbf{D}_\delta})$ by the specification of
    $\mname{sub}_{\epsilon\epsilon\epsilon\epsilon}$.  These imply
    $C(\synbrack{\textbf{B}_{\beta\gamma}})$ and
    $C(\synbrack{\textbf{D}_\delta})$ by the induction hypothesis.
    Then
    \begin{align} \setcounter{equation}{0}
    &
    \sV^{\cal M}_{\phi}(\sembrack{S(\synbrack{\textbf{B}_{\beta\gamma} \textbf{D}_\delta})}_\beta)\\
    & \QuasiEqual
    \sV^{\cal  M}_{\phi}(\sembrack{\mname{app}_{\epsilon\epsilon\epsilon} \, 
    S(\synbrack{\textbf{B}_{\beta\gamma}}) \, S(\synbrack{\textbf{D}_\delta})}_\alpha)\\
    &\QuasiEqual
    \sV^{\cal  M}_{\phi}(\sembrack{S(\synbrack{\textbf{B}_{\beta\gamma}})}_{\beta\gamma}
    \sembrack{S(\synbrack{\textbf{D}_\delta})}_\gamma)\\
    &\QuasiEqual
    \sV^{\cal  M}_{\phi}(\sembrack{S(\synbrack{\textbf{B}_{\beta\gamma}})}_{\beta\gamma})
    (\sV^{\cal  M}_{\phi}(\sembrack{S(\synbrack{\textbf{D}_\delta})}_\gamma))\\
    &\QuasiEqual
    \sV^{\cal  M}_{\phi[{\bf x}_\alpha \mapsto {\cal V}^{\cal M}_{\phi}({\bf A}_\alpha)]}(\textbf{B}_{\beta\gamma})
    (\sV^{\cal  M}_{\phi[{\bf x}_\alpha \mapsto {\cal V}^{\cal M}_{\phi}({\bf A}_\alpha)]}(\textbf{D}_\delta))\\
    &\QuasiEqual
    \sV^{\cal M}_{\phi[{\bf x}_\alpha \mapsto {\cal V}^{\cal M}_{\phi}({\bf A}_\alpha)]}
    (\textbf{B}_{\beta\gamma}\textbf{D}_\delta).
    \end{align}
    (2) is by the specification of
    $\mname{sub}_{\epsilon\epsilon\epsilon\epsilon}$; (3) is by the
    semantics of $\mname{app}_{\epsilon\epsilon\epsilon}$ and
    evaluation; (4) and (6) are by the semantics of application; and
    (5) is by $C(\synbrack{\textbf{B}_{\beta\gamma}})$ and
    $C(\synbrack{\textbf{D}_\delta})$.  Therefore,
    $C(\synbrack{\textbf{B}_{\beta\gamma} \textbf{D}_\delta})$ holds.

  \item[] \textbf{Case 5}: $\textbf{D}_\delta$ is $\lambda
    \textbf{x}_\alpha \textbf{B}_\beta$.
    $H(\synbrack{\lambda\textbf{x}_\alpha\textbf{B}_\beta})$ implies
    $\sM \models [\mname{cleanse}_{\epsilon\epsilon} \,
      \synbrack{\textbf{B}_\beta}] \IsDefApp$ by the specification of
    $\mname{sub}_{\epsilon\epsilon\epsilon\epsilon}$.  This implies
    that $\mname{cleanse}_{\epsilon\epsilon} \,
    \synbrack{\textbf{B}_\beta}$ is semantically closed by part 1 of
    Lemma~\ref{lem:cleanse}.  Then
    \begin{align} \setcounter{equation}{0}
    &
    \sV^{\cal M}_{\phi}(\sembrack{S(\lambda\textbf{x}_\alpha\textbf{B}_\beta)}_{\beta\alpha})\\
    &\QuasiEqual
    \sV^{\cal M}_{\phi}(\sembrack{\mname{abs}_{\epsilon\epsilon\epsilon} \, 
    \synbrack{\textbf{x}_\alpha} \, \mname{cleanse}_{\epsilon\epsilon} \, 
    \synbrack{\textbf{B}_\beta}}_{\beta\alpha})\\
    &\QuasiEqual
    \sV^{\cal M}_{\phi}(\lambda\textbf{x}_\alpha
    \sembrack{\mname{cleanse}_{\epsilon\epsilon} \, 
    \synbrack{\textbf{B}_\beta}}_\beta)\\
    &\QuasiEqual
    \sV^{\cal M}_{\phi}(\lambda\textbf{x}_\alpha\textbf{B}_\beta)\\
    &\QuasiEqual
    \sV^{\cal M}_{\phi[{\bf x}_\alpha \mapsto {\cal V}^{\cal M}_{\phi}
    ({\bf A}_\alpha)]}(\lambda\textbf{x}_\alpha\textbf{B}_\beta).
    \end{align}
    (2) is by the specification of
    $\mname{sub}_{\epsilon\epsilon\epsilon\epsilon}$; (3) is by the
    semantics of $\mname{abs}_{\epsilon\epsilon\epsilon}$ and
    evaluation and the fact that $\mname{cleanse}_{\epsilon\epsilon}
    \, \synbrack{\textbf{B}_\beta}$ is semantically closed; (4) is by
    part 4 of Lemma~\ref{lem:cleanse}; and (5) is by the fact
    that \[\sV^{\cal M}_{\phi[{\bf x}_\alpha \mapsto
        d]}(\textbf{B}_\beta) \QuasiEqual \sV^{\cal M}_{\phi[{\bf
          x}_\alpha \mapsto {\cal V}^{\cal M}_{\phi}({\bf
          A}_\alpha)][{\bf x}_\alpha \mapsto d]}(\textbf{B}_\beta)\]
    for all $d \in \sD_\alpha$.  Therefore,
    $C(\synbrack{\lambda\textbf{x}_\alpha\textbf{B}_\beta})$ holds.

  \item[] \textbf{Case 6}: $\textbf{D}_\delta$ is $\lambda
    \textbf{y}_\beta \textbf{B}_\gamma$ where $\textbf{x}_\alpha
    \not= \textbf{y}_\beta$.  
    $H(\synbrack{\lambda\textbf{y}_\beta\textbf{B}_\gamma})$ implies
    \[(\mbox{a})~\sV^{\cal
      M}_{\phi}(S(\synbrack{\lambda\textbf{y}_\beta\textbf{B}_\gamma}))
    \QuasiEqual \sV^{\cal
      M}_{\phi}(\mname{abs}_{\epsilon\epsilon\epsilon} \,
    \synbrack{\textbf{y}_\beta} \, S(\synbrack{\textbf{B}_\gamma})),\]
    $H(\synbrack{\textbf{B}_\gamma})$, and either $(\ast) \sglsp \sM
    \models \mname{not-free-in}_{o\epsilon\epsilon} \,
    \synbrack{\textbf{x}_\alpha} \, \synbrack{\textbf{B}_\gamma}$ or
    $(\ast\ast) \sglsp \sM \models
    \mname{not-free-in}_{o\epsilon\epsilon} \,
    \synbrack{\textbf{y}_\beta} \, \synbrack{\textbf{A}_\alpha}$ by
    the specification of
    $\mname{sub}_{\epsilon\epsilon\epsilon\epsilon}$.
    $H(\synbrack{\textbf{B}_\gamma})$ implies
    $C(\synbrack{\textbf{B}_\gamma})$ by the induction hypothesis and
    (b)~$S(\synbrack{\textbf{B}_\gamma})$ is semantically closed by
    part 1 of this lemma.  Then
    \begin{align} \setcounter{equation}{0}
    &
    \sV^{\cal M}_{\phi}(\sembrack{S(\lambda\textbf{y}_\beta\textbf{B}_\gamma)}_{\gamma\beta})\\
    &\QuasiEqual
    \sV^{\cal M}_{\phi}(\sembrack{\mname{abs}_{\epsilon\epsilon\epsilon} \, 
    \synbrack{\textbf{y}_\beta} \, S(\synbrack{\textbf{B}_\gamma})}_{\gamma\beta})\\
    &\QuasiEqual
    \sV^{\cal M}_{\phi}(\lambda\textbf{y}_\beta \sembrack{S(\synbrack{\textbf{B}_\gamma})}_\gamma)\\
    &\QuasiEqual
    \sV^{\cal M}_{\phi[{\bf x}_\alpha \mapsto {\cal
          V}^{\cal M}_{\phi}({\bf A}_\alpha)]}(\lambda\textbf{y}_\beta\textbf{B}_\gamma).
    \end{align}
    (2) is by (a); (3) is by (b) and the semantics of
    $\mname{abs}_{\epsilon\epsilon\epsilon}$ and evaluation; and (4)
    is by separate arguments for the two cases ($\ast$) and
    ($\ast\ast$).  In case ($\ast$),
    \begin{align} \setcounter{equation}{0}
    &
    \sV^{\cal M}_{\phi[{\bf y}_\alpha \mapsto d]}(\sembrack{S(\synbrack{\textbf{B}_\gamma})}_\gamma)\\
    &\QuasiEqual
    \sV^{\cal M}_{\phi[{\bf y}_\alpha \mapsto d]}(\textbf{B}_\gamma)\\
    &\QuasiEqual
    \sV^{\cal M}_{\phi[{\bf y}_\alpha \mapsto d][{\bf x}_\alpha \mapsto {\cal V}^{\cal M}_{\phi}({\bf A}_\alpha)]}
    (\textbf{B}_\gamma)\\
    &\QuasiEqual
    \sV^{\cal M}_{\phi[{\bf x}_\alpha \mapsto {\cal V}^{\cal M}_{\phi}({\bf A}_\alpha)][{\bf y}_\alpha \mapsto d]}
    (\textbf{B}_\gamma)
    \end{align}
    for all $d \in \sD_\alpha$.  (2) is by ($\ast$),
    $H(\synbrack{\textbf{B}_\gamma})$, and part 5 of this lemma; (3)
    is by ($\ast$) and part 2 of Lemma~\ref{lem:not-free-in}; and (4)
    follows from $\textbf{x}_\alpha \not= \textbf{y}_\beta$.  In case
    ($\ast\ast$),
    \begin{align} \setcounter{equation}{0}
    &
    \sV^{\cal M}_{\phi[{\bf y}_\alpha \mapsto d]}(\sembrack{S(\synbrack{\textbf{B}_\gamma})}_\gamma)\\
    &\QuasiEqual
    \sV^{\cal M}_{\phi[{\bf y}_\alpha \mapsto d][{\bf x}_\alpha \mapsto {\cal V}^{\cal M}_{\phi[{\bf y}_\alpha \mapsto d]}
    ({\bf A}_\alpha)]}(\textbf{B}_\gamma)\\
    &\QuasiEqual
    \sV^{\cal M}_{\phi[{\bf y}_\alpha \mapsto d][{\bf x}_\alpha \mapsto {\cal V}^{\cal M}_{\phi}({\bf A}_\alpha)]}
    (\textbf{B}_\gamma)\\
    &\QuasiEqual
    \sV^{\cal M}_{\phi[{\bf x}_\alpha \mapsto {\cal V}^{\cal M}_{\phi}({\bf A}_\alpha)][{\bf y}_\alpha \mapsto d]}
    (\textbf{B}_\gamma)
    \end{align}
    for all $d \in \sD_\alpha$.  (2) is by
    $C(\synbrack{\textbf{B}_\gamma})$; (3) is by ($\ast\ast$) and part
    2 of Lemma~\ref{lem:not-free-in}; and (4) follows from
    $\textbf{x}_\alpha \not= \textbf{y}_\beta$.  Therefore,
    $C(\synbrack{\lambda\textbf{y}_\beta\textbf{B}_\gamma})$ holds.

  \item[] \textbf{Case 7}: $\textbf{D}_\delta$ is $\If \, \textbf{A}_o
    \, \textbf{B}_\alpha \, \textbf{C}_\alpha$.  Similar to Case 4.

  \item[] \textbf{Case 8}: $\textbf{D}_\delta$ is
    $\synbrack{\textbf{B}_\beta}$.  Similar to Case 2.

  \item[] \bsp \textbf{Case 9}: $\textbf{D}_\delta$ is
    $\sembrack{\textbf{B}_\epsilon}_\beta$.
    $H(\synbrack{\sembrack{\textbf{B}_\epsilon}_\beta})$
    implies \[(\mbox{a})~\sM \models \mname{eval-free}^{\beta}_{o\epsilon} \,
    \sembrack{S(\synbrack{\textbf{B}_\epsilon})}_\epsilon\]
    and \[(\mbox{b})~\sM \models
    S(\synbrack{\sembrack{\textbf{B}_\epsilon}_\beta}) =
    S(\sembrack{S(\synbrack{\textbf{B}_\epsilon})}_\epsilon)\] by the
    specification of $\mname{sub}_{\epsilon\epsilon\epsilon\epsilon}$.
    (a), (b), and $H(\synbrack{\sembrack{\textbf{B}_\epsilon}_\beta})$ imply
    $H(\sembrack{S(\synbrack{\textbf{B}_\epsilon})}_\epsilon)$ by part
    3 of this lemma, and so (c)~$\sM \models
    \sembrack{S(\synbrack{\textbf{B}_\epsilon})}_\epsilon =
    \synbrack{\textbf{C}_\beta}$ for some evaluation-free
    $\textbf{C}_\beta$.
    $H(\sembrack{S(\synbrack{\textbf{B}_\epsilon})}_\epsilon)$ implies
    $H(\synbrack{\textbf{B}_\epsilon})$ by the specification of
    $\mname{sub}_{\epsilon\epsilon\epsilon\epsilon}$.
    $H(\synbrack{\textbf{B}_\epsilon})$ implies
    $C(\synbrack{\textbf{B}_\epsilon})$ by the induction hypothesis,
    and so \[(\mbox{d})~\sV^{\cal
      M}_{\phi}(\sembrack{S(\synbrack{\textbf{B}_\epsilon})}_\epsilon)
    \QuasiEqual \sV^{\cal M}_{\phi[{\bf x}_\alpha \mapsto {\cal
          V}^{\cal M}_{\phi}({\bf A}_\alpha)]}(\textbf{B}_\epsilon).\]
    (c) and (d) imply \[(\mbox{e})~\sV^{\cal M}_{\phi[{\bf x}_\alpha \mapsto
        {\cal V}^{\cal M}_{\phi}({\bf
          A}_\alpha)]}(\textbf{B}_\epsilon) = \sV^{\cal
      M}_{\phi}(\synbrack{\textbf{C}_\beta}) = \sV^{\cal M}_{\phi[{\bf
          x}_\alpha \mapsto {\cal V}^{\cal M}_{\phi}({\bf
          A}_\alpha)]}(\synbrack{\textbf{C}_\beta})\] since
    $\synbrack{\textbf{C}_\beta}$ is semantically closed.
    $H(\sembrack{S(\synbrack{\textbf{B}_\epsilon})}_\epsilon)$ and (c)
    imply $H(\synbrack{\textbf{C}_\beta})$, and
    $H(\synbrack{\textbf{C}_\beta})$ implies
    $C(\synbrack{\textbf{C}_\beta})$ by the inductive hypothesis.
    \begin{align} \setcounter{equation}{0}
    &
    \sV^{\cal M}_{\phi}(\sembrack{S(\synbrack{\sembrack{\textbf{B}_\epsilon}_\beta})}_\beta)\\
    &\QuasiEqual
    \sV^{\cal M}_{\phi}(\sembrack{S(\sembrack{S(\synbrack{\textbf{B}_\epsilon})}_\epsilon)}_\beta)\\
    &\QuasiEqual
    \sV^{\cal M}_{\phi}(\sembrack{S(\synbrack{\textbf{C}_\beta})}_\beta)\\ 
    &\QuasiEqual
    \sV^{\cal M}_{\phi[{\bf x}_\alpha\mapsto {\cal V}^{\cal M}_{\phi} ({\bf A}_\alpha)]}(\textbf{C}_\beta)\\
    &\QuasiEqual
    \sV^{\cal M}_{\phi[{\bf x}_\alpha\mapsto {\cal V}^{\cal M}_{\phi} ({\bf A}_\alpha)]}
    (\sembrack{\synbrack{\textbf{C}_\beta}}_\beta)\\
    &\QuasiEqual
    \sV^{\cal M}_{\phi[{\bf x}_\alpha\mapsto {\cal V}^{\cal M}_{\phi} ({\bf A}_\alpha)]}
    (\sembrack{\textbf{B}_\epsilon}_\beta).
    \end{align}
    (2) is by (b); (3) is by (d) and (e); (4) is by
    $C(\synbrack{\textbf{C}_\beta})$; (5) is by the semantics of
    evaluation and the fact $\textbf{C}_\beta$ is evaluation-free; and
    (6) is by (e).  Therefore,
    $C(\synbrack{\sembrack{\textbf{B}_\epsilon}_\beta})$ holds.\esp

\ee

\medskip

\noindent \textbf{Part 7} {\sglsp} Follows from the specifications of
$\mname{not-free-in}_{o\epsilon\epsilon}$,
$\mname{cleanse}_{\epsilon\epsilon}$, and
$\mname{sub}_{\epsilon\epsilon\epsilon\epsilon}$.

\end{proof}

\bigskip

The five requirements for
$\mname{sub}_{\epsilon\epsilon\epsilon\epsilon}$ are satisfied as
follows:

\be 

  \item \bsp Requirement 1 is satisfied by Specification 9 for
    $\mname{sub}_{\epsilon\epsilon\epsilon\epsilon}$.  Part~6 of
    Lemma~\ref{lem:sub} verifies that
    $\mname{sub}_{\epsilon\epsilon\epsilon\epsilon}$ performs
    substitution correctly.\esp

  \item \bsp Requirement 2 is satisfied by Specification 7 for
    $\mname{not-free-in}_{o\epsilon\epsilon}$ and Specification 9.6
    for $\mname{sub}_{\epsilon\epsilon\epsilon\epsilon}$.  Part~6 of
    Lemma~\ref{lem:sub} verifies that, when
    $\mname{sub}_{\epsilon\epsilon\epsilon\epsilon}
    \synbrack{\textbf{A}_\alpha} \, \synbrack{\textbf{x}_\alpha} \,
    \synbrack{\textbf{B}_\beta}$ is defined, variables are not
    captured.\esp

  \item Requirement 3 is satisfied by Specification 8 for
    $\mname{cleanse}_{\epsilon\epsilon}$ and Specifications 9.1, 9.5,
    and 9.9 for $\mname{sub}_{\epsilon\epsilon\epsilon\epsilon}$.
    Part 1 of Lemma~\ref{lem:sub} verifies that, when
    $\mname{sub}_{\epsilon\epsilon\epsilon\epsilon}
    \synbrack{\textbf{A}_\alpha} \, \synbrack{\textbf{x}_\alpha} \,
    \synbrack{\textbf{B}_\beta}$ is defined, it represents an
    evaluation-free $\wff{\beta}$.

  \item Requirement 4 is satisfied by Specification 7.9 for
    $\mname{not-free-in}_{o\epsilon\epsilon}$, Specification 8.7 for
    $\mname{cleanse}_{\epsilon\epsilon}$, and Specification 9.9 for
    $\mname{sub}_{\epsilon\epsilon\epsilon\epsilon}$.  Part 1 of
    Lemma~\ref{lem:sub} verifies that, when
    $\mname{sub}_{\epsilon\epsilon\epsilon\epsilon}
    \synbrack{\textbf{A}_\alpha} \, \synbrack{\textbf{x}_\alpha} \,
    \synbrack{\textbf{B}_\beta}$ is defined, it is semantically
    closed.

   \item Requirement 5 is satisfied by Specifications 7--9.  Part~7 of
     Lemma~\ref{lem:sub} verifies that
     $\mname{sub}_{\epsilon\epsilon\epsilon\epsilon}
     \synbrack{\textbf{A}_\alpha} \, \synbrack{\textbf{x}_\alpha} \,
     \synbrack{\textbf{B}_\beta}$ is defined in the cases
     corresponding to when substitution is defined in {\qzero}.

\ee

As a consequence of $\mname{sub}_{\epsilon\epsilon\epsilon\epsilon}$
satisfying Requirements 1--5, we can now prove that the law of
beta-reduction for {\qzerouqe} is valid in {\qzerouqe}:

\begin{thm}[Law of Beta-Reduction]\label{thm:beta-red-sub}\bsp
Let $\sM$ be a normal general model for {\qzerouqe}.  Then
\[\sM \models {[{\textbf{A}_\alpha \IsDefApp} \Andd
\mname{sub}_{\epsilon\epsilon\epsilon\epsilon} \,
\synbrack{\textbf{A}_\alpha} \, \synbrack{\textbf{x}_\alpha} \, \synbrack{\textbf{B}_\beta} = 
\synbrack{\textbf{C}_\beta}]} \Implies 
[\lambda \textbf{x}_\alpha \textbf{B}_\beta]\textbf{A}_\alpha \QuasiEqual \textbf{C}_\beta.\]
\esp
\end{thm}

\begin{proof} 
Let $\phi \in \mname{assign}(\sM)$.  Assume (a) $\sV^{\cal
  M}_{\phi}(\textbf{A}_\alpha)$ is defined and \[(\mbox{b})~\sV^{\cal
  M}_{\phi}(\mname{sub}_{\epsilon\epsilon\epsilon\epsilon} \,
\synbrack{\textbf{A}_\alpha} \, \synbrack{\textbf{x}_\alpha} \,
\synbrack{\textbf{B}_\beta} = \synbrack{\textbf{C}_\beta}) =
\mname{T}.\] We must show
\[\sV^{\cal M}_{\phi}([\lambda \textbf{x}_\alpha
    \textbf{B}_\beta]\textbf{A}_\alpha) \QuasiEqual \sV^{\cal
  M}_{\phi}(\textbf{C}_\beta).\] (b) implies \[(\mbox{c})~\sM \models
\mname{sub}_{\epsilon\epsilon\epsilon\epsilon} \,
\synbrack{\textbf{A}_\alpha} \, \synbrack{\textbf{x}_\alpha} \,
\synbrack{\textbf{B}_\beta} = \synbrack{\textbf{C}_\beta}\] and (d)
$\textbf{C}_\beta$ is evaluation-free by part 1 of
Lemma~\ref{lem:sub}.  Then
\begin{align} \setcounter{equation}{0}
& 
\sV^{\cal  M}_{\phi}([\lambda \textbf{x}_\alpha \textbf{B}_\beta]\textbf{A}_\alpha) \\
&\QuasiEqual
\sV^{\cal M}_{\phi[{\bf x}_\alpha \mapsto {\cal V}^{\cal M}_{\phi}({\bf A}_\alpha)]}(\textbf{B}_\beta) \\
&\QuasiEqual
\sV^{\cal M}_{\phi}(\sembrack{\mname{sub}_{\epsilon\epsilon\epsilon\epsilon} \,
\synbrack{\textbf{A}_\alpha} \, \synbrack{\textbf{x}_\alpha} \, \synbrack{\textbf{B}_\beta}}_\beta) \\
&\QuasiEqual
\sV^{\cal M}_{\phi}(\sembrack{\synbrack{\textbf{C}_\beta}}_\beta) \\
&\QuasiEqual
\sV^{\cal M}_{\phi}(\textbf{C}_\beta).
\end{align}
(2) is by (a) and the semantics of function application and function
abstraction; (3) is by (c) and part 6 of Lemma~\ref{lem:sub}; (4) is
by (b); and (5) is by (d) and the semantics of evaluation.  Therefore,
$\sV^{\cal M}_{\phi}([\lambda \textbf{x}_\alpha
  \textbf{B}_\beta]\textbf{A}_\alpha) \QuasiEqual \sV^{\cal
  M}_{\phi}(\textbf{C}_\beta)$.
\end{proof}

\subsection{Example: Double Substitution}

We mentioned above that
$\mname{sub}_{\epsilon\epsilon\epsilon\epsilon} \,
\synbrack{\textbf{A}_\alpha} \, \synbrack{\textbf{x}_\alpha} \,
\synbrack{\textbf{B}_\beta}$ may involve a ``double substitution''
when $\textbf{B}_\beta$ is an evaluation.  The following example
explores this possibility when $\textbf{B}_\beta$ is the simple
evaluation $\sembrack{\textbf{x}_{\epsilon}}_o$.

Let $\sM$ be any normal general model for {\qzerouqe}, $\phi \in
\mname{assign}(\sM)$, and $\textbf{A}_{o}$ be an evaluation-free wff
in which $\textbf{x}_{\epsilon}$ is not free.  Then
\begin{align} \setcounter{equation}{0}
& 
\sV^{\cal M}_{\phi}(\mname{sub}_{\epsilon\epsilon\epsilon\epsilon} \,
\synbrack{\synbrack{\textbf{A}_{o}}} \, 
\synbrack{\textbf{x}_{\epsilon}} \,
\synbrack{\sembrack{\textbf{x}_{\epsilon}}_o}) \\
&=
\sV^{\cal M}_{\phi}(\mname{sub}_{\epsilon\epsilon\epsilon\epsilon} \,
\synbrack{\synbrack{\textbf{A}_{o}}} \, 
\synbrack{\textbf{x}_{\epsilon}} \,
\sembrack{\mname{sub}_{\epsilon\epsilon\epsilon\epsilon} \,
\synbrack{\synbrack{\textbf{A}_{o}}} \, 
\synbrack{\textbf{x}_{\epsilon}} \,
\synbrack{\textbf{x}_{\epsilon}}}_\epsilon) \\
&=
\sV^{\cal M}_{\phi}(\mname{sub}_{\epsilon\epsilon\epsilon\epsilon} \,
\synbrack{\synbrack{\textbf{A}_{o}}} \, 
\synbrack{\textbf{x}_{\epsilon}} \,
\sembrack{\synbrack{\synbrack{\textbf{A}_{o}}}}_\epsilon) \\
&=
\sV^{\cal M}_{\phi}(\mname{sub}_{\epsilon\epsilon\epsilon\epsilon} \,
\synbrack{\synbrack{\textbf{A}_{o}}} \, 
\synbrack{\textbf{x}_{\epsilon}} \,
\synbrack{\textbf{A}_{o}}) \\
&=
\sV^{\cal M}_{\phi}(\synbrack{\textbf{A}_{o}}).
\end{align}
(2) is by the specification of
$\mname{sub}_{\epsilon\epsilon\epsilon\epsilon}$, the fact that
$\synbrack{\textbf{A}_{o}}$ is syntactically closed, and the fact that
$\textbf{A}_{o}$ is evaluation-free; (3) is by the specification of
$\mname{sub}_{\epsilon\epsilon\epsilon\epsilon}$; (4) is by the
semantics of evaluation and the fact that $\synbrack{\textbf{A}_{o}}$
is evaluation-free; and (5) is by the specification of
$\mname{sub}_{\epsilon\epsilon\epsilon\epsilon}$ and the fact that
$\textbf{x}_{\epsilon}$ is not free in $\textbf{A}_{o}$.  Therefore,
\[\sM \models \mname{sub}_{\epsilon\epsilon\epsilon\epsilon} \,
\synbrack{\synbrack{\textbf{A}_{o}}} \,
\synbrack{\textbf{x}_{\epsilon}} \,
\synbrack{\sembrack{\textbf{x}_{\epsilon}}_o}) =
\synbrack{\textbf{A}_{o}}\] and only the first substitution has an
effect.

Now consider the evaluation-free wff $\textbf{x}_{\epsilon} =
\textbf{x}_{\epsilon}$ (in which the variable $\textbf{x}_{\epsilon}$
is free).  Then
\begin{align} \setcounter{equation}{0}
& 
\sV^{\cal M}_{\phi}(\mname{sub}_{\epsilon\epsilon\epsilon\epsilon} \,
\synbrack{\synbrack{\textbf{x}_{\epsilon} = \textbf{x}_{\epsilon}}} \, 
\synbrack{\textbf{x}_{\epsilon}} \,
\synbrack{\sembrack{\textbf{x}_{\epsilon}}_o}) \\
&=
\sV^{\cal M}_{\phi}(\mname{sub}_{\epsilon\epsilon\epsilon\epsilon} \,
\synbrack{\synbrack{\textbf{x}_{\epsilon} = \textbf{x}_{\epsilon}}} \, 
\synbrack{\textbf{x}_{\epsilon}} \,
\sembrack{\mname{sub}_{\epsilon\epsilon\epsilon\epsilon} \,
\synbrack{\synbrack{\textbf{x}_{\epsilon} = \textbf{x}_{\epsilon}}} \, 
\synbrack{\textbf{x}_{\epsilon}} \,
\synbrack{\textbf{x}_{\epsilon}}}_\epsilon) \\
&=
\sV^{\cal M}_{\phi}(\mname{sub}_{\epsilon\epsilon\epsilon\epsilon} \,
\synbrack{\synbrack{\textbf{x}_{\epsilon} = \textbf{x}_{\epsilon}}} \, 
\synbrack{\textbf{x}_{\epsilon}} \,
\sembrack{\synbrack{\synbrack{\textbf{x}_{\epsilon} = \textbf{x}_{\epsilon}}}}_\epsilon) \\
&=
\sV^{\cal M}_{\phi}(\mname{sub}_{\epsilon\epsilon\epsilon\epsilon} \,
\synbrack{\synbrack{\textbf{x}_{\epsilon} = \textbf{x}_{\epsilon}}} \, 
\synbrack{\textbf{x}_{\epsilon}} \,
\synbrack{\textbf{x}_{\epsilon} = \textbf{x}_{\epsilon}}) \\
&=
\sV^{\cal M}_{\phi}(\synbrack{\synbrack{\textbf{x}_{\epsilon} = \textbf{x}_{\epsilon}}
= 
\synbrack{\textbf{x}_{\epsilon} = \textbf{x}_{\epsilon}}}).
\end{align}
(1)--(4) are by the same reasoning as above, and (5) is by the
specification of $\mname{sub}_{\epsilon\epsilon\epsilon\epsilon}$.
Therefore,
\[\sM \models \mname{sub}_{\epsilon\epsilon\epsilon\epsilon} \,
\synbrack{\synbrack{\textbf{x}_{\epsilon} = \textbf{x}_{\epsilon}}} \,
\synbrack{\textbf{x}_{\epsilon}} \,
\synbrack{\sembrack{\textbf{x}_{\epsilon}}_o} =
\synbrack{\synbrack{\textbf{x}_{\epsilon} = \textbf{x}_{\epsilon}} =
  \synbrack{\textbf{x}_{\epsilon} = \textbf{x}_{\epsilon}}}\] and both
substitutions have an effect. 

\subsection{Example: Variable Renaming}

In predicate logics like {\qzero}, bound variables can be renamed in a
wff (in certain ways) without changing the meaning the wff.  For
example, when the variable $x_\alpha$ is renamed to the variable
$y_\alpha$ (or any other variable of type $\alpha$) in the
evaluation-free wff $\lambda x_\alpha x_\alpha$, the result is the wff
$\lambda y_\alpha y_\alpha$.  $\lambda x_\alpha x_\alpha$ and $\lambda
y_\alpha y_\alpha$ are logically equivalent to each other, i.e., \[\sM
\models \lambda x_\alpha x_\alpha = \lambda y_\alpha y_\alpha.\] In
fact, a variable renaming that permutes the names of the variables
occurring in an evaluation-free wff of {\qzerouqe} without changing
the names of the wff's free variables preserves the meaning of the
wff.

\bsp Unfortunately, meaning-preserving variable renamings do not exist for
all the non-evaluation-free wffs of {\qzerouqe}.  As an example,
consider the two non-evaluation-free wffs $\lambda x_\epsilon
\sembrack{x_\epsilon}_{\seq{\epsilon\epsilon}}$ and $\lambda
y_\epsilon \sembrack{y_\epsilon}_{\seq{\epsilon\epsilon}}$ where
$x_\epsilon$ and $y_\epsilon$ are distinct variables.  Obviously,
$\lambda y_\epsilon \sembrack{y_\epsilon}_{\seq{\epsilon\epsilon}}$ is obtained from
$\lambda x_\epsilon \sembrack{x_\epsilon}_{\seq{\epsilon\epsilon}}$ by renaming
$x_\epsilon$ to be $y_\epsilon$.  If we forget about evaluation, we
would expect that $\lambda x_\epsilon
\sembrack{x_\epsilon}_{\seq{\epsilon\epsilon}}$ and $\lambda
y_\epsilon \sembrack{y_\epsilon}_{\seq{\epsilon\epsilon}}$ are
logically equivalent --- but they are not!  Let $\textbf{A}_\epsilon$
be $\synbrack{\mname{pair}_{\seq{\epsilon\epsilon}\epsilon\epsilon} \,
  x_\epsilon \, y_\epsilon}$, and suppose $\phi(x_\epsilon) =
\sE(x_\epsilon)$ and $\phi(y_\epsilon) = \sE(y_\epsilon)$.  Then
\begin{align*}
&
\sV^{\cal M}_{\phi}([\lambda
x_\epsilon \sembrack{x_\epsilon}_{\seq{\epsilon\epsilon}}]\textbf{A}_\epsilon) \\
&\QuasiEqual
\sV^{\cal M}_{\phi[x_\epsilon \mapsto {\cal V}_{\phi[{\bf A}_\epsilon]}]}
(\sembrack{x_\epsilon}_{\seq{\epsilon\epsilon}}) \\
&\QuasiEqual
\sV^{\cal M}_{\phi[x_\epsilon \mapsto {\cal V}_{\phi[{\bf A}_\epsilon]}]}
(\sE^{-1}(\sV^{\cal M}_{\phi[x_\epsilon \mapsto {\cal V}_{\phi[{\bf A}_\epsilon]}]}(x_\epsilon))) \\
&\QuasiEqual
\sV^{\cal M}_{\phi[x_\epsilon \mapsto {\cal V}_{\phi[{\bf A}_\epsilon]}]}(\sE^{-1}
(\sE(\mname{pair}_{\seq{\epsilon\epsilon}\epsilon\epsilon} \, x_\epsilon \, y_\epsilon))) \\
&\QuasiEqual
\sV^{\cal M}_{\phi[x_\epsilon \mapsto {\cal V}_{\phi[{\bf A}_\epsilon]}]}
(\mname{pair}_{\seq{\epsilon\epsilon}\epsilon\epsilon} \, x_\epsilon \, y_\epsilon) \\
&=
\seq{\sE(\mname{pair}_{\seq{\epsilon\epsilon}\epsilon\epsilon} \, x_\epsilon \, y_\epsilon),\sE(y_\epsilon)}.
\end{align*}
Similarly, 
\[\sV^{\cal M}_{\phi}([\lambda y_\epsilon \sembrack{y_\epsilon}_{\seq{\epsilon\epsilon}}]\textbf{A}_\epsilon)
\QuasiEqual
\seq{\sE(x_\epsilon),\sE(\mname{pair}_{\seq{\epsilon\epsilon}\epsilon\epsilon}
  \, x_\epsilon \, y_\epsilon)}.\] Therefore, $\lambda x_\epsilon
\sembrack{x_\epsilon}_{\seq{\epsilon\epsilon}}$ and $\lambda
y_\epsilon \sembrack{y_\epsilon}_{\seq{\epsilon\epsilon}}$ are not
logically equivalent, but the functions $\sV^{\cal M}_{\phi}(\lambda
x_\epsilon \sembrack{x_\epsilon}_{\seq{\epsilon\epsilon}})$ and
$\sV^{\cal M}_{\phi}(\lambda y_\epsilon
\sembrack{y_\epsilon}_{\seq{\epsilon\epsilon}})$ are equal on
constructions of the form $\sE(B_{\seq{\epsilon\epsilon}})$ where
$B_{\seq{\epsilon\epsilon}}$ is semantically closed.\esp

This example proves the following proposition:

\begin{prop}
\bsp Alpha-conversion is not valid in {\qzerouqe} for some
non-evaluation-free wffs.\esp
\end{prop}

\begin{note}[Nominal Data Types]\em
Since alpha-conversion is not universally valid in {\qzerouqe}, it is
not clear whether techniques for managing variable naming and binding
--- such as \emph{higher-order abstract
  syntax}~\cite{Miller00,PfenningElliot88} and \emph{nominal
  techniques}~\cite{GabbayPitts02,Pitts03} --- are applicable to
{\qzerouqe}.  However, the paper~\cite{NanevskiPfenning05} does
combine quotation/evaluation techniques with nominal techniques.
\end{note}

\subsection{Limitations of $\mname{sub}_{\epsilon\epsilon\epsilon\epsilon}$}\label{subsec:undef-sub}

Theorem~\ref{thm:beta-red-sub} shows beta-reduction can be computed
using $\mname{sub}_{\epsilon\epsilon\epsilon\epsilon}$.  However, it
is obviously not possible to use
$\mname{sub}_{\epsilon\epsilon\epsilon\epsilon}$ to compute a
beta-reduction when the corresponding application of
$\mname{sub}_{\epsilon\epsilon\epsilon\epsilon}$ is undefined.  There
are thus two questions that concern us:

\be

  \item When is an application of
    $\mname{sub}_{\epsilon\epsilon\epsilon\epsilon}$ undefined?

  \item When an application of
    $\mname{sub}_{\epsilon\epsilon\epsilon\epsilon}$ is undefined, is
    the the corresponding beta-reduction ever valid in {\qzerouqe}.

\ee

Let $\sM$ be a normal general model for {\qzerouqe}.  There are two
cases in which $\sM \models
[\mname{sub}_{\epsilon\epsilon\epsilon\epsilon}
  \synbrack{\textbf{A}_\alpha} \, \synbrack{\textbf{x}_\alpha} \,
  \synbrack{\textbf{B}_\beta}] \IsUndefApp$ will be true.  The first
case occurs when the naive substitution of $\textbf{A}_\alpha$ for the
free occurrences of $\textbf{x}_\alpha$ in $\textbf{B}_\beta$ causes a
variable capture.  In this case the corresponding beta-reduction is
not valid unless the bound variables in $\textbf{B}_\beta$ are renamed
so that the variable capture is avoided.  This can always be done if
$\textbf{B}_\beta$ is evaluation-free, but as we showed in the
previous subsection it is not always possible to rename variables in a
non-evaluation-free wff.

The second case in which $\sM \models
[\mname{sub}_{\epsilon\epsilon\epsilon\epsilon}
  \synbrack{\textbf{A}_\alpha} \, \synbrack{\textbf{x}_\alpha} \,
  \synbrack{\textbf{B}_\beta}] \IsUndefApp$ will be true occurs when
the naive cleansing of evaluations in the result of the substitution
causes a variable to escape outside of a quotation.  This happens when
the body of an evaluation is not semantically closed after the first
substitution.  In this case, the corresponding beta-reduction may be
valid.  We will illustrate this possibility with three examples.

\subsubsection*{Example 1}

Let $\textbf{A}_{\alpha\epsilon\epsilon}$ be the wff
\[\lambda x_\epsilon \lambda y_\epsilon 
\sembrack{\mname{app}_{\epsilon\epsilon\epsilon} \, x_\epsilon \,
  y_\epsilon}_\alpha\] and $\textbf{B}_{\alpha\beta}$ and
$\textbf{C}_\beta$ be syntactically closed evaluation-free wffs.  Then
\[\sM \models \textbf{A}_{\alpha\epsilon\epsilon}
\synbrack{\textbf{B}_{\alpha\beta}}\synbrack{\textbf{C}_\beta}
\QuasiEqual \textbf{B}_{\alpha\beta} \textbf{C}_\beta.\] However, we
also have
\[\sM \models [\mname{sub}_{\epsilon\epsilon\epsilon\epsilon} \,
\synbrack{\textbf{B}_{\alpha\beta}} \, \synbrack{x_\epsilon} \,
\synbrack{\textbf{A}_{\alpha\epsilon\epsilon}}] \IsUndefApp\] 
since the body of the evaluation contains $y_\epsilon$ 
after the first substitution.  Hence the beta-reduction of
$\textbf{A}_{\alpha\epsilon\epsilon}
\synbrack{\textbf{B}_{\alpha\beta}}\synbrack{\textbf{C}_\beta}$ is
valid in {\qzerouqe}, but the corresponding application of
$\mname{sub}_{\epsilon\epsilon\epsilon\epsilon}$ is undefined.

This is a significant limitation.  It means, for instance, that using
$\mname{sub}_{\epsilon\epsilon\epsilon\epsilon}$ we cannot instantiate
a formula with more than one variable within an evaluation (not in a
quotation).  An instance of specification 9.9 where the syntactic
variables are replaced with variables is an example of a formula with
this property.

In some cases this limitation can be overcome by instantiating all the
variables of type $\epsilon$ within an evaluation together as a group.
For example, let $\textbf{A}'_{\alpha\epsilon\epsilon}$ be the wff 
$\lambda x_{\seq{\epsilon\epsilon}}\textbf{D}_\alpha$ where $\textbf{D}_\alpha$ is
\[\sembrack{\mname{app}_{\epsilon\epsilon\epsilon} \, 
[\mname{fst}_{\epsilon\seq{\epsilon\epsilon}} \, x_{\seq{\epsilon\epsilon}}]
[\mname{snd}_{\epsilon\seq{\epsilon\epsilon}} \, x_{\seq{\epsilon\epsilon}}]}_\alpha.\]
Then
\begin{align*}
& 
\sV^{\cal M}_{\phi}(\mname{sub}_{\epsilon\epsilon\epsilon\epsilon} \,
\synbrack{\mname{pair}_{\seq{\epsilon\epsilon}\epsilon\epsilon} \, 
\synbrack{\textbf{B}_{\alpha\beta}} \, \synbrack{\textbf{C}_\beta}} \,
\synbrack{x_{\seq{\epsilon\epsilon}}} \, \synbrack{\textbf{D}_\alpha}) \\
&\QuasiEqual
\sV^{\cal M}_{\phi}(\mname{app}_{\epsilon\epsilon\epsilon} \,
[\mname{fst}_{\epsilon\seq{\epsilon\epsilon}} \, 
[\mname{pair}_{\seq{\epsilon\epsilon}\epsilon\epsilon} \, 
\synbrack{\textbf{B}_{\alpha\beta}} \, \synbrack{\textbf{C}_\beta}]] \, \\
&\hspace{14.3ex}
[\mname{snd}_{\epsilon\seq{\epsilon\epsilon}} \, 
[\mname{pair}_{\seq{\epsilon\epsilon}\epsilon\epsilon} \, 
\synbrack{\textbf{B}_{\alpha\beta}} \, \synbrack{\textbf{C}_\beta}]]) \\
&\QuasiEqual
\sV^{\cal M}_{\phi}(\mname{app}_{\epsilon\epsilon\epsilon} \,
\synbrack{\textbf{B}_{\alpha\beta}} \, \synbrack{\textbf{C}_\beta}) \\
&\QuasiEqual
\sV^{\cal M}_{\phi}(\synbrack{\textbf{B}_{\alpha\beta} \textbf{C}_\beta})
\end{align*}
for all $\phi \in \mname{assign}(\sM)$.  Hence
\begin{align*}
&
\sM \models \mname{sub}_{\epsilon\epsilon\epsilon\epsilon} \,
\synbrack{\mname{pair}_{\seq{\epsilon\epsilon}\epsilon\epsilon} \,
\synbrack{\textbf{B}_{\alpha\beta}} \, \synbrack{\textbf{C}_\beta}} \, \synbrack{x_\epsilon} \,
\synbrack{\textbf{D}_\alpha} = \synbrack{\textbf{B}_{\alpha\beta} \textbf{C}_\beta},
\end{align*}
and so by Theorem~\ref{thm:beta-red-sub}
\[\sM \models \textbf{A}'_{\alpha\epsilon\epsilon} \,
[\mname{pair}_{\seq{\epsilon\epsilon}\epsilon\epsilon} \,
\synbrack{\textbf{B}_{\alpha\beta}} \, \synbrack{\textbf{C}_\beta}] \QuasiEqual
\textbf{B}_{\alpha\beta} \textbf{C}_\beta.\]
The main reason we have introduced pairs in {\qzerouqe} is to allow us
to express function abstractions like
$\textbf{A}_{\alpha\epsilon\epsilon}$ in a form like
$\textbf{A}'_{\alpha\epsilon\epsilon}$ that can be beta-reduced using
$\mname{sub}_{\epsilon\epsilon\epsilon\epsilon}$.

\subsubsection*{Example 2}

Let $\textbf{C}_\alpha$ be the wff $[\lambda x_\epsilon
  x_\epsilon]\sembrack{x_\epsilon}_\alpha$.  Then
\[\sM \models [\lambda x_\epsilon x_\epsilon]\sembrack{x_\epsilon}_\alpha \QuasiEqual
\sembrack{x_\epsilon}_\alpha\]
but 
\[\sM \models [\mname{sub}_{\epsilon\epsilon\epsilon\epsilon} \,
\synbrack{\sembrack{x_\epsilon}_\alpha} \, \synbrack{x_\epsilon} \,
\synbrack{x_\epsilon}] \IsUndefApp\] 
since 
$\sM \models [\mname{cleanse}_{\epsilon\epsilon} \, \sembrack{x_\epsilon}_\alpha] \IsUndefApp$.  
We will overcome this limitation of
$\mname{sub}_{\epsilon\epsilon\epsilon\epsilon}$ by including 
\[[\lambda x_\epsilon x_\epsilon]\textbf{A}_\alpha \QuasiEqual \textbf{A}_\alpha\]
and the other basic properties of lambda-notation in the axioms of
{\pfsysuqe}.  These properties will be presented as schemas
similar to Axioms $4_1$--$4_5$ in~\cite{Andrews02}.

\subsubsection*{Example 3}

Let $\textbf{C}_\alpha$ be the wff $[\lambda x_\epsilon
\sembrack{x_\epsilon}_\alpha] x_\epsilon$.  Then
\[\sM \models [\lambda x_\epsilon \sembrack{x_\epsilon}_\alpha] x_\epsilon \QuasiEqual
\sembrack{x_\epsilon}_\alpha\]
but 
\[\sM \models [\mname{sub}_{\epsilon\epsilon\epsilon\epsilon} \,
\synbrack{x_\epsilon} \, \synbrack{x_\epsilon} \,
\synbrack{\sembrack{x_\epsilon}_\alpha}] \IsUndefApp\] 
since the body of the evaluation contains $x_\epsilon$ 
after the first substitution.  We will overcome this limitation of
$\mname{sub}_{\epsilon\epsilon\epsilon\epsilon}$ by including 
\[[\lambda x_\alpha \textbf{B}_\beta] x_\alpha \QuasiEqual \textbf{B}_\alpha\]
 in the axioms of {\pfsysuqe}.

\section{Proof System} \label{sec:proofsystem}

Now that we have defined a mechanism for substitution, we are ready to
present the proof system of {\qzerouqe} called {\pfsysuqe}.  It is
derived from {\pfsysu}, the proof system of {\qzerou}.  The presence
of undefinedness makes {\pfsysu} moderately more complicated than
$\sP$, the proof system of {\qzero}, but the presence of the type
$\epsilon$ and quotation and evaluation makes {\pfsysuqe}
significantly more complicated than {\pfsysu}.  A large part of the
complexity of {\qzerouqe} is due to the difficulty of beta-reducing
wffs that involve evaluations.

\subsection{Axioms}

{\pfsysuqe} consists of a set of axioms and a set of rules of
inference.  The axioms are given in this section, while the rules of
inference are given in the next section.  The axioms are organized
into groups.  The members of each group are presented using one or
more formula schemas.  A group is called an ``Axiom'' even though it
consists of infinitely many formulas.  

\bi

  \item[] \textbf{Axiom 1 (Truth Values)}

  \be

    \item[] $[\textbf{G}_{oo}T_o \Andd \textbf{G}_{oo}F_o] \Iff 
    \Forall \textbf{x}_o [\textbf{G}_{oo} \textbf{x}_o].$

  \ee

  \item[] \textbf{Axiom 2 (Leibniz' Law)} 

  \be

    \item[] $\textbf{A}_\alpha = \textbf{B}_\alpha \Implies [\textbf{H}_{o\alpha}\textbf{A}_\alpha \Iff
      \textbf{H}_{o\alpha}\textbf{B}_\alpha].$

  \ee

  \item[] \textbf{Axiom 3 (Extensionality)}

  \be

    \item[] $[{\textbf{F}_{\alpha\beta} \IsDefApp} \Andd {\textbf{G}_{\alpha\beta} \IsDefApp}] \Implies
      \textbf{F}_{\alpha\beta} = \textbf{G}_{\alpha\beta} \Iff 
      \Forall \textbf{x}_\beta 
      [\textbf{F}_{\alpha\beta}\textbf{x}_\beta \QuasiEqual \textbf{G}_{\alpha\beta}\textbf{x}_\beta].$

  \ee

  \item[] \textbf{Axiom 4 (Beta-Reduction)}

  \be

    \item ${[{\textbf{A}_\alpha \IsDefApp} \Andd
      \mname{sub}_{\epsilon\epsilon\epsilon\epsilon} \,
      \synbrack{\textbf{A}_\alpha} \, \synbrack{\textbf{x}_\alpha} \, \synbrack{\textbf{B}_\beta} = 
      \synbrack{\textbf{C}_\beta}]} \Implies 
      [\lambda \textbf{x}_\alpha \textbf{B}_\beta]\textbf{A}_\alpha \QuasiEqual \textbf{C}_\beta.$

    \vspace{1ex}

    \item $[\lambda \textbf{x}_\alpha \textbf{x}_\alpha]\textbf{A}_\alpha \QuasiEqual \textbf{A}_\alpha$.

    \vspace{1ex}

    \item ${\textbf{A}_\alpha \IsDefApp} \Implies 
      [\lambda \textbf{x}_\alpha \textbf{y}_\beta]\textbf{A}_\alpha \QuasiEqual \textbf{y}_\beta$ 
      {\sglsp} where $\textbf{x}_\alpha \not= \textbf{y}_\beta$.

    \vspace{1ex}

    \item ${\textbf{A}_\alpha \IsDefApp} \Implies 
      [\lambda \textbf{x}_\alpha \textbf{c}_\beta]\textbf{A}_\alpha \QuasiEqual \textbf{c}_\alpha$ 
      {\sglsp} where $\textbf{c}_\beta$ is a primitive constant.

    \vspace{1ex}

    \item $[\lambda \textbf{x}_\alpha [\textbf{B}_{\alpha\beta}\textbf{C}_\beta]]\textbf{A}_\alpha \QuasiEqual
          [[\lambda \textbf{x}_\alpha \textbf{B}_{\alpha\beta}]\textbf{A}_\alpha]
          [[\lambda \textbf{x}_\alpha \textbf{C}_\beta]\textbf{A}_\alpha]$.

    \vspace{1ex}

    \item ${\textbf{A}_\alpha \IsDefApp} \Implies
          [\lambda \textbf{x}_\alpha [\lambda \textbf{x}_\alpha \textbf{B}_\beta]]\textbf{A}_\alpha =
          \lambda \textbf{x}_\alpha \textbf{B}_\beta.$

    \vspace{1ex}

    \item ${\textbf{A}_\alpha \IsDefApp} \Andd 
      [\mname{not-free-in}_{o\epsilon\epsilon} \, \synbrack{\textbf{x}_\alpha} \,
      \synbrack{\textbf{B}_\gamma} \Or 
      \mname{not-free-in}_{o\epsilon\epsilon} \, \synbrack{\textbf{y}_\beta} \,
      \synbrack{\textbf{A}_\alpha}] \Implies {}$ \\[.5ex]
      \hspace*{4ex}
      $[\lambda \textbf{x}_\alpha [\lambda \textbf{y}_\beta \textbf{B}_\gamma]]\textbf{A}_\alpha =
      \lambda \textbf{y}_\beta [[\lambda \textbf{x}_\alpha\textbf{B}_\gamma]\textbf{A}_\alpha]$
      {\sglsp} where $\textbf{x}_\alpha \not= \textbf{y}_\alpha$.

    \vspace{1ex}

    \item $[\lambda \textbf{x}_\alpha 
      [\If \, \textbf{B}_o \, \textbf{C}_\beta \, \textbf{D}_\beta]]\textbf{A}_\alpha \QuasiEqual
      \If \, 
      [\lambda \textbf{x}_\alpha \textbf{B}_o]\textbf{A}_\alpha] \,
      [\lambda \textbf{x}_\alpha \textbf{C}_\beta]\textbf{A}_\alpha] \,
      [\lambda \textbf{x}_\alpha \textbf{D}_\beta]\textbf{A}_\alpha]$.

    \vspace{1ex}

    \item ${\textbf{A}_\alpha \IsDefApp} \Implies 
      [\lambda \textbf{x}_\alpha \synbrack{\textbf{B}_\beta}]\textbf{A}_\alpha \QuasiEqual 
      \synbrack{\textbf{B}_\beta}$.

    \vspace{1ex}

    \item $[\lambda \textbf{x}_\alpha \textbf{B}_\beta] \textbf{x}_\alpha \QuasiEqual \textbf{B}_\beta$.

  \ee

  \item[] \textbf{Axiom 5 (Tautologous Formulas)}

  \be

    \item[] $\textbf{A}_o$ {\sglsp} where $\textbf{A}_o$ is tautologous.

  \ee

  \item[] \textbf{Axiom 6 (Definedness)}

  \be

    \item ${\textbf{x}_\alpha\IsDefApp}.$

    \vspace{1ex}

    \item ${\textbf{c}_\alpha\IsDefApp}$ {\sglsp} where
      $\textbf{c}_\alpha$ is a primitive constant.\footnote{Notice
      that, for $\alpha \not= o$, $\textbf{c}_\alpha\IsDefApp$ is
      false if $\textbf{c}_\alpha$ is the defined constant
      $\Undefined_\alpha$.}

    \vspace{1ex}

    \item ${\textbf{A}_{o\beta}\textbf{B}_\beta \IsDefApp}.$

    \vspace{1ex}

    \item $[{\textbf{A}_{\alpha\beta}\IsUndefApp} \Or
      {\textbf{B}_\beta\IsUndefApp}] \Implies
      \textbf{A}_{\alpha\beta}\textbf{B}_\beta \QuasiEqual
      \Undefined_\alpha.$

    \vspace{1ex}

    \item ${[\lambda \textbf{x}_\alpha \textbf{B}_\beta]\IsDefApp}.$

    \vspace{1ex}

    \item $[\If \, \textbf{A}_o \, \textbf{B}_o \,
      \textbf{C}_o]\IsDefApp.$

    \vspace{1ex}

    \item $\synbrack{\textbf{A}_{\alpha}}\IsDefApp.$

    \vspace{1ex}

    \item $\sembrack{\textbf{A}_\epsilon}_o \IsDefApp.$

    \vspace{1ex}

    \item $\sembrack{\synbrack{\synbrack{\textbf{A}_\alpha}}}_\epsilon \IsDefApp.$

    \vspace{1ex}

    \item ${\NegAlt[\mname{eval-free}^{\alpha}_{o\epsilon} \,
        \textbf{A}_\epsilon]} \Implies
      \sembrack{\textbf{A}_\epsilon}_\alpha \QuasiEqual
      \Undefined_\alpha.$

    \vspace{1ex}

    \item $\Undefined_\alpha \IsUndefApp$ {\sglsp} where $\alpha \not= o$.

  \ee

  \item[] \textbf{Axiom 7 (Quasi-Equality)}

  \be

    \item $\textbf{A}_\alpha \QuasiEqual \textbf{A}_\alpha$.

  \ee

  \item[] \textbf{Axiom 8 (Definite Description)}

  \be

    \item ${\Forsome_1 \textbf{x}_\alpha \textbf{A}_o} \Iff {[\Iota
        \textbf{x}_\alpha \textbf{A}_o]\IsDefApp}$ {\sglsp} where
      $\alpha \not= o$.

    \vspace{1ex}

    \item ${[{\Forsome_1 \textbf{x}_\alpha \textbf{A}_o} \Andd 
   \mname{sub}_{\epsilon\epsilon\epsilon\epsilon} \,
   \synbrack{\Iota \textbf{x}_\alpha \textbf{A}_o} \,
   \synbrack{\textbf{x}_\alpha} \, \synbrack{\textbf{A}_o} = \synbrack{\textbf{B}_o}]}
   \Implies \textbf{B}_o$ {\sglsp} where $\alpha \not= o$.

  \ee

  \item[] \textbf{Axiom 9 (Ordered Pairs)}

  \be

    \item $[\mname{pair}_{\seq{\alpha\beta}\beta\alpha} \, \textbf{A}_\alpha
      \, \textbf{B}_\beta = \mname{pair}_{\seq{\alpha\beta}\beta\alpha} \,
      \textbf{C}_\alpha \, \textbf{D}_\beta] \Iff 
      [\textbf{A}_\alpha = \textbf{C}_\alpha \Andd \textbf{B}_\beta = \textbf{D}_\beta].$

    \vspace{1ex}

    \item ${\textbf{A}_{\seq{\alpha\beta}} \IsDefApp} \Implies
      \Forsome \textbf{x}_\alpha \Forsome \textbf{y}_\alpha 
      [\textbf{A}_{\seq{\alpha\beta}}
      = \mname{pair}_{\seq{\alpha\beta}\beta\alpha} \, \textbf{x}_\alpha \, \textbf{y}_\beta].$

  \ee

  \item[] \textbf{Axiom 10 (Conditionals)}

  \be

    \item $[\If \, T_o \, \textbf{B}_\alpha \, \textbf{C}_\alpha] \QuasiEqual \textbf{B}_\alpha.$

    \vspace{1ex}

    \item $[\If \, F_o \, \textbf{B}_\alpha \, \textbf{C}_\alpha] \QuasiEqual \textbf{C}_\alpha.$

    \vspace{1ex}

    \item $\sembrack{\If \, \textbf{A}_o \, \textbf{B}_\epsilon\, \textbf{C}_\epsilon}_\alpha \QuasiEqual 
          \If \, \textbf{A}_o \, \sembrack{\textbf{B}_\epsilon}_\alpha \, \sembrack{\textbf{C}_\epsilon}_\alpha.$

  \ee

  \item[] \textbf{Axiom 11 (Evaluation)}

  \be

    \item $\sembrack{\synbrack{\textbf{x}_\alpha}}_\alpha = \textbf{x}_\alpha.$

    \vspace{1ex}

    \item $\sembrack{\synbrack{\textbf{c}_\alpha}}_\alpha =
      \textbf{c}_\alpha$ {\sglsp} where $\textbf{c}_\alpha$ is primitive
      constant.

    \vspace{1ex}

    \item ${\mname{wff}^{\alpha\beta}_{o\epsilon} \, \textbf{A}_\epsilon}
      \Implies \sembrack{\mname{app}_{\epsilon\epsilon\epsilon} \,
        \textbf{A}_\epsilon \, \textbf{B}_\epsilon}_\alpha \QuasiEqual
      \sembrack{\textbf{A}_\epsilon}_{\alpha\beta}
      \sembrack{\textbf{B}_\epsilon}_\beta.$

    \vspace{1ex}

    \item $\mname{not-free-in}_{o\epsilon\epsilon} \, \synbrack{\textbf{x}_\alpha} \, 
      \synbrack{\textbf{B}_\epsilon} \Implies
      \sembrack{\mname{abs}_{\epsilon\epsilon\epsilon} \,
      \synbrack{\textbf{x}_\alpha} \, \textbf{B}_\epsilon}_{\beta\alpha}
      \QuasiEqual \lambda \textbf{x}_\alpha
      \sembrack{\textbf{B}_\epsilon}_{\beta}.$

    \vspace{1ex}

    \item $\sembrack{\mname{cond}_{\epsilon\epsilon\epsilon\epsilon} \, 
      \textbf{A}_\epsilon \, \textbf{B}_\epsilon \,\textbf{C}_\epsilon}_\alpha \QuasiEqual 
      \If \, \sembrack{\textbf{A}_\epsilon}_o \, 
      \sembrack{\textbf{B}_\epsilon}_\alpha \, \sembrack{\textbf{C}_\epsilon}_\alpha$.

    \vspace{1ex}

    \item ${\sembrack{\mname{quot}_{\epsilon\epsilon} \,
      \textbf{A}_\epsilon}_\epsilon \IsDefApp} \Implies
      {\sembrack{\mname{quot}_{\epsilon\epsilon} \,
        \textbf{A}_\epsilon}_\epsilon = \textbf{A}_\epsilon}.$

  \ee

  \item[] \textbf{Axiom 12 (Specifying Axioms)}

  \be

    \item[] $\textbf{A}_o$ {\sglsp} where $\textbf{A}_o$ is a
      specifying axiom in Specifications 1--9.

  \ee

\ei

\begin{note}[Overview of Axioms]\em
Axioms 1--4 of {\qzerouqe} correspond to the first four axioms of
{\qzero}.  Axioms 1 and 2 say essentially the same thing as the first
and second axioms of {\qzero} (see the next note).  A modification of
the third axiom of {\qzero}, Axiom~3 is the axiom of extensionality
for partial and total functions. Axiom~4 is the law of beta-reduction
for functions that may be partial and arguments that may be undefined.
Axiom 4.1 expresses the law of beta-reduction with substitution
represented by the logical constant
$\mname{sub}_{\epsilon\epsilon\epsilon\epsilon}$.  Axioms 4.2--9
express the law of beta-reduction using the basic properties of
lambda-notation.  Axiom 4.10 is an additional property of
lambda-notation.

Axiom 5 provides the tautologous formulas that are needed to discharge
the definedness conditions and substitution conditions on instances of
Axiom~4.  Axiom 6 deals with the definedness properties of wffs; the
first five parts of Axiom 6 address the three principles of the
traditional approach.  Axiom 7 states the reflexivity law for
quasi-equality.  Axioms 8 and 9 state the properties of the logical
constants $\iota_{\alpha(o\alpha)}$ and
$\mname{pair}_{\seq{\alpha\beta}\beta\alpha)}$, respectively.  Axiom
10 states the properties of conditionals.  Axioms 11 states the
properties of evaluation.  Axiom 12 gives the specifying axioms of the
12 logical constants involving the type $\epsilon$.
\end{note}

\begin{note}[Schemas vs.~Universal Formulas]\em
The proof systems {\pfsys} and {\pfsysu} are intended to be mimimalist
axiomatizations of {\qzero} and {\qzerou}.  For instance, in both
systems the first three axiom groups are single universal formulas
that express three different fundamental ideas.  In contrast, the
first three axiom groups of {\pfsysuqe} are formula schemas that
present all the instances of the three universal formulas.  The
instances of the these universal formulas are obtained in {\pfsys} and
{\pfsysu} by substitution.  Formulas schemas are employed in
{\pfsysuqe} instead of universal formulas for the sake of convenience
and uniformity.  In fact, the only axiom presented as a single formula
in Axioms 1--12 is Specification 4.29, the principle of induction for
type~$\epsilon$.
\end{note}

\begin{note}[Syntactic Side Conditions]\em
The syntactic conditions placed on the syntactic variables in the
schemas in Axioms 1--12 come in a few simple forms: 
\be

  \item A syntactic variable $\textbf{A}_\alpha$ can be any wff of
    type $\alpha$.

  \item A syntactic variable $\textbf{x}_\alpha$ can be any variable
    of type $\alpha$.

  \item A syntactic variable
    $\textbf{c}_\alpha$ with the condition ``$\textbf{c}_\alpha$ is
    a primitive constant'' can be any primitive constant of type
    $\alpha$, 

  \item A syntactic variable $\textbf{A}_\alpha$ with the condition
    ``$\textbf{A}_\alpha$ is a not a variable'' can be any wff of
    type $\alpha$ that is not a variable.

  \item A syntactic variable $\textbf{A}_\alpha$ with the condition
    ``$\textbf{A}_\alpha$ is a not a primitive constant'' can be any
    wff of type $\alpha$ that is not a primitive constant.

  \item Two variables must be distinct.

  \item Two primitive constants must be distinct.

  \item Two types must be distinct.

\ee 
Notice that none of these syntactic side conditions refer to notions
concerning free variables and substitution.
\end{note}

\subsection{Rules of Inference}

{\qzerouqe} has just two rules of inference:

\bi

  \item[] \textbf{Rule 1 (Quasi-Equality Substitution)}{\sglsp} From
    $\textbf{A}_\alpha \QuasiEqual \textbf{B}_\alpha$ and
    $\textbf{C}_o$ infer the result of replacing one occurrence of
    $\textbf{A}_\alpha$ in $\textbf{C}_o$ by an occurrence of
    $\textbf{B}_\alpha$, provided that the occurrence of
    $\textbf{A}_\alpha$ in $\textbf{C}_o$ is not within a quotation,
    not the first argument of a function abstraction, and not the
    second argument of an evaluation.

  \medskip

  \item[] \textbf{Rule 2 (Modus Ponens)}{\sglsp} From $\textbf{A}_o$ and
    $\textbf{A}_o \Implies \textbf{B}_o$ infer $\textbf{B}_o$.

\ei

\begin{note}[Overview of Rules of Inference]\em
{\qzerouqe} has the same two rules of inference as {\qzerou}.  Rule 1
(Quasi-Equality Substitution) corresponds to {\qzero}'s single rule of
inference, which is equality substitution.  These rules are exactly
the same except that the {\qzerouqe} rule requires only
\emph{quasi-equality} ($\QuasiEqual$) between the target wff and the
substitution wff, while the {\qzero} rule requires \emph{equality}
(=).  Rule 2 (Modus Ponens) is a primitive rule of inference in
{\qzerouqe}, but modus ponens is a derived rule of inference in
{\qzero}.  Modus ponens must be primitive in {\qzerouqe} since it is
needed to discharge the definedness conditions and substitution
conditions on instances of Axiom 4, the law of beta-reduction.
\end{note}

\subsection{Proofs}

Let $\textbf{A}_o$ be a formula and $\sH$ be a set of sentences (i.e.,
semantically closed formulas) of {\qzerouqe}.  A \emph{proof of
  $\textbf{A}_o$ from $\sH$ in {\pfsysuqe}} is a finite sequence of
$\wffs{o}$, ending with $\textbf{A}_o$, such that each member in the
sequence is an axiom of {\pfsysuqe}, a member of $\sH$, or is inferred
from preceding members in the sequence by a rule of inference of
{\pfsysuqe}.  We write $\proves{\sH}{\textbf{A}_o}$ to mean there is a
proof of $\textbf{A}_o$ from $\sH$ in {\pfsysuqe}.
$\proves{}{\textbf{A}_o}$ is written instead of
$\proves{\emptyset}{\textbf{A}_o}$.  $\textbf{A}_o$ is a
\emph{theorem} of {\pfsysuqe} if $\proves{}{\textbf{A}_o}$.

Now let $\sH$ be a set of syntactically closed evaluation-free
formulas of {\qzerouqe}.  (Recall that a syntactically closed
evaluation-free formula is also semantically closed by
Lemma~\ref{lem:sem-closed}.)  An \emph{evaluation-free proof of
  $\textbf{A}_o$ from $\sH$ in {\pfsysuqe}} is a proof of
$\textbf{A}_o$ from $\sH$ that is a sequence of evaluation-free
$\wffs{o}$. We write $\provesa{\sH}{\textbf{A}_o}$ to mean there is an
evaluation-free proof of $\textbf{A}_o$ from $\sH$ in {\pfsysuqe}.
Obviously, $\provesa{\sH}{\textbf{A}_o}$ implies
$\proves{\sH}{\textbf{A}_o}$.  $\provesa{}{\textbf{A}_o}$ is written
instead of $\provesa{\emptyset}{\textbf{A}_o}$.

$\sH$ is \emph{consistent} in {\pfsysuqe} if there is no proof of
$F_o$ from $\sH$ in {\pfsysuqe}.

\begin{note}[Proof from Hypotheses]\em
Andrews employs in~\cite{Andrews02} a more complicated notion of a
``proof from hypotheses'' in which a hypothesis is not required to be
semantically or syntactically closed.  We have chosen to use the
simpler notion since it is difficult to define Andrews' notion in the
presence of evaluations and we can manage well enough in this paper
with having only semantically or syntactically closed hypotheses.
\end{note}

\section{Soundness} \label{sec:soundness}

{\pfsysuqe} is \emph{sound for {\qzerouqe}} if
$\proves{\sH}{\textbf{A}_o}$ implies $\sH \modelsn \textbf{A}_o$
whenever $\textbf{A}_o$ is a formula and $\sH$ is a set of sentences
of {\qzerouqe}.  We will prove that the proof system {\pfsysuqe} is
sound for {\qzerouqe} by showing that its axioms are valid in every
normal general model for {\qzerouqe} and its rules of inference
preserve validity in every normal general model for {\qzerouqe}.

\subsection{Axioms and Rules of Inference}

\begin{lem} \label{lem:axioms}
Each axiom of {\pfsysuqe} is valid in every normal general model for
{\qzerouqe}.
\end{lem}

\begin{proof}
Let $\sM = \seq{\set{\sD_\alpha \;|\; \alpha \in \sT}, \sJ}$ be a
normal general model for {\qzerouqe} and $\phi \in
\mname{assign}(\sM)$.  There are 16 cases, one for each group of
axioms.

\medskip

\noindent \textbf{Axiom 1} {\sglsp} The proof is similar to the proof
of 5402 for Axiom 1 in~\cite[p.~241]{Andrews02} when $\sV^{\cal
  M}_{\phi}(\textbf{G}_{oo})$ is defined.  The proof is
straightforward when $\sV^{\cal M}_{\phi}(\textbf{G}_{oo})$ is
undefined.

\medskip

\noindent \textbf{Axiom 2} {\sglsp} The proof is similar to the proof
of 5402 for Axiom 2 in~\cite[p.~242]{Andrews02} when $\sV^{\cal
  M}_{\phi}(\textbf{H}_{o\alpha})$ is defined.  The proof is
straightforward when $\sV^{\cal M}_{\phi}(\textbf{H}_{o\alpha})$ is
undefined.

\medskip

\noindent \textbf{Axiom 3} {\sglsp} The proof is similar to the proof
of 5402 for Axiom 3 in~\cite[p.~242]{Andrews02}.

\medskip

\noindent \textbf{Axiom 4} {\sglsp} 

\bi

  \item[] \textbf{Axiom 4.1} {\sglsp} Each instance of Axiom 4.1 is
    valid in $\sM$ by Theorem~\ref{thm:beta-red-sub}.

  \item[] \textbf{Axiom 4.2} {\sglsp}
    We must show
    \[(\mbox{a})~\sV^{\cal M}_{\phi}([\lambda \textbf{x}_\alpha
    \textbf{x}_\alpha]\textbf{A}_\alpha) \QuasiEqual \sV^{\cal
      M}_{\phi}(\textbf{A}_\alpha)\] to prove Axiom 4.2 is valid in
    $\sM$.  If $\sV^{\cal M}_{\phi}(\textbf{A}_\alpha)$ is undefined,
    then clearly (a) is true.  So assume (b) $\sV^{\cal
      M}_{\phi}(\textbf{A}_\alpha)$ is defined.  Then
    \begin{align} \setcounter{equation}{0}
    &
    \sV^{\cal M}_{\phi}([\lambda \textbf{x}_\alpha \textbf{x}_\alpha]\textbf{A}_\alpha) \\
    & \QuasiEqual
    \sV^{\cal M}_{\phi[{\bf x}_\alpha \mapsto {\cal V}^{\cal M}_\phi({\bf A}_\alpha)]}
    (\textbf{x}_\alpha) \\
    & \QuasiEqual
    \sV^{\cal M}_{\phi}(\textbf{A}_\alpha) 
    \end{align}
    (2) is by (b) and the semantics of function application and function
    abstraction, and (3) is by the semantics of variables.

  \item[] \textbf{Axiom 4.3} {\sglsp} 
    Assume (a) $\sV^{\cal M}_{\phi}(\textbf{A}_\alpha)$ is defined and
    (b) $\textbf{x}_\alpha \not= \textbf{y}_\beta$.  We must show
    \[\sV^{\cal M}_{\phi}([\lambda \textbf{x}_\alpha
    \textbf{y}_\beta]\textbf{A}_\alpha) \QuasiEqual \sV^{\cal
      M}_{\phi}(\textbf{y}_\beta)\] to prove Axiom 4.3 is valid in
    $\sM$.  Then
    \begin{align} \setcounter{equation}{0}
    &
    \sV^{\cal M}_{\phi}([\lambda \textbf{x}_\alpha \textbf{y}_\beta]\textbf{A}_\alpha) \\
    & \QuasiEqual
    \sV^{\cal M}_{\phi[{\bf x}_\alpha \mapsto {\cal V}^{\cal M}_\phi({\bf A}_\alpha)]}
    (\textbf{y}_\beta) \\
    & \QuasiEqual
    \sV^{\cal M}_{\phi}(\textbf{y}_\beta). 
    \end{align}
    (2) is by (a) and the semantics of function application and
    function abstraction, and (3) is by (b) and the semantics of
    variables.

  \item[] \textbf{Axiom 4.4} {\sglsp} Similar to Axiom 4.3.

  \item[] \textbf{Axiom 4.5} {\sglsp} We must show
    \[(\mbox{a})~\sV^{\cal M}_{\phi}([\lambda \textbf{x}_\alpha 
    [\textbf{B}_{\alpha\beta}\textbf{C}_\beta]]\textbf{A}_\alpha)
    \QuasiEqual \sV^{\cal M}_{\phi}([[\lambda \textbf{x}_\alpha
        \textbf{B}_{\alpha\beta}]\textbf{A}_\alpha] [[\lambda
        \textbf{x}_\alpha \textbf{C}_\beta]\textbf{A}_\alpha])\] to
    prove Axiom 4.5 is valid in $\sM$.  If $\sV^{\cal
      M}_{\phi}(\textbf{A}_\alpha)$ is undefined, then clearly (a) is
    true.  So assume (b) $\sV^{\cal M}_{\phi}(\textbf{A}_\alpha)$ is
    defined.  Then
    \begin{align} \setcounter{equation}{0}
    &
    \sV^{\cal M}_{\phi}([\lambda \textbf{x}_\alpha [\textbf{B}_{\alpha\beta}\textbf{C}_\beta]]\textbf{A}_\alpha) \\
    & \QuasiEqual
    \sV^{\cal M}_{\phi[{\bf x}_\alpha \mapsto {\cal V}^{\cal M}_\phi({\bf A}_\alpha)]}
    (\textbf{B}_{\alpha\beta}\textbf{C}_\beta) \\
    & \QuasiEqual
    \sV^{\cal M}_{\phi[{\bf x}_\alpha \mapsto {\cal V}^{\cal M}_\phi({\bf A}_\alpha)]}(\textbf{B}_{\alpha\beta})
    (\sV^{\cal M}_{\phi[{\bf x}_\alpha \mapsto {\cal V}^{\cal M}_\phi({\bf A}_\alpha)]}(\textbf{C}_\beta)) \\
    & \QuasiEqual
    \sV^{\cal M}_\phi([\lambda \textbf{x}_\alpha \textbf{B}_{\alpha\beta}]\textbf{A}_\alpha)
    (\sV^{\cal M}_\phi([\lambda \textbf{x}_\alpha \textbf{C}_\beta]\textbf{A}_\alpha)) \\
    & \QuasiEqual
    \sV^{\cal M}_\phi([[\lambda \textbf{x}_\alpha \textbf{B}_{\alpha\beta}]\textbf{A}_\alpha] 
    [[\lambda \textbf{x}_\alpha \textbf{C}_\beta]\textbf{A}_\alpha]).
    \end{align}
    (2) and (4) are by (b) and the semantics of function application
    and function abstraction, and (3) and (5) are by the semantics of
    function application.

  \item[] \textbf{Axiom 4.6} {\sglsp} Assume (a) $\sV^{\cal
    M}_{\phi}(\textbf{A}_\alpha)$ is defined.  We must show
    \[\sV^{\cal M}_{\phi}([\lambda \textbf{x}_\alpha 
    [\lambda \textbf{x}_\alpha \textbf{B}_\beta]]\textbf{A}_\alpha)(d)
    \QuasiEqual \sV^{\cal M}_{\phi}(\lambda \textbf{x}_\alpha
    \textbf{B}_\beta)(d),\] where $d \in \sD_\alpha$, to prove Axiom
    4.6 is valid in $\sM$.
    \begin{align} \setcounter{equation}{0}
    &
    \sV^{\cal M}_{\phi}([\lambda \textbf{x}_\alpha 
    [\lambda \textbf{x}_\alpha \textbf{B}_\beta]]\textbf{A}_\alpha) (d)\\
    & \QuasiEqual
    \sV^{\cal M}_{\phi[{\bf x}_\alpha \mapsto {\cal V}^{\cal M}_\phi({\bf A}_\alpha)]}
    (\lambda \textbf{x}_\alpha \textbf{B}_\beta)(d)\\
    & \QuasiEqual
    \sV^{\cal M}_{\phi[{\bf x}_\alpha \mapsto {\cal V}^{\cal M}_\phi({\bf A}_\alpha)][{\bf x}_\alpha \mapsto d]}
    (\textbf{B}_\beta)\\
    & \QuasiEqual
    \sV^{\cal M}_{\phi[{\bf x}_\alpha \mapsto d]}
    (\textbf{B}_\beta)\\
    & \QuasiEqual
    \sV^{\cal M}_\phi(\lambda \textbf{x}_\alpha \textbf{B}_\beta])(d).
    \end{align}
    (2) is by (a) and the semantics of function application and
    function abstraction; (3) and (5) are by the semantics of function
    abstraction; and (4) is by
    \[\phi[{\bf x}_\alpha \mapsto {\cal V}^{\cal M}_\phi({\bf
        A}_\alpha)][{\bf x}_\alpha \mapsto d] = \phi[{\bf x}_\alpha
      \mapsto d].\]

  \item[] \textbf{Axiom 4.7} {\sglsp} Assume (a) $\sV^{\cal
    M}_{\phi}(\textbf{A}_\alpha)$ is defined, (b) $\textbf{x}_\alpha
    \not= \textbf{y}_\beta$, and 
    \begin{align*}
    &
    (\mbox{c})~\sM \models \mname{not-free-in}_{o\epsilon\epsilon} \,
    \synbrack{\textbf{x}_\alpha} \, \synbrack{\textbf{B}_\gamma} {\dblsp} \mbox{or} \\
    &\hspace*{3.75ex}
    \sM \models \mname{not-free-in}_{o\epsilon\epsilon}
    \, \synbrack{\textbf{y}_\beta} \, \synbrack{\textbf{A}_\alpha}.
    \end{align*}    
    We must show
    \[\sV^{\cal M}_{\phi}([\lambda \textbf{x}_\alpha 
    [\lambda \textbf{y}_\beta \textbf{B}_\gamma]]\textbf{A}_\alpha)(d)
    \QuasiEqual \sV^{\cal M}_{\phi}(\lambda \textbf{y}_\beta
      [[\lambda \textbf{x}_\alpha
        \textbf{B}_\gamma]\textbf{A}_\alpha])(d),\] where $d \in
    \sD_\beta$, to prove Axiom 4.7 is valid in $\sM$.
    \begin{align} \setcounter{equation}{0}
    &
    \sV^{\cal M}_{\phi}([\lambda \textbf{x}_\alpha 
    [\lambda \textbf{y}_\beta \textbf{B}_\gamma]]\textbf{A}_\alpha) (d)\\
    & \QuasiEqual
    \sV^{\cal M}_{\phi[{\bf x}_\alpha \mapsto {\cal V}^{\cal M}_\phi({\bf A}_\alpha)]}
    (\lambda \textbf{y}_\beta \textbf{B}_\gamma)(d)\\
    & \QuasiEqual
    \sV^{\cal M}_{\phi[{\bf x}_\alpha \mapsto {\cal V}^{\cal M}_\phi({\bf A}_\alpha)][{\bf y}_\beta \mapsto d]}
    (\textbf{B}_\gamma)\\
    & \QuasiEqual
    \sV^{\cal M}_{\phi[{\bf y}_\beta \mapsto d][{\bf x}_\alpha \mapsto {\cal V}^{\cal M}_\phi({\bf A}_\alpha)]}
    (\textbf{B}_\gamma)\\
    & \QuasiEqual
    \sV^{\cal M}_{\phi[{\bf y}_\beta \mapsto d]
    [{\bf x}_\alpha \mapsto {\cal V}^{\cal M}_{\phi[{\bf y}_\beta \mapsto d]}({\bf A}_\alpha)]}
    (\textbf{B}_\gamma)\\
    & \QuasiEqual
    \sV^{\cal M}_{\phi[{\bf y}_\beta \mapsto d]}
    ([\lambda \textbf{x}_\alpha \textbf{B}_\gamma]\textbf{A}_\alpha)\\
    & \QuasiEqual
    \sV^{\cal M}_\phi(\lambda \textbf{y}_\beta [[\lambda \textbf{x}_\alpha \textbf{B}_\gamma]\textbf{A}_\alpha])(d)
    \end{align}
    (2) and (6) are by (a) and the semantics of function application
    and function abstraction; (3) and (7) are by the semantics of
    function abstraction; (4) is by (b); and (5) is by (c) and part 2
    of Lemma~\ref{lem:not-free-in}.

  \item[] \textbf{Axiom 4.8} {\sglsp}  Similar to Axiom 4.5.

  \item[] \textbf{Axiom 4.9} {\sglsp}  Similar to Axiom 4.3.

  \item[] \textbf{Axiom 4.10} {\sglsp}
    We must show
    \[\sV^{\cal M}_{\phi}([\lambda \textbf{x}_\alpha
    \textbf{B}_\beta]\textbf{x}_\alpha) \QuasiEqual \sV^{\cal
      M}_{\phi}(\textbf{B}_\beta)\] to prove Axiom 4.10 is valid in
    $\sM$.
    \begin{align} \setcounter{equation}{0}
    &
    \sV^{\cal M}_{\phi}([\lambda \textbf{x}_\alpha \textbf{B}_\beta]\textbf{x}_\alpha) \\
    & \QuasiEqual
    \sV^{\cal M}_{\phi[{\bf x}_\alpha \mapsto \phi({\bf x}_\alpha)]}
    (\textbf{B}_\beta) \\
    & \QuasiEqual
    \sV^{\cal M}_{\phi}(\textbf{B}_\beta) 
    \end{align}
    (2) is by the semantics of function application, function
    abstraction, variables; and (3) is by $\phi = \phi[{\bf x}_\alpha
      \mapsto \phi({\bf x}_\alpha)]$.

\ei

\medskip

\noindent \textbf{Axiom 5} {\sglsp} The propositional constants $T_o$
and $F_o$ and the propositional connectives $\wedge_{ooo}$,
$\vee_{ooo}$, and $\Implies_{ooo}$ have their usual meanings in a
general model.  Hence any tautologous formula is valid in $\sM$.

\medskip

\noindent \textbf{Axiom 6} {\sglsp} $\sM \models
\textbf{A}_\alpha\IsDefApp$ iff $\sV^{\cal
  M}_{\phi}(\textbf{A}_\alpha)$ is defined for all $\phi \in
\mname{assign}(\sM)$.  Hence Axioms 6.1, 6.2, 6.3, 6.4, 6.5, 6.6, 6,7,
and 6.8 are valid in $\sM$ by conditions 1, 2, 3, 3, 4, 5, 6, and 7 in
the definition of a general model.  Axiom 6.9 is valid in $\sM$ by the
fact that quotations are evaluation-free and conditions 6 and 7 in
definition of a general model.  Axiom 6.10 is valid in $\sM$ by
Proposition~\ref{prop:eval-free} and condition 7 in the definition
of a general model.  Axiom 6.11 is valid in $\sM$ since
$\sJ(\iota_{\alpha(o\alpha)})$ is a unique member selector on
$\sD_\alpha$ and $\lambda x_\alpha [x_\alpha \not= x_\alpha]$
represents the empty set.

\medskip

\noindent \textbf{Axiom 7} {\sglsp} Clearly, $\sV^{\cal
  M}_{\phi}(\textbf{A}_\alpha \QuasiEqual \textbf{A}_\alpha) =
\mname{T}$ iff $\sV^{\cal M}_{\phi}(\textbf{A}_\alpha) \QuasiEqual
\sV^{\cal M}_{\phi}(\textbf{A}_\alpha)$, which is always true.  Hence
$\sM \models \textbf{A}_\alpha \QuasiEqual \textbf{A}_\alpha$.

\medskip

\noindent \textbf{Axiom 8} {\sglsp} 

\bi

  \item[] \textbf{Axiom 8.1} {\sglsp} Axiom 8.1 is valid in $\sM$ since
    $\sJ(\iota_{\alpha(o\alpha)})$ is a unique member selector on
    $\sD_\alpha$.

  \item[]\bsp \textbf{Axiom 8.2} {\sglsp} Assume (a) $\sV^{\cal
    M}_{\phi}(\Forsome_1 \textbf{x}_\alpha \textbf{A}_o) = \mname{T}$
    and \[\sV^{\cal
      M}_{\phi}(\mname{sub}_{\epsilon\epsilon\epsilon\epsilon} \,
    \synbrack{\Iota \textbf{x}_\alpha \textbf{A}_o} \,
    \synbrack{\textbf{x}_\alpha} \, \synbrack{\textbf{A}_o} =
    \synbrack{\textbf{B}_o}) = \mname{T}.\] We must show $\sV^{\cal
      M}_{\phi}(\textbf{B}_o) = \mname{T}$ to prove that Axiom 8.2 is
    valid in $\sM$.  Axiom~8.1 and (a) implies $\sV^{\cal
      M}_{\phi}(\Iota \textbf{x}_\alpha \textbf{A}_o)$ is defined.
    (a) and the fact that $\sJ(\iota_{\alpha(o\alpha)})$ is a unique
    member selector on $\sD_\alpha$ implies \[\sV^{\cal
      M}_{\phi}([\lambda \textbf{x}_\alpha \textbf{A}_o][\Iota
      \textbf{x}_\alpha \textbf{A}_o]) = \mname{T}.\]Then $\sV^{\cal
      M}_{\phi}([\lambda \textbf{x}_\alpha \textbf{A}_o][\Iota
      \textbf{x}_\alpha \textbf{A}_o] = \textbf{B}_o)$ by the proof
    for Axiom 4.  Thus $\sV^{\cal M}_{\phi}(\textbf{B}_o) =
    \mname{T}$.\esp

\ei

\medskip

\noindent \textbf{Axiom 9} {\sglsp} Axiom~9.1 is valid in $\sM$ since
$\sJ(\mname{pair}_{\seq{\alpha\beta}\beta\alpha})$ is a pairing
function on $\sD_\alpha$ and $\sD_\beta$.  Axiom 9.2 is valid in $\sM$
since every $p \in \sD_{\seq{\alpha\beta}}$ is a pair $\seq{a,b}$ where $a
\in \sD_\alpha$ and $b \in \sD_\beta$ and
$\sJ(\mname{pair}_{\seq{\alpha\beta}\beta\alpha})$ is a pairing
function on $\sD_\alpha$ and $\sD_\beta$.

\medskip

\noindent \textbf{Axiom 10} {\sglsp} Axioms 10.1 and 10.2 are valid in
$\sM$ by condition 5 in the definition of a general model.  $\sV^{\cal
  M}_{\phi}(\textbf{A}_o) = \mname{T}$ implies $\sV^{\cal
  M}_{\phi}(\sembrack{\If \, \textbf{A}_o \, \textbf{B}_\epsilon\,
  \textbf{C}_\epsilon}_\alpha) \QuasiEqual \sV^{\cal
  M}_{\phi}(\sembrack{\textbf{B}_\epsilon}_\alpha)$ and $\sV^{\cal
  M}_{\phi}(\textbf{A}_o) = \mname{F}$ implies $\sV^{\cal
  M}_{\phi}(\sembrack{\If \, \textbf{A}_o \, \textbf{B}_\epsilon\,
  \textbf{C}_\epsilon}_\alpha) \QuasiEqual \sV^{\cal
  M}_{\phi}(\sembrack{\textbf{C}_\epsilon}_\alpha)$ by conditions 5
and 7 in the definition of a general model.  Hence Axiom 10.3 is valid
in $\sM$.

\medskip

\noindent \textbf{Axiom 11} {\sglsp} 

\bi

  \item[] \textbf{Axioms 11.1 and 11.2} {\sglsp} Immediate by
    condition 7 in the definition of a general model since variables
    and primitive constants are evaluation-free.

  \item[] \textbf{Axiom 11.3} {\sglsp} Assume $\sV^{\cal
    M}_{\phi}(\mname{wff}^{\alpha\beta}_{o\epsilon} \,
    \textbf{A}_\epsilon) = \mname{T}$.  This implies (a)~$\sV^{\cal
      M}_{\phi}(\textbf{A}_\epsilon) = \sE(\textbf{C}_{\alpha\beta})$
    for some $\textbf{C}_{\alpha\beta}$.  We must show $X \QuasiEqual
    Y$ where $X$ is
    \[\sV^{\cal M}_{\phi}(\sembrack{\mname{app}_{\epsilon\epsilon\epsilon} \,
    \textbf{A}_\epsilon \, \textbf{B}_\epsilon}_\alpha)\] and $Y$
    is 
    \[\sV^{\cal M}_{\phi}(\sembrack{\textbf{A}_\epsilon }_{\alpha\beta} \sembrack{\textbf{B}_\epsilon}_\beta).\] 

    First, assume (b)~$\sV^{\cal M}_{\phi}(\textbf{B}_\epsilon)
    =\sE(\textbf{D}_{\beta})$ for some $\textbf{D}_{\beta}$.  If
    (c)~$\textbf{C}_{\alpha\beta}\textbf{D}_{\beta}$ is
    evaluation-free, then
    \begin{align} \setcounter{equation}{0}
    &
    \sV^{\cal M}_{\phi}(\sembrack{\mname{app}_{\epsilon\epsilon\epsilon} \,
    \textbf{A}_\epsilon \, \textbf{B}_\epsilon}_\alpha) \\
    &\QuasiEqual 
    \sV^{\cal  M}_{\phi}(\sembrack{\mname{app}_{\epsilon\epsilon\epsilon} \,
    \sE(\textbf{C}_{\alpha\beta}) \, \sE(\textbf{D}_{\beta})}_\alpha) \\
    &\QuasiEqual
    \sV^{\cal  M}_{\phi}(\sembrack{\sE(\textbf{C}_{\alpha\beta}\textbf{D}_{\beta})}_\alpha) \\
    &\QuasiEqual
    \sV^{\cal  M}_{\phi}(\textbf{C}_{\alpha\beta}\textbf{D}_{\beta})   \\
    &\QuasiEqual
    \sV^{\cal  M}_{\phi}(\textbf{C}_{\alpha\beta})(\sV^{\cal  M}_{\phi}(\textbf{D}_{\beta})) \\
    &\QuasiEqual
    \sV^{\cal  M}_{\phi}(\sembrack{\sE(\textbf{C}_{\alpha\beta})}_{\alpha\beta})
    (\sV^{\cal  M}_{\phi}(\sembrack{\sE(\textbf{D}_{\beta})}_\beta)) \\
    &\QuasiEqual
    \sV^{\cal  M}_{\phi}(\sembrack{\textbf{A}_\epsilon}_{\alpha\beta})
    (\sV^{\cal  M}_{\phi}(\sembrack{\textbf{B}_\epsilon}_\beta)) \\
    &\QuasiEqual
    \sV^{\cal  M}_{\phi}(\sembrack{\textbf{A}_\epsilon}_{\alpha\beta}
    \sembrack{\textbf{B}_\epsilon}_\beta).
    \end{align}
    (2) and (7) are by (a), (b), and Proposition~\ref{prop:sem-of-E}; (3) is
    by the definition of $\sE$; (4) and (6) are by (c) and the
    semantics of evaluation; and (5) and (8) are by the semantics of
    function application.  Hence $X \QuasiEqual Y$.  If
    $\textbf{C}_{\alpha\beta}\textbf{D}_{\beta}$ is not
    evaluation-free, then $\textbf{C}_{\alpha\beta}$ or
    $\textbf{D}_{\beta}$ is not evaluation-free.  Then $X$ and $Y$ are
    both undefined by the semantics of evaluation and the beginning
    and end of the derivation above.  Hence $X \QuasiEqual Y$.

    Second, assume $\sV^{\cal M}_{\phi}(\textbf{B}_\epsilon)
    =\sE(\textbf{D}_{\gamma})$ with $\gamma \not= \beta$.  Then $X$ is
    undefined by Specifications 6.4 and 6.5, the semantics of
    evaluation, and the beginning of the derivation above, and $Y$ is
    undefined by the semantics of evaluation and function application
    and the end of the derivation above.  Hence $X \QuasiEqual Y$.

    \bsp Third, assume $\sV^{\cal M}_{\phi}(\textbf{B}_\epsilon) = d$ where
    $d$ is a nonstandard construction.  Then $X$ is undefined by
    Lemmas~\ref{lem:eval-nonstand} and~\ref{lem:stand-constructions},
    and $Y$ is undefined by Lemma~\ref{lem:eval-nonstand} and the
    semantics of function application.  Hence $X \QuasiEqual Y$ in
    this case, and therefore, in every case.\esp

  \item[] \bsp \textbf{Axiom 11.4} {\sglsp} Let \[(\mbox{a})~\sV^{\cal
    M}_{\phi}(\mname{not-free-in}_{o\epsilon\epsilon} \,
    \synbrack{\textbf{x}_\alpha} \, \synbrack{\textbf{B}_\epsilon}) =
    \mname{T}\] and $d \in \sD_\alpha$.  It suffices to show $X(d)
    \QuasiEqual Y(d)$ where $X$ is \[\sV^{\cal
      M}_{\phi}(\sembrack{\mname{abs}_{\epsilon\epsilon\epsilon} \,
      \synbrack{\textbf{x}_\alpha} \,
      \textbf{B}_\epsilon}_{\beta\alpha})\] and $Y$ is \[\sV^{\cal
      M}_{\phi}(\lambda \textbf{x}_\alpha
    \sembrack{\textbf{B}_\epsilon}_{\beta}).\] First, assume (b)
    $\sV^{\cal M}_{\phi}(\textbf{B}_\epsilon) =
    \sE(\textbf{C}_{\beta})$ for some $\textbf{C}_{\beta}$.  This
    implies (c)~$\sV^{\cal M}_{\phi[{\bf x}_\alpha \mapsto
        d]}(\textbf{B}_\epsilon) = \sE(\textbf{C}_{\beta})$ by (a) and
    part 2 of Lemma~\ref{lem:not-free-in}.  If (d)
    $\textbf{C}_{\beta}$ is evaluation-free, then
    \begin{align} \setcounter{equation}{0}
    &
    \sV^{\cal M}_{\phi}(\sembrack{\mname{abs}_{\epsilon\epsilon\epsilon} \,
    \synbrack{\textbf{x}_\alpha} \, \textbf{B}_\epsilon}_{\beta\alpha})(d) \\
    &\QuasiEqual
    \sV^{\cal M}_{\phi}(\sembrack{\mname{abs}_{\epsilon\epsilon\epsilon} \,
    \synbrack{\textbf{x}_\alpha} \, \sE(\textbf{C}_{\beta})}_{\beta\alpha})(d) \\
    &\QuasiEqual
    \sV^{\cal M}_{\phi}(\sembrack{\sE(\lambda \textbf{x}_\alpha \textbf{C}_{\beta})}_{\beta\alpha})(d) \\
    &\QuasiEqual
    \sV^{\cal M}_{\phi}(\lambda \textbf{x}_\alpha \textbf{C}_{\beta})(d) \\
    &\QuasiEqual
    \sV^{\cal M}_{\phi[{\bf x}_\alpha \mapsto d]}(\textbf{C}_{\beta}) \\
    &\QuasiEqual
    \sV^{\cal M}_{\phi[{\bf x}_\alpha \mapsto d]}(\sembrack{\sE(\textbf{C}_{\beta})}_\beta) \\
    &\QuasiEqual
    \sV^{\cal M}_{\phi[{\bf x}_\alpha \mapsto d]}(\sembrack{\textbf{B}_\epsilon}_\beta) \\
    &\QuasiEqual
    \sV^{\cal M}_{\phi}(\lambda \textbf{x}_\alpha\sembrack{\textbf{B}_\epsilon}_\beta)(d). 
    \end{align}
    (2) is by (b) and Proposition~\ref{prop:sem-of-E}; (3) is by the
    definition of $\sE$; (4) and (6) are by the semantics of
    evaluation and (d); (5) and (8) are by the semantics of function
    abstraction; and (7) is by (c) and and
    Proposition~\ref{prop:sem-of-E}.  Hence $X(d) \QuasiEqual Y(d)$.
    If $\textbf{C}_{\beta}$ is not evaluation-free, then $X(d)$ and
    $Y(d)$ are both undefined by the semantics of evaluation and the
    beginning and end of the derivation above.  Hence $X(d)
    \QuasiEqual Y(d)$. \esp

    Second, assume $\sV^{\cal M}_{\phi}(\textbf{B}_\epsilon) =
    \sE(\textbf{C}_{\gamma})$ for some $\textbf{C}_{\gamma}$ where
    $\gamma \not= \beta$.  Then $X(d)$ is undefined by the semantics
    of evaluation and the beginning of the derivation above, and
    $Y(d)$ is undefined by the semantics of evaluation and the end of
    the derivation above.  Hence $X(d) \QuasiEqual Y(d)$.

    Third, assume $\sV^{\cal M}_{\phi}(\textbf{B}_\epsilon)$ is a
    nonstandard construction.  Then $X(d)$ is undefined by
    Lemmas~\ref{lem:eval-nonstand} and~\ref{lem:stand-constructions},
    and $Y(d)$ is undefined by Lemma~\ref{lem:eval-nonstand}.  Hence
    $X(d) \QuasiEqual Y(d)$ in this case, and therefore, in every
    case.

  \item[] \textbf{Axiom 11.5} {\sglsp} Similar to Axiom 11.3.

  \item[] \textbf{Axiom 11.6} {\sglsp} First, assume $\sV^{\cal
    M}_{\phi}(\textbf{A}_\epsilon) = \sE(\textbf{B}_\alpha)$ for some
    $\textbf{B}_\alpha$.  Then
  \begin{align} \setcounter{equation}{0}
  & 
  \sV^{\cal M}_{\phi}(\sembrack{\mname{quot}_{\epsilon\epsilon} \, \textbf{A}_\epsilon}_\epsilon). \\
  &\QuasiEqual
  \sV^{\cal M}_{\phi}(\sembrack{\mname{quot}_{\epsilon\epsilon} \, \sE(\textbf{B}_\alpha)}_\epsilon). \\
  &\QuasiEqual
  \sV^{\cal M}_{\phi}(\sembrack{\sE(\synbrack{\textbf{B}_\alpha})}_\epsilon). \\
  &\QuasiEqual
  \sV^{\cal M}_{\phi}(\synbrack{\textbf{B}_\alpha}). \\
  &\QuasiEqual
  \sV^{\cal M}_{\phi}(\textbf{A}_\epsilon).
  \end{align}
  (2) is by Proposition~\ref{prop:sem-of-E}; (3) is by the definition
  of $\sE$; (4) by the semantics of evaluation and the fact that
  quotations are evaluation-free; and (5) is by Specification 1.
  Hence $\sV^{\cal M}_{\phi}(\sembrack{\mname{quot}_{\epsilon\epsilon}
    \, \textbf{A}_\epsilon}_\epsilon) = \sV^{\cal
    M}_{\phi}(\textbf{A}_\epsilon)$ since $\sV^{\cal
    M}_{\phi}(\textbf{A}_\epsilon)$ is defined.

  Second, assume $\sV^{\cal M}_{\phi}(\textbf{A}_\epsilon) \not=
  \sE(\textbf{B}_\alpha)$ for all $\textbf{B}_\alpha$.  Then
  $\sV^{\cal M}_{\phi}(\mname{quot}_{\epsilon\epsilon} \,
  \textbf{A}_\epsilon) \not= \sE(\textbf{A}_\alpha)$ for all
  $\textbf{B}_\alpha$ by Lemma~\ref{lem:stand-constructions}.  Hence
  $\sV^{\cal M}_{\phi}(\sembrack{\mname{quot}_{\epsilon\epsilon} \,
    \textbf{A}_\epsilon}_\epsilon)$ is undefined by
  Lemma~\ref{lem:eval-nonstand}.  Therefore, Axiom 11.6 is valid in
  $\sM$ in both cases.

\ei

\medskip

\noindent \textbf{Axiom 12} {\sglsp} Each axiom of this group is a
specifying axiom and thus is valid in $\sM$ since $\sM$ is normal.
\end{proof}

\begin{lem} \label{lem:rules}
Each rule of inference of {\pfsysuqe} preserves validity in every
normal general model for {\qzerouqe}.
\end{lem}

\begin{proof} Let $\sM$ be a normal general model for {\qzerouqe}.  
We must show that Rules 1 and 2 preserve validity in $\sM$.

\medskip

\noindent \textbf{Rule 1} {\sglsp} Suppose $\textbf{C}_o$ and
$\textbf{C}'_o$ are wffs such that $\textbf{C}'_o$ is the result of
replacing one occurrence of $\textbf{A}_\alpha$ in $\textbf{C}_o$ by
an occurrence of $\textbf{B}_\alpha$, provided that the occurrence of
$\textbf{A}_\alpha$ in $\textbf{C}_o$ is not within a quotation, not
the first argument of a function abstraction, and not the second
argument of an evaluation.  Then it easily follows that $\sV^{\cal
  M}_{\phi}(\textbf{A}_\alpha) \QuasiEqual \sV^{\cal
  M}_{\phi}(\textbf{B}_\alpha)$ for all $\phi \in \mname{assign}(\sM)$
implies $\sV^{\cal M}_{\phi}(\textbf{C}_o) = \sV^{\cal
  M}_{\phi}(\textbf{C}'_o)$ for all $\phi \in \mname{assign}(\sM)$ by
induction on the size of $\textbf{C}_o$.  $\sM \models
\textbf{A}_\alpha \QuasiEqual \textbf{B}_\alpha$ implies $\sV^{\cal
  M}_{\phi}(\textbf{A}_\alpha) \QuasiEqual \sV^{\cal
  M}_{\phi}(\textbf{B}_\alpha)$ for all $\phi \in
\mname{assign}(\sM)$, and hence $\sM \models \textbf{C}_o$ implies
$\sV^{\cal M}_{\phi}(\textbf{C}_o) = \sV^{\cal
  M}_{\phi}(\textbf{C}'_o) = \mname{T}$ for all $\phi \in
\mname{assign}(\sM)$.  Therefore, $\sM \models \textbf{A}_\alpha
\QuasiEqual \textbf{B}_\alpha$ and $\sM \models \textbf{C}_o$ implies
$\sM \models \textbf{C}'_o$, and so Rule 1 preserves validity in
$\sM$.

\medskip

\noindent \textbf{Rule 2} {\sglsp} Since $\Implies_{ooo}$ has its usual
meaning in a general model, Rule 2 obviously preserves validity in $\sM$.
\end{proof}

\subsection{Soundness and Consistency Theorems}

\begin{thm} [Soundness Theorem] \label{thm:soundness}
{\pfsysuqe} is sound for {\qzerouqe}.
\end{thm}

\begin{proof}
Assume $\proves{\sH}{\textbf{A}_o}$ and $\sM \models \sH$ where
$\textbf{A}_o$ is a formula of {\qzerouqe}, $\sH$ is a set of
sentences of {\qzerouqe}, and $\sM$ is a normal general model for
{\qzerouqe}.  We must show that $\sM \models \textbf{A}_o$.  By
Lemma~\ref{lem:axioms}, each axiom of {\pfsysuqe} is valid in $\sM$,
and by Lemma~\ref{lem:rules}, each rule of inference of {\pfsysuqe}
preserve validity in $\sM$.  Therefore, $\proves{\sH}{\textbf{A}_o}$
implies $\sM \models \textbf{A}_o$.
\end{proof}

\begin{thm}[Consistency Theorem] \label{thm:consistency}\bsp
Let $\sH$ be a set of sentences of {\qzerouqe}.  If $\sH$ has a normal
general model, then $\sH$ is consistent in {\pfsysuqe}.\esp
\end{thm}

\begin{proof}
Let $\sM$ be a normal general model for $\sH$.  Assume that $\sH$ is
inconsistent in {\pfsysuqe}, i.e., that $\proves{\sH}{F_o}$.  Then, by
the Soundness Theorem, $\sH \modelsn F_o$ and hence $\sM \models F_o$.
This means that $\sV^{\cal M}_{\phi}(F_o) = \mname{T}$ and thus
$\sV^{\cal M}_{\phi}(F_o) \not= \mname{F}$ (for any assignment
$\phi$), which contradicts the definition of a general model.
\end{proof}

\section{Some Metatheorems} \label{sec:metatheorems}

We will prove several metatheorems of {\qzerouqe}.  Most of them will
be metatheorems that we need in order to prove the evaluation-free
completeness of {\qzerouqe} in section~\ref{sec:completeness} and the
results in section~\ref{sec:applications}.

\subsection{Analogs to Metatheorems of {\qzero}}

Most of the metatheorems we prove in this subsection are analogs of
the metatheorems of {\qzero} proven in section~52 of~\cite{Andrews02}.
There will be two versions for many of them, the first restricted to
evaluation-free proofs and the second unrestricted.  In this
subsection, let $\sH^{\rm ef}$ be a set of syntactically closed
evaluation-free formulas of {\qzerouqe} and $\sH$ be a set of
sentences of {\qzerouqe}.

\begin{prop}[Analog of 5200 in~\cite{Andrews02}]\label{prop:5200}
\be

  \item[]

  \item $\provesa{}{\textbf{A}_\alpha \QuasiEqual \textbf{A}_\alpha}$
    where $\textbf{A}_\alpha$ is evaluation-free.

  \item $\proves{}{\textbf{A}_\alpha \QuasiEqual \textbf{A}_\alpha}$.

\ee
\end{prop}

\begin{proof}
By Axiom 7 for both parts.
\end{proof}

\begin{thm}[Tautology Theorem: Analog of 5234]\label{thm:tautology}
\be

  \item[]

  \item Let $\textbf{A}^{1}_{o}, \ldots, \textbf{A}^{n}_{o},
    \textbf{B}_o$ be evaluation-free.  If $\provesa{\sH^{\rm
        ef}}{\textbf{A}^{1}_{o}}$, \ldots, $\provesa{\sH^{\rm
        ef}}{\textbf{A}^{n}_{o}}$ and $[\textbf{A}^{1}_{o} \Andd
      \cdots \Andd \textbf{A}^{n}_{o}] \Implies \textbf{B}_o$ is
    tautologous for $n \ge 1$, then $\provesa{\sH^{\rm
        ef}}{\textbf{B}_o}$.  Also, if $\textbf{B}_o$ is tautologous,
    then $\provesa{\sH^{\rm ef}}{\textbf{B}_o}$.

  \item If $\proves{\sH}{\textbf{A}^{1}_{o}}$, \ldots,
    $\proves{\sH}{\textbf{A}^{n}_{o}}$ and $[\textbf{A}^{1}_{o} \Andd
    \cdots \Andd \textbf{A}^{n}_{o}] \Implies \textbf{B}_o$ is
    tautologous for $n \ge 1$, then $\proves{\sH}{\textbf{B}_o}$.
    Also, if $\textbf{B}_o$ is tautologous, then
    $\proves{\sH}{\textbf{B}_o}$.

\ee
\end{thm}

\begin{proof}
Follows from Axiom 5 (Tautologous Formulas) and Rule 2 (Modus Ponens)
for both parts.
\end{proof}

\begin{lem}\label{lem:quasi-equality}
\be

  \item[]

  \item $\provesa{}{[\textbf{A}_\alpha = \textbf{B}_\alpha] \Implies
    [\textbf{A}_\alpha \QuasiEqual \textbf{B}_\alpha]}$ where
    $\textbf{A}_\alpha$ and $\textbf{B}_\alpha$ are evaluation-free.

  \item $\proves{}{[\textbf{A}_\alpha = \textbf{B}_\alpha] \Implies
    [\textbf{A}_\alpha \QuasiEqual \textbf{B}_\alpha]}$.

\ee
\end{lem}

\begin{proof}
Follows from the definition of $\QuasiEqual$ and the Tautology
Theorem for both parts.
\end{proof}

\begin{lem}\label{lem:equality}
\be

  \item[]

  \item If $\provesa{\sH^{\rm ef}}{{\textbf{A}_\alpha\IsDefApp}}$ or
    $\provesa{\sH^{\rm ef}}{{\textbf{B}_\alpha\IsDefApp}}$, then
    $\provesa{\sH^{\rm ef}}{\textbf{A}_\alpha \QuasiEqual
    \textbf{B}_\alpha}$ implies $\provesa{\sH^{\rm
      ef}}{\textbf{A}_\alpha = \textbf{B}_\alpha}$ where
    $\textbf{A}_\alpha$ and $\textbf{B}_\alpha$ are evaluation-free.

  \item If $\proves{\sH}{{\textbf{A}_\alpha\IsDefApp}}$ or
    $\proves{\sH}{{\textbf{B}_\alpha\IsDefApp}}$, then
    $\proves{\sH}{\textbf{A}_\alpha \QuasiEqual \textbf{B}_\alpha}$
    implies $\proves{\sH}{\textbf{A}_\alpha = \textbf{B}_\alpha}$.

\ee
\end{lem}

\begin{proof}
Follows from the definition of $\QuasiEqual$ and the Tautology Theorem
for both parts.
\end{proof}

\begin{cor}\label{cor:T}
$\provesa{}{T_o}$.
\end{cor}

\begin{proof}
By the definition of $T_o$, Axiom 6.2, Lemma~\ref{lem:equality}, and
Proposition~\ref{prop:5200}. \phantom{XX}
\end{proof}

\bigskip

Both versions of the Quasi-Equality Rules (analog of the Equality
Rules (5201)) follow from Lemma~\ref{prop:5200} and Rule 1.  By virtue
of Lemmas~\ref{lem:quasi-equality} and the Quasi-Equality Rules,
Rule~1 is valid if the hypothesis $\textbf{A}_\alpha \QuasiEqual
\textbf{B}_\alpha$ is replaced by $\textbf{B}_\alpha \QuasiEqual
\textbf{A}_\alpha$, $\textbf{A}_\alpha = \textbf{B}_\alpha$, or
$\textbf{B}_\alpha = \textbf{A}_\alpha$,

\begin{prop}\label{prop:o-defined}
\be

  \item[]

  \item $\provesa{}{{\textbf{A}_o\IsDefApp}}$ where $\textbf{A}_o$ is
    evaluation-free.

  \item $\proves{}{{\textbf{A}_o\IsDefApp}}$.

\ee
\end{prop}

\begin{proof}
By Axioms 6.1--3 and 6.5--8 for both parts.
\end{proof}

\begin{lem}\label{lem:eval-free-sub}
Let $\textbf{A}_\alpha$ and $\textbf{B}_\beta$ be evaluation-free.
Either \[\provesa{}{[\mname{sub}_{\epsilon\epsilon\epsilon\epsilon} \,
    \synbrack{\textbf{A}_\alpha} \, \synbrack{\textbf{x}_\alpha} \,
    \synbrack{\textbf{B}_\beta}]\IsUndefApp}\] or
\[\provesa{}{\mname{sub}_{\epsilon\epsilon\epsilon\epsilon} \,
  \synbrack{\textbf{A}_\alpha} \, \synbrack{\textbf{x}_\alpha} \,
  \synbrack{\textbf{B}_\beta} = \synbrack{\textbf{C}_\beta}}\] for
some (evaluation-free) wff $\textbf{C}_\beta$.
\end{lem}

\begin{proof}
Follows from Axiom 6.11, Axiom 10, Lemma~\ref{lem:quasi-equality},
Specifications 7--9, the Tautology Theorem, and Rule 1.
\end{proof}

\bigskip

\bsp When $\textbf{A}_\alpha$ and $\textbf{B}_\beta$ are
evaluation-free, let $\mname{S}^{{\bf x}_\alpha}_{{\bf A}_\alpha}
\textbf{B}_\beta$ be the wff
$\mname{sub}_{\epsilon\epsilon\epsilon\epsilon} \,
\synbrack{\textbf{A}_\alpha} \, \synbrack{\textbf{x}_\alpha} \,
\synbrack{\textbf{B}_\beta}$ denotes if
$\mname{sub}_{\epsilon\epsilon\epsilon\epsilon} \,
\synbrack{\textbf{A}_\alpha} \, \synbrack{\textbf{x}_\alpha} \,
\synbrack{\textbf{B}_\beta}$ is defined and be undefined
otherwise.\esp

\newpage

\begin{thm}[Beta-Reduction Theorem: Analog of 5207]
\be

  \item[]

  \item $\provesa{}{{\textbf{A}_\alpha \IsDefApp} \Implies [\lambda
      \textbf{x}_\alpha \textbf{B}_\beta]\textbf{A}_\alpha \QuasiEqual
    \mname{S}^{{\bf x}_\alpha}_{{\bf A}_\alpha} \textbf{B}_\beta}$,
    provided $\mname{S}^{{\bf x}_\alpha}_{{\bf A}_\alpha}
    \textbf{B}_\beta$ is defined, where $\textbf{A}_\alpha$ and
    $\textbf{B}_\beta$ are evaluation-free.

   \item $\proves{}{[{\textbf{A}_\alpha \IsDefApp} \Andd
       \mname{sub}_{\epsilon\epsilon\epsilon\epsilon} \,
       \synbrack{\textbf{A}_\alpha} \, \synbrack{\textbf{x}_\alpha} \,
       \synbrack{\textbf{B}_\beta} = \synbrack{\textbf{C}_\beta}]}
     \Implies [\lambda \textbf{x}_\alpha
       \textbf{B}_\beta]\textbf{A}_\alpha \QuasiEqual
     \textbf{C}_\beta.$

\ee
\end{thm}

\begin{proof}
Part 1 is by Axiom 4, Lemma~\ref{lem:eval-free-sub}, and the Tautology
Theorem.  Part~2 is immediately by Axiom~4.1.
\end{proof}

\begin{thm}[Universal Instantiation: Analog of 5215]\label{thm:5215}
\be

  \item[]

  \item If $\provesa{\sH^{\rm ef}}{{\textbf{A}_\alpha \IsDefApp}}$ and
    $\provesa{\sH^{\rm ef}}{\forall \textbf{x}_\alpha \textbf{B}_o}$, then
    $\provesa{\sH^{\rm ef}}{\mname{S}^{{\bf x}_\alpha}_{{\bf A}_\alpha}
    \textbf{B}_o}$, provided $\mname{S}^{{\bf x}_\alpha}_{{\bf
      A}_\alpha} \textbf{B}_\beta$ is defined, where
    $\textbf{A}_\alpha$ and $\textbf{B}_\beta$ are evaluation-free.

  \item If $\proves{\sH}{{\textbf{A}_\alpha \IsDefApp}}$,
    $\proves{\sH}{\mname{sub}_{\epsilon\epsilon\epsilon\epsilon} \,
    \synbrack{\textbf{A}_\alpha} \, \synbrack{\textbf{x}_\alpha} \,
    \synbrack{\textbf{B}_o} = \synbrack{\textbf{C}_o}}$, and
    $\proves{\sH}{\forall \textbf{x}_\alpha \textbf{B}_o}$, then
    $\proves{\sH}{ \textbf{C}_o}$.

  \item If $\proves{\sH}{[\lambda \textbf{x}_\alpha
      \textbf{B}_o]\textbf{A}_\alpha = \textbf{C}_o}$ and
    $\proves{\sH}{\forall \textbf{x}_\alpha \textbf{B}_o}$, then
    $\proves{\sH}{ \textbf{C}_o}$.

  \item If $\proves{\sH}{\forall \textbf{x}_\alpha \textbf{B}_o}$,
    then $\proves{\sH}{ \textbf{B}_o}$.

\ee
\end{thm}

\begin{proof}

\medskip

\noindent \textbf{Part 1} {\sglsp} 
\begin{align} \setcounter{equation}{0}
\sH^{\rm ef}
&\vdash^{\rm ef}
\lambda \textbf{x}_\alpha T_o = \lambda \textbf{x}_\alpha \textbf{B}_o. \\
\sH^{\rm ef}
&\vdash^{\rm ef}
[\lambda \textbf{x}_\alpha T_o]\textbf{A}_\alpha \QuasiEqual 
[\lambda \textbf{x}_\alpha \textbf{B}_o]\textbf{A}_\alpha. \\
\sH^{\rm ef}
&\vdash^{\rm ef}
T_o \QuasiEqual
\mname{S}^{{\bf x}_\alpha}_{{\bf A}_\alpha} \textbf{B}_\beta. \\
\sH^{\rm ef}
&\vdash^{\rm ef}
\mname{S}^{{\bf x}_\alpha}_{{\bf A}_\alpha} \textbf{B}_\beta.
\end{align}
(1) is by the definition of $\Forall$; (2) follows from (1) by the
Quasi-Equality Rules; (3) follows from (2) by the first hypothesis,
the Beta-Reduction Theorem (part~1), and Rule 1; and (4) follows from
(3) and Corollary~\ref{cor:T} by Rule 1.

\medskip

\noindent \textbf{Part 2} {\sglsp} Similar to Part 1.

\medskip

\noindent \textbf{Part 3} {\sglsp} Similar to Part 1.

\medskip

\noindent \textbf{Part 4} {\sglsp} Follows from Axiom 4.10,
Lemma~\ref{lem:equality}, and part 3 of this theorem. \phantom{XX}
\end{proof}

\begin{thm}[Universal Generalization: Analog of 5220]\label{thm:gen}
\be

  \item[]

  \item If $\provesa{\sH^{\rm ef}}{\textbf{A}_o}$, then
    $\provesa{\sH^{\rm ef}}{\forall \textbf{x}_\alpha \textbf{A}_o}$
    where $\textbf{A}_o$ is evaluation-free.

  \item If $\proves{\sH}{\textbf{A}_o}$, then $\proves{\sH}{\forall
    \textbf{x}_\alpha \textbf{A}_o}$.

\ee
\end{thm}

\begin{proof}
\medskip

\noindent \textbf{Part 1} {\sglsp} 
\begin{align} \setcounter{equation}{0}
\sH^{\rm ef}
&\vdash^{\rm ef}
\textbf{A}_o \\
\sH^{\rm ef}
&\vdash^{\rm ef}
T_o = \textbf{A}_o \\
\sH^{\rm ef}
&\vdash^{\rm ef}
\lambda x_\alpha T_o = \lambda x_\alpha T_o \\
\sH^{\rm ef}
&\vdash^{\rm ef}
\forall \textbf{x}_\alpha \textbf{A}_o.
\end{align}
(1) is by hypothesis; (2) follows from (1) by the Tautology Theorem;
(3) is by Axiom 6.5, Lemma~\ref{lem:equality}, and
Proposition~\ref{prop:5200}; and (4) follows from (2) and (3) by Rule
1 and the definition of $\Forall$.

\medskip

\noindent \textbf{Part 2} {\sglsp} Similar to Part 1.
\end{proof}

\begin{lem}[Analog of 5209]\label{lem:5209}
If $\provesa{}{{\textbf{A}_\alpha\IsDefApp}}$ and
$\provesa{}{\textbf{B}_\beta \QuasiEqual \textbf{C}_\beta}$, then
$\provesa{}\mname{S}^{{\bf x}_\alpha}_{{\bf
    A}_\alpha}[\textbf{B}_\beta \QuasiEqual \textbf{C}_\beta]$,
provided $\mname{S}^{{\bf x}_\alpha}_{{\bf A}_\alpha}[\textbf{B}_\beta
  \QuasiEqual \textbf{C}_\beta]$ is defined.
\end{lem}

\begin{proof}
Similar to the proof of 5209 in~\cite{Andrews02}.  It uses
Proposition~\ref{prop:5200}, the Beta-Reduction Theorem (part 1), and
Rule 1.
\end{proof}

\begin{cor}\label{cor:5209}
If $\provesa{}{{\textbf{A}_\alpha\IsDefApp}}$ and
$\provesa{}{\textbf{B}_o = \textbf{C}_o}$, then
$\provesa{}\mname{S}^{{\bf x}_\alpha}_{{\bf A}_\alpha}[\textbf{B}_o =
  \textbf{C}_o]$, provided $\mname{S}^{{\bf x}_\alpha}_{{\bf
    A}_\alpha}[\textbf{B}_o = \textbf{C}_o]$ is defined.
\end{cor}

\begin{proof}
By Lemma~\ref{lem:quasi-equality}, Lemma~\ref{lem:5209},
Proposition~\ref{prop:o-defined}, and the Tautology Theorem.
\end{proof}

\begin{lem}[Analog of 5205]\label{lem:5205}
$\provesa{}{f_{\alpha\beta} = \lambda \textbf{y}_\beta[f_{\alpha\beta}
      \textbf{y}_\beta]}.$
\end{lem}

\begin{proof}
Similar to the proof of 5205 in~\cite{Andrews02}.  It uses Axiom 3,
Axioms 6.1 and 6.5, Corollary~\ref{cor:5209},
Lemmas~\ref{lem:quasi-equality} and~\ref{lem:equality}, the
Quasi-Equality Rules, the Beta-Reduction Theorem (part 1), and Rule 1.
\end{proof}

\begin{lem}[Analog of 5206]\label{lem:5206}
$\provesa{}{\lambda \textbf{x}_\beta \textbf{A}_\alpha = \lambda
    \textbf{z}_\beta \mname{S}^{{\bf x}_\beta}_{{\bf
        z}_\beta}\textbf{A}_\alpha}$, provided $\textbf{z}_\beta$ is
  not free in $\textbf{A}_\alpha$ and $\mname{S}^{{\bf x}_\beta}_{{\bf
      z}_\beta}\textbf{A}_\alpha$ is defined.
\end{lem}

\begin{proof}
Similar to the proof of 5206 in~\cite{Andrews02}.  It employs Axioms
6.1 and 6.5, Corollary~\ref{cor:5209}, Lemma~\ref{lem:5205}, the
Beta-Reduction Theorem (part 1), and Rule~1.
\end{proof}

\bigskip

Analogs of $\alpha$-conversion, $\beta$-conversion, and
$\eta$-conversion in~\cite{Andrews02} for evaluation-free proof are
obtained directly from Lemma~\ref{lem:5206}, the Beta-Reduction
Theorem (part 1), Lemma~\ref{lem:5205} using Lemma~\ref{lem:5209} and
Rule 1.

\begin{thm}[Deduction Theorem: Analog of 5240]\label{thm:deduction}\bsp
Let $\textbf{A}_o$ and $\textbf{H}_o$ be syntactically closed
evaluation-free formulas.  If $\provesa{\sH^{\rm ef} \cup
  \set{\textbf{H}_o}}{\textbf{A}_o}$, then $\provesa{\sH^{\rm
    ef}}{\textbf{H}_o \Implies \textbf{A}_o}$.\esp
\end{thm}

\begin{proof}
Similar to the proof of 5240 in~\cite{Andrews02}.  It uses Axioms 1--3
and 6, the Tautology Theorem, the Beta-Reduction Theorem (part~1),
Universal Instantiation (part~1), Universal Generalization,
$\alpha$-conversion, and Rule 1.
\end{proof}

\subsection{Other Metatheorems}

The metatheorems we prove in this subsection are not analogs of
metatheorems of {\qzero}; they involve ordered pairs, quotation, and evaluation.

\begin{lem}[Ordered Pairs]\label{lem:pairs}
\be

  \item[]

  \item $\provesa{}{\Forall x_\alpha \Forall y_\beta
    [\mname{pair}_{\seq{\alpha\beta}\beta\alpha} \, x_\alpha \, y_\beta]  \IsDefApp}$.

  \item $\provesa{}{\Forall x_\alpha \Forall y_\beta
    [\mname{fst}_{\alpha\seq{\alpha\beta}} \, 
    [\mname{pair}_{\seq{\alpha\beta}\beta\alpha} \, x_\alpha \, y_\beta] = x_\alpha]}$.

  \item $\provesa{}{\Forall x_\alpha \Forall y_\beta
    [\mname{snd}_{\beta\seq{\alpha\beta}} \, 
    [\mname{pair}_{\seq{\alpha\beta}\beta\alpha} \, x_\alpha \, y_\beta] = v_\beta]}$.

  \item $\provesa{}{\Forall z_{\seq{\alpha\beta}} 
    [\mname{pair}_{\seq{\alpha\beta}\beta\alpha} 
    [\mname{fst}_{\alpha\seq{\alpha\beta}} \, z_{\seq{\alpha\beta}}]
    [\mname{snd}_{\beta\seq{\alpha\beta}} \, z_{\seq{\alpha\beta}}] 
    = z_{\seq{\alpha\beta}}]}$.

\ee
\end{lem}

\begin{proof}
These four metatheorems of {\qzerouqe} can be straightforwardly proved
using the definitions of $\mname{fst}_{\alpha\seq{\alpha\beta}}$ and
$\mname{snd}_{\beta\seq{\alpha\beta}}$ and Axioms 8 and 9.
\end{proof}

\begin{thm}[Injectiveness of Quotation]\label{lem:quo-inj}
If $\provesa{}{\synbrack{\textbf{A}_\alpha}} =
\synbrack{\textbf{B}_\alpha}$, then $\textbf{A}_\alpha =
\textbf{B}_\alpha$.
\end{thm}

\begin{proof}  Assume $\provesa{}\synbrack{\textbf{A}_\alpha} =
\synbrack{\textbf{B}_\alpha}$.  By Specification 1 and Rule 1, this
implies $\provesa{}{\sE(\textbf{A}_\alpha) = \sE(\textbf{B}_\beta)}$.
From this and Specifications 4.1--28, we can prove that
$\textbf{A}_\alpha = \textbf{B}_\alpha$ by induction on the size of
$\textbf{A}_\alpha$.
\end{proof}

\begin{thm}[Disquotation Theorem] \label{thm:disquote}
\bsp If $\textbf{D}_\delta$ is evaluation-free, then
\mbox{$\proves{}{\sembrack{\synbrack{\textbf{D}_\delta}}_\alpha
  \QuasiEqual \textbf{D}_\delta}$}.\esp
\end{thm}

\begin{proof}
The proof is by induction on the size of $\textbf{D}_\delta$.

\be

  \item[] \textbf{Case 1}: $\textbf{D}_\delta$ is $\textbf{x}_\alpha$.
    Then $\proves{}{\sembrack{\synbrack{\textbf{x}_\alpha}}_\alpha} =
    \textbf{x}_\alpha$ by Axiom 11.1.

  \item[] \textbf{Case 2}: $\textbf{D}_\delta$ is a primitive constant
    $\textbf{c}_\alpha$.  Then
    $\proves{}{\sembrack{\synbrack{\textbf{c}_\alpha}}_\alpha} =
    \textbf{c}_\alpha$ by Axiom~11.2.

  \item[] \textbf{Case 3}: $\textbf{D}_\delta$ is
    $\textbf{A}_{\alpha\beta}\textbf{B}_\beta$.  Assume (a)
    $\textbf{A}_{\alpha\beta}\textbf{B}_\beta$ is evaluation-free.
    (a) implies (b) $\textbf{A}_{\alpha\beta}$ and $\textbf{B}_\beta$
    are evaluation-free.  Then we can derive the conclusion of the
    theorem as follows:
    \begin{align} \setcounter{equation}{0}
    &\vdash
    \sembrack{\synbrack{\textbf{A}_{\alpha\beta}\textbf{B}_\beta}}_\alpha \QuasiEqual
    \sembrack{\synbrack{\textbf{A}_{\alpha\beta}\textbf{B}_\beta}}_\alpha. \\
    &\vdash
    \sembrack{\synbrack{\textbf{A}_{\alpha\beta}\textbf{B}_\beta}}_\alpha \QuasiEqual
    \sembrack{\sE(\textbf{A}_{\alpha\beta}\textbf{B}_\beta)}_\alpha. \\
    &\vdash
    \sembrack{\synbrack{\textbf{A}_{\alpha\beta}\textbf{B}_\beta}}_\alpha \QuasiEqual
    \sembrack{\mname{app}_{\epsilon\epsilon\epsilon} \,
    \sE(\textbf{A}_{\alpha\beta}) \, \sE(\textbf{B}_\beta)}_{\alpha}. \\
    &\vdash
    \sembrack{\synbrack{\textbf{A}_{\alpha\beta}\textbf{B}_\beta}}_\alpha \QuasiEqual
    \sembrack{\mname{app}_{\epsilon\epsilon\epsilon} \,
    \synbrack{\textbf{A}_{\alpha\beta}} \, \synbrack{\textbf{B}_\beta}}_{\alpha}. \\
    &\vdash
    \sembrack{\mname{app}_{\epsilon\epsilon\epsilon} \,
    \synbrack{\textbf{A}_{\alpha\beta}} \, \synbrack{\textbf{B}_\beta}}_{\alpha} \QuasiEqual
    \sembrack{\synbrack{\textbf{A}_{\alpha\beta}}}_{\alpha\beta}
    \sembrack{\synbrack{\textbf{B}_\beta}}_\beta. \\
    &\vdash
    \sembrack{\synbrack{\textbf{A}_{\alpha\beta}\textbf{B}_\beta}}_\alpha \QuasiEqual
    \sembrack{\synbrack{\textbf{A}_{\alpha\beta}}}_{\alpha\beta}
    \sembrack{\synbrack{\textbf{B}_\beta}}_\beta. \\
    &\vdash
    \sembrack{\synbrack{\textbf{A}_{\alpha\beta}\textbf{B}_\beta}}_\alpha \QuasiEqual
    \textbf{A}_{\alpha\beta}\textbf{B}_\beta.
    \end{align}
    (1) is by Proposition~\ref{prop:5200}; (2) and (4) follow from the
    (1) and (3), respectively, and Specification 1 by Rule 1; (3)
    follows from (2) by the definition of $\sE$; (5) is by Axiom 11.3;
    (6) follows from (4) and (5) by Rule~1; (7) follows from (b), the
    induction hypothesis, and (6) by Rule~1.

  \item[] \textbf{Case 4}: $\textbf{D}_\delta$ is
    $\lambda\textbf{x}_\beta\textbf{A}_\alpha$.  Similar to Case 2.
    It is necessary to use the fact that $\sE(\textbf{A}_\alpha)$ is
    semantically closed.

  \item[] \textbf{Case 5}: $\textbf{D}_\delta$ is $\If \, \textbf{A}_o
    \, \textbf{B}_\alpha \, \textbf{C}_\alpha$.  Similar to Case 2.

  \item[] \textbf{Case 6}: $\textbf{D}_\delta$ is
    $\synbrack{\textbf{A}_\alpha}$.  Then we can derive the conclusion
    of the theorem as follows:
    \begin{align} \setcounter{equation}{0}
    &\vdash
    \sembrack{\synbrack{\synbrack{\textbf{A}_\alpha}}}_\epsilon \QuasiEqual
    \sembrack{\synbrack{\synbrack{\textbf{A}_\alpha}}}_\epsilon. \\
    &\vdash
    \sembrack{\synbrack{\synbrack{\textbf{A}_\alpha}}}_\epsilon \QuasiEqual
    \sembrack{\sE(\synbrack{\textbf{A}_\alpha})}_\epsilon. \\
    &\vdash
    \sembrack{\synbrack{\synbrack{\textbf{A}_\alpha}}}_\epsilon \QuasiEqual
    \sembrack{\mname{quot}_{\epsilon\epsilon} \, \sE(\textbf{A}_\alpha)}_\epsilon. \\
    &\vdash
    \sembrack{\synbrack{\synbrack{\textbf{A}_\alpha}}}_\epsilon \QuasiEqual
    \sembrack{\mname{quot}_{\epsilon\epsilon} \, \synbrack{\textbf{A}_\alpha}}_\epsilon. \\
    &\vdash
    \sembrack{\mname{quot}_{\epsilon\epsilon} \, \synbrack{\textbf{A}_\alpha}}_\epsilon \QuasiEqual 
    \If \, \sembrack{\mname{quot}_{\epsilon\epsilon} \, \synbrack{\textbf{A}_\alpha}}_\epsilon \IsDefApp \, 
    \synbrack{\textbf{A}_\alpha} \, \Undefined_\epsilon. \\
    &\vdash
    \sembrack{\mname{quot}_{\epsilon\epsilon} \, \synbrack{\textbf{A}_\alpha}}_\epsilon \QuasiEqual 
    \If \, \sembrack{\synbrack{\synbrack{\textbf{A}_\alpha}}}_\epsilon \IsDefApp \, 
    \synbrack{\textbf{A}_\alpha} \, \Undefined_\epsilon. \\
    &\vdash
    \sembrack{\synbrack{\synbrack{\textbf{A}_\alpha}}}_\epsilon \IsDefApp. \\ 
    &\vdash
    \sembrack{\mname{quot}_{\epsilon\epsilon} \, \synbrack{\textbf{A}_\alpha}}_\epsilon \QuasiEqual 
    \synbrack{\textbf{A}_\alpha}.\\
    &\vdash
    \sembrack{\synbrack{\synbrack{\textbf{A}_\alpha}}}_\epsilon \QuasiEqual
    \synbrack{\textbf{A}_\alpha}.
    \end{align}
    (1) is by Proposition~\ref{prop:5200}; (2) and (4) follow from the
    (1) and (3), respectively, and Specification 1 by Rule 1; (3)
    follows from (2) by the definition of $\sE$; (5) is by Axiom 11.6;
    (6) follows from (4) and (5) by Rule~1; (7) is by Axiom~6.8; (8)
    follows from (6) and (7) by Axiom 10.1 and Rule~1; and (9) follows
    from (4) and (8) by Rule 1.

  \item[] \textbf{Case 7}: $\textbf{D}_\delta$ is
    $\sembrack{\textbf{A}_{\epsilon}}_\alpha$.  The theorem holds
    trivially in this case since $\textbf{D}_\delta$ is not
    evaluation-free.

\ee

\end{proof}

\section{Completeness} \label{sec:completeness}

{\pfsysuqe} is \emph{complete for {\qzerouqe}} if $\sH \modelsn
\textbf{A}_o$ implies $\proves{\sH}{\textbf{A}_o}$ whenever
$\textbf{A}_o$ is a formula and $\sH$ is a set of sentences of
       {\qzerouqe}.  However, {\pfsysuqe} is actually not complete for
       {\qzerouqe}.  For instance, let $A_o$ be the sentence
\[[\lambda x_\epsilon \lambda y_\epsilon 
\sembrack{\mname{app}_{\epsilon\epsilon\epsilon} \, x_\epsilon \,
  y_\epsilon}_\alpha]\synbrack{\NegAlt_{oo}}\synbrack{T_o} =
F_o.\] Then, as observed in subsection~\ref{subsec:undef-sub}, ${}
\modelsn A_o$ holds but $\proves{}{A_o}$ does not hold.

{\pfsysuqe} is \emph{evaluation-free complete for {\qzerouqe}} if $\sH
\modelsna \textbf{A}_o$ implies $\provesa{\sH}{\textbf{A}_o}$ whenever
$\textbf{A}_o$ is an evaluation-free formula and $\sH$ is a set of
syntactically closed evaluation-free formulas of {\qzerouqe}.  We will
prove that {\pfsysuqe} is evaluation-free complete.  Our proof will
closely follow the proof of Theorem 22 (Henkin's Completeness Theorem
for {\qzerou}) in~\cite{Farmer08a} which itself is based on the proof
of 5502 (Henkin's Completeness and Soundness Theorem)
in~\cite{Andrews02}.

\subsection{Extension Lemma}

For any set $S$, let $\mbox{card}(S)$ be the cardinality of $S$.  Let
$\sL(\mbox{\qzerouqe})$ be the set of wffs of {\qzerouqe}, let $\kappa
= \mbox{card}(\sL(\mbox{\qzerouqe}))$, let $\sC_\alpha$ be a
well-ordered set of cardinality $\kappa$ of new primitive constants of
type $\alpha$ for each $\alpha \in \sT$, and let $\sC =\bigcup_{\alpha
  \in {\cal T}} \sC_\alpha$.

Define {\qzerouqeplus} to be the logic that extends {\qzerouqe} as
follows.  The syntax of {\qzerouqeplus} is obtained from the syntax of
{\qzerouqe} by adding the members of $\sC$ to the primitive constants
of {\qzerouqe} without extending the set of quotations of {\qzerouqe}.
That is, $\synbrack{\textbf{c}_\alpha}$ is not a wff of
{\qzerouqeplus} for all $\textbf{c}_\alpha \in \sC$, and $\sE$ is
still only defined on the wffs of {\qzerouqe}.  Let
$\sL(\mbox{\qzerouqeplus})$ be the set of wffs of {\qzerouqeplus}.
Obviously, $\mbox{card}(\sL(\mbox{\qzerouqeplus})) = \kappa$.  The
semantics of {\qzerouqeplus} is the same as the semantics of
{\qzerouqe} except that a general or evaluation-free model for
{\qzerouqeplus} is a general or evaluation-free model
$\seq{\set{\sD_\alpha \;|\; \alpha \in \sT}, \sJ}$ for {\qzerouqe}
where the domain of $\sJ$ has been extended to include~$\sC$.  Let
{\pfsysuqeplus} be the proof system that is obtained from {\pfsysuqe}
by replacing the phrase ``primitive constant'' with the phrase
``primitive constant not in $\sC$'' in each formula schema in
Specifications 1--9 and Axioms 1--12 except Axiom 6.2.  Since
$\sL(\mbox{\qzerouqeplus})$ is a proper superset of
$\sL(\mbox{\qzerouqe})$, the axioms of {\pfsysuqeplus} are a proper
superset of the axioms of {\pfsysuqe}.  {\pfsysuqeplus} has the same
rules of inference as {\pfsysuqe}.  Let $\provesb{\sH}{\textbf{A}_o}$
mean there is an evaluation-free proof of $\textbf{A}_o$ from $\sH$ in
{\pfsysuqeplus}.  Assume {\qzerouqeplus} inherits all the other
definitions of {\qzerouqe}.

An \emph{xwff} of {\qzerouqeplus} is a syntactically closed
evaluation-free wff of {\qzerouqeplus}.  An $\textit{xwff}_\alpha$ is
an xwff of type $\alpha$.  Let $\sH$ be a set of $\xwffs{o}$ of
{\qzerouqeplus}.  $\sH$ is \emph{evaluation-free complete} in
{\pfsysuqeplus} if, for every $\xwff{o}$ $\textbf{A}_o$ of
{\qzerouqeplus}, either $\provesb{\sH}{\textbf{A}_o}$ or
$\provesb{\sH}{\NegAlt\textbf{A}_o}$.  $\sH$ is \emph{evaluation-free
  extensionally complete} in {\pfsysuqeplus} if, for every $\xwff{o}$
of the form $\textbf{A}_{\alpha\beta} = \textbf{B}_{\alpha\beta}$ of
{\qzerouqeplus}, there is an xwff $\textbf{C}_\beta$ such that:

\be

  \item $\provesb{\sH}{{\textbf{C}_\beta\IsDefApp}}$.

  \item $\provesb{\sH}{[{\textbf{A}_{\alpha\beta}\IsDefApp} \Andd
  {\textbf{B}_{\alpha\beta}\IsDefApp} \Andd [\textbf{A}_{\alpha\beta}
  \textbf{C}_\beta \QuasiEqual \textbf{B}_{\alpha\beta}
  \textbf{C}_\beta]] \Implies [\textbf{A}_{\alpha\beta} =
  \textbf{B}_{\alpha\beta}]}$.

\ee

\medskip

\begin{lem}[Extension Lemma]\bsp
Let {\sG} be a set of $\xwffs{o}$ of {\qzerouqe} consistent in
{\pfsysuqe}.  Then there is a set $\sH$ of $\xwffs{o}$ of
{\qzerouqeplus} such that:

\be

  \item $\sG \subseteq \sH$.

  \item $\sH$ is consistent in {\pfsysuqeplus}.

  \item $\sH$ is evaluation-free complete in {\pfsysuqeplus}.

  \item $\sH$ is evaluation-free extensionally complete in
    {\pfsysuqeplus}.

\ee
\esp
\end{lem}

\begin{proof}
The proof is very close to the proof of 5500 in~\cite{Andrews02}.  By
transfinite induction, a set $\sG_\tau$ of $\xwffs{\alpha}$ is
defined for each ordinal $\tau \le \kappa$.  The main difference
between our proof and the proof of 5500 is that, in case (c) of the
definition of $\sG_{\tau + 1}$,
\[\sG_{\tau + 1} = \sG_\tau \cup
\set{\NegAlt[{\textbf{A}_{\alpha\beta}\IsDefApp} \Andd
    {\textbf{B}_{\alpha\beta}\IsDefApp} \Andd [\textbf{A}_{\alpha\beta}
      \textbf{c}_\beta \QuasiEqual \textbf{B}_{\alpha\beta}
      \textbf{c}_\beta]]}\] where $\textbf{c}_\beta$ is the first
constant in $\sC_\beta$ that does not occur in $\sG_\tau$ or
$\textbf{A}_{\alpha\beta} = \textbf{B}_{\alpha\beta}$.  (Notice that
$\provesb{}{\textbf{c}_\beta\IsDefApp}$ by Axiom 6.2.)

To prove that $\sG_{\tau + 1}$ is consistent in {\pfsysuqeplus}
assuming $\sG_\tau$ is consistent in {\pfsysuqeplus} when $\sG_{\tau +
  1}$ is obtained by case (c) , it is necessary to show that, if
\[\provesb{\sG_\tau}{{\textbf{A}_{\alpha\beta}\IsDefApp} \Andd
{\textbf{B}_{\alpha\beta}\IsDefApp} \Andd [\textbf{A}_{\alpha\beta}
  \textbf{c}_\beta \QuasiEqual \textbf{B}_{\alpha\beta}
  \textbf{c}_\beta]},\] then
$\provesb{\sG_\tau}{\textbf{A}_{\alpha\beta} =
  \textbf{B}_{\alpha\beta}}$.  Assume the hypothesis of this
statement.  Let $P$ be an evaluation-free proof of
\[{\textbf{A}_{\alpha\beta}\IsDefApp} \Andd 
{\textbf{B}_{\alpha\beta}\IsDefApp} \Andd [\textbf{A}_{\alpha\beta}
  \textbf{c}_\beta \QuasiEqual \textbf{B}_{\alpha\beta}
  \textbf{c}_\beta]\] from a finite subset $\sS$ of $\sG_\tau$, and
let $\textbf{x}_\beta$ be a variable that does not occur in $P$ or
$\sS$.  Since $\textbf{c}_\beta$ does not occur in $\sG_\tau$,
$\textbf{A}_{\alpha\beta}$, or $\textbf{B}_{\alpha\beta}$ and
$\textbf{c}_\beta \in\sC$, the result of substituting
$\textbf{x}_\beta$ for each occurrence of $\textbf{c}_\beta$ in $P$
is an evaluation-free proof of
\[{\textbf{A}_{\alpha\beta}\IsDefApp} \Andd
{\textbf{B}_{\alpha\beta}\IsDefApp} \Andd [\textbf{A}_{\alpha\beta}
\textbf{x}_\beta \QuasiEqual \textbf{B}_{\alpha\beta}
\textbf{x}_\beta]\] from $\sS$.  Therefore,
\[\provesb{\sS}{{\textbf{A}_{\alpha\beta}\IsDefApp} \Andd
{\textbf{B}_{\alpha\beta}\IsDefApp} \Andd [\textbf{A}_{\alpha\beta}
\textbf{x}_\beta \QuasiEqual \textbf{B}_{\alpha\beta}
\textbf{x}_\beta]}.\] This implies
\[\provesb{\sS}{{\textbf{A}_{\alpha\beta}\IsDefApp}}, \sglsp
\provesb{\sS}{{\textbf{B}_{\alpha\beta}\IsDefApp}}, \sglsp
\provesb{\sS}{\forall \textbf{x}_\beta [\textbf{A}_{\alpha\beta}
    \textbf{x}_\beta \QuasiEqual \textbf{B}_{\alpha\beta}
    \textbf{x}_\beta]}\] by the Tautology Theorem and Universal
Generalization.  It follows from these that
$\provesb{\sG_\tau}{\textbf{A}_{\alpha\beta} =
  \textbf{B}_{\alpha\beta}}$ by Axiom 3, the Tautology Theorem, and
Rule~1.

The rest of the proof is essentially the same as the proof of 5500.
\end{proof}

\subsection{Henkin's Theorem}

\bsp A general or evaluation-free model $\seq{\set{\sD_\alpha \;|\;
    \alpha \in \sT}, \sJ}$ for {\qzerouqe} is \emph{frugal} if
$\mbox{card}(\sD_\alpha) \le \mbox{card}(\sL(\mbox{\qzerouqe}))$ for
all $\alpha \in \sT$. \esp

\begin{thm}[Henkin's Theorem for {\pfsysuqe}] \label{thm:henkin}
Every set of syntactically closed evaluation-free formulas of
{\qzerouqe} consistent in {\pfsysuqe} has a frugal normal
evaluation-free model.
\end{thm}

\begin{proof}
The proof is very close to the proof of Theorem 21
in~\cite{Farmer08a}.  Let $\sG$ be a set of $\xwffs{\alpha}$ of
{\qzerouqe} consistent in {\pfsysuqe}, and let $\sH$ be a set of
$\xwffs{\alpha}$ of {\qzerouqeplus} that satisfies the four statements
of the Extension Lemma.

\medskip

\noindent \textbf{Step 1} {\sglsp} We define simultaneously, by
recursion on $\gamma \in \sT$, a frame $\set{\sD_\alpha \;|\; \alpha
  \in \sT}$ and a partial function $\sV$ whose domain is the set of
xwffs of {\qzerouqeplus} so that the following conditions hold for all
$\gamma \in \sT$:

\bi

  \item [$(1^\gamma)$] $\sD_\gamma = \set{\sV(\textbf{A}_\gamma) \;|\;
  \textbf{A}_\gamma \mbox{ is a } \xwff{\gamma} \mbox{ and }
  \provesb{\sH}{{\textbf{A}_\gamma\IsDefApp}}}$.

  \item [$(2^\gamma)$] $\sV(\textbf{A}_\gamma)$ is defined \,iff\,
  $\provesb{\sH}{{\textbf{A}_\gamma\IsDefApp}}$ for all xwffs
  $\textbf{A}_\gamma$.

  \item [$(3^\gamma)$] $\sV(\textbf{A}_\gamma) =
  \sV(\textbf{B}_\gamma)$ \,iff\, $\provesb{\sH}{\textbf{A}_\gamma =
  \textbf{B}_\gamma}$ for all xwffs $\textbf{A}_\gamma$ and
  $\textbf{B}_\gamma$.

\ei
Let $\sV(x) \QuasiEqual \sV(y)$ mean either $\sV(x)$ and $\sV(y)$ are
both defined and equal or $\sV(x)$ and $\sV(y)$ are both undefined.

\medskip

\noindent \textbf{Step 1.1} {\sglsp} We define $\sD_{\iotaAltS}$
and $\sV$ on $\xwffs{\iotaAltS}$.  For each xwff
$\textbf{A}_{\iotaAltS}$, if
$\provesb{\sH}{{\textbf{A}_{\iotaAltS}\IsDefApp}}$, let
\[\sV(\textbf{A}_{\iotaAltS}) = \set{\textbf{B}_{\iotaAltS} \;|\;
\textbf{B}_{\iotaAltS} \mbox{ is a } \xwff{\iotaAltS} \mbox{ and }
\provesb{\sH}{\textbf{A}_{\iotaAltS} = \textbf{B}_{\iotaAltS}}},\]
and otherwise let $\sV(\textbf{A}_{\iotaAltS})$ be undefined.  Also, let
\[\sD_{\iotaAltS} = \set{\sV(\textbf{A}_{\iotaAltS}) \;|\; 
\textbf{A}_{\iotaAltS} \mbox{ is a } \xwff{\iotaAltS} \mbox{ and }
\provesb{\sH}{\textbf{A}_{\iotaAltS}\IsDefApp}}.\]
$(1^{\iotaAltS})$, $(2^{\iotaAltS})$, and $(3^{\iotaAltS})$ are clearly
satisfied.

\medskip

\noindent \textbf{Step 1.2} {\sglsp} We define $\sD_o$ and $\sV$ on
$\xwffs{o}$.  For each xwff $\textbf{A}_o$, if
$\provesb{\sH}{\textbf{A}_o}$, let $\sV(\textbf{A}_o) = \mname{T}$,
and otherwise let $\sV(\textbf{A}_o) = \mname{F}$.  Also, let $\sD_o =
\set{\mname{T},\mname{F}}$.  By the consistency and evaluation-free
completeness of $\sH$, exactly one of $\provesb{\sH}{\textbf{A}_o}$
and $\provesb{\sH}{\NegAlt\textbf{A}_o}$ holds.  By
Proposition~\ref{prop:o-defined},
$\provesb{\sH}{\textbf{A}_o\IsDefApp}$ for all wffs $\textbf{A}_o$.
Hence $(1^o)$, $(2^o)$, and $(3^o)$ are satisfied.

\medskip

\noindent \textbf{Step 1.3} {\sglsp} We define $\sD_\epsilon$ and
$\sV$ on $\xwffs{\epsilon}$.  Let \[\sD_\epsilon =
\set{\sE(\textbf{A}_\alpha) \;|\; \textbf{A}_\alpha \mbox{ is a wff of
    {\qzerouqe}}}.\] Choose a mapping $f$ from
$\set{\textbf{A}_\epsilon \;|\; \textbf{A}_\epsilon \mbox{ is an }
  \xwff{\epsilon} \mbox{ and }
  \provesb{\sH}{\textbf{A}_{\epsilon}\IsDefApp}}$ to $\sD_\epsilon$
such that:

\be

  \item $f(\textbf{A}_\epsilon) = f(\textbf{B}_\epsilon)$ iff
    $\provesb{\sH}{\textbf{A}_\epsilon = \textbf{B}_\epsilon}$.

  \item If $\provesb{\sH}{\textbf{A}_\epsilon = \sE(\textbf{C}_\gamma)}$, 
    then $f(\textbf{A}_\epsilon) = \sE(\textbf{C}_\gamma)$.

  \item If $\provesb{\sH}{\mname{wff}^{\alpha}_{o\epsilon} \,
    \textbf{A}_\epsilon}$, then $f(\textbf{A}_\epsilon) =
    \sE(\textbf{C}_\alpha)$ for some wff $\textbf{C}_\alpha$.

\ee

\bsp \noindent It is possible to choose such a mapping by
Lemma~\ref{lem:quo-inj}, Specification 6.13, and the fact that
$\mbox{card}(\sL(\mbox{\qzerouqe})) =
\mbox{card}(\sL(\mbox{\qzerouqeplus}))$.  For each xwff
$\textbf{A}_\epsilon$, if
$\provesb{\sH}{{\textbf{A}_{\epsilon}\IsDefApp}}$, let
$\sV(\textbf{A}_\epsilon) = f(\textbf{A}_\epsilon)$, and otherwise let
$\sV(\textbf{A}_\epsilon$ be undefined.  $(2^{\epsilon})$ and
$(3^{\epsilon})$ are clearly satisfied; $(1^{\epsilon})$ is satisfied
since, for all wffs $\textbf{A}_\alpha$ of {\qzerouqe},
$\sE(\textbf{A}_\alpha)$ is an $\xwff{\epsilon}$ by the definition of
$\sE$ and Specification~7 and
$\provesb{\sH}{\sE(\textbf{A}_\alpha)\IsDefApp}$ by Axiom 6.7 and
Specification~1.\esp

\medskip

\noindent \textbf{Step 1.4} {\sglsp} We define $\sD_{\alpha\beta}$
and $\sV$ on $\xwffs{\alpha\beta}$ for all $\alpha, \beta \in \sT$.
Now suppose that $\sD_\alpha$ and $\sD_{\beta}$ are defined and that
the conditions hold for $\alpha$ and $\beta$.  For each xwff
$\textbf{A}_{\alpha\beta}$, if
$\provesb{\sH}{{\textbf{A}_{\alpha\beta}\IsDefApp}}$, let
$\sV(\textbf{A}_{\alpha\beta})$ be the (partial or total) function
from $\sD_{\beta}$ to $\sD_\alpha$ whose value, for any argument
$\sV(\textbf{B}_\beta) \in \sD_\beta$, is
$\sV(\textbf{A}_{\alpha\beta}\textbf{B}_\beta)$ if
$\sV(\textbf{A}_{\alpha\beta}\textbf{B}_\beta)$ is defined and is
undefined if $\sV(\textbf{A}_{\alpha\beta}\textbf{B}_\beta)$ is
undefined, and otherwise let $\sV(\textbf{A}_{\alpha\beta})$ be
undefined.  We must show that this definition is independent of the
particular xwff $\textbf{B}_\beta$ used to represent the argument.  So
suppose $\sV(\textbf{B}_\beta) = \sV(\textbf{C}_\beta)$; then
$\provesb{\sH}{\textbf{B}_\beta = \textbf{C}_\beta}$ by $(3^{\beta})$,
so $\provesb{\sH}{\textbf{A}_{\alpha\beta}\textbf{B}_\beta \QuasiEqual
  \textbf{A}_{\alpha\beta}\textbf{C}_\beta}$ by
Lemma~\ref{lem:quasi-equality} and the Quasi-Equality Rules, and so
$\sV(\textbf{A}_{\alpha\beta}\textbf{B}_\beta) \QuasiEqual
\sV(\textbf{A}_{\alpha\beta}\textbf{C}_\beta)$ by $(2^{\alpha})$ and
$(3^{\alpha})$.  Finally, let \[\sD_{\alpha\beta} =
\set{\sV(\textbf{A}_{\alpha\beta}) \;|\; \textbf{A}_{\alpha\beta}
  \mbox{ is a } \xwff{\alpha\beta} \mbox{ and }
  \provesb{\sH}{{\textbf{A}_{\alpha\beta}\IsDefApp}}}.\]
$(1^{\alpha\beta})$ and $(2^{\alpha\beta})$ are clearly satisfied; we
must show that $(3^{\alpha\beta})$ is satisfied.  Suppose
$\sV(\textbf{A}_{\alpha\beta}) = \sV(\textbf{B}_{\alpha\beta})$.  Then
$\provesb{\sH}{{\textbf{A}_{\alpha\beta}\IsDefApp}}$ and
$\provesb{\sH}{{\textbf{B}_{\alpha\beta}\IsDefApp}}$.  Since $\sH$ is
evaluation-free extensionally complete, there is a $\textbf{C}_\beta$
such that $\provesb{\sH}{{\textbf{C}_{\beta}\IsDefApp}}$ and
\[\provesb{\sH}{[{\textbf{A}_{\alpha\beta}\IsDefApp} \Andd
{\textbf{B}_{\alpha\beta}\IsDefApp} \Andd [\textbf{A}_{\alpha\beta}
  \textbf{C}_\beta \QuasiEqual \textbf{B}_{\alpha\beta}
  \textbf{C}_\beta]] \Implies [\textbf{A}_{\alpha\beta} =
    \textbf{B}_{\alpha\beta}]}.\] Then
$\sV(\textbf{A}_{\alpha\beta}\textbf{C}_\beta) \QuasiEqual
\sV(\textbf{A}_{\alpha\beta})(\sV(\textbf{C}_\beta)) \QuasiEqual
\sV(\textbf{B}_{\alpha\beta})(\sV(\textbf{C}_\beta)) \QuasiEqual
\sV(\textbf{B}_{\alpha\beta}\textbf{C}_\beta),$ so
$\provesb{\sH}{\textbf{A}_{\alpha\beta}\textbf{C}_\beta \QuasiEqual
  \textbf{B}_{\alpha\beta}\textbf{C}_\beta}$ by $(2^\alpha)$ and
$(3^\alpha)$, and so $\provesb{\sH}{\textbf{A}_{\alpha\beta} =
  \textbf{B}_{\alpha\beta}}$.  Now suppose
$\provesb{\sH}{\textbf{A}_{\alpha\beta} = \textbf{B}_{\alpha\beta}}$.
Then, for all xwffs $\textbf{C}_\beta \in \sD_\beta$,
$\provesb{\sH}{\textbf{A}_{\alpha\beta}\textbf{C}_\beta \QuasiEqual
  \textbf{B}_{\alpha\beta}\textbf{C}_\beta}$ by
Lemma~\ref{lem:quasi-equality} and the Quasi-Equality Rules, and so
$\sV(\textbf{A}_{\alpha\beta})(\sV(\textbf{C}_\beta)) \QuasiEqual
\sV(\textbf{A}_{\alpha\beta}\textbf{C}_\beta) \QuasiEqual
\sV(\textbf{B}_{\alpha\beta}\textbf{C}_\beta) \QuasiEqual
\sV(\textbf{B}_{\alpha\beta})(\sV(\textbf{C}_\beta)).$ Hence
$\sV(\textbf{A}_{\alpha\beta}) = \sV(\textbf{B}_{\alpha\beta})$.

\medskip

\noindent \textbf{Step 1.5} {\sglsp} We define
$\sD_{\seq{\alpha\beta}}$ and $\sV$ on $\xwffs{\seq{\alpha\beta}}$ for
all $\alpha, \beta \in \sT$.  Now suppose that $\sD_\alpha$ and
$\sD_{\beta}$ are defined and that the conditions hold for $\alpha$
and $\beta$.  For each xwff $\textbf{A}_{\seq{\alpha\beta}}$, if
$\provesb{\sH}{{\textbf{A}_{\seq{\alpha\beta}}\IsDefApp}}$, let
\[\sV(\textbf{A}_{\seq{\alpha\beta}}) = 
\seq{\sV(\mname{fst}_{\alpha\seq{\alpha\beta}} \,
\textbf{A}_{\seq{\alpha\beta}}),
\sV(\mname{snd}_{\beta\seq{\alpha\beta}} \,
\textbf{A}_{\seq{\alpha\beta}})},\] and otherwise let
$\sV(\textbf{A}_{\seq{\alpha\beta}})$ be undefined.  Also, let
\[\sD_{\seq{\alpha\beta}} =
\set{\sV(\textbf{A}_{\seq{\alpha\beta}}) \;|\;
  \textbf{A}_{\seq{\alpha\beta}} \mbox{ is a }
  \xwff{\seq{\alpha\beta}} \mbox{ and }
  \provesb{\sH}{{\textbf{A}_{\seq{\alpha\beta}}\IsDefApp}}}.\]
$(1^{\seq{\alpha\beta}})$ and $(2^{\seq{\alpha\beta}})$ are clearly
satisfied; we must show that $(3^{\seq{\alpha\beta}})$ is satisfied.
\begin{align} \setcounter{equation}{0}
&
\sV(\textbf{A}_{\seq{\alpha\beta}}) = \sV(\textbf{B}_{\seq{\alpha\beta}}) \\
&\mathrel{\mbox{iff}} 
\langle(\sV(\mname{fst}_{\alpha\seq{\alpha\beta}} \,
\textbf{A}_{\seq{\alpha\beta}}),
\sV(\mname{snd}_{\beta\seq{\alpha\beta}} \,
\textbf{A}_{\seq{\alpha\beta}})), \notag\\
&\hspace*{4.2ex}
(\sV(\mname{fst}_{\alpha\seq{\alpha\beta}} \,
\textbf{B}_{\seq{\alpha\beta}}),
\sV(\mname{snd}_{\beta\seq{\alpha\beta}} \,
\textbf{B}_{\seq{\alpha\beta}}))\rangle \\
&\mathrel{\mbox{iff}} 
\sV(\mname{fst}_{\alpha\seq{\alpha\beta}} \, \textbf{A}_{\seq{\alpha\beta}}) =
\sV(\mname{fst}_{\alpha\seq{\alpha\beta}} \, \textbf{B}_{\seq{\alpha\beta}}) \mbox{ and} \notag\\
&\hspace*{3.3ex}
\sV(\mname{snd}_{\beta\seq{\alpha\beta}} \, \textbf{A}_{\seq{\alpha\beta}}) =
\sV(\mname{snd}_{\beta\seq{\alpha\beta}} \, \textbf{B}_{\seq{\alpha\beta}}) \\
&\mathrel{\mbox{iff}}  
\provesb{\sH}{\mname{fst}_{\alpha\seq{\alpha\beta}} \, \textbf{A}_{\seq{\alpha\beta}} =
\mname{fst}_{\alpha\seq{\alpha\beta}} \, \textbf{B}_{\seq{\alpha\beta}}}  \mbox{ and} \notag\\
&\hspace*{3.3ex} 
\provesb{\sH}{\mname{snd}_{\beta\seq{\alpha\beta}} \, \textbf{A}_{\seq{\alpha\beta}} =
\mname{snd}_{\beta\seq{\alpha\beta}} \, \textbf{B}_{\seq{\alpha\beta}}} \\
&\mathrel{\mbox{iff}} 
\provesb{\sH}{\mname{pair}_{\seq{\alpha\beta}\beta\alpha} \,
[\mname{fst}_{\alpha\seq{\alpha\beta}} \, \textbf{A}_{\seq{\alpha\beta}}] \,
[\mname{snd}_{\beta\seq{\alpha\beta}} \, \textbf{A}_{\seq{\alpha\beta}}] = \notag\\
&\hspace{9.6ex}
\mname{pair}_{\seq{\alpha\beta}\beta\alpha} \,
[\mname{fst}_{\alpha\seq{\alpha\beta}} \, \textbf{B}_{\seq{\alpha\beta}}] \,
[\mname{snd}_{\beta\seq{\alpha\beta}} \, \textbf{B}_{\seq{\alpha\beta}}}] \\
&\mathrel{\mbox{iff}} 
\provesb{\sH}{\textbf{A}_{\seq{\alpha\beta}} = \textbf{B}_{\seq{\alpha\beta}}}.
\end{align}
(2) is by the definition of $\sV$ on $\xwffs{\seq{\alpha\beta}}$; (3)
is by definition of ordered pairs; (4) is by $(3^\alpha)$ and
$(3^\beta)$; (5) is by Axiom 9.1; and (6) is by Axioms 9.1 and
9.2. Hence $(3^{\seq{\alpha\beta}})$ is satisfied.

\medskip

\noindent \textbf{Step 2} {\sglsp} We claim that $\sM =
\seq{\set{\sD_\alpha \;|\; \alpha \in \sT}, \sV}$ is an
interpretation.  For each primitive constant $\textbf{c}_\gamma$ of
{\qzerouqeplus}, $\textbf{c}_\gamma$ is an $\xwff{\gamma}$ and
$\provesb{\sH}{\textbf{c}_\gamma\IsDefApp}$ by Axiom 6.2, and thus
$\sV$ maps each primitive constant of {\qzerouqeplus} of type $\gamma$
into $\sD_\gamma$ by $(1^{\gamma})$ and $(2^{\gamma})$.

\medskip

\bsp \noindent \textbf{Step 2.1} {\sglsp} We must show
$\sV(\mname{Q}_{o\alpha\alpha}) = \sJ(\mname{Q}_{o\alpha\alpha})$,
i.e., that $\sV(\mname{Q}_{o\alpha\alpha})$ is the identity relation
on $\sD_\alpha$.  Let $\sV(\textbf{A}_\alpha)$ and
$\sV(\textbf{B}_\alpha)$ be arbitrary members of $\sD_\alpha$.  Then
$\sV(\textbf{A}_\alpha) = \sV(\textbf{B}_\alpha)$ iff
$\provesb{\sH}{\textbf{A}_\alpha = \textbf{B}_\alpha}$ iff
$\provesb{\sH}{\mname{Q}_{o\alpha\alpha}\textbf{A}_\alpha\textbf{B}_\alpha}$
iff $\mname{T} =
\sV(\mname{Q}_{o\alpha\alpha}\textbf{A}_\alpha\textbf{B}_\alpha) =
\sV(\mname{Q}_{o\alpha\alpha})
(\sV(\textbf{A}_\alpha))(\sV(\textbf{B}_\alpha)).$ Thus
$\sV(\mname{Q}_{o\alpha\alpha})$ is the identity relation on
$\sD_\alpha$. \esp

\medskip

\noindent \textbf{Step 2.2} {\sglsp} We must show that
$\sV(\iota_{\alpha(o\alpha)}) = \sJ(\iota_{\alpha(o\alpha)})$, i.e.,
that, for $\alpha \not= o$, $\sV(\iota_{\alpha(o\alpha)})$ is the
unique member selector on $\sD_\alpha$.  For $\alpha \not= o$, let
$\sV(\textbf{A}_{o\alpha})$ be an arbitrary member of $\sD_{o\alpha}$
and $\textbf{x}_\alpha$ be a variable that does not occur in
$\textbf{A}_{o\alpha}$.  Suppose $\sV(\textbf{A}_{o\alpha}) =
\sV(\mname{Q}_{o\alpha\alpha} \textbf{B}_\alpha)$.  We must show that
$\sV(\iota_{\alpha(o\alpha)})(\sV(\textbf{A}_{o\alpha})) = \sV
(\textbf{B}_\alpha)$.  The hypothesis implies
$\provesb{\sH}{\textbf{A}_{o\alpha} = \mname{Q}_{o\alpha\alpha}
  \textbf{B}_\alpha}$, so $\provesb{\sH}{\Forsome_1 \textbf{x}_\alpha
  [\textbf{A}_{o\alpha}\textbf{x}_\alpha]}$ by the definition of
$\Forsome_1$, and so $\provesb{\sH}{\textbf{A}_{o\alpha}[\Iota
    \textbf{x}_\alpha \textbf{A}_{o\alpha}]}$ by Axiom 8.1 and Axiom
8.2.  Hence
$\provesb{\sH}{\mname{Q}_{o\alpha\alpha}\textbf{B}_\alpha[\Iota
    \textbf{x}_\alpha \textbf{A}_{o\alpha}]}$ by Rule 1, and so
$\sV(\textbf{B}_\alpha) = \sV(\Iota \textbf{x}_\alpha
\textbf{A}_{o\alpha}) =
\sV(\iota_{\alpha(o\alpha)}\textbf{A}_{o\alpha}) =
\sV(\iota_{\alpha(o\alpha)})(\sV(\textbf{A}_{o\alpha}))$.

\bsp Now suppose that $\sV(\Forall \textbf{x}_\alpha [\textbf{A}_{o\alpha}
  \not= \mname{Q}_{o\alpha\alpha} \textbf{x}_\alpha]) = \mname{T}$.
We must show that
$\sV(\iota_{\alpha(o\alpha)})(\sV(\textbf{A}_{o\alpha}))$ is
undefined.  The hypothesis implies $\provesb{\sH}{\Forall \textbf{x}_\alpha
  [\textbf{A}_{o\alpha} \not= \mname{Q}_{o\alpha\alpha}
    \textbf{x}_\alpha]}$, so $\provesb{\sH}{\NegAlt[\Forsome_1
    \textbf{x}_\alpha [\textbf{A}_{o\alpha}\textbf{x}_\alpha]]}$ by
the definition of $\Forsome_1$, and so $\provesb{\sH}{[\Iota
    \textbf{x}_\alpha
           [\textbf{A}_{o\alpha}\textbf{x}_\alpha]]\IsUndefApp}$ by
Axiom 8.1.  Hence $\sV(\Iota \textbf{x}_\alpha
[\textbf{A}_{o\alpha}\textbf{x}_\alpha]) \QuasiEqual
\sV(\iota_{\alpha(o\alpha)}\textbf{A}_{o\alpha}) \QuasiEqual
\sV(\iota_{\alpha(o\alpha)})(\sV(\textbf{A}_{o\alpha}))$ is undefined. \esp

\medskip

\bsp \noindent \textbf{Step 2.4} {\sglsp} We must show that
$\sV(\mname{pair}_{\seq{\alpha\beta}\beta\alpha}) =
\sJ(\mname{pair}_{\seq{\alpha\beta}\beta\alpha})$.  Let
$\sV(\textbf{A}_\alpha)$ be an arbitrary member of $\sD_\alpha$ and
$\sV(\textbf{B}_\beta)$ be an arbitrary member of $\sD_\beta$.  We
must show that $\sV(\mname{pair}_{\seq{\alpha\beta}\beta\alpha}\,
\textbf{A}_\alpha, \textbf{B}_\beta) = \seq{\sV(\textbf{A}_\alpha),
\sV(\textbf{B}_\beta)}$.  $\sV(\mname{fst}_{\alpha\seq{\alpha\beta}}
\, [\mname{pair}_{\seq{\alpha\beta}\beta\alpha}\, \textbf{A}_\alpha,
  \textbf{B}_\beta]) = \sV(\textbf{A}_\alpha)$ iff
$\provesb{\sH}{\mname{fst}_{\alpha\seq{\alpha\beta}} \,
  [\mname{pair}_{\seq{\alpha\beta}\beta\alpha}\, \textbf{A}_\alpha,
    \textbf{B}_\beta] = \textbf{A}_\alpha}$, which holds by the
definition of $\mname{fst}_{\alpha\seq{\alpha\beta}}$ and Axiom 9.1.
Similarly, $\sV(\mname{snd}_{\beta\seq{\alpha\beta}} \,
[\mname{pair}_{\seq{\alpha\beta}\beta\alpha}\, \textbf{A}_\alpha,
  \textbf{B}_\beta]) = \sV(\textbf{B}_\beta)$ holds by the definition
of $\mname{snd}_{\beta\seq{\alpha\beta}}$ and Axiom 9.1.  Hence
\begin{align*}
&
\sV(\mname{pair}_{\seq{\alpha\beta}\beta\alpha} \, \textbf{A}_\alpha, \textbf{B}_\beta) \\
&= 
\seq{\sV(\mname{fst}_{\alpha\seq{\alpha\beta}} \,
[\mname{pair}_{\seq{\alpha\beta}\beta\alpha} \, \textbf{A}_\alpha, \textbf{B}_\beta]),
\sV(\mname{snd}_{\beta\seq{\alpha\beta}} \,
[\mname{pair}_{\seq{\alpha\beta}\beta\alpha} \, \textbf{A}_\alpha, \textbf{B}_\beta])} \\
&=
\seq{\sV(\textbf{A}_\alpha), \sV(\textbf{B}_\beta)}.
\end{align*}
\esp

\noindent Thus $\sM$ is an interpretation.  

\medskip

\noindent \textbf{Step 3} {\sglsp} We claim further that $\sM$ is an
evaluation-free model for {\qzerouqeplus}.  For each assignment $\phi
\in \mname{assign}(\sM)$ and evaluation-free wff $\textbf{D}_\delta$,
let \[\textbf{D}^{\phi}_\delta = \mname{S}^{{\bf x}^{1}_{\delta_1}
  \cdots {\bf x}^{n}_{\delta_n}}_{{\bf E}^{1}_{\delta_1} \cdots {\bf
    E}^{n}_{\delta_n}} \textbf{D}_\delta = \mname{S}^{{\bf
    x}^{1}_{\delta_1}}_{{\bf E}^{1}_{\delta_1}} \mname{S}^{{\bf
    x}^{2}_{\delta_2}}_{{\bf E}^{2}_{\delta_2}} \cdots \mname{S}^{{\bf
    x}^{n}_{\delta_n}}_{{\bf E}^{n}_{\delta_n}} \textbf{D}_\delta\]
where ${\bf x}^{1}_{\delta_1} \cdots {\bf x}^{n}_{\delta_n}$ are the
free variables of $\textbf{D}_\delta$ and ${\bf E}^{i}_{\delta_i}$ is
the first xwff (in some fixed enumeration) of
$\sL(\mbox{\qzerouqeplus})$ such that $\phi({\bf x}^{i}_{\delta_i}) =
\sV({\bf E}^{i}_{\delta_i})$ for all $i$ with $1 \le i \le n$.  Since
each ${\bf E}^{i}_{\delta_i}$ is syntactically closed,
$\textbf{D}^{\phi}_\delta$ is always defined.

Let $\sV^{\cal M}_{\phi}(\textbf{D}_\delta) \QuasiEqual
\sV(\textbf{D}^{\phi}_\delta)$.  $\textbf{D}^{\phi}_\delta$ is
clearly a $\xwff{\delta}$, so $\sV^{\cal M}_{\phi}(\textbf{D}_\delta) \in
\sD_\delta$ if $\sV^{\cal M}_{\phi}(\textbf{D}_\delta)$ is defined.  We will show
that the six conditions of an evaluation-free model are satisfied as
follows:.

\be

  \item Let $\textbf{D}_\delta$ be a variable
  $\textbf{x}_\delta$.  Choose $\textbf{E}_\delta$ so that
  $\phi(\textbf{x}_\delta) = \sV(\textbf{E}_\delta)$ as above.  Then
  $\sV^{\cal M}_{\phi}(\textbf{D}_\delta) = \sV^{\cal M}_{\phi}(\textbf{x}_\delta) =
  \sV(\textbf{x}^{\phi}_{\delta}) = \sV(\textbf{E}_\delta) =
  \phi(\textbf{x}_\delta)$.

  \item Let $\textbf{D}_\delta$ be a primitive constant.  Then
    $\sV^{\cal M}_{\phi}(\textbf{D}_\delta) = \sV(\textbf{D}^{\phi}_\delta) =
    \sV(\textbf{D}_\delta) = \sJ(\textbf{D}_\delta)$.

  \item \bsp Let $\textbf{D}_\delta$ be
    $[\textbf{A}_{\alpha\beta} \textbf{B}_\beta]$.  If
    $\sV^{\cal M}_{\phi}(\textbf{A}_{\alpha\beta})$ is defined,
    $\sV^{\cal M}_{\phi}(\textbf{B}_\beta)$ is defined, and
    $\sV^{\cal M}_{\phi}(\textbf{A}_{\alpha\beta})$ is defined at
    $\sV^{\cal M}_{\phi}(\textbf{B}_\beta)$, then
    $\sV^{\cal M}_{\phi}(\textbf{D}_\delta) =
    \sV^{\cal M}_{\phi}(\textbf{A}_{\alpha\beta}\textbf{B}_\beta) =
    \sV(\textbf{A}^{\phi}_{\alpha\beta}\textbf{B}^{\phi}_\beta) =
    \sV(\textbf{A}^{\phi}_{\alpha\beta})(\sV(\textbf{B}^{\phi}_\beta))
    =
    \sV^{\cal M}_{\phi}(\textbf{A}_{\alpha\beta})(\sV^{\cal M}_{\phi}(\textbf{B}_\beta))$.
    Now assume $\sV^{\cal M}_{\phi}(\textbf{A}_{\alpha\beta})$ is undefined,
    $\sV^{\cal M}_{\phi}(\textbf{B}_\beta)$ is undefined, or
    $\sV^{\cal M}_{\phi}(\textbf{A}_{\alpha\beta})$ is not defined at
    $\sV^{\cal M}_{\phi}(\textbf{B}_\beta)$.  Then
    $\provesb{\sH}{{\textbf{A}^{\phi}_{\alpha\beta}\IsUndefApp}}$,
    $\provesb{\sH}{{\textbf{B}^{\phi}_{\beta}\IsUndefApp}}$, or
    $\sV(\textbf{A}^{\phi}_{\alpha\beta}\textbf{B}^{\phi}_{\beta})$ is
    undefined.
    $\provesb{\sH}{{\textbf{A}^{\phi}_{\alpha\beta}\IsUndefApp}}$ or
    $\provesb{\sH}{{\textbf{B}^{\phi}_{\beta}\IsUndefApp}}$ implies
    $\provesb{\sH}{\textbf{A}^{\phi}_{\alpha\beta}\textbf{B}^{\phi}_{\beta}
      \QuasiEqual \Undefined_\alpha}$ by Axiom 6.4.  If $\alpha = o$, then
    $\sV^{\cal M}_{\phi}(\textbf{D}_\delta) =
    \sV^{\cal M}_{\phi}(\textbf{A}_{\alpha\beta}\textbf{B}_\beta) =
    \sV(\textbf{A}^{\phi}_{\alpha\beta}\textbf{B}^{\phi}_{\beta}) =
    \sV(F_o) = \mname{F}$.  If $\alpha \not= o$, then
    $\sV^{\cal M}_{\phi}(\textbf{D}_\delta) \QuasiEqual
    \sV^{\cal M}_{\phi}(\textbf{A}_{\alpha\beta}\textbf{B}_\beta) \QuasiEqual
    \sV(\textbf{A}^{\phi}_{\alpha\beta}\textbf{B}^{\phi}_{\beta})
    \QuasiEqual \sV(\Undefined_\alpha)$ is undefined by Axiom 6.11.

 \esp

  \item \bsp Let $\textbf{D}_\delta$ be
    $[\lambda\textbf{x}_\alpha\textbf{B}_\beta]$.  Let
    $\sV(\textbf{E}_\alpha)$ be an arbitrary member of $\sD_\alpha$,
    and so $\textbf{E}_\alpha$ is an xwff and
    $\provesb{\sH}{{\textbf{E}_\alpha\IsDefApp}}$.  Given an
    assignment $\phi \in \mname{assign}(\sM)$, let $\psi = \phi[
      \textbf{x}_\alpha \mapsto \sV(\textbf{E}_\alpha)]$.  It follows
    from the Beta-Reduction Theorem (part 1) that
    $\provesb{\sH}{[\lambda \textbf{x}_\alpha \textbf{B}_\beta]^{\phi}
      \textbf{E}_{\alpha} \QuasiEqual \textbf{B}^{\psi}_{\beta}}$.
    Then $\sV^{\cal
      M}_{\phi}(\textbf{D}_\delta)(\sV(\textbf{E}_\alpha)) \QuasiEqual
    \sV^{\cal M}_{\phi}([\lambda\textbf{x}_\alpha\textbf{B}_\beta])
    (\sV(\textbf{E}_\alpha)) \QuasiEqual
    \sV([\lambda\textbf{x}_\alpha\textbf{B}_\beta]^{\phi})(\sV(\textbf{E}_\alpha))
    \QuasiEqual \sV([\lambda\textbf{x}_\alpha\textbf{B}_\beta]^{\phi}
    \textbf{E}_\alpha)) \QuasiEqual \sV(\textbf{B}^{\psi}_{\beta})
    \QuasiEqual \sV^{\cal M}_\psi(\textbf{B}_{\beta})$ as required. \esp

  \item \bsp Let $\textbf{D}_\delta$ be
    $[\mname{c}\textbf{A}_o\textbf{B}_\alpha\textbf{C}_\alpha]$.  If
    $\sV^{\cal M}_{\phi}(\textbf{A}_o) = \mname{T}$, then
    $\sV^{\cal M}_{\phi}(\textbf{D}_\delta)
    \QuasiEqual
    \sV^{\cal M}_{\phi}(\mname{c}\textbf{A}_o\textbf{B}_\alpha\textbf{C}_\alpha)
    \QuasiEqual
    \sV([\mname{c}\textbf{A}_o\textbf{B}_\alpha\textbf{C}_\alpha]^\phi)
    \QuasiEqual
    \sV(\mname{c}\textbf{A}^{\phi}_{o}\textbf{B}^{\phi}_{\alpha}\textbf{C}^{\phi}_{\alpha})
    \QuasiEqual
    \sV(\mname{c}\,T_o\textbf{B}^{\phi}_{\alpha}\textbf{C}^{\phi}_{\alpha})
    \QuasiEqual
    \sV(\textbf{B}^{\phi}_{\alpha})
    \QuasiEqual
    \sV^{\cal M}_{\phi}(\textbf{B}_\alpha)$
    by Axiom 10.1.  Similarly, if $\sV^{\cal M}_{\phi}(\textbf{A}_o) =
    \mname{F}$, then $\sV^{\cal M}_{\phi}(\textbf{D}_\delta)
    \QuasiEqual \sV^{\cal M}_{\phi}(\textbf{B}_\alpha)$ by Axiom
    10.2.\esp

  \item \bsp Let $\textbf{D}_\delta$ be
    $[\mname{q}\textbf{A}_\alpha]$.  Then 
    $\sV^{\cal M}_{\phi}(\textbf{D}_\delta) 
    = 
    \sV^{\cal M}_{\phi}([\mname{q}\textbf{A}_\alpha]) 
    =
    \sV([\mname{q}\textbf{A}_\alpha]^\phi) 
    =
    \sV([\mname{q}\textbf{A}_\alpha]) 
    =
    \sE(\textbf{A}_\alpha)$.\esp

\ee
Thus $\sM$ is an evaluation-free model for {\qzerouqeplus}.

\medskip

\bsp \noindent \textbf{Step 4} {\sglsp} We must show that $\sM$ is
normal and frugal.  If $\textbf{A}_o$ is an evaluation-free specifying
axiom given by Specifications 1--9, then $\provesb{\sH}{\textbf{A}_o}$
by Axiom 12, so $\sV(\textbf{A}_o) = \mname{T}$ and $\sM \models
\textbf{A}_o$, and so $\sM$ is normal.  Clearly,
(a)~$\mbox{card}(\sD_\alpha) \le \mbox{card}(\sL(\mbox{\qzerouqe}))$
since $\sV$ maps a subset of the $\xwffs{\alpha}$ of {\qzerouqeplus}
onto $\sD_\alpha$ and (b)~$\mbox{card}(\sL(\mbox{\qzerouqeplus})) =
\mbox{card}(\sL(\mbox{\qzerouqe}))$, and so $\sM$ is frugal.\esp

\medskip

\noindent \textbf{Step 5} {\sglsp} We must show that $\sM$ is a frugal
normal evaluation-free model for $\sG$.  We have shown that $\sM$ is a
frugal normal evaluation-free model for {\qzerouqeplus}.  Clearly,
$\sM$ is also a frugal normal evaluation-free model for {\qzerouqe}.
If $\textbf{A}_o \in \sG$, then $\textbf{A}_o \in \sH$, so
$\provesb{\sH}{\textbf{A}_o}$, so $\sV(\textbf{A}_o) = \mname{T}$ and
$\sM \models \textbf{A}_o$, and so $\sM$ is an evaluation-free model
for $\sG$.
\end{proof}

\subsection{Evaluation-Free Completeness Theorem}

\begin{thm}[Evaluation-Free Completeness Theorem for {\pfsysuqe}] \label{thm:completeness} \hspace{1ex}\\
{\pfsysuqe} is evaluation-free complete for {\qzerouqe}.
\end{thm}

\begin{proof} 
Let $\textbf{A}_o$ be an evaluation-free formula and $\sH$ be a set of
syntactically closed evaluation-free formulas of {\qzerouqe}.  Assume
$\sH \modelsna \textbf{A}_o$, and let $\textbf{B}_o$ be a universal
closure of $\textbf{A}_o$.  Then $\textbf{B}_o$ is syntactically
closed and $\sH \modelsna \textbf{B}_o$ by Lemma~\ref{lem:uni-close}.
Suppose $\sH \cup \{\NegAlt \textbf{B}_o\}$ is consistent in
{\pfsysuqe}.  Then, by Henkin's Theorem, there is a normal
evaluation-free model $\sM$ for $\sH \cup \set{\NegAlt \textbf{B}_o}$,
and so $\sM \models \NegAlt \textbf{B}_o$.  Since $\sM$ is also a
normal evaluation-free model for $\sH$, $\sM \models \textbf{B}_o$.
From this contradiction it follows that $\sH \cup \set{\NegAlt
  \textbf{B}_o}$ is inconsistent in {\pfsysuqe}.  Hence
$\provesa{\sH}{\textbf{B}_o}$ by the Deduction Theorem and the
Tautology Theorem.  Therefore, $\provesa{\sH}{\textbf{A}_o}$ by
Universal Instantiation (part 1) and Axiom 6.1.
\end{proof}

\section{Applications} \label{sec:applications}

We will now look at some applications of the machinery in {\qzerouqe}
for reasoning about the interplay of the syntax and semantics of
{\qzerouqe} expressions (i.e., wffs).  We will consider three kinds of
applications.  The first kind uses the type $\epsilon$ machinery to
reason about the syntactic structure of wffs; see the examples in
subsection~\ref{subsec:implication}.  The second kind uses evaluation
applied to variables of type~$\epsilon$ to express syntactic variables
as employed, for example, in schemas; see the examples in
subsections~\ref{subsec:lem} and~\ref{subsec:beta-red}.  The third
kind uses the full machinery of {\qzerouqe} to formalize syntax-based
mathematical algorithms in the manner described in~\cite{Farmer13}; see the
example in subsection~\ref{subsec:conj-cons}.

\subsection{Example: Implications} \label{subsec:implication}

We will illustrate how the type $\epsilon$ machinery in {\qzerouqe}
can be used to reason about the syntactic structure of wffs by
defining some useful constants for analyzing and manipulating
\emph{implications}, i.e., formulas of the form $\textbf{A}_o \Implies
\textbf{B}_o$.

Let $\mname{implies}_{\epsilon\epsilon\epsilon}$ be a defined constant
that stands for \[\lambda x_\epsilon \lambda y_\epsilon
[\mname{app}_{\epsilon\epsilon\epsilon} \,
  [\mname{app}_{\epsilon\epsilon\epsilon} \, \synbrack{\Implies_{ooo}}
    \, x_\epsilon] \, y_\epsilon].\]

\begin{lem} \label{lem:implies}
For all formulas $\textbf{A}_o$ and $\textbf{B}_o$, 
\bi

  \item[] $\proves{}{\mname{implies}_{\epsilon\epsilon\epsilon} \,
    \synbrack{\textbf{A}_o} \, \synbrack{\textbf{B}_o} =
    \synbrack{\textbf{A}_o \Implies \textbf{B}_o}}$.

\ei
\end{lem}

\begin{proof}
\begin{align} \setcounter{equation}{0}
&\vdash
\mname{implies}_{\epsilon\epsilon\epsilon} \, \synbrack{\textbf{A}_o} \, \synbrack{\textbf{B}_o} \QuasiEqual
\mname{implies}_{\epsilon\epsilon\epsilon} \, \synbrack{\textbf{A}_o} \, \synbrack{\textbf{B}_o}. \\
&\vdash
\mname{implies}_{\epsilon\epsilon\epsilon} \, \synbrack{\textbf{A}_o} \, \synbrack{\textbf{B}_o}
\QuasiEqual {} \notag\\
&\hspace*{2.7ex}
[\lambda x_\epsilon \lambda y_\epsilon [\mname{app}_{\epsilon\epsilon\epsilon} \,
[\mname{app}_{\epsilon\epsilon\epsilon} \, \synbrack{\Implies_{ooo}} \, x_\epsilon] \, 
y_\epsilon]] \, \synbrack{\textbf{A}_o} \, \synbrack{\textbf{B}_o}. \\
&\vdash
\mname{implies}_{\epsilon\epsilon\epsilon} \, \synbrack{\textbf{A}_o} \, \synbrack{\textbf{B}_o} \QuasiEqual
{\mname{app}_{\epsilon\epsilon\epsilon} \,
[\mname{app}_{\epsilon\epsilon\epsilon} \, \synbrack{\Implies_{ooo}} \, \synbrack{\textbf{A}_o}] \, 
\synbrack{\textbf{B}_o}}. \\
&\vdash
\mname{implies}_{\epsilon\epsilon\epsilon} \, \synbrack{\textbf{A}_o} \, \synbrack{\textbf{B}_o} \QuasiEqual
\mname{app}_{\epsilon\epsilon\epsilon} \,
\synbrack{\Implies_{ooo} \, \textbf{A}_o} \, \synbrack{\textbf{B}_o}. \\
&\vdash
\mname{implies}_{\epsilon\epsilon\epsilon} \, \synbrack{\textbf{A}_o} \, \synbrack{\textbf{B}_o} =
\synbrack{\Implies_{ooo} \textbf{A}_o \textbf{B}_o}. \\
&\vdash
\mname{implies}_{\epsilon\epsilon\epsilon} \, \synbrack{\textbf{A}_o} \, \synbrack{\textbf{B}_o} =
\synbrack{\textbf{A}_o \Implies \textbf{B}_o}.
\end{align} 
(1) is by Proposition~\ref{prop:5200}; (2) follows from (1) by the
definition of $\mname{implies}_{\epsilon\epsilon\epsilon}$; (3)
follows from (2) and Axioms 4.2--5 and 6.6 by Rules 1 and 2; (4)
follows from (3) by Specification 1; (5) follows from (4) by
Specification 1; and (6) follows from (5) by abbreviation.
\end{proof}

\bigskip

That is, $\mname{implies}_{\epsilon\epsilon\epsilon}$ is an
\emph{implication constructor}: the application of it to the syntactic
representations of two formulas $\textbf{A}_o$ and $\textbf{B}_o$
denotes the syntactic representation of the implication $\textbf{A}_o
\Implies \textbf{B}_o$.

Let $\mname{is-implication}_{o\epsilon}$ be a defined constant that
stands for \[\lambda x_\epsilon \Forsome y_\epsilon \Forsome
z_\epsilon [x_\epsilon = \mname{implies}_{\epsilon\epsilon\epsilon}
  \,y_\epsilon \, z_\epsilon].\] That is,
$\mname{is-implication}_{o\epsilon}$ is an \emph{implication recognizer}:
the application of it to the syntactic representation of a formula
$\textbf{A}_o$ has the value $\mname{T}$ iff $\textbf{A}_o$ has the
form $\textbf{B}_o \Implies \textbf{C}_o$.

Let $\mname{antecedent}_{\epsilon\epsilon}$ and
$\mname{succedent}_{\epsilon\epsilon}$ be the defined constants that,
respectively, stand for \[\lambda x_\epsilon \Iota y_\epsilon \Forsome
z_\epsilon [x_\epsilon = \mname{implies}_{\epsilon\epsilon\epsilon}
  \,y_\epsilon \, z_\epsilon]\] and \[\lambda x_\epsilon \Iota
z_\epsilon \Forsome y_\epsilon [x_\epsilon =
  \mname{implies}_{\epsilon\epsilon\epsilon} \,y_\epsilon \,
  z_\epsilon].\] Then 
\[\proves{}{\mname{antecedent}_{\epsilon\epsilon} \, \synbrack{\textbf{A}_o \Implies
  \textbf{B}_o} = \synbrack{\textbf{A}_o}}\] and
\[\proves{}{\mname{succedent}_{\epsilon\epsilon} \, \synbrack{\textbf{A}_o \Implies
  \textbf{B}_o} = \synbrack{\textbf{B}_o}}.\]  That is,
$\mname{antecedent}_{\epsilon\epsilon}$ and $\mname{succedent}_{\epsilon\epsilon}$
are \emph{implication deconstructors}: the applications of them to the
syntactic representation of a formula $\textbf{A}_o$ denote the
syntactic representations of the antecedent and succedent,
respectively, of $\textbf{A}_o$ if $\textbf{A}_o$ is an implication
and are undefined otherwise.

Let $\mname{converse}_{\epsilon\epsilon}$ be a defined constant that
stands for \[\lambda x_\epsilon
[\mname{implies}_{\epsilon\epsilon\epsilon} \,
  [\mname{succedent}_{\epsilon\epsilon} \, x_\epsilon] \,
  [\mname{antecedent}_{\epsilon\epsilon} \, x_\epsilon]].\]
Then \[\proves{}{\mname{converse}_{\epsilon\epsilon} \,
  \synbrack{\textbf{A}_o \Implies \textbf{B}_o} =
  \synbrack{\textbf{B}_o \Implies \textbf{A}_o}}.\] That is,
$\mname{converse}_{\epsilon\epsilon}$ is an \emph{implication
  converser}: the application of it to the syntactic representation of
a formula $\textbf{A}_o$ denotes the syntactic representation of the
converse of $\textbf{A}_o$ if $\textbf{A}_o$ is an implication and is
undefined otherwise.

\subsection{Example: Law of Excluded Middle}\label{subsec:lem}

The value of a wff of the form $\sembrack{\textbf{x}_\epsilon}_\alpha$
ranges over the values of wffs of type $\alpha$.  Thus wffs like
$\sembrack{\textbf{x}_\epsilon}_\alpha$ can be used in other wffs as
syntactic variables.  It is thus possible to express schemas as single
wffs in {\qzerouqe}.  As an example, let us consider the \emph{law of
  excluded middle (LEM)} which is usually written as a formula schema
like
\[\textbf{A}_o \Or \NegAlt\textbf{A}_o\] where $\textbf{A}_o$ 
ranges over all formulas.  LEM can be naively represented in
{\qzerouqe} as \[\Forall x_\epsilon [\sembrack{x_\epsilon}_o \Or
  \NegAlt\sembrack{x_\epsilon}_o].\] The variable $x_\epsilon$ ranges
over the syntactic representations of all wffs, not just formulas.
However, $\sembrack{x_\epsilon}_o$ is false when the value of
$x_\epsilon$ is not an evaluation-free formula.  A more intensionally
correct representation of LEM is \[\Forall
x_\epsilon[\mname{eval-free}^{o}_{o\epsilon} \, x_\epsilon \Implies
  [\sembrack{x_\epsilon}_o \Or \NegAlt\sembrack{x_\epsilon}_o]]\]
where $x_\epsilon$ is restricted to the syntactic representations of
evaluation-free formulas.  This representation of LEM is a theorem of
{\pfsysuqe}:

\begin{lem}
$\proves{}{\Forall x_\epsilon[\mname{eval-free}^{o}_{o\epsilon} \,
      x_\epsilon \Implies [\sembrack{x_\epsilon}_o \Or
        \NegAlt\sembrack{x_\epsilon}_o]]}$.
\end{lem}

\begin{proof}
\begin{align} \setcounter{equation}{0}
&\vdash 
x_o \Or \NegAlt x_o. \\
&\vdash 
\Forall x_o [x_o \Or \NegAlt x_o]. \\
&\vdash 
[\lambda x_o [x_o \Or \NegAlt x_o]] 
[\If \, [\mname{eval-free}^{o}_{o\epsilon} \, x_\epsilon] \, \sembrack{x_\epsilon}_o \, \Undefined_o] 
\QuasiEqual {} \notag\\
&\hspace*{2.7ex}
[\If \, [\mname{eval-free}^{o}_{o\epsilon} \, x_\epsilon] \,
\sembrack{x_\epsilon}_o \, \Undefined_o] \Or 
\NegAlt[\If \, [\mname{eval-free}^{o}_{o\epsilon} \, x_\epsilon] \,
\sembrack{x_\epsilon}_o \, \Undefined_o]. \\
&\vdash 
[\If \, [\mname{eval-free}^{o}_{o\epsilon} \, x_\epsilon] \,
\sembrack{x_\epsilon}_o \, \Undefined_o] \Or 
\NegAlt[\If \, [\mname{eval-free}^{o}_{o\epsilon} \, x_\epsilon] \,
\sembrack{x_\epsilon}_o \, \Undefined_o]. \\
&\vdash  
\mname{eval-free}^{o}_{o\epsilon} \, x_\epsilon \Implies
[\sembrack{x_\epsilon}_o \Or \NegAlt\sembrack{x_\epsilon}_o]. \\
&\vdash  
\Forall x_\epsilon[\mname{eval-free}^{o}_{o\epsilon} \, x_\epsilon \Implies
  [\sembrack{x_\epsilon}_o \Or \NegAlt\sembrack{x_\epsilon}_o]].
\end{align}
(1) is by Axiom 5; (2) follows from (1) by Universal Generalization;
(3) follows from Axioms 4.2--4 and Proposition~\ref{prop:o-defined} by
Rules 1 and 2; (4) follows from (2) and (3) by Universal Instantiation
(part 3); (5) follows from (4) by the Axiom 10 and the Tautology
Theorem; and (6) follows from (5) by Universal Generalization.
\end{proof}

\subsection{Example: Law of Beta-Reduction}\label{subsec:beta-red}

Axiom 4.1, the law of beta-reduction for {\qzerouqe}, can be expressed
as the following schema whose only syntactic variables are $\alpha$
and $\beta$:

\begin{align*}
&
\Forall x_\epsilon \Forall y_\epsilon \Forall z_\epsilon \Forall z'_\epsilon 
[{\sembrack{x_\epsilon}_\alpha \IsDefApp} \Andd
[\mname{var}^{\alpha}_{o\epsilon} \, y_\epsilon] \Andd
[\mname{wff}^{\beta}_{o\epsilon} \, z_\epsilon] \Andd
[\mname{wff}^{\beta}_{o\epsilon} \, z'_\epsilon] \Andd {} \\
&\hspace*{4ex}
\mname{sub}_{\epsilon\epsilon\epsilon\epsilon} \, 
x_\epsilon \, y_\epsilon \, z_\epsilon = z'_\epsilon] \Implies
\sembrack{\mname{app}_{\epsilon\epsilon\epsilon}
[\mname{abs}_{\epsilon\epsilon\epsilon} \,  y_\epsilon \, z_\epsilon] \, x_\epsilon}_\beta \QuasiEqual 
\sembrack{z'_\epsilon}_\beta
\end{align*}
Each instance of this schema (for a chosen $\alpha$ and $\beta$) is
valid in {\qzerouqe} but not provable in {\pfsysuqe}.  Moreover, the
instances of an instance $\textbf{A}_o$ of this schema are not
provable in {\pfsysuqe} from $\textbf{A}_o$ since $\textbf{A}_o$
contains the evaluation
\[\sembrack{\mname{app}_{\epsilon\epsilon\epsilon}
  [\mname{abs}_{\epsilon\epsilon\epsilon} \, y_\epsilon \, z_\epsilon]
  \, x_\epsilon}_\beta\] in which more than one variable is free.

Using the technique of grouping variables together described in
Example~1 in subsection~\ref{subsec:undef-sub}, we can also express
Axiom 4.1 as the following schema that contains just the single
variable $x_{\seq{\seq{\epsilon\epsilon}\seq{\epsilon\epsilon}}}$:
\begin{align*}
&
\Forall x_{\seq{\seq{\epsilon\epsilon}\seq{\epsilon\epsilon}}}
[{[\sembrack{\textbf{X}_\epsilon}_\alpha \IsDefApp} \Andd
[\mname{var}^{\alpha}_{o\epsilon} \, \textbf{Y}_\epsilon] \Andd
[\mname{wff}^{\beta}_{o\epsilon} \, \textbf{Z}_\epsilon] \Andd
[\mname{wff}^{\beta}_{o\epsilon} \, \textbf{Z}'_\epsilon] \Andd {} \\
&\hspace*{4ex}
\mname{sub}_{\epsilon\epsilon\epsilon\epsilon} \, 
\textbf{X}_\epsilon \, \textbf{Y}_\epsilon \, \textbf{Z}_\epsilon = \textbf{Z}'_\epsilon] \Implies 
\sembrack{\mname{app}_{\epsilon\epsilon\epsilon}
[\mname{abs}_{\epsilon\epsilon\epsilon} \, 
\textbf{Y}_\epsilon \, \textbf{Z}_\epsilon] \, \textbf{X}_\epsilon}_\beta \QuasiEqual 
\sembrack{\textbf{Z}'_\epsilon}_\beta] \\
&\mbox{where:}\\
&\hspace*{4ex}
\textbf{X}_\epsilon \sglsp \mbox{is} \sglsp 
\mname{fst}_{\epsilon\seq{\epsilon\epsilon}} \, 
[\mname{fst}_{\seq{\epsilon\epsilon}\seq{\seq{\epsilon\epsilon}\seq{\epsilon\epsilon}}} \,
x_{\seq{\seq{\epsilon\epsilon}\seq{\epsilon\epsilon}}}] \\
&\hspace*{4ex}
\textbf{Y}_\epsilon \sglsp \mbox{is} \sglsp 
\mname{snd}_{\epsilon\seq{\epsilon\epsilon}} \, 
[\mname{fst}_{\seq{\epsilon\epsilon}\seq{\seq{\epsilon\epsilon}\seq{\epsilon\epsilon}}} \,
x_{\seq{\seq{\epsilon\epsilon}\seq{\epsilon\epsilon}}}] \\
&    \hspace*{4ex}
\textbf{Z}_\epsilon \sglsp \mbox{is} \sglsp 
\mname{fst}_{\epsilon\seq{\epsilon\epsilon}} \, 
[\mname{snd}_{\seq{\epsilon\epsilon}\seq{\seq{\epsilon\epsilon}\seq{\epsilon\epsilon}}} \,
x_{\seq{\seq{\epsilon\epsilon}\seq{\epsilon\epsilon}}}] \\
&\hspace*{4ex}
\textbf{Z}'_\epsilon \sglsp \mbox{is} \sglsp 
\mname{snd}_{\epsilon\seq{\epsilon\epsilon}} \, 
[\mname{snd}_{\seq{\epsilon\epsilon}\seq{\seq{\epsilon\epsilon}\seq{\epsilon\epsilon}}} \,
x_{\seq{\seq{\epsilon\epsilon}\seq{\epsilon\epsilon}}}].
\end{align*}
Like the first schema, each instance of this second schema is valid in
{\qzerouqe} but not provable in {\pfsysuqe}.  However, unlike the
first schema, the instances of an instance $\textbf{A}_o$ of the
second schema are provable in {\pfsysuqe} from $\textbf{A}_o$.

\subsection{Example: Conjunction Construction} \label{subsec:conj-cons}

Suppose $A$ is an algorithm that, given two formulas $\textbf{A}_o$
and $\textbf{B}_o$ as input, returns as output (1) $\textbf{B}_o$ if
$\textbf{A}_o$ is $T_o$, (2) $\textbf{A}_o$ if $\textbf{B}_o$ is
$T_o$, (3) $F_o$ if either $\textbf{A}_o$ or $\textbf{B}_o$ is $F_o$,
or (4) $\textbf{A}_o \Andd \textbf{B}_o$ otherwise.  Although this is
a trivial algorithm, we can use it to illustrate how a syntax-based
mathematical algorithm can be formalized in {\qzerouqe}.  As described
in~\cite{Farmer13} we need to do the following three things to
formalize $A$ in {\qzerouqe}.

\be

  \item Define an operator $O_A$ in {\qzerouqe} as a constant that
    represents $A$.

  \item Prove in {\pfsysuqe} that $O_A$ is mathematically correct.

  \item Devise a mechanism for using $O_A$ in {\qzerouqe}.

\ee

Let $\mname{and}_{\epsilon\epsilon\epsilon}$ be a defined
constant that stands for
\[\lambda x_\epsilon \lambda y_\epsilon
[\mname{app}_{\epsilon\epsilon\epsilon} \,
  [\mname{app}_{\epsilon\epsilon\epsilon} \, \synbrack{\wedge_{ooo}}
    \, x_\epsilon] \, y_\epsilon].\]
Then \[\proves{}{\mname{and}_{\epsilon\epsilon\epsilon} \,
  \synbrack{\textbf{A}_o} \, \synbrack{\textbf{B}_o} =
  \synbrack{\textbf{A}_o \Andd \textbf{B}_o}}\] for all formulas
$\textbf{A}_o$ and $\textbf{B}_o$ as shown by a derivation similar to
the one for $\mname{implies}_{\epsilon\epsilon\epsilon}$ in the proof
of Lemma~\ref{lem:implies}.  Define $O_A$ to be
$\mname{and-simp}_{\epsilon\epsilon\epsilon}$, a defined constant that
stands for
\begin{align*}
&
\lambda x_\epsilon \lambda y_\epsilon [\If \, [x_\epsilon = \synbrack{T_o}] \, y_\epsilon \\
&\hspace*{7.1ex}
[\If \, [y_\epsilon = \synbrack{T_o}] \, x_\epsilon \\
&\hspace*{7.1ex}
[\If \, [x_\epsilon = \synbrack{F_o}] \, \synbrack{F_o} \\
&\hspace*{7.1ex}
[\If \, [y_\epsilon = \synbrack{F_o}] \, \synbrack{F_o} \\
&\hspace*{7.1ex}
[\mname{and}_{\epsilon\epsilon\epsilon} \, x_\epsilon \, y_\epsilon]]]]].
\end{align*}
The sentence
\begin{align*}
&
\Forall x_\epsilon \Forall y_\epsilon 
[[\mname{wff}^{o}_{o\epsilon} \, x_\epsilon \Andd
\mname{wff}^{o}_{o\epsilon} \, y_\epsilon] \Implies {}\\
&\hspace*{4ex}
[[x_\epsilon = \synbrack{T_o} \Implies 
\mname{and-simp}_{\epsilon\epsilon\epsilon} \, x_\epsilon \, y_\epsilon =  y_\epsilon] \Andd {}\\
&\hspace*{4.6ex}
[y_\epsilon = \synbrack{T_o} \Implies 
\mname{and-simp}_{\epsilon\epsilon\epsilon} \, x_\epsilon \, y_\epsilon =  x_\epsilon] \Andd {}\\
&\hspace*{4.6ex}
[[x_\epsilon = \synbrack{F_o} \Or y_\epsilon = \synbrack{F_o}] \Implies 
\mname{and-simp}_{\epsilon\epsilon\epsilon} \, x_\epsilon \, y_\epsilon =  \synbrack{F_o}] \Andd {}\\
&\hspace*{4.6ex}
[[x_\epsilon \not = \synbrack{T_o} \Andd y_\epsilon \not= \synbrack{T_o} \Andd
x_\epsilon \not = \synbrack{F_o} \Andd y_\epsilon \not= \synbrack{F_o}] \Implies {} \\
&\hspace{8ex}
\mname{and-simp}_{\epsilon\epsilon\epsilon} \, x_\epsilon \, y_\epsilon =  
\mname{and}_{\epsilon\epsilon\epsilon} \, x_\epsilon \, y_\epsilon]]],
\end{align*}
called \mname{CompBehavior}, specifies the intended
\emph{computational behavior} of $O_A$.  

\begin{thm} [Computational Behavior of $\mname{and-simp}_{\epsilon\epsilon\epsilon}$]
\[\proves{}{\mname{CompBehavior}}.\]
\end{thm}

\begin{proof}
\mname{CompBehavior} follows easily in {\pfsysuqe} from the definition
of $\mname{and-simp}_{\epsilon\epsilon\epsilon}$.  \phantom{XX}
\end{proof}

\bigskip

\noindent 
Hence $O_A$ represents $A$ by virtue of having the same computational
behavior as that of $A$.

Let us make the following definitions:
\bi

  \item[] $\textbf{P}_o$ is 
  $\sembrack{\mname{and}_{\epsilon\epsilon\epsilon} \,
  [\mname{fst}_{\epsilon\seq{\epsilon\epsilon}} \, x_{\seq{\epsilon\epsilon}}]  \,
  [\mname{snd}_{\epsilon\seq{\epsilon\epsilon}} \, x_{\seq{\epsilon\epsilon}}]}_o$.

  \item[] $\textbf{Q}_o$ is 
  $\sembrack{\mname{and-simp}_{\epsilon\epsilon\epsilon} \,
  [\mname{fst}_{\epsilon\seq{\epsilon\epsilon}} \, x_{\seq{\epsilon\epsilon}}]  \,
  [\mname{snd}_{\epsilon\seq{\epsilon\epsilon}} \, x_{\seq{\epsilon\epsilon}}]}_o$.

  \item[] $\textbf{R}_o$ is 
  $\sembrack{\mname{fst}_{\epsilon\seq{\epsilon\epsilon}} \, x_{\seq{\epsilon\epsilon}}}_o  \Andd
  \sembrack{\mname{snd}_{\epsilon\seq{\epsilon\epsilon}} \, x_{\seq{\epsilon\epsilon}}}_o$.

  \item[] $\textbf{S}_o$ is 
  $[\If \, [[\mname{fst}_{\epsilon\seq{\epsilon\epsilon}} \, x_{\seq{\epsilon\epsilon}}] = 
  \synbrack{T_o}] \, 
  \sembrack{\mname{snd}_{\epsilon\seq{\epsilon\epsilon}} \, x_{\seq{\epsilon\epsilon}}}_o \\[.5ex]
  \hspace*{5.7ex}
  [\If \, [[\mname{snd}_{\epsilon\seq{\epsilon\epsilon}} \, x_{\seq{\epsilon\epsilon}}] = \synbrack{T_o}] \, 
  \sembrack{\mname{fst}_{\epsilon\seq{\epsilon\epsilon}} \, x_{\seq{\epsilon\epsilon}}}_o \\[.5ex]
  \hspace*{5.7ex}
  [\If \, [[\mname{fst}_{\epsilon\seq{\epsilon\epsilon}} \, x_{\seq{\epsilon\epsilon}}] = \synbrack{F_o}] \, 
  \sembrack{\synbrack{F_o}}_o \\[.5ex]
  \hspace*{5.7ex}
  [\If \, [[\mname{snd}_{\epsilon\seq{\epsilon\epsilon}} \, x_{\seq{\epsilon\epsilon}}] = \synbrack{F_o}] \, 
  \sembrack{\synbrack{F_o}}_o \\[.5ex]
  \hspace*{5.7ex}
  \textbf{P}_o]]]]$.

\ei
The formula \[\Forall x_{\seq{\epsilon\epsilon}}
[[\mname{wff}^{o}_{o\epsilon} \,
    [\mname{fst}_{\epsilon\seq{\epsilon\epsilon}} \,
      x_{\seq{\epsilon\epsilon}}] \Andd \mname{wff}^{o}_{o\epsilon} \,
    [\mname{snd}_{\epsilon\seq{\epsilon\epsilon}} \,
      x_{\seq{\epsilon\epsilon}}]] \Implies [\textbf{Q}_o \Iff
    \textbf{R}_o]],\] called \mname{MathMeaning}, expresses the intended
\emph{mathematical meaning} of $O_A$.  We will show that
\mname{MathMeaning} is a theorem of {\pfsysuqe} via a series of
lemmas.

The first lemma asserts that the analog of the \mname{MathMeaning} for
$\mname{and}_{\epsilon\epsilon\epsilon}$ is a theorem of
  {\pfsysuqe}:

\begin{lem} \label{lem:sem-and}
$\proves{}{\Forall x_{\seq{\epsilon\epsilon}}
[[\mname{wff}^{o}_{o\epsilon} \,
  [\mname{fst}_{\epsilon\seq{\epsilon\epsilon}} \, x_{\seq{\epsilon\epsilon}}] \Andd
  \mname{wff}^{o}_{o\epsilon} \,      
  [\mname{snd}_{\epsilon\seq{\epsilon\epsilon}} \, x_{\seq{\epsilon\epsilon}}]] \Implies 
  [\textbf{P}_o \Iff \textbf{R}_o]]}$.
\end{lem}

\begin{proof}
\begin{align} \setcounter{equation}{0}
&\vdash 
[\mname{wff}^{o}_{o\epsilon} \,
[\mname{fst}_{\epsilon\seq{\epsilon\epsilon}} \, x_{\seq{\epsilon\epsilon}}] \Andd
\mname{wff}^{o}_{o\epsilon} \,      
[\mname{snd}_{\epsilon\seq{\epsilon\epsilon}} \, x_{\seq{\epsilon\epsilon}}]] \Implies
\textbf{P}_o \Iff \textbf{P}_o \\
&\vdash
[\mname{wff}^{o}_{o\epsilon} \,
[\mname{fst}_{\epsilon\seq{\epsilon\epsilon}} \, x_{\seq{\epsilon\epsilon}}] \Andd
\mname{wff}^{o}_{o\epsilon} \,      
[\mname{snd}_{\epsilon\seq{\epsilon\epsilon}} \, x_{\seq{\epsilon\epsilon}}]] \Implies \notag\\
&\hspace{2.7ex}
\textbf{P}_o \Iff 
\sembrack{\mname{app}_{\epsilon\epsilon\epsilon} \,
[\mname{app}_{\epsilon\epsilon\epsilon} \, \synbrack{\wedge_{ooo}} \,
[\mname{fst}_{\epsilon\seq{\epsilon\epsilon}} \, x_{\seq{\epsilon\epsilon}}]] \, 
[\mname{snd}_{\epsilon\seq{\epsilon\epsilon}} \, x_{\seq{\epsilon\epsilon}}]}_o. \\
&\vdash 
[\mname{wff}^{o}_{o\epsilon} \,
[\mname{fst}_{\epsilon\seq{\epsilon\epsilon}} \, x_{\seq{\epsilon\epsilon}}] \Andd
\mname{wff}^{o}_{o\epsilon} \,      
[\mname{snd}_{\epsilon\seq{\epsilon\epsilon}} \, x_{\seq{\epsilon\epsilon}}]] \Implies \notag\\
&\hspace{2.7ex}
\textbf{P}_o \Iff
\sembrack{[\mname{app}_{\epsilon\epsilon\epsilon} \, \synbrack{\wedge_{ooo}} \,
[\mname{fst}_{\epsilon\seq{\epsilon\epsilon}} \, x_{\seq{\epsilon\epsilon}}]]}_{oo}
\sembrack{\mname{snd}_{\epsilon\seq{\epsilon\epsilon}} \, x_{\seq{\epsilon\epsilon}}}_o. \\
&\vdash 
[\mname{wff}^{o}_{o\epsilon} \,
[\mname{fst}_{\epsilon\seq{\epsilon\epsilon}} \, x_{\seq{\epsilon\epsilon}}] \Andd
\mname{wff}^{o}_{o\epsilon} \,      
[\mname{snd}_{\epsilon\seq{\epsilon\epsilon}} \, x_{\seq{\epsilon\epsilon}}]] \Implies \notag\\
&\hspace{2.7ex}
\textbf{P}_o \Iff
\sembrack{\synbrack{\wedge_{ooo}}}_{ooo}
\sembrack{\mname{fst}_{\epsilon\seq{\epsilon\epsilon}} \, x_{\seq{\epsilon\epsilon}}}_o
\sembrack{\mname{snd}_{\epsilon\seq{\epsilon\epsilon}} \, x_{\seq{\epsilon\epsilon}}}_o. \\
&\vdash  
[\mname{wff}^{o}_{o\epsilon} \,
[\mname{fst}_{\epsilon\seq{\epsilon\epsilon}} \, x_{\seq{\epsilon\epsilon}}] \Andd
\mname{wff}^{o}_{o\epsilon} \,      
[\mname{snd}_{\epsilon\seq{\epsilon\epsilon}} \, x_{\seq{\epsilon\epsilon}}]] \Implies\notag\\
&\hspace{2.7ex}
\textbf{P}_o \Iff
\wedge_{ooo}
\sembrack{\mname{fst}_{\epsilon\seq{\epsilon\epsilon}} \, x_{\seq{\epsilon\epsilon}}}_o
\sembrack{\mname{snd}_{\epsilon\seq{\epsilon\epsilon}} \, x_{\seq{\epsilon\epsilon}}}_o. \\
&\vdash  
[\mname{wff}^{o}_{o\epsilon} \,
[\mname{fst}_{\epsilon\seq{\epsilon\epsilon}} \, x_{\seq{\epsilon\epsilon}}] \Andd
\mname{wff}^{o}_{o\epsilon} \,      
[\mname{snd}_{\epsilon\seq{\epsilon\epsilon}} \, x_{\seq{\epsilon\epsilon}}]] \Implies 
\textbf{P}_o \Iff \textbf{R}_o. \\
&\vdash 
\Forall x_{\seq{\epsilon\epsilon}}
[[\mname{wff}^{o}_{o\epsilon} \,
[\mname{fst}_{\epsilon\seq{\epsilon\epsilon}} \, x_{\seq{\epsilon\epsilon}}] \Andd
\mname{wff}^{o}_{o\epsilon} \,      
[\mname{snd}_{\epsilon\seq{\epsilon\epsilon}} \, x_{\seq{\epsilon\epsilon}}]] \Implies
[\textbf{P}_o \Iff \textbf{R}_o]].
\end{align}
(1) is by Lemmas~\ref{lem:equality}, Propositions~\ref{prop:5200}
and~\ref{prop:o-defined}, and the Tautology Theorem; (2) follows from
(1) by the definition of $\mname{and}_{\epsilon\epsilon\epsilon}$,
Axioms 4.2--5, and Lemma~\ref{lem:pairs}; (3) and (4) follow from (2)
and (3), respectively, by Axiom 11.3, Lemma~\ref{lem:pairs}, and
Specification 6; (5) follows from (4) by Axiom 11.2; (6) is by
abbreviation; and (7) is by Universal Generalization.
\end{proof}

\bigskip

The second lemma shows how $\textbf{Q}_o$ can be reduced:

\begin{lem} \label{lem:Q-reduction}
$\proves{}{\textbf{Q}_o \Iff \textbf{S}_o}$.
\end{lem}

\begin{proof}
The right side of the equation is obtained from the left side in three
steps.  First, $\mname{and-simp}_{\epsilon\epsilon\epsilon}$ is
replaced by its definition.  Second, the resulting formula is
beta-reduced using Axioms 4.2--5 and 4.8 and parts 2 and 3 of
Lemma~\ref{lem:pairs}.  And third, evaluations are pushed inward using
Axiom 10.3.
\end{proof}

\bigskip 

The next lemma consists of five theorems of {\pfsysuqe}:

\begin{lem} \label{lem:math-meaning}

\be

  \item[]

  \item $\vdash x_{\seq{\epsilon\epsilon}} = 
  [\mname{pair}_{\seq{oo}oo} \, \synbrack{T_o} \, \synbrack{T_o}] \Implies 
  [\textbf{Q}_o \Iff \textbf{R}_o]$.

  \item $\vdash x_{\seq{\epsilon\epsilon}} = 
  [\mname{pair}_{\seq{oo}oo} \, \synbrack{F_o} \, \synbrack{T_o}] \Implies 
  [\textbf{Q}_o \Iff \textbf{R}_o]$.

  \item $\vdash x_{\seq{\epsilon\epsilon}} = 
  [\mname{pair}_{\seq{oo}oo} \, \synbrack{T_o} \, \synbrack{F_o}] \Implies 
  [\textbf{Q}_o \Iff \textbf{R}_o]$.

  \item $\vdash x_{\seq{\epsilon\epsilon}} = 
  [\mname{pair}_{\seq{oo}oo} \, \synbrack{F_o} \, \synbrack{F_o}] \Implies 
  [\textbf{Q}_o \Iff \textbf{R}_o]$.

  \item $\vdash [{x_{\seq{\epsilon\epsilon}} \not= 
  [\mname{pair}_{\seq{oo}oo} \, \synbrack{T_o} \, \synbrack{T_o}]} \Andd
  {x_{\seq{\epsilon\epsilon}} \not= 
  [\mname{pair}_{\seq{oo}oo} \, \synbrack{F_o} \, \synbrack{T_o}]} \Andd {} \\[.5ex]
  \hspace*{2.7ex}
  {x_{\seq{\epsilon\epsilon}} \not= 
  [\mname{pair}_{\seq{oo}oo} \, \synbrack{T_o} \, \synbrack{F_o}]} \Andd
  x_{\seq{\epsilon\epsilon}} \not= 
  [\mname{pair}_{\seq{oo}oo} \, \synbrack{F_o} \, \synbrack{F_o}] \Andd {} \\[.5ex]
  \hspace*{2.1ex}
  [\mname{wff}^{o}_{o\epsilon} \,
  [\mname{fst}_{\epsilon\seq{\epsilon\epsilon}} \, x_{\seq{\epsilon\epsilon}}] \Andd
  \mname{wff}^{o}_{o\epsilon} \,      
  [\mname{snd}_{\epsilon\seq{\epsilon\epsilon}} \, x_{\seq{\epsilon\epsilon}}]]] \Implies 
  [\textbf{Q}_o \Iff \textbf{R}_o]$.

\ee
\end{lem}

\begin{proof}

\medskip

\noindent \textbf{Part 1} {\sglsp}
\begin{align} \setcounter{equation}{0}
&\vdash 
x_{\seq{\epsilon\epsilon}} =  [\mname{pair}_{\seq{oo}oo} \, \synbrack{T_o} \, \synbrack{T_o}] \Implies {} \notag\\
&\hspace*{2.5ex}
[[\lambda x_{\seq{\epsilon\epsilon}} [\textbf{Q}_o \Iff \textbf{R}_o]] x_{\seq{\epsilon\epsilon}} \Iff
[\lambda x_{\seq{\epsilon\epsilon}} [\textbf{Q}_o \Iff \textbf{R}_o]] 
[\mname{pair}_{\seq{oo}oo} \, \synbrack{T_o} \, \synbrack{T_o}]]. \\
&\vdash 
x_{\seq{\epsilon\epsilon}} =  [\mname{pair}_{\seq{oo}oo} \, \synbrack{T_o} \, \synbrack{T_o}] \Implies {} \notag\\
&\hspace*{2.5ex}
[[\textbf{Q}_o \Iff \textbf{R}_o] \Iff
[\lambda x_{\seq{\epsilon\epsilon}} [\textbf{Q}_o \Iff \textbf{R}_o]] 
[\mname{pair}_{\seq{oo}oo} \, \synbrack{T_o} \, \synbrack{T_o}]]. \\
&\vdash 
x_{\seq{\epsilon\epsilon}} =  [\mname{pair}_{\seq{oo}oo} \, \synbrack{T_o} \, \synbrack{T_o}] \Implies {} \notag\\
&\hspace*{2.5ex}
[[\textbf{Q}_o \Iff \textbf{R}_o] \Iff
[\lambda x_{\seq{\epsilon\epsilon}} [\textbf{S}_o \Iff \textbf{R}_o]] 
[\mname{pair}_{\seq{oo}oo} \, \synbrack{T_o} \, \synbrack{T_o}]]. \\
&\vdash 
x_{\seq{\epsilon\epsilon}} =  [\mname{pair}_{\seq{oo}oo} \, \synbrack{T_o} \, \synbrack{T_o}] \Implies {} \notag\\
&\hspace*{2.5ex}
[[\textbf{Q}_o \Iff \textbf{R}_o] \Iff
[\sembrack{\synbrack{T_o}}_o \Iff
\sembrack{\synbrack{T_o}}_o  \Andd \sembrack{\synbrack{T_o}}_o]]. \\
&\vdash 
x_{\seq{\epsilon\epsilon}} =  [\mname{pair}_{\seq{oo}oo} \, \synbrack{T_o} \, \synbrack{T_o}] \Implies {} \notag\\
&\hspace*{2.5ex}
[[\textbf{Q}_o \Iff \textbf{R}_o] \Iff T_o]. \\
&\vdash 
x_{\seq{\epsilon\epsilon}} =  [\mname{pair}_{\seq{oo}oo} \, \synbrack{T_o} \, \synbrack{T_o}] \Implies 
[\textbf{Q}_o \Iff \textbf{R}_o].
\end{align}
(1) is by Axiom 2; (2) follows from (1) by Axiom 4.10; (3) follows
from (2) by Lemma~\ref{lem:Q-reduction} and Rule~1; (4) follows from
(3) by Lemma~\ref{lem:pairs}, part 2 of the Beta-Reduction Theorem,
Axiom 10.1, and the Tautology Theorem; (5) follows (4) by Axiom 11.2
and the Tautology Theorem; (6) follows from (5) by the Tautology
Theorem.

\medskip

\noindent \textbf{Part 2} {\sglsp} Similar to Part 1.

\medskip

\noindent \textbf{Part 3} {\sglsp} Similar to Part 1.

\medskip

\noindent \textbf{Part 4} {\sglsp} Similar to Part 1.

\medskip

\noindent \textbf{Part 5} {\sglsp} Let $\textbf{A}_o$ be the
antecedent of the implication in part 5 of the lemma.
\begin{align} \setcounter{equation}{0}
&\vdash
\textbf{P}_o \Iff \textbf{R}_o \\
&\vdash
\textbf{A}_o \Implies [\textbf{P}_o \Iff \textbf{R}_o] \\
&\vdash
\textbf{A}_o \Implies [\textbf{S}_o \Iff \textbf{R}_o]  \\
&\vdash
\textbf{A}_o \Implies [\textbf{Q}_o \Iff \textbf{R}_o] 
\end{align}
(1) is by Lemma~\ref{lem:sem-and} and part 4 of Universal
Instantiation; (2) follows from (1) by the Tautology Theorem; (3)
follows from (2) by Axiom 10.2, Lemma~\ref{lem:pairs}, and the
Tautology Theorem; and (4) follows from (3) by
Lemma~\ref{lem:Q-reduction} and Rule~1.
\end{proof}

\bigskip

Finally, the theorem below shows that \mname{MathMeaning} is a theorem
of {\pfsysuqe}:

\begin{thm} [Mathematical Meaning of $\mname{and-simp}_{\epsilon\epsilon\epsilon}$]
\[\proves{}{\mname{MathMeaning}}.\]
\end{thm}

\begin{proof}
\[\proves{}{[\mname{wff}^{o}_{o\epsilon} \,
  [\mname{fst}_{\epsilon\seq{\epsilon\epsilon}} \,
    x_{\seq{\epsilon\epsilon}}] \Andd \mname{wff}^{o}_{o\epsilon} \,
  [\mname{snd}_{\epsilon\seq{\epsilon\epsilon}} \,
    x_{\seq{\epsilon\epsilon}}]] \Implies [\textbf{Q}_o \Iff
    \textbf{R}_o]}\] follows from Lemmas~\ref{lem:pairs} and
\ref{lem:math-meaning} and the Tautology Theorem.  Then
\[\proves{}{\Forall x_{\seq{\epsilon\epsilon}} 
[[\mname{wff}^{o}_{o\epsilon} \,
  [\mname{fst}_{\epsilon\seq{\epsilon\epsilon}} \, x_{\seq{\epsilon\epsilon}}] \Andd
  \mname{wff}^{o}_{o\epsilon} \,      
  [\mname{snd}_{\epsilon\seq{\epsilon\epsilon}} \, x_{\seq{\epsilon\epsilon}}]]] \Implies
[\textbf{Q}_o \Iff \textbf{R}_o]]}\] follows from this by Universal Generalization.
\end{proof}

\bigskip

\noindent 
Hence $O_A$ is mathematically correct. 

While $A$ manipulates formulas,
$\mname{and-simp}_{\epsilon\epsilon\epsilon}$ manipulates syntactic
representations of formulas.  An application of $O_A$ has the form
$\mname{and-simp}_{\epsilon\epsilon\epsilon} \,
\synbrack{\textbf{A}_o} \, \synbrack{\textbf{B}_o}$.  Its value can be
computed by expanding its definition, beta-reducing using Axiom 4, and
then rewriting the resulting wff using Axiom 10 and Specification~1.
If $\textbf{A}_o$ and $\textbf{B}_o$ are evaluation-free, its meaning
can be obtained by instantiating the universal formula
\mname{MathMeaning} with the wff
$\seq{\synbrack{\textbf{A}_o},\synbrack{\textbf{B}_o}}$ and then
simplifying.

\section{Conclusion} \label{sec:conclusion}

\subsection{Summary of Results}

We have presented a version of simple type theory called {\qzerouqe}
that admits undefined expressions, quotations, and evaluations.
{\qzerouqe} is based on {\qzero}, a version of Church's type
theory~\cite{Church40} developed by Peter~B.~Andrews~\cite{Andrews02}.
{\qzerouqe} directly formalizes the traditional approach to
undefinedness~\cite{Farmer04} in which undefined expressions are
treated as legitimate, nondenoting expressions that can be components
of meaningful statements.  It has the same facility for reasoning
about undefinedness as {\qzerou}~\cite{Farmer08a} that is derived from
{\qzero}.  In addition, it has a facility for reasoning about the
syntax of expressions based on quotation and evaluation.

The syntax of {\qzerouqe} differs from the syntax of {\qzero} by
having the following new machinery: a base type $\epsilon$ that
denotes a domain of syntactic values, a quotation operator, an
evaluation operator, and several constants involving the type
$\epsilon$.  {\qzerouqe} also has some additional new machinery for
ordered pairs and conditionals: a type constructor for forming types
that denote domains of ordered pairs, a constant for forming ordered
pairs, and an expression constructor for forming conditionals.  The
semantics of {\qzerouqe} is based on Henkin-style general
models~\cite{Henkin50} that include partial functions as well as total
functions and in which expressions may be undefined.  The expression
constructor for conditionals is nonstrict with respect to
undefinedness.  An application of the quotation operator to an
expression denotes a syntactic value that represents the expression.
An application of the evaluation operator to an expression $E$ denotes
the value of the expression represented by the value of $E$.  To avoid
the Evaluation Problem mentioned in the Introduction, an evaluation
$\sembrack{\synbrack{\textbf{A}_\epsilon}}_\alpha$ is undefined when
$\textbf{A}_\epsilon$ is not evaluation-free.

The syntax and semantics of {\qzerouqe} are modest modifications of
the syntax and semantics of {\qzerou}, but {\pfsysuqe}, the proof
system of {\qzerouqe}, is a major modification of {\pfsysu}, the proof
system of {\qzerou}.  The substitution operation that is needed to
perform beta-reduction is defined in the metalogic of {\qzerou}, while
it is represented in {\qzerouqe} by a primitive constant
$\mname{sub}_{\epsilon\epsilon\epsilon\epsilon}$.  To avoid the
Variable Problem mentioned in the Introduction,
$\mname{sub}_{\epsilon\epsilon\epsilon\epsilon}$ defines a
semantics-dependent form of substitution.  Moreover, the syntactic
side conditions concerning free variables and substitution that are
expressed in the metalogic of a traditional logic are expressed in the
language of {\qzerouqe}.  We prove that {\pfsysuqe} is sound with
respect to the semantics of {\qzerouqe} (Theorem~\ref{thm:soundness}),
but it is not complete.  However, it is complete for evaluation-free
formulas (Theorem~\ref{thm:completeness}).

{\qzerouqe} is not complete because it is not possible to beta-reduce
all applications of function abstraction.  There are two ways of
performing beta-reduction in {\qzerouqe}.  The first way uses the
specifying axioms of the primitive constant
$\mname{sub}_{\epsilon\epsilon\epsilon\epsilon}$ to perform
substitution as expressed by Axiom 4.1.  This first way works for all
applications of function abstraction involving just evaluation-free
wffs, but it works for only some applications involving evaluations.
The second way uses the basic properties of lambda-notation as
expressed by Axioms 4.2--10.  Like the first way, this second way
works for all applications of function abstraction involving just
evaluation-free wffs, but it works for only some applications
involving evaluations.  However, the two ways complement each other
because they work for different applications of function abstraction
involving evaluations.

\subsection{Significance of Results}

The construction of {\qzerouqe} demonstrates how the replete approach
to reasoning about syntax~\cite{Farmer13} --- in which it is possible
to reason about the syntax of the entire language of the logic using
quotation and evaluation operators defined in the logic --- can be
implemented in Church's type theory~\cite{Church40}.  Moreover, the
implementation ideas employed in {\qzerouqe} can be applied to other
traditional logics like first-order logic.  Even though the proof
system of {\qzerouqe} is not complete, it is powerful enough to be
useful.  We have illustrated how {\qzerouqe} can be used to (1)~reason
about the syntactic structure of expressions, (2) represent and
instantiate schemas with syntactic variables, and (3) formalize
syntax-based mathematical algorithms in the sense given
in~\cite{Farmer13}.  We believe {\qzerouqe} is the first
implementation of the replete approach in a traditional logic.

The most innovative and complex part of {\qzerouqe} is the
semantics-based form of substitution represented by the primitive
constant $\mname{sub}_{\epsilon\epsilon\epsilon\epsilon}$.  It
provides the means to instantiate both variables occurring in
evaluations and variables resulting from evaluations.  In particular,
it enables schemas expressed using evaluation (e.g., as given in
subsections~\ref{subsec:lem} and~\ref{subsec:beta-red}) to be
instantiated.  We showed that the substitution mechanism is correct by
proving the law of beta-reduction formulated using
$\mname{sub}_{\epsilon\epsilon\epsilon\epsilon}$
(Theorem~\ref{thm:beta-red-sub}).  The proof of this theorem is
intricate and involves many lemmas.

{\qzerouqe} is intended primarily for theoretical purposes; it is not
designed to be used in practice.  A more practical version of
{\qzerouqe} could be obtained by extending it in some of the ways
discussed in~\cite{Farmer08}.  For instance, {\qzerouqe} could be
extended to include type variables as in the logic of the {\hol}
theorem proving system~\cite{GordonMelham93} and its
successors~\cite{Harrison09,Lemma1Ltd00,Paulson94} or
subtypes as in the logic of the {\imps} theorem proving
system~\cite{FarmerEtAl93,FarmerEtAl96}.  These additions would
significantly raise the practical expressivity of the logic but would
further raise the complexity of the logic.  Many of these kinds of
practical measures are implemented together in the logic
Chiron~\cite{Farmer07a,Farmer07}, a derivative of
von-Neumann-Bernays-G\"odel ({\nbg}) set theory that admits undefined
expressions, has a rich type system, and is equipped with a facility
of reasoning about syntax that is very similar to {\qzerouqe}'s.

\subsection{Related Work}

\subsubsection*{Reasoning in Logic about Syntax}

Reasoning in a logic about syntax begins with Kurt G\"odel's famous
use of \emph{G\"odel numbers} in~\cite{Goedel31} to encode
expressions.  G\"odel, Tarski, and others used reasoning about syntax
to show some of the limits of formal logic by \emph{reflecting} the
metalogic of a logic into the logic itself. \emph{Reflection} is a
technique to embed reasoning about a reasoning system (i.e.,
metareasoning) in the reasoning system itself.  It very often involves
the syntactic manipulation of expressions.  Reflection has been
employed in logic both for theoretical purposes~\cite{Koellner09} and
practical purposes~\cite{Harrison95}.

The technique of \emph{deep embedding} is used to reason in a logic
about the syntax of a particular
language~\cite{BoultonEtAl93,ContejeanEtAl07,WildmoserNipkow04}.  This
is usually done with the local approach but could also be done with
the global approach.  A deep embedding can also provide a basis for
formalizing syntax-based mathematical algorithms.  Examples include
the ring tactic implemented in Coq~\cite{Coq8.4} and Wojciech
Jedynak's semiring solver in
Agda~\cite{Norell07,Norell09,VanDerWalt12}.

Florian Rabe proposes in~\cite{Rabe15} a method for freely adding
literals for the values in a given semantic domain.  This method can
be used for reasoning about syntax by choosing a language of
expressions as the semantic domain.  Rabe's approach provides a
quotation operation that is more general than the quotation operation
we have defined for {\qzerouqe}.  However, his approach does not
provide an escape from obstacles like the Evaluation Problem and the
Variable Problem described in section 1.

\subsubsection*{Reasoning in the Lambda Calculus about Syntax}

Corrado B\"ohm and Alessandro Berarducci present
in~\cite{BoehmBerarducci85} a method for representing an inductive
type of values as a collection of lambda-terms.  Then functions
defined on the members of the inductive type can also be represented
as lambda terms.  Both the lambda terms representing the values and
those representing the functions defined on the values can be typed in
the second-order lambda calculus (System F)~\cite{Girard72,Reynolds74}
as shown in~\cite{BoehmBerarducci85}.  C. B\"ohm and his collaborators
present in~\cite{BerarducciBoehm93,BoehmEtAl94} a second, more
powerful method for representing inductive types as collections of
lambda-terms in which the lambda terms are not as easily typeable as
in the first method.  These two methods provide the means to
efficiently formalize syntax-based mathematical algorithms in the
lambda calculus.

Using the fact that inductive types can be directly represented in the
lambda calculus, Torben \AE. Mogensen in~\cite{Mogensen94} represents
the inductive type of lambda terms in lambda calculus itself as well
as defines an evaluation operator in the lambda calculus.  He thus
shows that the replete approach to reasoning about syntax, minus the
presence of a built-in quotation operator, can be realized in the
lambda calculus.  (See Henk Barendregt's survey
paper~\cite{Barendregt97} on the impact of the lambda calculus for a
nice description of this work.)

\subsubsection*{Metaprogramming}


\emph{Metaprogramming} is writing computer programs to manipulate and
generate computer programs in some programming language $L$.
Metaprogramming is especially useful when the ``metaprograms'' can be
written in $L$ itself.  This is facilitated by implementing in $L$
metareasoning techniques for $L$ that involve the manipulation of
program code.  See~\cite{DemersMalenfant95} for a survey of how this
kind of ``reflection'' can be done for the major programming
paradigms.  The programming languages we mentioned in the Introduction
support metaprogramming: Lisp, Agda~\cite{Norell07,Norell09},
Elixir~\cite{Elixir15}, F\#~\cite{FSharp15},
MetaML~\cite{TahaSheard00}, MetaOCaml~\cite{MetaOCaml11},
reFLect~\cite{GrundyEtAl06}, and Template
Haskell~\cite{SheardJones02}.  These languages represent fragments of
computer code as values in an inductive type and include quotation,
quasiquotation, and evaluation operations.  For example, these
operations are called \emph{quote}, \emph{backquote}, and \emph{eval}
in the Lisp programming language.  Thus metaprogramming languages
take, more or less, the replete approach to reasoning about the syntax
of programs.  The metaprogramming language Archon~\cite{Stump09}
developed by Aaron Stump offers an interesting alternate approach in
which program code is manipulated directly instead of manipulating
representations of computer code.

\subsubsection*{Theories of Truth}

Truth is a major subject in philosophy~\cite{Truth13}.  A theory of
truth seeks to explain what truth is and how the liar and other
related paradoxes can be resolved.  A \emph{semantics theory of truth}
defines a truth predicate for a formal language, while an
\emph{axiomatic theory of truth}~\cite{Halbach11,AxiomThyTruth13}
specifies a truth predicate for a formal language by means of an
axiomatic theory.  We have mentioned in Note~\ref{note:truth} that an
evaluation of the form $\sembrack{\textbf{A}_\epsilon}_o$ is a truth
predicate on $\wffs{\epsilon}$ $\textbf{A}_\epsilon$ that represent
formulas.  Thus {\qzerouqe} provides a semantic theory of truth via it
semantics and an axiomatic theory of truth via its proof system
{\pfsysuqe}.

Since our goal is not to explicate the nature of truth, it is not
surprising that the semantic and axiomatic theories of truth provided
by {\qzerouqe} are not very innovative.  Theories of truth ---
starting with Tarski's work~\cite{Tarski33,Tarski35,Tarski35a} in the
1930s --- have traditionally been restricted to the truth of
sentences, i.e., formulas with no free variables.  However, the
{\qzerouqe} semantic and axiomatic theories of truth admit formulas
with free variables.

\subsection{Future Work}

2Our future research will seek to answer the following questions:

\be

  \item Can nontrivial syntax-based mathematical algorithms --- such
    as those that compute derivatives symbolically --- be formalized
    in {\qzerouqe} in the sense given in~\cite{Farmer13}?

  \item Can a logic equipped with the machinery of {\qzerouqe} for
    reasoning about undefinedness and syntax be effectively
    implemented as a software system?

  \item Can the replete approach to reasoning about syntax serve as a
    basis to integrate axiomatic and algorithmic mathematics?

\ee
We will discuss each of these research questions in turn.

\subsubsection*{Formalizing Syntax-Based Mathematical Algorithms}

We conjecture that it is possible to formalize nontrivial syntax-based
mathematical algorithms in {\qzerouqe} in the sense given
in~\cite{Farmer13}.  We intend to work out the details for the
well-known algorithm for the symbolic differentiation of polynomials
as described in~\cite{Farmer13}.  First, we will define a theory $R$
of the real numbers in {\qzerouqe}.  Second, we will define in $R$ the
basic ideas of calculus including the notions of a derivative and a
polynomial.  Third, we will define a constant in $R$ that represents
the symbolic differentiation algorithm for polynomials.  Fourth, we
will specify in $R$ the intended computational behavior of the
algorithm and prove that the constant satisfies that specification.
Fifth, we will specify in $R$ the intended mathematical meaning of the
algorithm and prove that the constant satisfies that specification.
And, finally, we will show how the constant can be used to compute
derivatives of polynomial functions in $R$.

Polynomial functions are total (i.e., they are defined at all points
on the real line) and their derivatives are also total.  Hence no
issues of definedness arise in the specification of the mathematical
meaning of the differentiation algorithm for polynomials.  However,
functions more general than polynomial functions as well as their
derivatives may be undefined at some points.  This means that
specifying the mathematical meaning of a symbolic differentiation
algorithm for more general functions will require using the
undefinedness facility of {\qzerouqe}.

\subsubsection*{Implementation of the {\qzerouqe} Machinery}

It remains an open question whether a logic like {\qzerouqe} can be
effectively implemented as a computer program.  The undefinedness
component of {\qzerouqe} has been implemented in the {\imps} theorem
proving system~\cite{FarmerEtAl93,FarmerEtAl96} which has been
successfully used to prove hundreds of theorems in traditional
mathematics, especially in mathematical analysis.  However, quotation
and evaluation would add another level of complexity to a theorem
proving system like {\imps} that can deal directly with undefinedness.

There are three approaches for implementing the syntax reasoning
machinery of {\qzerouqe}.  The first is to implement part of the
machinery of {\qzerouqe} in an existing implementation of Church's
type theory such as John Harrison's HOL Light~\cite{Harrison09} in
order to conduct experiments concerning reasoning about syntax.  For
example, a worthy experiment would be to formalize a syntax-based
mathematical algorithm like the symbolic differentiation algorithm for
polynomials mentioned above.  The second is to directly implement
{\qzerouqe} --- a version of {\qzerouqe} with perhaps some practical
additions --- to test the entire design of {\qzerouqe}.  And the third
is to implement {\qzerouqe}'s syntax reasoning machinery as part of
the implementation of a general purpose logic for mechanized
mathematics.  We have engineered Chiron~\cite{Farmer07a,Farmer07} to
be just such as logic.  It contains essentially the same syntax
reasoning machinery as {\qzerouqe}, and we have a rudimentary
implementation of it~\cite{CaretteEtAl11}.

Implementing the ideas in {\qzerouqe} will be challenging.  Reasoning
about the interplay of syntax and semantics in a logic instead of a
metalogic is tricky.  It is easy for both developers and users to
become confused --- just ask any Lisp programmer.  A practical proof
system will require new axioms and rules of inference as well as an
effective means to perform substitution in the presence of
evaluations.  The latter, as we have seen, is fraught with
difficulties.  Finally, new notation and user-interface techniques are
needed to shield the user, as much as possible, from low-level
syntactic manipulations.

\subsubsection*{Integration of Axiomatic and Algorithmic Mathematics}

The MathScheme project~\cite{CaretteEtAl11}, led by Jacques Carette
and the author, is a long-term project being pursued at McMaster
University with the aim of producing a framework in which formal
deduction and symbolic computation are tightly integrated.  A key part
of the framework is the notion of a \emph{biform
  theory}~\cite{CaretteFarmer08,Farmer07b} that is a combination of an
axiomatic theory and an algorithm theory.  A biform theory is a basic
unit of mathematical knowledge that consists of a set of
\emph{concepts} that denote mathematical values, \emph{transformers}
that denote syntax-based algorithms, and \emph{facts} about the
concepts and transformers.  Since transformers manipulate the syntax
of expressions, biform theories are difficult to formalize in a
traditional logic.  One of the main goals of the MathScheme project is
to see if a logic like {\qzerouqe} that implements the replete
approach to syntax reasoning can be used develop a library of biform
theories.

\addcontentsline{toc}{section}{Acknowledgments}
\section*{Acknowledgments} 

The author is grateful to Marc Bender, Jacques Carette, Michael
Kohlhase, Pouya Larjani, and Florian Rabe for many valuable
discussions on the use of quotation and evaluation in logic.  Peter
Andrews deserves special thanks for writing \emph{An Introduction to
  Mathematical Logic and Type Theory: To Truth through
  Proof}~\cite{Andrews02}.  The ideas embodied in {\qzerouqe} heavily
depend on the presentation of {\qzero} given in this superb textbook.

\addcontentsline{toc}{section}{References}
\bibliography{$HOME/research/lib/imps}
\bibliographystyle{plain}

\end{document}

%% file: stt-with-uqe-def.tex
\newcommand{\be}{\begin{enumerate}}
\newcommand{\ee}{\end{enumerate}}
\newcommand{\bi}{\begin{itemize}}
\newcommand{\ei}{\end{itemize}}
\newcommand{\bc}{\begin{center}}
\newcommand{\ec}{\end{center}}
\newcommand{\bsp}{\begin{sloppypar}}
\newcommand{\esp}{\end{sloppypar}}

\newtheorem{thm}{Theorem}[subsection]
\newtheorem{cor}[thm]{Corollary}
\newtheorem{lem}[thm]{Lemma}
\newtheorem{prop}[thm]{Proposition}

\newtheorem{note}{Note}

\newenvironment{proof}{\par\noindent{\bf Proof\sglsp}}{\hfill$\Box$}

\newcommand{\sglsp}{\ }
\newcommand{\dblsp}{\ \ }

\newcommand{\sC}{\mbox{$\cal C$}}
\newcommand{\sD}{\mbox{$\cal D$}}
\newcommand{\sE}{\mbox{$\cal E$}}
\newcommand{\sG}{\mbox{$\cal G$}}
\newcommand{\sH}{\mbox{$\cal H$}}
\newcommand{\sJ}{\mbox{$\cal J$}}
\newcommand{\sL}{\mbox{$\cal L$}}
\newcommand{\sM}{\mbox{$\cal M$}}
\newcommand{\sP}{\mbox{$\cal P$}}
\newcommand{\sS}{\mbox{$\cal S$}}
\newcommand{\sT}{\mbox{$\cal T$}}
\newcommand{\sV}{\mbox{$\cal V$}}
\newcommand{\sW}{\mbox{$\cal W$}}

\renewcommand{\phi}{\varphi}

\newcommand{\iotaAlt}{\mbox{\it \i}}
\newcommand{\iotaAltS}{\mbox{{\scriptsize \it \i}}}
\newcommand{\qzero}{${\cal Q}_0$}
\newcommand{\qzerou}{${\cal Q}^{\rm u}_{0}$}

\newcommand{\qzerouqeplus}{$\overline{{\cal Q}^{\rm uqe}_{0}}$}
\newcommand{\qzerouqe}{${\cal Q}^{\rm uqe}_{0}$}
\newcommand{\pfsys}{${\cal P}$}
\newcommand{\pfsysu}{${\cal P}^{\rm u}$}

\newcommand{\pfsysuqeplus}{$\overline{{\cal P}^{\rm uqe}}$}
\newcommand{\pfsysuqe}{${\cal P}^{\rm uqe}$}
\newcommand{\sub}{\d{\sf S}}

\newcommand{\seq}[1]{{\langle #1 \rangle}}
\newcommand{\set}[1]{{\{ #1 \}}}
\newcommand{\sembrack}[1]{\llbracket#1\rrbracket}
\newcommand{\synbrack}[1]{\ulcorner#1\urcorner}
\newcommand{\commabrack}[1]{\lfloor#1\rfloor}
\newcommand{\mname}[1]{\mbox{\sf #1}}

\newcommand{\tarrow}{\rightarrow}
\newcommand{\NegAlt}{{\sim}}
\newcommand{\Andd}{\wedge}
\newcommand{\Implies}{\supset}
\newcommand{\Or}{\vee}
\newcommand{\Iff}{\equiv}
\newcommand{\Forall}{\forall}

\newcommand{\Forsome}{\exists}

\newcommand{\Iota}{\mbox{\rm I}}
\newcommand{\IsDef}{\downarrow}
\newcommand{\IsUndef}{\uparrow}

\newcommand{\QuasiEqual}{\simeq}
\newcommand{\Undefined}{\bot}
\newcommand{\If}{\mname{if}}
\newcommand{\IsDefApp}{\!\IsDef}
\newcommand{\IsUndefApp}{\!\IsUndef}
\newcommand{\invertediota}{\rotatebox[origin=c]{180}{$\iota$}}

\newcommand{\imps}{\mbox{\sc imps}}

\newcommand{\hol}{\mbox{\sc hol}}

\newcommand{\lutins}{\mbox{\sc lutins}}

\newcommand{\modelsa}{\mathrel{\models^{\rm ef}\!}}

\newcommand{\modelsn}{\mathrel{\models_{\rm n}}}
\newcommand{\modelsna}{\mathrel{\models^{\rm ef\!}_{\rm n}}}

\newcommand{\proves}[2]{#1 \vdash #2}
\newcommand{\provesa}[2]{#1 \mathrel{\vdash^{\rm ef}\!} #2}
\newcommand{\provesb}[2]{#1 \mathrel{\overline{\vdash^{\rm ef}}\!} #2}
\newcommand{\wff}[1]{{\rm wff}_{#1}}
\newcommand{\xwff}[1]{{\rm xwff}_{#1}}
\newcommand{\wffs}[1]{{\rm wffs}_{#1}}
\newcommand{\xwffs}[1]{{\rm xwffs}_{#1}}

\newcommand{\nbg}{\mbox{\sc nbg}}